\documentclass{article}
\usepackage{a4wide}

\usepackage{amssymb,amsmath,amsthm}
\usepackage{newpxtext,newpxmath} 
\usepackage{mathtools}
\usepackage{multicol}
\usepackage[normalem]{ulem}

\usepackage{tikz}
\usetikzlibrary{arrows.meta}
\usepackage{mathrsfs,enumerate}

\usepackage[bookmarks=false]{hyperref}

\usepackage{color}

\definecolor{refkey}{gray}{.3}
  \definecolor{labelkey}{gray}{.3}

\numberwithin{equation}{section}  
\newtheorem{theorem}{Theorem}[section]  
\newtheorem{lemma}[theorem]{Lemma}
\newtheorem{corollary}[theorem]{Corollary}
\newtheorem{proposition}[theorem]{Proposition}

\newtheorem{remark}[theorem]{Remark}

\newtheorem{definition}[theorem]{Definition}
\newtheorem{problem}[theorem]{Problem}

\theoremstyle{definition}

\renewcommand{\otimes}{\times}

\newcommand{\N}{\mathbb{N}}

\renewcommand{\P}{\mathbb{P}}

\newcommand{\R}{\mathbb{R}}
\renewcommand{\S}{\mathbb{S}}
\newcommand{\T}{\mathbb{T}}

\newcommand{\cB}{{\ensuremath{\mathcal B}}}

\newcommand{\cG}{{\ensuremath{\mathcal G}}}

\newcommand{\cO}{{\ensuremath{\mathcal O}}}

\newcommand{\cS}{{\ensuremath{\mathcal S}}}

\newcommand{\PP}{\mathscr{P}}

\newcommand{\HH}{\mathscr{H}}

\newcommand{\XX}{\mathscr{X}}

\renewcommand{\vv}{{\boldsymbol v}}

\newcommand{\xx}{{\boldsymbol x}}

\newcommand{\xX}{{\boldsymbol X}}

\newcommand{\aalpha}{{\boldsymbol \alpha}}
\newcommand{\bbeta}{{\boldsymbol \beta}}

\newcommand{\mmu}{{\boldsymbol \mu}}

\newcommand{\sfa}{{\mathsf a}}
\newcommand{\sfb}{{\mathsf b}}
\newcommand{\sfc}{{\mathsf c}}
\newcommand{\sfd}{{\mathsf d}}
\newcommand{\sfe}{{\mathsf e}}

\newcommand{\sfs}{{\mathsf s}}

\newcommand{\sfu}{{\mathsf u}}

\newcommand{\sfx}{{\mathsf x}}
\newcommand{\sfy}{{\mathsf y}}

\newcommand{\sfA}{{\mathsf A}}
\newcommand{\sfB}{{\mathsf B}}

\newcommand{\sfS}{{\mathsf S}}

\newcommand{\rmc}{{\mathrm c}}

\newcommand{\rme}{{\mathrm e}}

\newcommand{\rmA}{{\mathrm A}}

\newcommand{\rmD}{{\mathrm D}}

\newcommand{\rmI}{{\mathrm I}}

\newcommand{\rmX}{{\mathrm X}}

\newcommand{\Kliminf}{K\kern-3pt-\kern-2pt\mathop{\rm lim\,inf}\limits}  
\newcommand{\supp}{\mathop{\rm supp}\nolimits}   
\newcommand{\argmin}{\mathop{\rm argmin}\limits}   
 
\newcommand{\restr}[1]{\lower3pt\hbox{$|_{#1}$}}
\newcommand{\Restriction}[1]{\lower3pt\hbox{$|_{#1}$}}  
\newcommand{\Leb}[1]{{\mathscr L}^{#1}}      

\newcommand{\down}{\downarrow}              

\newcommand{\eps}{\varepsilon}  

\newcommand{\Pc}[2]{\overline{#1}\kern-2pt^{\vphantom 0}_{#2}}
\newcommand{\Pcshift}[3]{\overline{#1}\kern-1pt^{#3}_{#2}}
\newcommand{\Pch}[2]{\overline{#1}^{\kern1pt h}_{\kern-2pt#2}}
\newcommand{\Pck}[2]{\overline{#1}^{\kern1pt k}_{\kern-2pt#2}}
\newcommand{\Pcinfty}[2]{\overline{#1}^{\kern1pt \infty}_{\kern-2pt#2}}
\newcommand{\Pchn}[2]{\overline{#1}^{\kern1pt h_n}_{\kern-2pt#2}}
\newcommand{\Pchnk}[2]{\overline{#1}^{\kern1pt h'_{n_k}}_{\kern-2pt#2}}
\newcommand{\Pchk}[2]{\overline{#1}^{\kern1pt h_{k}}_{\kern-2pt#2}}

\newcommand{\Pl}[2]{{#1}\kern-2pt^{\vphantom 0}_{#2}}
\newcommand{\Pcb}[2]{\underline{\phantom{u}}\kern-6pt #1_{#2}}

\newcommand{\nchi}{{\raise.3ex\hbox{$\chi$}}}

\newcommand{\GeoCon}[3]{\mathrm{Geo}_{#1}[#2\kern-1pt\to\kern-1pt#3]}
\newcommand{\GeoConPhi}[4]{\mathrm{Geo}^{\phi,#2}_{#1}[#3\kern-1pt\to\kern-1pt#4]}
\newcommand{\GeoConArg}[5]{\mathrm{Geo}^{#5,#2}_{#1}[#3\kern-1pt\to\kern-1pt#4]}

\newcommand{\GDom}[1]{[0,1]}

\renewcommand{\d}{{\mathrm d}}
\newcommand{\dt}{{\d t}}

\newcommand{\lambdam}{{\lambda\kern-2pt^-}}

\newcommand{\Urd}{{\frac{ \d}{\d t}\!\!}^+}

\newcommand{\Lrd}{{\frac{ \d}{\dt}\!}_+}

\newcommand{\Uld}{{\frac{ \d}{\d t}\!\!}^-}

\newcommand{\Lld}{{\frac{ \d}{\dt}\!}_-}

\newcommand{\Dom}[1]{\ensuremath{\mathrm{Dom}(#1)}}

\newcommand{\nc}{\normalcolor}

\newcommand{\XXX}{{\rmX}}

\newcommand{\DDD}{{\rmD}}
\newcommand{\dX}{\sfd_\XXX}

\newcommand{\lipmp}[1]{{
\mathfrak {lip}\kern-1pt^\mp\kern-1pt#1}}
\newcommand{\lipm}[1]{{
\mathfrak {lip}\kern-1pt^-\kern-1pt#1}}
\newcommand{\lipp}[1]{{
\mathfrak {lip}\kern-1pt^+\kern-1pt#1}}

\newcommand{\sku}[2]{{\sk_1(#1,#2)}}
\newcommand{\skuname}{{\sk_1}}
\newcommand{\skd}[2]{{\skdname(#1,#2)}}
\newcommand{\bskd}[2]{{\bskdname(#1,#2)}}
\newcommand{\skdname}{{\sk}} 
\newcommand{\bskdname}{{\sk}} 

\newcommand{\blskdname}{{\bskdname^{\cvx}}}

\newcommand{\bskuname}{{\sk_1}}
\newcommand{\blskuname}{{\sk_1^{\cvx}}}
\newcommand{\blskd}[2]{{\blskdname(#1,#2)}}

\newcommand{\blsku}[2]{{\blskuname(#1,#2)}}

\newcommand{\bsku}[2]{{\bskuname(#1,#2)}}
\newcommand{\hatsk}{\skew{-6}{\hat}{\sk}}
\newcommand{\cbskdname}{{\hatsk}}
\newcommand{\cblskdname}{{\hatsk^{\cvx}}}
\newcommand{\cblskd}[2]{{\cblskdname(#1,#2)}}

\newcommand{\cblskuname}{{\hatsk_1^{\cvx}}}
\newcommand{\cblsku}[2]{{\cblskuname(#1,#2)}}

\newcommand{\sk}{\sfb}

\newcommand{\lskdname}{{\sk^{\cvx}}}

\newcommand{\lskuname}{{\sk_1^{\cvx}}}

\newcommand{\lskd}[2]{{\lskdname(#1,#2)}}

\newcommand{\lsku}[2]{{\lskuname(#1,#2)}}

\newcommand{\ol}[1]{\overline{#1}}
\newcommand{\ul}[1]{\underline{#1}}

\newcommand{\sEVI}{$\sfs$-$\mathrm{EVI}$}

\newcommand{\urmder}[2]{|\dot{#1}^+|(#2)}
\newcommand{\ulmder}[2]{|\dot{#1}^-|(#2)}
\newcommand{\lrmder}[2]{|\dot{#1}_+|(#2)}
\newcommand{\llmder}[2]{|\dot{#1}_-|(#2)}
\newcommand{\mder}[2]{|\dot{#1}|(#2)}
\newcommand{\skslope}[2]{\Delta^{#1}(#2)}
\newcommand{\slope}[1]{\Delta(#1)}
\newcommand{\domainslope}[1]{\DDD(\Delta^{#1})}
\newcommand{\bif}{\mathsf b}

\usepackage{tikz}
\usetikzlibrary{calc}

\def\<#1,#2>{\left\langle #1,\,#2\right\rangle}
\DeclarePairedDelimiter\norm{\lVert}{\rVert}
\newcommand{\Wass}{\mathscr W}
\newcommand{\ddt}{\frac{\d}{\d t}}
\newcommand{\dds}{\frac{\d}{\d s}}

\newcommand{\Id}{\text{Id}}
\usepackage{longtable}
\usepackage{amssymb}
\usepackage{bm}
\usepackage{ifthen}
\renewcommand{\div}[1]{\mathrm{div}\left( #1 \right)}
\newcommand{\cvx}{\kappa} 
\newcommand{\dXi}{\sfd_{\XXX^{(i)}}}
\newcommand{\simplex}[1]{\bm{\bigtriangleup}_{#1}}

\newcommand{\bary}[2]{\boldsymbol{\mathsf #1}(#2)}
\newcommand{\baryname}[1]{\boldsymbol{\mathsf #1}}
\newcommand{\baryc}[2]{\baryname{#1}_{#2}}
\newcommand{\barysystem}[1]{\cB_{#1}}
\newcommand{\baryeval}[1]{\mathsf E[#1]}
\newcommand{\baryevalname}{\mathsf E}

\newif\ifincludeallexamples 
\includeallexamplesfalse

\usepackage{thmtools} 
\declaretheorem[name=Example]{example}
\makeatletter
\thmt@define@thmuse@key{tag}{%
  \thmt@suspendcounter{\thmt@envname}{#1}%
}
\makeatother

\begin{document}
 \title{Dissipative Evolutions in Metric Spaces}
 \author{Lauren Conger\thanks{Department of Mathematics, Stanford University, Stanford, USA. \texttt{lconger@stanford.edu}}, Franca Hoffmann\thanks{Computing and Mathematical Sciences, California Institute of Technology, Pasadena, USA. \texttt{franca.hoffmann@caltech.edu}}, Giuseppe Savar\'e\thanks{Department of Decision Sciences and BIDSA, Bocconi University, Milano, Italy. \texttt{giuseppe.savare@unibocconi.it}}}
\date{}
\maketitle
\begin{abstract}
We study a novel class of 
evolution variational inequalities (EVI) driven by 
bifunctions $\bif:\DDD \times \DDD \to \R$ 
defined on a subset $\DDD$ of a metric space $(\XXX,\dX)$ and satisfying a
natural dissipativity condition
$$\bif(x,y)+\bif(y,x)\le \eta \dX^2(x,y).$$
We provide general conditions for existence,  
stability, regularity, approximation 
and asymptotic behavior of solutions, 
 by proposing a metric framework
that generalizes the classical 
structure of evolutions
driven by monotone operators in Hilbert spaces.

 A motivating application is the setting of minimax 
and multispecies coupled gradient flows, in which each of $N$
species evolves in the direction of steepest descent of its own energy functional and the joint system is not  a gradient flow of any single energy.

To construct solutions, we introduce a 
variational movement scheme (VMS), a time discretization scheme in which each update is a saddle point of a metrically  
penalized bifunction, generalizing the classical minimizing movement/ JKO 
 scheme for single species gradient descent. In the coupled  multispecies case, the VMS reduces to a Nash equilibrium problem. We prove existence of discrete solutions 
to the VMS via 
a general approach 
which combines dissipativity of $\bif$
with a new notion of 
convexity along suitable 
\textit{barycentric interpolations} 
(which is inspired by 
convexity along generalized geodesics, a crucial notion for gradient flows).

Together, these conditions provide a zeroth-order notion of game-theoretic monotonicity applicable to general metric spaces without linear structure.
\end{abstract}

\tableofcontents

\section{Introduction}\label{sec:intro}
The theory of gradient flows on a metric space $(\XXX,\dX)$
  driven by an energy $\varphi:\DDD\subset \XXX\to \R$
  with complete sublevels is by now classical.
  Among the various formulations, the strongest one involves a system
  of
  Evolution Variational Inequalities (EVI): 
  for a locally Lipschitz curve $\sfx: [ 0,\infty)\to {\DDD}$ we ask that 
  \begin{equation}
  \label{eq:GF-EVI}
    \frac12\ddt\dX^2(y,\sfx(t))
   + \frac{\cvx}{2}\dX^2(y,\sfx(t))
   \le \varphi(y) - \varphi(\sfx(t)),
   \qquad \text{for a.e.~}\,t>0\text{ and for all } y\in\DDD,
\end{equation}
where $\cvx\in \R$ is a parameter which is tailored to 
  the second order derivative of $\varphi$ along geodesics.

  Sufficient conditions to solve 
  \eqref{eq:GF-EVI} 
  can be expressed through the generating function
    \begin{equation}
        \label{eq:PHI}
        \Phi(\tau,\bar x;x):=
        \frac 1{2\tau}
        \dX^2(\bar x,x)+\varphi(x),\qquad
        \tau>0,\ \bar x,x\in \DDD,
    \end{equation}
    which describes the 
    celebrated  JKO (or Minimizing Movement) scheme
\cite{JordanKinderlehrerOtto98,AGS08}
 with step size $\tau>0$,
\begin{equation}
\label{eq:JKO}
X_\tau^n \in \arg\min_{x\in\DDD}
\Phi(\tau,X_{\tau}^{n-1};x)
 \quad n=1,2,\cdots;\quad 
X_\tau^0=x_0\in \DDD \text{ given.}
\end{equation}
If 
\begin{equation}
    \label{eq:to-be-quoted}
    \begin{gathered}
        \text{every pair $x_1,x_2\in \DDD$
    can be connected by a curve in $\DDD$ 
    along which}\\
    \text{the functionals 
    $x\mapsto \Phi(\tau,\bar x;x)$
    is $(\tau^{-1}+\cvx)$-convex}
    \end{gathered}
\end{equation}    
  then the minimizing movement (MM) scheme 
  \eqref{eq:JKO} is solvable and 
  its solutions converge to a 
  locally Lipschitz curve solving 
  \eqref{eq:GF-EVI} as the time step $\tau$ goes to $0$.
  EVI and JKO have become canonical machinery for evolutions
driven by a single energy \cite{AGS08};
 condition \eqref{eq:to-be-quoted}
provides a purely metric criterion encompassing several relevant cases, including NPC spaces and functionals that are convex along generalized geodesics in Wasserstein spaces. 
\medskip

{Many problems of interest, however, are not driven by a
 single energy.
  Saddle-point
evolutions, coupled multispecies systems, and Nash-type games are
described
in product spaces $\XXX=\prod \XXX^{(i)}$ 
by an antisymmetric Lagrangian
$F(y^{(1)},x^{(2)})-F(x^{(1)},y^{(2)})$ or by $N$ energies
$(F^{(i)})$; they give rise to a 
\textbf{bifunction} $\bif:\DDD\times\DDD\to\R$, which reduces to the
gradient flow case above when
$\bif(y,x)=\varphi(y)-\varphi(x)$.}

\medskip

{Taking also inspiration by the theory of contraction semigroups generated by monotone operators, this paper aims to extend the metric EVI / MM scheme machinery to
bifunction-driven evolutions, where the bifunction $\bif$ takes the
place of $\varphi$ and 
a new variational movement scheme
takes the place
of \eqref{eq:JKO}. 
We propose a new framework using a natural extension of \eqref{eq:GF-EVI} that 
 tries to capture dissipativity at a metric level without any underlying Hilbertian structure,  
both providing existence of solutions and estimates for long-time behavior.

To that end, given 
a driving bifunction 
$\bif:\DDD \times \DDD \to \R$, 
we look for curves $\sfx:[0,\infty)\to \DDD$ that satisfy the
EVI
\begin{align}\label{introeq:EVI}
    \ddt \frac12 \dX^2(y,\sfx(t)) + \frac \cvx 2 \dX^2(y,\sfx(t))  \le \skd y{\sfx(t)}  \,, \quad \text{a.e.~in } (0,\infty), 
    \text{ for all } y\in\DDD\,. \tag{EVI}
\end{align}
The value $\cvx\in\R$  depends on the convexity 
(or lack of convexity) of $\skd \cdot x$. 

The aim of this work is to provide general conditions under which 
for every initial datum $x_0\in \DDD$ there exists some solution to 
\eqref{introeq:EVI} 
starting from $x_0$, 
determine its properties, and characterize its asymptotic behavior. {As a byproduct, we also obtain conditions under which equilibria to \eqref{introeq:EVI} exist.}

\paragraph{The driving bifunctions.}
Several familiar settings fit into \eqref{introeq:EVI} through
specific choices of $\bif$. 
We single out four choices of bifunctions $\bif$
which recur as running examples
throughout the paper. These
range from the classical monotone-operator evolution (I), the only one
requiring a linear Hilbert structure and in which dissipativity is of
pure interaction type, through the convexity-driven gradient and min-max
flows (II)--(III), where $\bif$ is antisymmetric, to the coupled
multispecies systems (IV), in which the two mechanisms act together.
\smallskip

\textbf{(I) Dissipative evolutions in Hilbert spaces.} On a Hilbert space $\XXX=H$,
$\dX(x,y)=\norm{x-y}$, a monotone operator $\sfA:\DDD\to H$ in a convex set $\DDD$ drives
\begin{equation}
    \label{eq:monotone}
   \dot\sfx=-\sfA(\sfx)\,,\qquad
   \skd yx=\<\sfA(x),y-x>\,,\qquad\text{for all } x,y\in\DDD .
\end{equation} 
\nc 

\textbf{(II) Single-species gradient flow.} 
We have already seen that for an energy
$\varphi:\DDD\to\R$, setting
\begin{equation}
    \label{eq:gradient}
   \bif(y,x) = \varphi(y) - \varphi(x), \qquad \text{for all } x,y\in\DDD,
\end{equation}
\eqref{introeq:EVI} 
coincides with the usual 
EVI characterization of the
metric gradient flow of $\varphi$ 
\eqref{eq:GF-EVI}.

{\textbf{(III) Two-species min-max.}} $\XXX=\XXX^{(1)}\times \XXX^{(2)}$
and we consider a Lagrangian 
$F:\DDD^{(1)}\times\DDD^{(2)}\to\R$, with one species minimizing and
the other maximizing $F$. 
When $F$ is a regular function in a 
product of two Hilbert spaces or Riemannian manifolds, 
the evolution corresponds to the system
\begin{displaymath}
    \dot x^{(1)}(t)=-\nabla_{1}F(x^{(1)}(t),x^{(2)}(t)),\quad 
    \dot x^{(2)}(t)=\nabla_{2}F(x^{(1)}(t),x^{(2)}(t)).
\end{displaymath}
For convex--concave $F$ this corresponds to the evolution driven by 
the monotone operator associated with a convex--concave
Lagrangian \cite{rockafellar_monotone_1970}.
If we set 
\begin{equation}
    \label{eqintro:bifunction_minmax}
   \bif(y,x) = F(y^{(1)},x^{(2)}) - F(x^{(1)},y^{(2)}),
   \qquad \text{for all } x=(x^{(1)},x^{(2)}),\, y=(y^{(1)},y^{(2)}) \in\DDD,
\end{equation}
then 
solutions can be  characterized by 
\eqref{introeq:EVI}, which therefore 
can be used to describe a metric 
saddle-point evolution. 
\medskip

\textbf{\textbf{(IV) Multispecies coupled gradient flow.}} 
In a system of $N$ species, each species $i$ is equipped with a metric space structure $(\XXX^{(i)},\dXi)$ and we consider the product space $\XXX \coloneqq \prod_{i=1}^N \XXX^{(i)}$. Each species aims to minimize its own energy functional $F^{(i)}:\DDD \to \R$, which may depend on the other species, and evolves in the direction of steepest descent for the species $i$. We use the notation $F^{(j)}(x^{(i)},y^{(-i)}):=F^{(j)}(y^{(1)},\dots,y^{(i-1)},x^{(i)},y^{(i+1)},\dots,y^{(N)})$ to denote that the $i$-{th} entry in $F^{(j)}$ is $x^{(i)}$, and all other entries are the entries of $y$. Selecting the bifunction
\begin{align}\label{intro:bifunction_multispecies}
    \skd yx = \sum_{i=1}^N 
    F^{(i)}(y^{(i)},
    x^{(-i)} )-F^{(i)}(x^{(i)}, x^{(-i)} )\,, \quad \text{for all }x,y\in\DDD\,,
\end{align}
encodes, for smooth energies $F^{(i)}$ in Euclidean spaces or Riemannian manifolds, the coupled system of gradient flows in
which each species descends its own energy,
\begin{align*}
   \begin{bmatrix}
    \dot x^{(1)} \\ \vdots \\ \dot x^{(N)}
   \end{bmatrix} = -\begin{bmatrix}
       \nabla_{1} F^{(1)}(x^{(1)}, \dots, x^{(N)}) \\ \vdots \\ \nabla_{N} F^{(N)}(x^{(1)}, \dots, x^{(N)})
   \end{bmatrix},
\end{align*}
in the sense that any curve $\sfx(t)$ solving \eqref{introeq:EVI} for $\bif$ 
given by \eqref{intro:bifunction_multispecies} is a
solution of the system. 
Here, while each species evolves in the direction of maximal slope, the joint system is not necessarily evolving according to a gradient flow, because there is not necessarily a single energy that parameterizes the evolution. Additionally, the individual energies can increase, decrease, or oscillate in time even while the entire system contracts to a unique steady state. Moreover, standard notions of convexity for each $F^{(i)}(\cdot,x^{(-i)})$ are not sufficient for guaranteeing existence of joint critical points, which are candidates for steady states of the system.

\paragraph{Interaction dissipativity.}
It is clear that stability properties of solutions to \eqref{introeq:EVI}
should be related to suitable metric dissipativity conditions on $\bif$ involving
$\bif(x,y)+\bif(y,x)$. 
One can easily check that the bifunction $\bif$ 
of the previous examples (II) and (III) 
satisfies the structural condition
\begin{equation}
    \label{introeq:diss0}
    \bif(x,y)+\bif(y,x)=0\quad\text{for every }
    x,y\in \DDD.
\end{equation}
Such a condition is particularly useful to prove
an exponential stability estimate among two solutions $\sfx,\sfy$ of \eqref{introeq:EVI}
of the form
\begin{equation}
    \label{introeq:stability}
    \dX(\sfx(t),\sfy(t))\le 
    \rme^{\lambda (t-s)}\dX(\sfx(s),\sfy(s))
    \quad 
    0\le s\le t,
\end{equation}
for $\lambda:=-\kappa.$
We will systematically adopt a more general and natural condition, which turns out to be  sufficiently flexible to cover interesting examples such as (IV);
it can be expressed by the 
$\eta$-interaction dissipativity condition
\begin{subequations}
\begin{equation}
    \label{introeq:disseta}
    \bif(x,y)+\bif(y,x)\le \eta
    \dX^2(x,y)\quad\text{for every }
    x,y\in \DDD,
\end{equation}
which yields \eqref{introeq:stability}
for $\lambda:=\eta-\kappa$ (see Theorem~\ref{thm:cEVI_solutions}).
We will also show that such a condition, as well as
the diagonal property
\begin{equation}
    \label{introeq:diagonal}
    \bif(x,x)\ge0\quad \text{for every }x\in \DDD,
\end{equation}
\end{subequations}
is useful to study the solvability of the
discrete scheme we propose to solve
\eqref{introeq:EVI}.
 
\paragraph{The Variational Movement Scheme and the Discrete EVI.}
In order to discretize \eqref{introeq:EVI}, we look for a
time-stepping scheme whose iterates satisfy a discrete analogue of
the EVI. 
We have seen that in the case (II) of 
the single-species
gradient flow $\bif(y,x)=\varphi(y)-\varphi(x)$,  the JKO
scheme is obtained by a recursive minimization associated to the generating function
$\Phi$ \eqref{eq:PHI}.

For a general bifunction $\bif$, minimization of a single
functional is no longer available; 
we propose 
 to replace it with a 
system of variational 
inequalities    
associated with the generating bifunction
\begin{equation}
\label{introeq:generatingB}
\sfB(\tau,\bar x; y, x)
   := \frac{1}{2\tau}\dX^2(\bar x, y) - \frac{1}{2\tau}\dX^2(\bar x, x)
      + \bif(y, x),
\end{equation}
which in case (II) reduces to
$\sfB(\tau,\bar x; y, x)=
\Phi(\tau,\bar x;y)-
\Phi(\tau,\bar x;x)$.
For a step size $\tau>0$, starting
from $X_\tau^0:=x_0\in \DDD$, 
the scheme reads:
\begin{equation*}
\boxed{\
\begin{aligned}
&\textbf{Variational movement scheme (VMS).}\\
&\text{Given } X_\tau^{n-1}\in\DDD,\ \text{find } X_\tau^{n}\in\DDD
\text{ such that}\\[4pt]
&
0\;\le\;
\sfB(\tau,X_\tau^{n-1}; Y, X_\tau^{n})
\quad \text{for every }Y\in \DDD.
\end{aligned}\ }
\end{equation*}
We will show that, under general assumptions on $\bif$ and
on $(\XXX,\dX)$, 
involving a reinforced convexity condition 
of the map 
$y\mapsto \sfB(\tau,\bar x;y,x)$,
 which could be considered a
natural counterpart of 
\eqref{eq:to-be-quoted} 
involving arbitrary finite collections of points,
the scheme admits at least one solution at every
step. Moreover, under the 
$\eta$-interaction dissipativity
conditions \eqref{introeq:disseta}, 
\eqref{introeq:diagonal},
solutions to (VMS) can be equivalently 
formulated in terms of the 
dual variational inequality
\begin{equation}
    \label{introeq:dual}
    \sfB(\tau,X_\tau^{n-1}; X_\tau^{n}, X)
\;\le\;0
\quad\text{for every }X\in \DDD,
\end{equation}
so that 
$(X_\tau^{n}, X_\tau^{n})$ is a saddle point of
$(x,y)\mapsto \sfB(\tau, X_\tau^{n-1}; y, x)$ over $\DDD\times\DDD$.

Still assuming 
that 
$y\mapsto \sfB(\tau,\bar x;y,x)$ 
satisfies the $(1/\tau+\kappa)$-convexity condition
as in \eqref{eq:to-be-quoted}, 
we can show that 
any solution 
$(X_\tau^n)_n$ of (VMS) 
satisfies the discrete EVI 
\begin{equation}
\label{introeq:bEVI}
	\frac1{2\tau}\Big(\dX^2(X^n_\tau,Y)-\dX^2(X^{n-1}_\tau,Y)\Big)
		+
	\frac 1{2\tau}	\dX^2(X^n_\tau,X^{n-1}_\tau) + \frac{\cvx}{2}\dX^2(X_\tau^n,Y)
	 \le \skd Y{X^n_\tau} 
		\,, \quad \text{for every}\ Y\in\DDD\,, \ n\ge 1.
	\end{equation}
As in the gradient flow setting, 
we will prove that the
piecewise constant interpolant constructed from (VMS)
 converges, as
$\tau\to 0$, to a solution of \eqref{introeq:EVI}.

   \medskip 
The Variational Movement Scheme
has a nice interpretation in the
reference cases (I)-(II)-(III)-(IV).
In the monotone case (I), it coincides with the implicit (backward) Euler discretization of the dynamics \eqref{eq:monotone}. 
Moreover, in case (II)
the scheme reduces to a JKO
step.
In the saddle-Lagrangian case (III), 
every solution $X^n_\tau=(X^{n,(1)}_\tau,X^{n,(2)}_\tau)$
is a saddle point of the penalized Lagrangian 
\begin{displaymath}
    (x_1,x_2)\mapsto \frac 1{2\tau}\mathsf d_{\rmX^{(1)}}^2(
    X^{n-1,(1)}_\tau,x_1)-
    \frac 1{2\tau}\mathsf d_{\rmX^{(2)}}^2(
    X^{n-1,(2)}_\tau,x_2)+
    F(x_1,x_2).
\end{displaymath}
\nc 
   When $\bif$ is selected to parameterize a multispecies coupled gradient flow \eqref{intro:bifunction_multispecies} as in (IV), the variational inequality is, in fact, a Nash equilibrium problem. A Nash equilibrium $x_*$ of a set of energy functionals $(F^{(i)})$ satisfies
\begin{align*}
    F^{(i)}(x_*^{(i)},x^{(-i)}_*) \le F^{(i)}(y^{(i)},x_*^{(-i)})\qquad \text{for every }y \in \XXX\,, \ \text{for every }i\in[1,\dots,N]\,.
\end{align*}
A Nash equilibrium $x_*\in\XXX$ is such that no species can decrease its energy, given that all the others are fixed.

We can then form the corresponding
distance-penalized functions as in 
\eqref{eq:PHI}
\begin{equation}
    \label{eq:PHIi}
    \Phi^{(i)}(\tau,\bar x;x^{(i)},x^{(-i)}):=
    F^{(i)}(x^{(i)},x^{(-i)})+
    \frac 1{2\tau}
    \dXi^2(x^{(i)},\bar x^{(i)})\,.
\end{equation}
It is immediate to check that $X_\tau^n$ is a solution to one step of (VMS) starting from
$X_{\tau}^{n-1}$ if and only if
\begin{align*}
        X_\tau^{n} \quad \text{is a Nash equilibrium of} \quad \Big\{x \mapsto 
        \Phi^{(i)}(\tau,X_\tau^{n-1};x^{(i)},x^{(-i)})\Big\}_i\,.
\end{align*}

\paragraph{Static versus Evolutionary Games.}
For a static game, when we are interested in finding Nash equilibria, in general, no metric space structure is required; a linear structure together with a topology can give rise to sufficient convexity and compactness properties for the equilibrium problem to admit a solution. 
For evolutionary games, as for gradient flows, one needs to introduce a metric structure that allows differentials to be converted into velocity fields. Here we do not aim at treating genuinely doubly nonlinear settings, such as those induced by non-Hilbertian norms, but rather take inspiration from settings such as Hilbert spaces, NPC spaces, and Euclidean Wasserstein spaces, where one expects a behaviour analogous to the Euclidean case. Defining evolutions by means of the EVI formulation, we are then able to move beyond a strictly Hilbertian framework and obtain a notion of evolutionary game dynamics relying only on the underlying metric structure.

From a game theoretic perspective, the key questions are: (1) existence of equilibria (in a suitable sense) for static games, and (2) existence of time-dependent solutions and long-time behavior for evolutionary games. Studying (1) relies on a notion of equilibrium for the bifunction $\sfb:\DDD\times\DDD\to\R$, whereas existence of solutions for the evolution in (2) is obtained by taking a limit of discrete solutions to (VMS), for which existence of solutions reduces to an equilibrium problem for the generating bifunction $\sfB(\tau,\bar x; \cdot \,, \cdot):\DDD\times\DDD\to\R$ defined in \eqref{introeq:generatingB}. To discuss existence of solutions to equilibrium problems in a common framework for (1) and (2), we therefore state our results for general bifunctions $\sfa:\DDD\times\DDD\to\R$; these results can then be applied to $\sfa=\sfb$ for (1) and to $\sfa=\sfB$ for (2). 

In what follows, we discuss the existence of discrete VMS solutions (requiring to solve an equilibrium problem for $\sfB$), and the discrete to continuum limit providing the existence of solutions to the evolutionary game \eqref{introeq:EVI} driven by $\sfb$. We conclude this introductory exposition by summarizing relevant terminology, relating the notions of dissipativity, monotonicity, accretivity and convexity. Finally, we present four running examples that are used throughout the paper to concretize the meaning of our results in special settings that the reader may be more familiar with.

\paragraph{Existence of Discrete VMS Solutions.} 
A number of challenges arise in proving existence of solutions to VMS, and more generally, existence of solutions to equilibrium problems in general metric spaces defined by a bifunction $\sfa:\DDD\times\DDD\to \R$. If the metric space has some linear structure
and 
$y\mapsto \sfa(y,x)$
has compact convex sublevels,
one can hope to apply 
general existence theorems 
for solutions of 
systems of variational inequalities
(see, e.g., \cite{Baiocchi-Capelo84}).
In the multispecies coupled gradient flow setting with an unbounded domain, from the finite-dimensional game theory literature we expect such an equilibrium point to exist under \textit{strong monotonicity}, which is a first-order (gradient) condition on the energy functionals that guarantees contraction of coupled gradient descent dynamics \cite{conger_monotone_2025}. However, in general metric spaces, we prefer a zeroth-order notion of monotonicity, since the appropriate notion of gradient is often difficult to determine. 

To address these challenges, we employ a notion of \emph{barycentric convexity}, 
inspired by 
\eqref{eq:to-be-quoted}, 
that provides a surrogate of a linear structure, and \textit{interaction dissipativity}, which enables contraction of the dynamics.

Barycentric convexity is a generalization of convexity along a curve which connects two points in $\XXX$ to an interpolation among $J$ points in $\XXX$. 
We thus consider systems $\barysystem{J}$
of barycentric maps 
$\baryname x : \simplex J  \to \DDD$ 
from the $J-1$ dimensional simplex $\simplex J$ to  $\DDD\subset \XXX$ such that 
for every choice of $J$ points $(x_j)_{j=1}^J$
in $\DDD$ there exists a map $\baryname x\in \barysystem{J}$ such that 
$x_j=\bary x {e^j_J}=:\baryc x j$,
where $(e^j_J)_j$ is the canonical basis of $\R^J$.
We say that 
a function $\psi:\DDD  \to \R$ is \textit{barycentrically $\cvx$-convexlike}
	with respect to a system $(\barysystem{J})_{J\in \N_{\ge2}}$
	of barycentric maps
    if for every integer $J\in \N_{\ge 2}$, 
	every $\baryname x\in \barysystem{J}$, and
	every $\aalpha\in \simplex {J}$,
	\begin{equation*}
		\psi(\bary x\aalpha)\le 
	\sum_{j=1}^J\alpha_j\psi(\baryc xj)-
	\frac \cvx4\sum_{j,k=1}^J
	\alpha_j\alpha_k\dX^2(\baryc xj,\baryc xk)\,.
	\end{equation*} 
In the case where $J=2$, the convexity condition becomes the standard two-point convexity inequality along a curve parameterized by $\aalpha=(\alpha,1-\alpha)$, and if the curve parameterized by $\alpha$ is a geodesic, the inequality corresponds to geodesic $\cvx$-convexity.

Equipped with barycentric convexity and interaction dissipativity, we take two approaches for showing existence of the discrete VMS: one via compactness, and one via interaction dissipativity. The first provides a primal solution to the variational equilibrium problem (VMS),  which requires compactness of superlevel sets of $\sfB(y,\cdot)$.
 The linear structure of the simplex $\simplex J$ allows us to prove a version of the Knaster–Kuratowski–Mazurkiewicz (KKM) theorem in the metric space setting. While the required compactness condition for the first approach is more general than the assumptions in the second approach, it can be difficult to show. The second approach uses the dual formulation of the variational equilibrium problem, and relies on barycentric $\cvx$-convexity of $y\mapsto \sfB(y,x)$ and $\eta$ interaction dissipativity, both of which replace the compactness assumption needed in the primal formulation.

\paragraph{Discrete to Continuous Limit.}
In order to take $\tau\to 0$ for the sequence $(X_\tau^n)_n$, the continuous-time piecewise constant interpolant $\ol X_\tau:[0,\infty)\to\XXX$ is constructed,
\begin{equation*}
		\ol X_\tau(0):=X^0_\tau,\quad
		\ol X_\tau(t):=X^n_\tau\quad\text{if }t\in ((n-1)\tau,n\tau].
	\end{equation*}
We show a  Cauchy condition for  piecewise-constant interpolations of $(X_\tau^n)$
which results in the existence of a limit $x(t)$ as $\tau \to 0$.
This estimate is achieved adapting the doubling integration method from \cite{Nochetto-Savare06} to this setting. See also \cite{Crandall86,Crandall-Evans75,kobayashi_difference_1975} for the Crandall-Liggett error estimation method for Banach spaces.
Then, taking the limit of the discrete EVI, the continuous EVI is recovered. The interaction dissipativity inequality in combination with the continuous EVI enables the $\lambda=\eta-\cvx$ dissipativity estimate
\begin{equation}
		\label{introeq:contraction}
		\dX(\sfy(t),\sfx(t))\le \rme^{\lambda(t-s)}\dX(\sfy(s),\sfx(s))\quad\text{whenever}\quad
		0\le s\le t.
\end{equation}
For any curves $\sfx(t),\sfy(t)$ satisfying the continuous EVI, when $\eta < \kappa$ and $(\DDD,\dX)$ is complete, the contraction enables us to establish the existence of a unique steady state. In the multispecies gradient flow setting, under the stronger threshold $\cvx>2\eta_+$, we can also show that the steady state is the unique Nash equilibrium of $(F^{(i)})$ (Corollary~\ref{cor:ss_is_Nash}).

To understand the roles of $\eta$ and $\cvx$ in determining $\lambda$ dissipativity, consider the following four examples for $\bif$
 associated with a monotone flow, a gradient flow,
a min-max flow,
and a general multispecies coupled gradient flow
in a Hilbert space: 
\begin{align*}
\text{(I)}\quad
   \skd yx &= \<\sfA(x),y-x>\,,& \cvx&=0\,,\ \lambda=\eta \quad  \text{ (with } \eta= 0 \text{ if } \sfA \text{ is monotone)}& \\
\text{(II)}\quad
      \skd yx &= \varphi(y)-\varphi(x)\,, & \eta&=0\,, \ \lambda=-\kappa &\\
\text{(III)}\quad
   \skd yx &= F(y^{(1)},x^{(2)})-F(x^{(1)},y^{(2)})\,,& \eta&=0\,, \ \lambda=-\cvx& \\
\text{(IV)}\quad
   \skd yx &= \sum_{i=1}^N F^{(i)}(y^{(i)},x^{(-i)})-F^{(i)}(x^{(i)},x^{(-i)})\,,& \lambda&=\eta-\cvx\,.& 
\end{align*}
In the second and third examples, we have that $\skd xy + \skd yx = 0$, so $\eta=0$ and the dissipativity is entirely dependent on the convexity of $y \mapsto \skd yx$. In the case of a convex-concave Lagrangian $F$ in the third example, we also have $\kappa=0$. 
In contrast, for the first example, $y \mapsto \skd yx$ is $0$-convex, so the dissipativity comes entirely from the interaction dissipativity
\begin{displaymath}
    \bif(x,y)+\bif(y,x)=-
    \langle \sfA(x)-\sfA(y),x-y\rangle\,,
\end{displaymath}
and $\sfA$ is called $\eta$-monotone if interaction dissipativity holds with $\eta\le 0$.
In the last example, both $\cvx$ convexity and $\eta$ interaction dissipativity contribute to the overall system dissipativity. In this paper, we prove that in the multispecies coupled gradient flow setting, $\eta\ge 0$ (Proposition~\ref{prop:eta_ge0}). This means that, for an uncoupled gradient flow system of $N$ species, adding interaction terms to the energies cannot improve the convergence rate, only making it worse. Notice that among these four examples, only the first case needs a linear Hilbertian structure.

\begin{figure}[h]
\centering
\resizebox{\textwidth}{!}{%
\begin{tikzpicture}[
    base_box/.style={rounded corners=20pt, dashed, line width=1pt},
    outer_style/.style={base_box, draw=blue, fill=cyan!10},
    middle_style/.style={base_box, draw=green!60!black, fill=cyan!5!green!5!white!90},
    inner_style/.style={rounded corners=20pt,  draw=none, fill=cyan!5!green!5!white!90},
    title_font/.style={font=\bfseries\Large},
    placeholder_font/.style={font=\normalsize}
]

\def\outerboxwd{24} 
\def\outerboxht{9.5} 
\def\middlemargin{8} 
\def\innermargin{16} 

\draw [outer_style] (0,0) rectangle (\outerboxwd,\outerboxht);

\draw [middle_style] (\middlemargin, 0) rectangle (\outerboxwd, \outerboxht-1.5);

\draw [inner_style] (\innermargin, 0) rectangle (\outerboxwd, \outerboxht-3.);

\node [anchor=center] at (\outerboxwd/2, \outerboxht-0.75) {%
    \Large EVI:   $\quad \ddt \dX^2(y,x(t)) + \frac \kappa 2 \dX^2(y,x(t)) \le \skd y{x(t)}$
};

\node [anchor=north east, title_font] at (\outerboxwd-0.5, \outerboxht-0.5) {Metric Spaces};
\node [anchor=north west, placeholder_font, text width=6.5cm, align=left] at (0.5, \outerboxht-2.5) {%
    \Large 
    Parameters: \\
    \vspace{0.5 cm}
    1. $\eta$ interaction dissipativity \\
    \vspace{0.1 cm}
    $\skd xy + \skd yx \le \eta \dX^2(x,y)$. \\
    \vspace{0.5 cm}
    2. $\kappa$ convexity \\
    \vspace{0.1 cm}
    $y \mapsto \skd yx$ is $\kappa$-convex. \\
    \vspace{0.5 cm} 
    3. $\lambda=\eta-\kappa$ dissipativity 
    \\
    \vspace{0.1 cm}
   \large $\dX(\sfy(t),\sfx(t)) \le e^{\lambda t} \dX(\sfy(0),\sfx(0))$.
   \Large
};

\node [anchor=north east, title_font] at (\outerboxwd-0.5, \outerboxht-2.0) {Hilbert Spaces};
\node [anchor=north west, placeholder_font, text width=15cm, align=left] at (\middlemargin+0.5, \outerboxht-2.5) {%
    \Large 
    -- accretivity= -- monotonicity
    = dissipativity \\ 
   \begin{multicols}{2}
    
   Example: $\XXX=(H,\norm{\cdot})$, \\
    $\skd yx = \langle\sfA(x),y-x\rangle$. \\
    \vspace{0.5 cm}
    1. \large $\langle\sfA(x)-\sfA(y),x-y\rangle\ge -\eta \norm{x-y}^2$. \\
    \Large
   \vspace{0.5 cm}
   2. $y \mapsto \skd yx$ is $\kappa=0$ convex. \\
    \vspace{0.5 cm}
    3. $\sfA$ is $-\lambda=-\eta$ monotone.

   \columnbreak
       Example: $\XXX=(H,\norm{\cdot})$, \\
     $\skd yx = \varphi(y)-\varphi(x)$. \\
    \vspace{0.5 cm}
   \Large 1. $\skd xy + \skd yx =  0$; $\eta=0$.  \\
   \vspace{0.5 cm}
   2. $\varphi$ is $\kappa$-convex.   \\
   \vspace{0.5 cm}
   3. The gradient flow system is $-\lambda=\kappa$ monotone.
   \end{multicols}
};


\end{tikzpicture}%
} 
\caption{The nomenclature and coefficients for convexity, monotonicity, accretivity, and dissipativity as used in this work are illustrated in particular settings.  }
\label{fig:nomenclature}
\end{figure}
\paragraph{Dissipativity, monotonicity, accretivity, and convexity. } 
The solution $x$ to the EVI is $(\lambda=\eta-\kappa)$-dissipative if $\sfb$ is $\eta$-interaction dissipative and $\sfb(\cdot,x)$ is barycentrically $\cvx$-convexlike; 
for $\lambda=0$ the evolution is non-expansive, for $\lambda<0$, the solution is exponentially contractive. This is a generalization of several well-established notions. In Hilbert spaces, 
dissipativity usually refers to the opposite behaviour of
monotonicity, that is, a system is $\lambda$ dissipative if it is $-\lambda$ monotone. This sign convention is also used in game theory, where a monotone system is contracting if it has a positive monotonicity parameter.  
A positive $\cvx$ convexity parameter corresponds to a strongly convex function, carrying the standard sign convention for convexity.
A summary of this nomenclature is provided in Figure~\ref{fig:nomenclature}, along with the examples illustrating the various convexity and dissipativity parameters. 
 In this work, we will use the term \textit{dissipativity} to describe the contractivity. In the particular setting of multispecies coupled gradient flows, we will use the term \textit{game-theoretic monotonicity} for notions parallel with monotonicity in the game theory literature, using a consistent sign for $\lambda$ throughout: $\lambda<0$ means a contracting system.

\paragraph{Running Examples.}
Throughout this paper, we accompany the exposition of the general theory with four running examples (I)--(IV) introduced above.  
A number of other examples are included throughout the paper, and  Section~\ref{sec:multispecies} contains specific choices of energy functionals for multispecies gradient flows in Euclidean and Wasserstein-2 spaces. 
\begin{itemize}
    \item \textbf{Example (I) Evolutions via monotone operators on Hilbert spaces}:
    A simple example is $H=L^2(\R^d)$, $\DDD=\big\{u\in H^1(\R^d):
    u\ge 0\big\}$ and $\sfA u:=-\Delta u+\boldsymbol f\cdot \nabla u$, the Laplacian operator perturbed by Lipschitz drift term $\boldsymbol f$,
    \begin{displaymath}
        \bif(v,u)=\int_{\R^d} \Big(\nabla u\cdot \nabla (v-u) + \boldsymbol f\cdot \nabla u (v-u)\Big)\,\d x\,\qquad u,v\in \DDD.
    \end{displaymath}
    In this setting, we explicitly need a linear structure and the use of an inner product, and thus it is not clear how to envision monotone operators in general metric spaces. Thanks to this special structure, lacking interactions between the $x$ and $y$ components, many assumptions of our theory in general metric spaces are trivially satisfied, with all the conditions reducing to properties of the monotone operator $\sfA(x)$, and $y$ is merely playing the role of a test variable. We do not claim to reproduce the well-known theory for maximal monotone operators,
    however this example is useful to illustrate the role of the dissipativity parameter $\eta$ in the general theory and in more complex systems where Hilbertian and purely metric components interact.

\item \textbf{Example (II) Single-Species Gradient Flows:} 
Here, we consider in particular the Wasserstein-2 space, 
$\XXX=\PP_2(\R^d)$, $\dX(\rho,\sigma)=\Wass_2(\rho,\sigma)$,  yielding a partial differential equation 
driven by the Otto-Wasserstein gradient:
    \begin{align*}
        \partial_t \rho=
        \div{\rho \nabla_{\Wass_2}\varphi[\rho]}\,\qquad
        \nabla_{\Wass_2}\varphi[\rho]:=\nabla \frac{\delta \varphi[\rho]}{\delta \rho}\,.
    \end{align*}
Since the seminal works \cite{Otto01,JordanKinderlehrerOtto98}, it is by now well-known that many partial differential equations in continuity-equation form can be understood in this framework. 

 \item \textbf{Example (III) Min-Max Gradient Flows:} 
Again, we consider in particular the Wasserstein-2 space, where each player is represented by a probability measure on a different feature space, 
$\XXX^{(i)}=\PP_2(\R^{d_i})$, $\dXi(\rho^{(i)},\sigma^{(i)})=\Wass_2(\rho^{(i)},\sigma^{(i)})$,
$$\XXX=\XXX^{(1)}\times \XXX^{(2)},\quad  
\dX^2((\rho^{(1)},\rho^{(2)}),(\sigma^{(1)},\sigma^{(2)}))=\Wass_2^2(\rho^{(1)},\sigma^{(1)})+\Wass_2^2(\rho^{(2)},\sigma^{(2)}),$$
and the solution $\sfx(t)=[\rho^{(1)}(t),\rho^{(2)}(t)]$ satisfies the system of PDEs
\begin{align*}
    \partial_t \rho^{(1)}=
        \div{\rho^{(1)} \nabla_{x^{(1)}} \frac{\delta F[\rho^{(1)},\rho^{(2)}]}{\delta \rho^{(1)}}},\qquad 
       \partial_t \rho^{(2)}=  - \div{\rho^{(2)} \nabla_{x^{(2)}} \frac{\delta F[\rho^{(1)},\rho^{(2)}]}{\delta \rho^{(2)}}}\,.
\end{align*}
In this setting the bifunctional \eqref{eqintro:bifunction_minmax} 
$$\bif([\rho^{(1)},\rho^{(2)}],[\sigma^{(1)},\sigma^{(2)}])=F(\sigma^{(1)},\rho^{(2)})-F(\rho^{(1)},\sigma^{(2)})$$
is associated with a single Lagrangian $F(\rho^{(1)},\rho^{(2)})$, 
and the antisymmetry property $\sfb(y,x)+\sfb(x,y)=0$ means that key assumptions of our theory reduce to the Lagrangian $F$ being (geodesically) convex-concave. Here, $x$ and $y$ play similar roles, since convexity in $x$ implies concavity in $y$ and vice versa. 

\item \textbf{Example (IV) Multispecies Coupled Gradient Flows:}  
We consider $N$ species $\rho^{(i)}\in \XXX^{(i)}=\PP_2(\R^{d_i})$ for $i\in[1,\dots,N]$, each minimizing their own energy functional $F^{(i)}:\XXX\to\R$, where 
$$\XXX=\Pi_{i=1}^N \XXX^{(i)}\quad\text{endowed with}\quad
\dX^2(\rho,\sigma):=\sum_{i=1}^N \Wass_2^2(\rho^{(i)},\sigma^{(i)})\,\quad
\rho=(\rho^{(1)},\dots,\rho^{(N)})$$
and so the $i$th energy may also depend on the state of all the other species. 
The corresponding dynamics become a coupled system of partial differential equations
    \begin{displaymath}
        \begin{bmatrix}
            \partial_t \rho^{(1)} \vphantom{\tfrac{\delta F[\rho^{(1)},\cdots,\rho^{(N)}]}{\delta \rho^{(1)}}}\\
            \vdots \\
            \partial_t \rho^{(N)} \vphantom{\frac{\delta F[\rho^{(1)},\cdots,\rho^{(N)}]}{\delta \rho^{(N)}}} 
        \end{bmatrix}
    =\begin{bmatrix}
       \operatorname{div} 
       \big(\rho^{(1)} 
       \nabla_{x^{(1)}} 
       \tfrac{\delta F[\rho^{(1)},\cdots,\rho^{(N)}]}
       {\delta \rho^{(1)}}
       \big) 
       \\ \vdots \\ \operatorname{div} 
       \big(\rho^{(N)} \nabla_{x^{(N)}} \frac{\delta F[\rho^{(1)},\cdots,\rho^{(N)}]}{\delta \rho^{(N)}}\big)
    \end{bmatrix} \,.
    \end{displaymath}
Many special cases of such systems have been analyzed in the literature, but a general theory is still lacking. A multispecies coupled gradient flow system is not necessarily a joint gradient flow system, and thus the analysis of such systems is challenging especially when the energies contain nonlinearities, nonlocalities, and intricate interplay between species. 
The long-time behavior of such multispecies systems was recently studied in \cite{conger_monotone_2025}, leaving well-posedness as an open question, which our framework is able to address. In particular, we show formally in Lemma~\ref{lem:W2GF_satisfies_EVI} that solutions to coupled Wasserstein-2 dynamics satisfy the EVI for the choice of $\sfb$ above, and Corollary~\ref{cor:ss_is_Nash} identifies assumptions on $(F^{(i)})$ under which such solutions exist. Note that Example (IV) reduces to Example (III) when $N=2$ and $F^{(1)}(\rho^{(1)},\rho^{(2)})=-F^{(2)}(\rho^{(2)},\rho^{(1)})$,  which includes the single-species setting Example (II) for $N=1$, as illustrated in Figure~\ref{fig:multispecies_gradient_flows}.
\end{itemize}

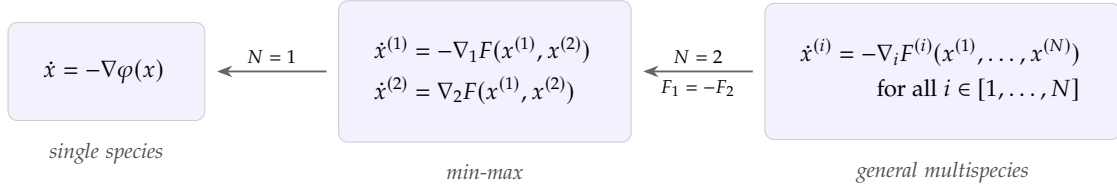
\begin{figure}
    \centering
\resizebox{\textwidth}{!}{%
\begin{tikzpicture}[
    eqnblock/.style={
        rectangle, 
        draw=black!20, 
        fill=blue!5, 
        rounded corners, 
        inner sep=15pt, 
        align=center,
        anchor=west 
    },
    desc/.style={
        text width=3cm, 
        align=center, 
        font=\small\itshape,
        text=black!70,
        anchor=north 
    },
    myarrow/.style={
        ->, >=Stealth, thick, draw=black!60, shorten <=5pt, shorten >=5pt
    }
]

    \node[eqnblock] (eq1) at (0,0) {
        $\begin{aligned}
           \dot{x}=-\nabla \varphi(x)
        \end{aligned}$
    };
    \node[desc, yshift=-0.2cm] at (eq1.south) {single species};

    \node[eqnblock, xshift=2cm] (eq2) at (eq1.east) {
        $\begin{aligned}
           \dot x^{(1)} &= -\nabla_{1} F(x^{(1)},x^{(2)}) \\
            \dot x^{(2)}  &= \nabla_{2} F(x^{(1)},x^{(2)})
        \end{aligned}$
    };
    \node[desc, yshift=-0.2cm] at (eq2.south) {min-max};

    \node[eqnblock, xshift=2cm] (eq3) at (eq2.east) {
        $\begin{aligned}
           {\dot x}^{(i)} = -\nabla_{i} F^{(i)}(x^{(1)},\dots,x^{(N)}) \\
          \text{for all } i \in [1,\dots,N]
        \end{aligned}$
    };
    \node[desc, yshift=-0.2cm] at (eq3.south) {general multispecies};

    \draw[myarrow] (eq2.west) -- node[above, font=\footnotesize] {$N=1$} (eq1.east);
    \draw[myarrow] (eq3.west) -- node[above, font=\footnotesize] {$N=2$ } node[below, font=\footnotesize] {$F_1=-F_2$} (eq2.east);

\end{tikzpicture}
}
    \caption{Relationship between Examples (II), (III), (IV).}
    \label{fig:multispecies_gradient_flows}
\end{figure}

\subsection{Contributions}\label{sec:contributions}
The main contributions of this paper are:
\begin{enumerate}
    \item[(C1)] Existence of solutions to \eqref{introeq:EVI} driven by a bifunction $\bif$ (Section~\ref{sec:metric_dissipative_evolutions}), which includes the classes of systems outlined in our running examples (I), (II), (III), and (IV). In particular, the existence of solutions to general coupled PDE systems of the form in Example (IV) is new. 
    
    \item[(C2)] Existence of solutions to variational equilibrium problems in general metric spaces (Section~\ref{sec:equilibrium_problems}) using three different approaches:
    \begin{enumerate}
        \item[(i)] Primal problem formulation only using compactness and a suitable convexity assumption (Section~\ref{sec:primal});
        \item[(ii)] Dual problem formulation, removing compactness thanks to dissipativity (Section~\ref{sec:dual}); 
        \item[(iii)] Antisymmetric bifunctions only using convexity, identifying a corresponding saddle point problem that is characterized by a Lagrangian (Section~\ref{sec:saddle});
    \end{enumerate}
    as well as relations between the primal and dual formulations in this setting (Section~\ref{sec:dual}).
    \item[(C3)] A zeroth-order definition of game-theoretic monotonicity for $N$-player games in general metric spaces (Definition~\ref{def:monotonicity_zeroth_order}), including finite-dimensional games as well as infinite-dimensional games 
    in spaces of probability measures.
    This notion of monotonicity provides contraction of the system to a unique Nash equilibrium according to \cite{conger_monotone_2025}. We provide a detailed discussion of properties that are implied by our theory in the setting of $N$-player games (Section~\ref{sec:multispecies}), which corresponds to the general multispecies systems discussed in Example IV.
\end{enumerate}
Contributions (C1) and (C2) rely on a suitable notion of admissible bifunctions $\sfb$ (Section~\ref{sec:admissible-b}), which then informs the game-theoretic monotonicity in (C3). In particular, we identify two separate properties of the driving bifunctions $\sfb$ that each contribute differently to dissipativity for evolutions in general metric spaces: interaction-dissipativity \eqref{introeq:disseta}, which by itself is a new notion, and barycentric convexity (Definition~\ref{def:convexlike}). In fact, we show that under sufficient dissipativity we obtain not only existence of solutions to \eqref{introeq:EVI}, but also contraction of the dynamics (Theorem~\ref{thm:main_existence}). These results rely on a suitable notion of solutions to \eqref{introeq:EVI}, which we introduce in Definition~\ref{def:solutions}.

In fact, our results for (C1) crucially rely on (C2). Our strategy for showing existence of solutions for the dynamics in (C1) goes via a discrete scheme that approximates solutions to \eqref{introeq:EVI}. One of our main contributions is to identify a suitable approximation scheme (Definition~\ref{def:VMS}), which we call \emph{Variational Movement Scheme} (VMS); it represents a generalization of the minimizing movement scheme for gradient flows to general bifunctions. Our general existence results from (C2) can be applied to (VMS) to obtain discrete EVI-solutions (Theorem~\ref{thm:saddle-sEVI}), and we show that their piecewise constant interpolations in time converge to a solution of the continuous-time dynamics \eqref{introeq:EVI} in a suitable sense (Theorem~\ref{thm:tau_to_zero}). Thus, we note in particular that Contribution (C2) provides three different approaches for the existence of solutions to \eqref{introeq:EVI}, depending on the properties of the driving bifunction $\sfb$. Finally, (C2) also provides conditions for existence of a variational equilibrium for \eqref{introeq:EVI} itself, solving the corresponding static game, not just for the generating bifunction appearing in (VMS) that is needed to achieve (C1).

In the following informal theorem, we state a representative special case of the results from (C1) and (C2), corresponding to the dissipative setting; the general well-posedness statements, including the antisymmetric and compact-domain cases, are given in Theorems~\ref{thm:main_existence}, \ref{thm:antisym_existence}, and~\ref{thm:tau_to_zero}.
\begin{theorem}[Informal]\label{thm:main}
    Let $(\DDD,\dX)$ be complete, let $\tau_o>0$ be such that
    $\tau_o^{-1}+\cvx > 2\eta_+$ and suppose that
    \begin{itemize}
            \item $\sfb$ is $\eta$ interaction dissipative and nonnegative on the diagonal 
        according to 
        \eqref{introeq:disseta}-\eqref{introeq:diagonal}; 
        \item for every $\tau\in (0,\tau_o)$,
	every $\bar x,x\in \DDD$,
	the functions $y \mapsto \sfB(\tau,\bar x;y,x)$, given by
    \begin{align*}
        \sfB(\tau,\bar x;y,x)=\frac 1{2\tau}\dX^2(\bar x,y) - \frac{1}{2\tau} \dX^2(\bar x,x) + \skd yx\,,
    \end{align*}
    are barycentrically $(\tau^{-1}+\cvx)$ convexlike with respect to some
    continuous barycentric system $\cB$;
        \item $x \mapsto \skd yx$ satisfies a suitable (conditional) upper semicontinuity. 
    \end{itemize}
     Then for any initial condition $x_0\in\DDD$, there exists a solution $(X_\tau^n)_{n\ge 0}$ to (VMS), which satisfies the discrete EVI \eqref{introeq:bEVI}.
     There exists a continuous curve $\sfx : [0,+\infty) \to \DDD$ with $\sfx(0)=x_0$ that satisfies the EVI
     \begin{align*}
         \frac12\frac{\d} \dt \dX^2(\sfx(t),y)+\frac{\cvx}{2} \dX^2(\sfx(t),y)\le \bskd y{\sfx(t)}\quad\text{$t\in (0,+\infty)$,}\quad 
	\text{for every }y\in \DDD\,.
     \end{align*}
     Any continuous curves $\sfx,\sfy$ which satisfy the EVI above also satisfy the contraction estimate
     \begin{align*}
         \dX(\sfx(t),\sfy(t)) \le e^{\lambda(t-s)} \dX(\sfx(s),\sfy(s))\,, \quad \text{for all }0\le s\le t\,.
     \end{align*}
\end{theorem}

We state below our zeroth-order definition of game-theoretic monotonicity from Contribution (C3) that makes use of the notions of barycentric convexity and interaction dissipativity introduced above; a larger convexity parameter $\cvx$ increases monotonicity whereas a larger interaction dissipativity parameter $\eta$ decreases monotonicity. 
\begin{definition}[Game-Theoretic Monotonicity]\label{def:monotonicity_zeroth_order}
    A set of energy functionals $(F^{(i)})$ are $(\eta,\cvx)$-monotone, 
    with rate $\lambda = \eta-\cvx$, if
    \begin{itemize}
        \item $\cvx$ convexity: $y\mapsto \bif(y,x)$ is barycentrically $\cvx$-convexlike with respect to a system $\cB$.
        \item $\eta$ dissipativity: $\skd yx + \skd xy \le \eta \dX^2(x,y)$.
    \end{itemize}
    Equivalently, each $F^{(i)}(\cdot,x^{(-i)})$ is barycentrically $\cvx$-convexlike and for every $x, y \in \DDD$, 
        \begin{align*}
    \sum_{i=1}^N 
   F^{(i)}(y^{(i)},x^{(-i)})-  F^{(i)}(x^{(i)},x^{(-i)}) + F^{(i)}(x^{(i)},y^{(-i)}) - F^{(i)}(y^{(i)},y^{(-i)}) \le \eta \,\dX^2(x,y)  \,.
    \end{align*}
\end{definition}
We show that, for particular choices of $\cB$, this notion of game-theoretic monotonicity implies the classical definitions in Euclidean space and in the space of measures equipped with the Wasserstein-2 metric, as defined in \cite{conger_monotone_2025,wang_local_2026}. In \cite{conger_monotone_2025}, monotonicity is defined with respect to geodesic curves, whereas in \cite{wang_local_2026} it is defined along linear interpolations of probability measures. 
Note that monotonicity of the generating bifunction $\sfB$ defined in terms of $\sfb$ provides existence of solutions to \eqref{introeq:EVI} via existence of discrete iterates to (VMS); whereas monotonicity of $\sfb$ provides contraction of the dynamics (see \eqref{introeq:contraction} above).

\subsection{Related Literature}
This work extends results for accretive and dissipative evolutions to the metric space setting, using  variational equilibrium problem tools. An important application of the results is the setting of multispecies gradient flows, which can be studied with a game theory point of view, and for which general existence results were left as an open question in \cite{conger_monotone_2025}. Related literature from each of these fields is discussed in this section.

\paragraph{Accretive and Dissipative Evolutions.} The theory of nonlinear evolutions driven by accretive operators in Banach and Hilbert spaces provides a classical foundation for the single-species special case of our framework. The foundational Crandall--Liggett generation theorem \cite{Crandall-Liggett71} established that accretive operators on general Banach spaces generate semigroups, with solutions obtained as limits of implicit time-discretizations. The connection between accretive operators and evolution inequalities in Hilbert spaces was further developed by Brézis \cite{Brezis73}. For the case of $\lambda$-dissipative systems (corresponding to strongly accretive operators), contraction and long-time behavior estimates analogous to our dissipativity inequality were obtained by Kobayashi \cite{kobayashi_difference_1975}; see also Crandall \cite{Crandall86} for a broader survey of nonlinear semigroup theory and its connections to accretive operators. The optimal error estimates between piecewise-constant interpolants in the $\tau \to 0$ limit that we rely upon are adapted from those in \cite{Nochetto-Savare06}, which extended the earlier error analysis of Crandall and Evans \cite{Crandall-Evans75} 
to derive optimal a posteriori error estimates; 
the method of doubling the time variable used therein originates in \cite{kobayashi_difference_1975}.
\paragraph{Variational Equilibrium Problems.}
The existence theory for saddle points and minimax problems in convex-concave settings is classical, with foundational contributions due to von Neumann \cite{vonNeumann_1928}.
The connection between the sub-super differential
of a convex-concave Lagrangian and monotonicity 
has been developed by Rockafellar \cite{Rockafellar-saddle70}. 
The general theory of variational inequalities was developed in the works of Browder \cite{browder_1968} and later unified with fixed-point and KKM theory; see Fan \cite{fan_1984} for the Fan--KKM theorem, which underlies our primal existence proof (Theorem~\ref{thm:primal_soln}).
\paragraph{Gradient Flows and Systems of Coupled Gradient Flows.} 
The development of the theory of gradient flows in Wasserstein and metric spaces received a decisive impetus from \cite{JordanKinderlehrerOtto98,Otto01} which introduced what is now called the JKO scheme
and 
provided many tools and geometric insights to interpret a large class of diffusion equations as gradient flows in the Wasserstein-2 metric. 
In the broader metric setting, \cite{AGS08} provides a rigorous theory for gradient flows in general metric spaces, revisiting 
De Giorgi's approach to curves of maximal slope
\cite{DeGiorgi-Degiovanni-Marino-Tosques83} and introducing the notion of EVI \eqref{eq:GF-EVI} and convexity along generalized geodesics, motivated by applications in the Wasserstein spaces inspired by \cite{JordanKinderlehrerOtto98}. A systematic study of EVI properties for gradient flows 
has been developed by
\cite{Muratori-SavareI}. Recent work \cite{aubin-frankowski_evolution_2026} generalizes the EVI framework from metrics to general cost functions, and analyzing explicit (rather than implicit) update schemes.

When more than one species evolves according to the direction of steepest descent of its own energy, the joint system is not necessarily a gradient flow. This occurs, for example, in cross-diffusion systems. The work \cite{di_francesco_measure_2013} proved existence of solutions for two-species aggregation systems using a semi-implicit JKO scheme in which each species minimizes its own energy with the other held fixed, an approach equivalent to the implicit-explicit update introduced above rather than the fully implicit Nash update (also see Remark~\ref{rmk:choice-scheme} for comments on the choice of scheme). 
The extension to diffusive systems was treated in \cite{benamou_augmented_2016}. 
The coupling-induced loss of contractivity is a central theme in \cite{beck_matthes_zizza_expcv_coupled_diffusion}, which studies two Wasserstein gradient flows coupled through a cross-diffusion term; while coupling destroys the individual contractivity, exponential convergence to a unique steady state is recovered for sufficiently small coupling. Other coupled multispecies Wasserstein-2 gradient flow well-posedness works include \cite{di_francesco_nonlinear_2018,carrillo_zoology_2018,carrillo_measure_2020}; see also references within.  For settings with an infinite number of species, well-posedness was studied in   \cite{debiec_finite_2025,cances_continuum_2024}.

In other cases, long-time behavior of coupled systems has been shown while leaving existence as an open problem. In the min-max setting, \cite{conger_coupled_2024} shows conditions under which a min-max Wasserstein gradient flow converges to a unique steady state. The convergence of two-timescale gradient descent-ascent in a nonconvex setting is analyzed in \cite{an_lu_2025}, which establishes long-time convergence. Similarly, in the  setting of measure optimization, \cite{dvurechensky_zhu_2024} studies a class of functional saddle-point problems, modeling the dynamics as a Fisher-Rao-RKHS gradient flow, proposing a primal-dual kernel mirror prox algorithm and providing convergence guarantees. Existence of solutions in the min-max setting was then addressed in \cite{isobe_gradient_2025}, and our work generalizes their results by moving beyond the min-max setting.  The long-time behavior of general multispecies setting in the Wasserstein-2 metric was analyzed in \cite{conger_monotone_2025}, and our work provides conditions under which solutions exist to that system. Similarly, \cite{wang_local_2026} considers a multispecies coupled Wasserstein-2 gradient flow with potential functions and entropy, and our work provides conditions under which solutions exist to the evolution studied there.  Other coupled gradient flow systems that fall within our more general framework  include Cahn-Hilliard/Allen-Cahn system with cross-kinetic coupling, with existence studied in \cite{BRUNK2024104051} using Galerkin approximations.

There are  coupled gradient flow settings where our results cannot be directly applied, either due to convexity issues with the metric or with the energies. Two recent works address spherical Hellinger gradient flows, also known as Fisher-Rao gradient flows. In \cite{carrillo_fisher-rao_2024}, the authors present convexity properties of $f$-divergence energy functionals along geodesics; however, the squared spherical Hellinger metric is not 1-convex along geodesics due to its positive curvature. While nonnegative cross curvature results in \cite{Leger-Todeschi-Vialard25} (see Example~\ref{ex:NNCC}) provide a set of curves along which the squared metric is 1-convex, it is not clear if nontrivial energies are also convex along these curves. In \cite{lascu_fisher-rao_2024}, the authors analyze the long-time behavior of min-max flows in this metric.
In cases where the joint energy is not jointly convex---such as cross-diffusion systems in which the pressure depends on the total density--- our approach cannot be directly applied. In \cite{elbar_santambrogio_2025}, a cross-diffusion system for two interacting populations is studied in which diffusion is governed by the aggregate density through a fast-diffusion pressure law; the non-convexity of the pressure energy is a central obstruction to a direct gradient flow treatment, and global weak solutions are obtained under a mixing condition on the initial data. The existence theory for this same system was subsequently completed in full generality, without restrictions on the initial data, in \cite{meszaros_parker_2025}. The present framework does not directly apply to these settings without either convexity along some set of curves for each of the individual energies, or suitable interaction-dissipativity.

\paragraph{Game Theory.}
In finite dimensional games, the existence of Nash equilibria follows from compactness and continuity assumptions via the Kakutani-Brouwer fixed-point theorem~\cite{nash_1951,debreu_1952}, and first-order monotonicity conditions on cost functions in games have also been used to prove existence and uniqueness \cite{rosen_existence_1965}. For a summary of Nash equilibrium existence results and their connections to variational inequalities, see \cite{facchinei_finite-dimensional_2003}; our notion of game-theoretic monotonicity extends these definitions to metric spaces. In infinite-dimensional settings, results typically rely on compactness assumptions and  linear structure in the metric space \cite{glicksberg_1952} and our aim is to provide results for spaces without such a structure. For an extensive treatment of Nash equilibria in unbounded, strongly monotone Euclidean settings, see \cite{Facchinei_Pang_2009}. In mean field games, a notion of \textit{displacement monotonicity} was recently proposed in \cite{gangbo_mean_2022}, with several works such as \cite{meszaros_mean_2024} generalizing their results; this is a convexity property along Wasserstein-2 geodesics. The earlier \textit{Lasry-Lions monotonicity} \cite{lasry_mean_2007} is a convexity property along linear interpolants. Both of these definitions are for a single-species mean field game, with dynamics that are not a coupled gradient flow system.

\subsection{Plan of the paper}
In Section~\ref{sec:preliminaries}, we provide preliminary information such as notation and definitions used throughout. Section~\ref{sec:equilibrium_problems} introduces the main variational equilibrium problems and results for existence of solutions to those problems. Section~\ref{sec:soln_to_discrete_EVI} applies the equilibrium problem results to the discrete EVI.  In Section~\ref{sec:metric_dissipative_evolutions}, we define notions of solutions and the corresponding EVIs and the main approximation, stability, and existence theorems. 
The proof of the Cauchy estimates needed for the convergence of the VMS method is postponed to Appendix~\ref{app:cauchy_estimate_proof}. 
Finally, Section~\ref{sec:multispecies} contains results for coupled multispecies gradient flow systems to which our results can be applied and shows how our zeroth-order notion of game theoretic monotonicity implies the classical first-order condition in Euclidean space and recent notions in the Wasserstein-2 space.

\subsection*{List of main notation}
The notation used in this paper is listed in the table below. 
\begin{longtable}{p{3.0cm} l}
$\sfa$ & bifunction $\sfa:\DDD \times \DDD \to \R$ \\
$\aalpha$, $\bbeta$ & $\aalpha\in \simplex {J}$ point in the simplex \\
 $\sfb$ & bifunction $\sfb:\DDD\times \DDD\to \R$ driving \eqref{introeq:EVI}\\
 $\lskdname$ & shifted bifunction $\lskdname(y,x):=\skd yx-\frac{\cvx}{2}\dX^2(x,y)$ for $x,y\in \DDD\times \DDD$\\
 $\cbskdname$ & conjugate bifunction $\cbskdname(y,x):=\eta\,\dX^2(x,y)-\skd xy$\nc\\
$\barysystem J,\ \barysystem{}$ & barysystem with $J$ reference points and 
collections of all $\barysystem J$, $J\in \N_{\ge 2}$\\
    $\barysystem{\bar x,2}$ &  barysystem with $2$ reference points and dependent on $\bar x$ \\
    $\sfB$ & bifunction $\sfB:\DDD\times \DDD\to \R$ in variational movement scheme (VMS) \\
    $\skslope{\cvx}x$ & bifunction slope $\skslope{\cvx}{x}:=
    \sup_{y\in \DDD}\frac{\lskd yx_-}{\dX(x,y)}\in [0,+\infty]$ for $x\in \DDD$\\
    $\slope x$ & local bifunction slope $\slope x:=\limsup_{y\to x}
    \frac{(\skd yx)_-}{\dX(x,y)}\in [0,+\infty]$ for $x\in\DDD$\\
        $\DDD$  &  domain of the bifunction, a subset of the metric space $\XXX$\\
    $\domainslope{\cvx},\ \domainslope{}$ & domains in $\DDD$ where bifunction slopes $\skslope{\cvx}\cdot,\ \slope\cdot$ are finite \\
    $(\sfe_J^i)_{i=1}^J$, $\simplex {J}$ & 
    canonical basis of $\R^J$ and its reference convex simplex for $J$ points \\
    $\cG_2$ & set of geodesic curves, particular choice of $\barysystem 2$\\
    $\Gamma(\mu,\nu),\Gamma_o(\mu,\nu)$ & set of joint ($o$: optimal) probability measures  with marginals $\mu$ and $\nu$
    \\ 
    $H$ & Hilbert space \\
 $\cvx_i$, $\cvx$ &  convexity parameter for $i^{th}$ species, convexity of a given bifunction \\
 $\lambda$ & dissipativity rate with $\lambda=\eta-\kappa$\\
  $\ell$, $\ell_\tau$ & piecewise linear time interpolators \\
    $\eta$ & interaction dissipativity \\
    $\PP,\PP_p$ & space of probability measures (subscript $p$: with 
     finite $p^{th}$ moment) \\
    $\sigma_-$, $\sigma^+$ &  sublevel and superlevel sets of a given function\\
    $\sfS$, $\sfS_t$ & semigroup operators \\
    $\tau$ & time step in a given discrete scheme \\
      $|\dot{\sfu}^{\pm}|(t)$, $|\dot{\sfu}_{\pm}|(t)$ & left/right, upper/lower metric derivatives \\
      $v$, $v_i$, $w$ & elements in vector space \\
$\Wass_2$ & Wasserstein-2 metric \\
     $(\XXX,\dX)$ & reference metric space\\
     $\baryname x$ & barycentric map from simplex to $D$ \\
    $\xx$, $(x_i)_{i=1}^J$ & collection of $J$ points in $\XXX$\\
    $x$, $x^{(i)}$, $x^{(-i)}$ & point in $\XXX$, point in $\XXX^{(i)}$, point in $\prod_{j\ne i}\XXX^{(j)}$ \\
    $x^*$, $x_*$ & primal and dual variational problem solutions \\
    $Y,X_\tau^n$ & object in discrete time, sequence in $\XXX$ parameterized by $\tau$ with index $n$ \\
    $\overline X_\tau(t)$, $\underline X_\tau(t)$  & piecewise constant left/right continuous interpolants of $X_\tau^n$\\
\end{longtable}

\vspace{6pt}

\section{Preliminaries: metric spaces, slopes, and convexlike functions}
\label{sec:preliminaries}

In this section, we establish definitions of convexity in general metric spaces, along with various notions of slopes (Section~\ref{sec:convexity-def}). We illustrate the convexity notions for several examples (Section~\ref{sec:convexity-ex-all}), including for our running examples (I)-(IV). We conclude this section with definitions of semicontinuity and related properties for functions (Section~\ref{sec:semicontinuity}), and definitions of metric and Dini derivatives (Section~\ref{sec:derivatives-def}).

The convexity definitions introduced in Section~\ref{sec:convexity-def} center around the idea of barycentric convexity, which generalizes 
the notion of convexity along a curve connecting two points
by considering general interpolations between an arbitrary number of points. The primary motivation for this 
more restrictive condition (at least in general nonlinear spaces) is to show the existence of solutions to the equilibrium problem that appears in the time-discretized approximation of flow satisfying the evolution variational inequality \eqref{introeq:EVI}. In particular, to prove existence of the equilibrium solution to the variational movement scheme (VMS), we will implement three different arguments
combining 
barycentric convexity of 
$y\mapsto \sfB(y,x)$ 
with topological or dissipativity properties of 
$\sfB$.

A secondary application, just as in the classical gradient flow setting, is to derive the discrete EVI property \eqref{introeq:bEVI} for solutions to the (VMS): 
this introduces a compatibility condition, related to 1-convexity of the distance function, requiring more flexibility for the choice of curves along which convexity is defined. As for the case of generalized geodesics 
for Wasserstein-2 gradient flows, 
this notion involves barycentric systems
depending on a further base point and 
should take into account the 
rescaling of the squared distance function by 
the time step $\tau>0$.

\subsection{Barycentric Convexity and Related Notions}\label{sec:convexity-def}
Let $(\XXX,\dX)$ be a metric space. A geodesic in $\XXX$ is a curve
$\sfx:[0,1]\to \XXX$ such that 
\begin{equation}
	\label{eq:geodesic}
	\dX(\sfx(s),\sfx(t))=|t-s|\dX(\sfx(0),\sfx(1))\quad 
	\text{for every }s,t\in [0,1].
\end{equation}
Given $x_0,x_1\in \XXX$, 
we say that $x_\vartheta$ is a $\vartheta$-intermediate point
between $x_0$ and $x_1$ if 
\begin{equation*}
	\dX(\sfx_0,\sfx_\vartheta)=\vartheta\dX(\sfx_0,\sfx_1)\,\quad
	\dX(\sfx_\vartheta,\sfx_1)=(1-\vartheta)\dX(\sfx_0,\sfx_1)\,, \quad \text{ for all }\vartheta \in (0,1)\,.
\end{equation*}
We denote by 
$(\sfe_J^i)_{i=1}^J$ the canonical basis of $\R^J$ and by 
$\simplex J$ the $J-1$-dimensional 
probability simplex in $\R^J$:
\begin{equation}
	\label{eq:simplex}
	\simplex {J}:=
	\operatorname{co}\{\sfe_J^1,\sfe_J^2,\cdots,\sfe_J^J\}=
	\Big\{\aalpha=(\alpha_1,\cdots,\alpha_J)\in \R^J:
	\alpha_j\ge 0,\ \sum_{j=1}^J\alpha_j=1\Big\}.
\end{equation}
\begin{definition}[Systems of barycentric maps]
    \label{def:barycentric maps}
    A $J$-barycentric map in $\DDD\subset\XXX$ is
a map
$\baryname x:\simplex J\to \DDD$.
We set $\baryc xi:=\bary x{\sfe_J^i}$.

A system $\barysystem J$ of $J$-barycentric maps 
in $\DDD$ is a collection
of $J$-barycentric maps $\baryname x$
such that the evaluation map
\begin{equation}
	\baryevalname 
	:\barysystem J\to \DDD^J,
	\quad
	\baryeval{\baryname x}:=\big(\baryc x1,
	\baryc x2,\cdots 
	\baryc xJ\big)
	\quad\text{is surjective,}
\end{equation}
i.e.~for every collection of $J$ points in $\DDD$, denoted $\xx\in \DDD^J$,
there exists $\baryname x\in \barysystem J$ 
such that $x_j=\baryc xj=\bary x{\sfe_J^j}$, $j=1,\cdots,J$.

A system $\cB$ of barycentric maps
is the union of
$J$-systems $\barysystem J$, for every $J\in \N_{\ge 2}$.
\end{definition}
For an example of a 3-barycentric map for the sphere in two dimensions, see Figure~\ref{fig:barycentric_map}.
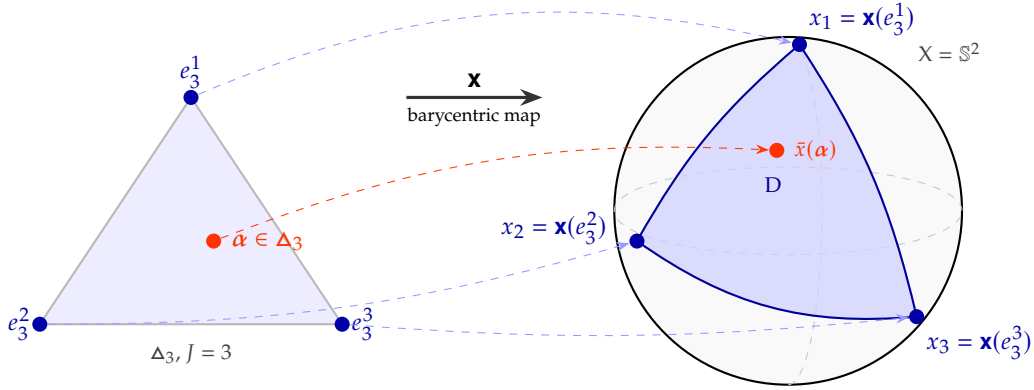
\begin{figure}
    \centering



\begin{tikzpicture}[
  >=Stealth,
  font=\small,
  vtx/.style   = {circle, fill, inner sep=2pt},
  vtxB/.style  = {vtx, fill=blue!65!black},
  vtxR/.style  = {vtx, fill=red!60!orange},
  darrow/.style= {->, dashed, thin, shorten >=3pt, shorten <=3pt},
]

\coordinate (E1) at ( 0.00,  3.00);   
\coordinate (E2) at (-2.00,  0.00);   
\coordinate (E3) at ( 2.00,  0.00);   
\coordinate (A)  at ( 0.30,  1.10);   

\fill[blue!7]            (E1)--(E2)--(E3)--cycle;
\draw[gray!55, thick]    (E1)--(E2)--(E3)--cycle;

\node[vtxB]                      at (E1) {};
\node[above,       blue!65!black] at (E1) {$e^1_3$};
\node[vtxB]                      at (E2) {};
\node[left,        blue!65!black] at (E2) {$e^2_3$};
\node[vtxB]                      at (E3) {};
\node[right,       blue!65!black] at (E3) {$e^3_3$};

\node[vtxR] at (A) {};
\node[right=3pt, red!60!orange] at (A) {$\boldsymbol{\alpha}\in\simplex{3}$};

\node[below=4pt, gray!50!black, font=\footnotesize] at (0,0)
  {$\simplex{3}$, $J=3$};

\draw[->, line width=1.2pt, gray!40!black]
  (2.85, 3.0) -- (4.65, 3.0)
  node[midway, above, black, font=\normalsize] {$\baryname x$}
  node[midway, below, black, font=\scriptsize] {barycentric map};

\def\Cx{7.90}  \def\Cy{1.50}
\def\R{2.30}

\draw[fill=gray!5, thick] (\Cx,\Cy) circle (\R cm);
\draw[dashed, gray!45, thin]
  (\Cx,\Cy) ellipse (\R cm and 0.58cm);
\draw[dashed, gray!45, thin]
  (\Cx, {(\Cy)-\R}) arc(-90:90:0.45 and \R);

\coordinate (X1) at (8.05, 3.70);   
\coordinate (X2) at (5.90, 1.10);   
\coordinate (X3) at (9.60, 0.10);   
\coordinate (XA) at (7.75, 2.30);   

\fill[blue!18, opacity=0.75]
  (X1) to[bend right=10] (X2)
       to[bend right=20] (X3)
       to[bend right=10] (X1);
\draw[blue!60!black, thick]
  (X1) to[bend right=10] (X2)
       to[bend right=20] (X3)
       to[bend right=10] (X1);

\node[blue!55!black, font=\footnotesize] at (7.70, 1.85) {$\DDD$};

\node[vtxB]                         at (X1) {};
\node[above right, blue!65!black]   at (X1) {$x_1=\bary{x}{e^1_3}$};
\node[vtxB]                         at (X2) {};
\node[yshift=5pt, left=8pt,  blue!65!black]   at (X2) {$x_2=\bary{x}{e^2_3}$};
\node[vtxB]                         at (X3) {};
\node[below right, blue!65!black]   at (X3) {$x_3=\bary{x}{e^3_3}$};

\node[vtxR] at (XA) {};
\node[right=3pt, red!60!orange, font=\footnotesize] at (XA)
  {$\bar{x}(\boldsymbol{\alpha})$};

\node[gray!50!black, font=\footnotesize, above right] at (9.50, 3.35)
  {$\XXX = \mathbb{S}^2$};

\draw[darrow, blue!40]        (E1) to[bend left=18]  (X1);
\draw[darrow, blue!40]        (E2) to[bend right=8]  (X2);
\draw[darrow, blue!40]        (E3) to[bend right=5]  (X3);
\draw[darrow, red!55!orange]  (A)  to[bend left=12]  (XA);

\end{tikzpicture}
    \caption{Example of a 3-barycentric map to the 2D-sphere: The barycentric map $\baryname x$ maps from the simplex $\simplex {3}$ to the space $\XXX = \S^2$. The simplex vertices $(e_3^j)_j$ map to points in $\XXX$, and interpolations in $\simplex {3}$ map to interpolations in $\XXX$.}
    \label{fig:barycentric_map}
\end{figure}
Notice that no continuity of the maps
$\aalpha\mapsto \bary x\aalpha$ is required: when a semicontinuity
property of the compositions with barycentric maps is needed (as in
Theorem~\ref{thm:primal_soln} below), it will be assumed explicitly.
It is useful to recall that 
in a Hilbert space $\XXX$ 
a function 
$\varphi:\XXX \to (-\infty,+\infty]$ 
    is $\kappa$-convex if for all $\aalpha \in \simplex J$, $J\ge 2$,
    and for every $(x_j)_{j=1}^J\in \XXX^J$
    \begin{align*}
        \varphi 
        \Big(\sum_{j=1}^J\alpha_j x_j\Big)
        \le \sum_{j=1}^J \alpha_j \varphi(x_j)
        -\frac{\kappa}{4}
        \sum_{j,k=1}^J\alpha_j \alpha_k 
        \|x_j-x_k\|^2\,.
    \end{align*}
It is well known that in order to check $\kappa$-convexity
it is sufficient to consider the case $J=2$.
A function $\varphi$ is quasi-convex
if 
\begin{displaymath}
    \varphi 
        \Big(\sum_{j=1}^J\alpha_j x_j\Big)
        \le 
        \max_{1\le j\le J}\varphi(x_j).
\end{displaymath} 
\begin{definition}[Convexlike, Barycentrically Convexlike and Quasi-convexlike Functions]
	\label{def:convexlike}
	Let $\varphi:
     \DDD\to \R$,
    $\DDD\subset \XXX$.
	\begin{enumerate}[\rm (1)]
		\item 
	We say 
	that $\varphi$ is \textbf{$\cvx$-convexlike} 
	with respect to 
	a system $\cB_2$ of 2-barycentric maps
    in $\DDD$ 
	if 
	for every 
	$\baryname x\in \barysystem 2$ and
	$\aalpha\in \simplex2$,
	\begin{equation}
		\label{eq:B2-convex}
		\varphi(\bary x\aalpha)\le 
		\alpha_1\varphi(\baryc x1)+\alpha_2\varphi(\baryc x2)
		-\frac\cvx2\alpha_1\alpha_2
		\dX^2(\baryc x1,\baryc x2).
	\end{equation}
	\item We say that $\varphi$ is \textbf{barycentrically $\cvx$-convexlike}
	with respect to a system $\barysystem{}$
	of barycentric maps in $\DDD$
	if for every integer $J\in \N_{\ge 2}$, 
	every $\baryname x\in \barysystem{}$, and
	every $\aalpha\in \simplex {J}$
	\begin{equation}
		\label{eq:convex}
		\varphi(\bary x\aalpha)\le 
	\sum_{j=1}^J\alpha_j\varphi(\baryc xj)-
	\frac \cvx 4\sum_{j,k=1}^J
	\alpha_j\alpha_k\dX^2(\baryc xj,\baryc xk)
	\end{equation} 
\item We say that $\varphi$ is \textbf{barycentrically quasi-convexlike}
	with respect to a system $\barysystem{}$
	of barycentric maps in $\DDD$
	if for every integer $J\in \N_{\ge 2}$, 
	every $\baryname x\in \barysystem{}$, and
	every $\aalpha\in \simplex {J}$
	\begin{equation}
		\label{eq:quasi-convex}
		\varphi(\bary x\aalpha)\le 
		\max \Big\{\varphi(\baryc xj):
	j\in \{1,\cdots,J\},
	\alpha_j>0\Big\}
	\end{equation} 
	\item
	A collection $(\varphi_a)_{a\in A}$ of functions
is \textbf{jointly} $\cvx$-convexlike (resp.~barycentrically $\cvx$ convexlike, 
barycentrically quasi-convexlike)
if every function
$\varphi_a$ is $\cvx$-convexlike 
(resp.~barycentrically $\cvx$-convexlike,
barycentrically quasi-convexlike)
with respect to a common system $\barysystem 2$
of $2$-barycentric maps
 (resp.~a 
common system $\barysystem {}$ of barycentric maps)
	\end{enumerate}	
\end{definition}
\begin{definition}[Radial Convexity]\label{def:radial-convexity}
    A bifunction $\sfa$ is $\xi$ {\em radially convex} if for every $x,y\in\DDD$ there is a family
	$(z_t)_{t\in[0,1]}$ in $\DDD$, with $z_0=x$ and $z_1=y$, such that
    \begin{equation}
		\label{eq:radial-convex}
		\sfa(z_t,x)\le (1-t)\,\sfa(x,x)+t\,\sfa(y,x)
		-\frac\xi2\, t(1-t)\,\dX^2(x,y),\qquad t\in[0,1],
	\end{equation}
\end{definition}
Radial convexity is barycentric convexity of $\sfa(\cdot, x)$ with respect to $\cB_2$ restricted to maps where one of the end points coincides with the second argument of $\sfa$. This is a weaker assumption than $y \mapsto \sfa(y,x)$ is $\xi$-barycentrically convex with respect to $\cB_2$.  $\xi$ radial convexity holds, in particular, whenever the maps
	$y\mapsto\sfa(y,x)$ are jointly $\xi$-convexlike \eqref{eq:B2-convex}
	(take $z_t=\bary z{(1-t),t}$ along a map joining $x$ to $y$).

\subsection{Examples}\label{sec:convexity-ex-all}
First, we present a number of examples for metric spaces and common choices of curves along which convexity can be verified (Section~\ref{sec:convexity-ex}). Then we discuss convexity properties for each of our  four running examples (Section~\ref{sec:running-ex-convexity}).

\subsubsection{Notions of convexity on common choices of metric spaces}\label{sec:convexity-ex}
\begin{example}[Geodesics]
\label{ex:convex1} 
	In a geodesic metric space $\XXX$
	there is a natural system $\cG_2$ of $2$-barycentric maps given by all the minimal, constant speed geodesic paths
    associated with 
	$\sfx:[0,1]\to\XXX$
	satisfying \eqref{eq:geodesic}.
    We just set
    $\bary x{(1-t),t}:=\sfx_t$.
    
	A function $\varphi$ is geodesically $\cvx$-convex if
	it is $\cvx$-convexlike 
	with respect to a subsystem $\cB_2\subset \cG_2$.	
    \end{example}
    \begin{example}[Convex interpolation in vector spaces]\label{ex:interp_vector_space}
	In a vector space $V$ 
	we can  use 
	the canonical convex combinations
	$\bary x{\vv;\aalpha}=
	\sum_{j=1}^J\alpha_jv_j$
	parametrized by the points $\vv
	=(v_1,\cdots,v_J)\in V^J$.
	Notice that 
	the family of shifted norms
	$\varphi_w(x):=\|x-w\|_{V}$, $w\in V$, 
	is jointly convex.
	If the corresponding squared family 
	$\psi_z(x):=\frac 12 \|x-z\|^2_{V}$ with $z\in V$ 
	is jointly $1$-convex then $V$ is a Hilbert space. See \cite{jordan_inner_1935} for how the parallelogram law ensures that a Banach space is induced by an inner product, $\<x,y>=\frac{1}{4}(\norm{x+y}^2 - \norm{x-y}^2)$.
    \end{example} 
	\begin{example}[Non Positively Curved metric spaces]
    \label{ex:NPC}
     A geodesic space $(\XXX,\dX)$ is a nonpositively curved (NPC) space in the sense of Alexandrov if 
     the collection of functionals
     $x\mapsto 
     \frac 12\dX^2(x,z)$, $z\in \XXX$, 
     is jointly $1$-convex 
     with respect to the 
     geodesic system $\cG_2$
     (see Example~\ref{ex:convex1}).
     
     Equivalently, 
     for every 
     $z,x_0,x_1\in \XXX$ 
     there exists a geodesic $\sfx$ 
     joining $x_0$ to $x_1$ such that 
    \begin{align}\label{eq:cn}
        \dX^2(\sfx(t),z) \le (1-t)\dX^2(x_0,z) + t\dX^2(x_1,z) - t(1-t) \dX^2(x_0,x_1),
        \quad t\in [0,1].
    \end{align}
See \cite{bridson_metric_1999,Laokul_2020} for more details. 
	It is worth noticing 
    that in a NPC space 
	for every $\xx=
    (x_1,\cdots,x_J)\in \XXX^J$ 
	we can define
	\begin{displaymath}
		\bary x{\xx;\aalpha}:=
		\argmin_{x\in \XXX}
		\sum_{j=1}^J \alpha_j\dX^2(x,x_j);
	\end{displaymath}
	The unique minimizer defines a 
	$J$-barycentric map 
	and the collections $\cG$ 
    of all these maps
	is a barycentric system.
	The family
	$x\mapsto \frac 12\dX^2(x,z)$, $z\in \XXX$,
	is jointly $1$-convex 
    with respect to $\cG$ 
    for every choice of $z\in \XXX$.

    Similarly, if $\varphi:\XXX\to (-\infty,+\infty]$
    is geodesically $\kappa$-convex 
    with respect to $\cG_2$ 
    then it is barycentrically
    $\kappa$-convex with respect to $\cG$.
    \end{example}

	\begin{example}[Linear interpolation in Wasserstein spaces]
		Let $\XX=\mathscr P_p(\XXX),$ $p\ge 1,$ 
	$\XXX$ being a complete and separable metric space. We endow $\XX$ with the $p$-Wasserstein metric $\Wass_p$. As for the first example,
	we can consider the collection 
	$\barysystem J$ of maps parametrized by 
	$\mmu=(\mu_1,\cdots,\mu_J)\in \XX^J$:
	$\bary x{\mmu;\aalpha}=\sum_{j=1}^J\alpha_j \mu_j$.
	The family $\mu\mapsto \Wass_p(\mu,\nu)$, 
     $\nu\in \XX$, is
	jointly convex.
	\end{example}
	\begin{example}[Coupling interpolation in Wasserstein spaces]
	    Let us suppose that 
	$\XXX$ is a separable Banach space, with 
	 $\XX=\mathscr P_p(\XXX),$
	 we can consider a reference standard Borel probability space
	 $(\Omega,\mathfrak B,\P)$ with an atomless Borel probability measure $\P.$
	We introduce the collection 
	$\barysystem J$ of maps parametrized by 
	$\xX=(X_1,\cdots,X_J)\in (L^p(\Omega,\P,\XXX))^J$:
	$\bary x{\xX;\aalpha}=\bigg(\sum_{j=1}^J\alpha_j X_j\bigg)_\sharp \P$.
	A function which is 
	jointly convex with respect to $\barysystem{}$ is also called
	\textbf{totally convex}.
	\end{example}
	\begin{example}[Generalized geodesics]
    \label{ex:generalized-geo}
	    Let us now 
	consider $\XX=\PP_2(H)$, $H$ a Hilbert space.
	For a given reference measure
	$\bar \mu\in \XX$ we consider the sets 
	$$\cO_{\bar\mu,J}:=
	\Big\{(\bar X,X_1,X_2,\cdots,X_J)\in (L^2(\Omega,\P,H))^{J+1}\nc:
	\bar X_\sharp \P=\bar \mu,\ 
	(\bar X,X_j)_\sharp\P\text{ optimal in }\XX\Big\}.$$
	For every $\xX\in \cO_{\bar\mu,J}$ 
	we set 
	$\bary x{\xX;\aalpha}:=\bigg(\sum_{j=1}^J\alpha_j X_j\bigg)_\sharp \P$.
	The collection $\barysystem {\bar\mu}$
	(depending on the reference measure $\bar\mu$) of these maps
	provides a natural extension
	to the notion of generalized geodesics.
	Notice that 
	$\mu\mapsto \frac 12 \Wass_2^2(\mu,\bar\mu)$ is
	$1$-convex with respect to 
	$\barysystem {\bar\mu}$
	and 
	the examples considered in 
	\cite[Sec. 9.2]{AGS08}
	are also barycentrically convexlike 
	with respect to $\barysystem{\bar\mu}$ 
	for every choice of $\bar\mu$. Indeed, the barycentric map $\bary x{\xX;\aalpha}$ is the push-forward of the convex combination $\sum_{j}\alpha_j X_j$ of the optimal maps $X_j$, which range in a convex subset of $L^2(\Omega,\P,H)$. Since these energies lift to functionals of the map, their convexity along two-point interpolations---the case established in \cite[Sec.~9.2]{AGS08}---is convexity in the usual sense, and therefore extends to the barycenter $\sum_{j}\alpha_j X_j$ of arbitrarily many maps.
	\end{example}
    \begin{example}[NNCC spaces]\label{ex:NNCC}
        According to 
        \cite{Leger-Todeschi-Vialard25}
        a metric space $(\XXX,\dX)$ 
        endowed with the cost function 
        $\sfc(x,y):=\dX^2(x,y)$
        has nonnegative cross curvature (NNCC)
        if 
        for every $\bar x\in \XXX$
        the collection of functions
        \begin{displaymath}
            x\mapsto \dX^2(x,\bar x)-
            \dX^2(x,y),\quad y\in \XXX,
        \end{displaymath}
        are jointly convexlike
        (with respect to a common system $\mathcal{G}_{\bar x,2}$
        of $2$-barycentric maps called \emph{variational $c$-segments}, which may depend on $\bar x$, and function as generalized geodesics).
        It is worth noticing that 
        in NNCC spaces 
        the collection
        $x\mapsto \frac 12\dX^2(x,\bar x)$
        are jointly $1$-convex
        with respect to 
        $\cG_{\bar x,2}$.  The space of probability measures endowed with the spherical Hellinger metric is an example of an NNCC space; see Example~\ref{ex:convexity-spherical-Hellinger}.    
    \end{example} 
    \begin{example}[Barycentric interpolation in product spaces]
      Consider the product space $\XXX=\PP_2(\R^d) \times \R^d$. Such a space was used for the system describing the dynamics of strategic agents interacting with algorithms, studied in \cite{conger_strategic_2023}. We can define a different interpolation for the component of $\PP_2(\R^d)$ than that in $\R^d$, resulting many different interpolation options. We can select the barycentric maps defined in 6. for $\PP_2(\R^d)$ above with reference measure $\rho\in\PP_2(\R^d)$  and the barycentric maps defined in 1. for $\R^d$,
    \begin{align*}
        \bary x{\aalpha} :=\bigg(\sum_{j=1}^J \aalpha_j \baryc x j ^{(1)}\bigg)_\sharp \P \times  \sum_{j=1}^J \aalpha_j \baryc x j ^{(2)}\,,  \qquad (\baryc x j , \rho)_\sharp \P \text{ optimal in }\Wass_2\,.
    \end{align*}
    Then $x \mapsto \frac{1}{2}\Wass_2^2(x^{(1)},\rho)+\frac 12\norm{x^{(2)}}^2$ is 1-convex with respect to
    $\cB$ parameterized by $\baryname x$ above.
    Alternatively, we could select 4. for $\PP_2(\R^d)$ and 3. for $\R^d$,
    \begin{align*}
        \bary x{\aalpha} :=\sum_{j=1}^J \aalpha_j \baryc x j ^{(1)}\times  \argmin_{x^{(2)}\in\R^d}\sum_{j=1}^J \aalpha_j \norm{\baryc x j ^{(2)}-x^{(2)}}^2_V\,,
    \end{align*}
     where $\norm{\cdot}_V$ is any Euclidean norm. However, while this interpolant results in 1-convexity for the family $x^{(2)} \mapsto \frac12 \norm{x^{(2)}-\bar x^{(2)}}^2$ for any $\bar x^{(2)}\in\R^d$, the family $x^{(1)} \mapsto \Wass_2^2(x^{(1)},\rho)$ for any $\rho\in\PP_2(\R^d)$ is not 1-convex, which precludes the conditions required for taking $\tau\to 0$ in the movement scheme.
    \end{example}
	
Given a barycentric system $\cB$, the convexity of $\alpha \mapsto \dX^2(\baryname x(\alpha),\bar x)$ for any $\baryname x \in \cB$ will be critical for taking the time step to zero in Section~\ref{sec:soln_to_discrete_EVI}; since the choice of $\cB$ determines the notion of convexity, selection of the appropriate set of interpolants is important. Additionally,  the selection of $\cB$ changes the convexity parameter of energy functionals. This is illustrated in the following example.
\begin{example}[Selection of $\cB$ and Energies] 
    Consider the setting $\XXX=\PP_2(\R^d)$ endowed with $\Wass_2$ and energies $E_1(\rho),E_2(\rho):\PP_2(\R^d)\to \R\cup \{+\infty\}$ defined by 
    \begin{align*}
        E_1(\rho)=\int V(z)\d\rho(z)\,, \quad E_2(\rho)=\begin{cases}
            \int \rho(z)\log \rho(\d z) & \text{if } \rho\in L^1(\R^d)\,, \\
            +\infty & \text{otherwise}
        \end{cases}\,,
    \end{align*}
     with $V\in C^2(\R^d,\R)$ satisfying $\nabla^2 V \succeq \cvx \Id$. We set $\baryname x$ according to Wasserstein-2 geodesics $\cG_2$, and $\baryname y$ according to linear interpolation.
     Checking the convexity conditions with respect to each of these barysystems,
     \begin{align*}
         E_1(\baryname x (\alpha)) &\le \alpha_1 E_1(x_1) + \alpha_2 E_1(x_2) -\frac{\cvx}{2} \alpha_1 \alpha_2 \Wass_2^2(x_1,x_2) \\
         E_1(\baryname y(\alpha)) &= \alpha_1 E_1(y_1) + \alpha_2 E_1(y_2) \\
          E_2(\baryname x (\alpha)) &\le \alpha_1 E_2(x_1) + \alpha_2 E_2(x_2)  \\
           E_2(\baryname y(\alpha)) &\le \alpha_1 E_2(y_1) + \alpha_2 E_2(y_2) \,.
     \end{align*}
     If $\cvx>0$, then for $E_1$ $\baryname x$ has a "better" (more positive) barycentric convexity parameter; if $\cvx<0$, then since $\baryname y$ results in 0-convexity, it is a better choice. $E_1$ is barycentrically quasi-convexlike with respect to both $\baryname x$ and $\baryname y$. For $E_2$, geodesics and linear interpolations both result in 0-barycentric convexity. 
\end{example}

\subsubsection{Running Examples: Convexity properties}\label{sec:running-ex-convexity}
For each example, consider a bifunction $\bif:\DDD\times \DDD \to \R$ where $\DDD \subset \XXX$. 
\begin{itemize}
    \item \textbf{Example (I): Monotone operator.}
Let $\XXX=H$ be a Hilbert space and 
$\DDD$ be a convex subset of $H$.
Let 
$\bif(x,y):=\langle A(x),y-x\rangle$, where
$A:\DDD\to H$ is an operator.
 Then 
 $y \mapsto \skd yx=\<A(x),y-x>$ is jointly barycentrically $0$-convexlike for the 
 canonical choice of $\cB$
 given by Example \ref{ex:interp_vector_space} in Section~\ref{sec:convexity-ex}. In fact, barycentric convexity is a trivial thanks to the linearity in $y$; this is a direct consequence of the linear structure provided by a Hilbert space.
 
\item \textbf{Example (II): Single species gradient flow.} Let $\XXX=\PP_2(H)$ and 
let $\varphi:\DDD\subset\PP_2 \to \R$ be an energy functional.
Let 
$\bif(x,y):=\varphi(y)-\varphi(x)$, where
 $\varphi:\DDD\to \R$ is a $\kappa$-displacement convex function, in the sense of McCann~\cite{McCann97}. Then $y \mapsto \skd yx$ is barycentrically $\cvx$-convexlike for the 
  choice of $\cB$
 given by geodesics, as in Example~\ref{ex:convex1} in Section~\ref{sec:convexity-ex}. For a number of common choices of $\varphi$, such as  the potential energy, entropy, and nonlocal energy, convexity along geodesics implies convexity along generalized geodesics, as in Example~\ref{ex:generalized-geo} in Section~\ref{sec:convexity-ex}; see \cite[Section 9.3]{AGS08} for details. 
 
\item \textbf{Example (III): Min-max gradient flows.} Let $\XXX^{(i)}=\PP_2(\R^{d_i})$ with $\XXX=\PP_2(\R^{d_1}) \times \PP_2(\R^{d_2})$, and let $F:\DDD \times \DDD \to \R$ be an energy functional. For $\skd yx = F(y^{(1)},x^{(2)})-F(x^{(1)},y^{(2)})$,  $y \mapsto \skd yx$ is $\cvx$-barycentrically convexlike with respect to $\cB$ if $F(\cdot,x^{(2)})$ and $-F(x^{(1)},\cdot)$ are both $\cvx$-barycentrically convex with respect to $\cB$ for all $x \in \DDD$, where $\cB$ can be chosen similarly to Example (II).

\item \textbf{Example (IV): Multispecies gradient flows.} Let $\XXX^{(i)}=\PP_2(\R^{d_i})$ and recall $\XXX = \prod_{i=1}^N \XXX^{(i)}$. 
In the coupled gradient flow setting,  $\skd yx = \sum_{i=1}^N F^{(i)}(y^{(i)},x^{(-i)})-F^{(i)}(x^{(i)},x^{(-i)}))$; any $\kappa$-barycentrically convexlike $y \mapsto \skd yx$ implies that $F^{(i)}(\cdot, \rho^{(-i)})$ is $\kappa$-barycentrically convexlike for each species $i$.  In the Wasserstein-2 space, two different choices of $J=2$ barysystems result in classical types of convexity. Let $\Gamma_o(\mu,\nu)$ be the set of Wasserstein-2 optimal transport maps with marginals $\mu$ and $\nu$. If we select $\cB=\cG_2$ to be Wasserstein-2 geodesics, then the interpolant $\bary \rho \aalpha$ is given by
\begin{align*}
    \bary {\rho^{(i)}} \aalpha = [\aalpha_1 z_1 + \aalpha_2 z_2]_\sharp \gamma^{(i)}\,, \qquad \text{where } \gamma^{(i)}\in\Gamma_o(\baryc {\rho^{(i)}} 1,\baryc {\rho^{(i)}} 2) \quad \text{and } 1-\alpha_1 = \alpha_2\,.
\end{align*}
Then the convexity inequality $\rho \to \skd \rho \mu$ is
\begin{align*}
     &\sum_i F^{(i)}(\bary  {\rho^{(i)}} \aalpha ,\mu^{(-i)} ) \\
     &\qquad \qquad \le  \aalpha_1  \sum_i F^{(i)} (\baryc {\rho^{(i)}} 1,\mu^{(-i)}) + \aalpha_2 
     \sum_i F^{(i)}(\baryc {\rho^{(i)}} 2,\mu^{(-i)})- \frac\kappa2  \aalpha_1 \aalpha_2 \sum_i \Wass_2^2(\baryc {\rho^{(i)}} 1, \baryc {\rho^{(i)}} 2)\,,
\end{align*}
which follows from $\kappa$ displacement convexity of $F^{(i)}(\cdot,\mu^{(-i)})$ in the sense of McCann \cite{McCann97}.
Later, we will see in Section~\ref{sec:ex-VMS} that
one must consider convexity with respect to generalized geodesics for the generating bifunction $\sfB$, since the Wasserstein-2 distance function is not geodesically 1-convex. In this case, the barysystem will be dependent on a reference measure $\bar \rho$:
\begin{align*}
    &\bary {\rho^{(i)}} \aalpha = [\aalpha_1 z_1 + \aalpha_2 z_3]_\sharp \hat \gamma^{(i)}\,, \quad \alpha_1=1-\alpha_2\,, \\
    &\quad \text{where } \hat \gamma\in \PP_2((\R^d)^3 ) \text{ satisfies }  [z_1,z_2]_\sharp \hat \gamma^{(i)} \in \Gamma_o(\baryname \rho_1^{(i)},\bar \rho^{(i)})\,, \ [z_2,z_3]_\sharp \hat\gamma \in \Gamma_o(\bar \rho^{(i)},\baryname \rho_2^{(i)})\,.
\end{align*}
Convexity along generalized geodesics implies geodesic convexity \cite[Remark 9.2.8]{AGS08}; the reverse implication does not necessarily hold. It does hold, however, for common choices  such as  potential, self-interaction, and internal energy functionals; see \cite[Section 9.3]{AGS08} for details.

\end{itemize}
}

\subsection{Semicontinuity of functions}\label{sec:semicontinuity}

It is useful to recall that 
in a Hilbert space $\XXX$ 
a function 
$\varphi:\XXX \to (-\infty,+\infty]$ is upper hemicontinuous
if its restriction to segments is upper semicontinuous.
\begin{definition}
	\label{def:hemicontinuity}
	Let $\barysystem{}'$ be a collection of barycentric maps in $\DDD$
	(e.g.~a system $\barysystem J$ of $J$-barycentric maps or a system
	$\barysystem{}$ of barycentric maps, as in
	Definition~\ref{def:barycentric maps}).
	We say that $\varphi:\DDD\to(-\infty,+\infty]$ is upper
	hemicontinuous with respect to $\barysystem{}'$ if for every
	$\baryname x\in \barysystem{}'$ the composition
	$\varphi\circ\baryname x:
	\aalpha\mapsto \varphi(\bary x\aalpha)$
	is upper semicontinuous in its simplex of definition.
\end{definition}

\begin{definition}
	We say that a function $\varphi:\DDD\to
    (-\infty,+\infty]$ 
    has 	
	complete sublevels if
	for every $c\in \R$ the corresponding sublevel set 
	$\{x\in \DDD:\varphi(x)\le c\}$ 
	is complete (possibly empty). 
	In particular $\varphi$ is lower  semicontinuous.
	
	$\psi:\DDD \to [-\infty,+\infty)$ has complete superlevels 
	if
	for every $c\in \R$ the corresponding superlevel set 
	$\{x\in \DDD:\psi(x)\ge c\}$ 	is complete. 
	In particular $\psi$ is upper semicontinuous.
\end{definition}
\begin{remark}
	\label{rem:trivial}
	If $\XXX$ is complete
    and 
     $\DDD$ is closed, 
     then $\varphi$ has complete sublevels 
	(resp.~superlevels) iff $\varphi$ 
	is lower (resp.~upper) semicontinuous.

    If $\XXX$ is complete but 
    $\DDD$ is not closed, 
    these conditions are in general stronger than
    lower (resp.~upper) semicontinuity: in fact, 
    they are 
    respectively equivalent to the following property
    \begin{equation}
        \label{eq:complet-sub}
        \begin{aligned}
            x_n\in \DDD,\ x\in \XXX,\ x_n\to x
        \text{ as }n\to\infty,\ 
        \ \varphi(x_n)\le c<\infty 
        \quad
        &\Rightarrow
        \quad
        x\in \DDD,\ \varphi(x)\le c,\\
        x_n\in \DDD,\ x\in \XXX,\ x_n\to x
        \text{ as }n\to\infty,\ 
        \ \varphi(x_n)\ge c>-\infty
        \quad
        &\Rightarrow
        \quad
        x\in \DDD,\ \varphi(x)\ge c.
        \end{aligned}
    \end{equation}
\end{remark}
\begin{lemma}
	\label{le:obvious}
	If $\varphi_\alpha:\XXX\to (-\infty,+\infty]$, $\alpha\in \rmA$, is a collection of 
	functions with complete sublevels, also
	$\varphi:=\sup_{\alpha\in \rmA}\varphi_\alpha$ has complete sublevels.
	Similarly, if
	$\psi_\alpha:\XXX\to [-\infty,+\infty)$, $\alpha\in \rmA$, is a collection of 
	functions with complete superlevels, then
	$\psi:=\inf_{\alpha\in \rmA}\psi_\alpha$ has complete superlevels.
\end{lemma}
\begin{proof}
	For every $c\in \R$ we have
	\begin{displaymath}
		\{x\in \XXX:\varphi(x)\le c\}=
		\bigcap_{\alpha\in \rmA}\{x\in \XXX:\varphi_\alpha(x)\le c\},
	\end{displaymath}
	so that all the sublevels are intersection of complete subsets of $\XXX$ 
	and are therefore complete as well. The same follows for $\psi$.
\end{proof}
If $\varphi:\XXX\to \overline\R=\R\cup\{\pm\infty\}$,
we set
\begin{equation}
	\Dom\varphi:=\big\{x\in \XXX:\varphi\in \R\big\}
\end{equation}
and we say that $\varphi$ is  \textbf{proper if $\Dom\varphi$ is not empty.}

\begin{proposition}
	\label{prop:minima}
	Let
	$\varphi: \DDD  \to (-\infty,+\infty] $ be a proper function
	with complete sublevels.
	If $\varphi$ is $\cvx$-convexlike
	for some $\cvx>0$ 
	with respect to some system $\barysystem 2$ of 
	$2$-barycentric maps in $\DDD$, 
	then 
	\begin{enumerate}
	\item there exists a unique minimizing point 
	$\bar x\in\DDD $ such that $\varphi(\bar x)=\inf_{x\in \DDD }\varphi(x)$;
	\item 
	every infimizing sequence
	$(x_n)_{n\in\N}$ in $\DDD $ converges to $\bar x$.
	\end{enumerate}
\end{proposition}
\begin{proof}
	\textbf{1.} follows by \cite[Lemma 2.4.8]{AGS08}.
	To show \textbf{2.},
	let $\bar \varphi:=\inf_{x\in \DDD }\varphi(x)\in \R.$ Fix any $\eps>0$.
	We argue as in the proof of Lemma 2.4.8 \cite{AGS08}
	using the estimate
	\begin{equation}
		\label{eq:inf-estimate}
		\varphi(x_0)\le \bar\varphi+\eps,\quad
		\varphi(x_1)\le \bar\varphi+\eps\quad\Rightarrow\quad
		\dX^2(x_0,x_1)\le \frac {8\eps}\cvx\,, \quad \text{for all } x_0,x_1 \in \DDD\,,
	\end{equation}
    which holds due to the convexity inequality, for any $\baryname x \in \cB$ such that $x_0 =\bary x {\sfe_2^1}$ and $x_1 = \bary x {\sfe_2^2}$,
\begin{align*}
    \bar \varphi \le \varphi(\bary x {1/2,1/2}) \le \frac{1}{2}\varphi(x_0) + \frac{1}{2}\varphi(x_1) - \frac{\cvx}{8}\dX^2(x_0,x_1) \quad 
    \Rightarrow  \quad \frac{\cvx}{8}
    \dX^2(x_0,x_1)  \le \varepsilon \,.
\end{align*}
	\eqref{eq:inf-estimate} immediately shows that 
	an infimizing sequence
	satisfies the Cauchy condition and therefore converges to 
	a point $\bar x\in\DDD$ 
    (since the sequence belongs to a sublevel of $\varphi$, which is complete) and $\bar x$ is a minimizing point for $\varphi$ 
	thanks to its lower semicontinuity. 
    \eqref{eq:inf-estimate}
	also shows that minimizing points are unique.	
\end{proof}
\subsection{Metric and Dini derivatives}\label{sec:derivatives-def}
\begin{definition}
    Let $\sfu:(a,b)\to \XXX$ be a curve. The upper/lower right/left metric derivatives and the metric derivative
of $\sfu$ (if they exist) are defined by
\begin{equation}
	\label{eq:urmder}
	\begin{alignedat}{2}
		\urmder{\sfu}{t}:=
		{}&\limsup_{h\downarrow0}\frac{\dX(\sfu(t+h),\sfu(t))}h,&\quad
		\lrmder{\sfu}{t}:=
		{}&\liminf_{h\downarrow0}\frac{\dX(\sfu(t+h),\sfu(t))}h,\\
				\ulmder{\sfu}{t}:=
		{}&\liminf_{h\downarrow0}\frac{\dX(\sfu(t),\sfu(t-h))}h,&\quad
		\llmder{\sfu}{t}:=
		{}&\limsup_{h\downarrow0}\frac{\dX(\sfu(t),\sfu(t-h))}h,\\
			 \mder{\sfu}{t}:={}&\lim_{h\to0}\frac{\dX(\sfu(t+h),\sfu(t))}{|h|}.
	\end{alignedat}
\end{equation}
\end{definition}
A curve $\sfu:(a,b)\to \XXX$ is absolutely continuous if there exists a positive function 
$m\in L^1(a,b)$ such that 
\begin{equation}
	\label{eq:AC}
	\dX(\sfu(t),\sfu(s))\le \int_s^t m(r)\,\d r\quad\text{for every }a\le s\le t\le b.
\end{equation}
Let $\Leb 1$ be the Lebesgue measure on $\R$. We recall the following general properties:
\begin{itemize}
\item  
If $\sfu:(a,b) \to \XXX$ 
is absolutely continuous, then 
the limit \eqref{eq:urmder} defining its metric derivative $|\dot \sfu|$
exists $\Leb 1$-a.e.~in $(a,b)$ 
and $\mder\sfu t\le m(t)$ for $\Leb 1$-a.e.~$t\in (a,b)$.
In this case,
\begin{equation*}
    d_X(\sfu(t),\sfu(s))\le \int_s^t \mder{\sfu}{\tau} \d \tau \quad \text{for every }a\le s\le t\le b\,.
\end{equation*}
\item Any absolutely continuous curve can be reparametrized to become 1-Lipschitz.
\end{itemize}

\begin{definition}
    For a real function $\zeta:(a,b)\to \R$ and $t\in (a,b)$ 
we consider the upper and lower right/left Dini derivatives
\begin{equation}
	\label{eq:Dini}
	\begin{aligned}
\Urd\zeta(t):={}&\limsup_{h\down0}\frac{\zeta(t+h)-\zeta(t)}h,&
	\Lrd \zeta(t):={}&\liminf_{h\down0}\frac{\zeta(t+h)-\zeta(t)}h\\
	\Uld\zeta(t):={}&\limsup_{h\down0}\frac{\zeta(t)-\zeta(t-h)}h,&
	\Lld \zeta(t):={}&\liminf_{h\down0}\frac{\zeta(t)-\zeta(t-h)}h
	\end{aligned}
\end{equation}
\end{definition}

\begin{lemma}
	\label{le:ulDini}
	
    Let $\zeta:(a,b)\to [0,\infty)$ be continuous
	and let $\upsilon:(a,b)\to \R$ be 
    a function satisfying
    \begin{enumerate}
        \item $\upsilon(t)=0$
        in $Z:=\{t\in (a,b): \zeta(t)=0\}$,
        \item 
        $\upsilon$ is upper semicontinuous
        in $(a,b)\setminus Z$,
        \item $\upsilon$ is locally bounded from above in $(a,b).$
    \end{enumerate}
    If 
    \begin{equation}
        \label{eq:lower-Dini}
        \Lld \zeta(t)\le \upsilon(t)
        \quad \text{for every }t\in (a,b)\setminus Z,
    \end{equation}
	then $\upsilon$ is locally integrable in $(a,b)$
    and for every 
    $a_0\in (a,b)$ 
    the function
	\begin{equation}
		\label{eq:auxiliary}
		\tilde\zeta(t):=\zeta(t)-\int_{a_0}^t \upsilon(s)\,\d s
		\quad\text{is nonincreasing and continuous.}
	\end{equation}
    In particular 
    \eqref{eq:lower-Dini} 
    is equivalent to the 
    inequality with the upper left Dini derivative
     \begin{equation}
        \label{eq:upper-Dini}
        \Uld \zeta(t)\le \upsilon(t)
        \quad \text{for every }t\in (a,b),
    \end{equation}
    and to the distributional inequality
	\begin{equation}
		\label{eq:ddistributional}
		\frac\d{\dt}\zeta\le \upsilon \quad\text{in }\mathscr D'(a,b),
	\end{equation}
    where $\mathscr D'(a,b)$ is the space of distributions dual to $C_c^\infty((a,b),\R)$.
    
	Finally, if $\zeta$ is locally absolutely continuous 
	(in particular if $\zeta$ is locally Lipschitz) 
	then the above conditions are equivalent to
	\begin{equation}
		\frac\d{\dt}\zeta(t)\le \upsilon(t)\quad \text{for a.e. }t\in (a,b)
	\end{equation}
\end{lemma}
\begin{proof}
	The proof 
	can be obtained as an adaptation of 
    \cite[Lemma A.1]{Muratori-SavareI}.
\end{proof}

\section{Equilibrium problems}\label{sec:equilibrium_problems}

Given a general bifunction $\sfa:\DDD\times\DDD\to\R$ defined on a subset $\DDD$ of a metric space $(\XXX,\dX)$, we can use it to define a primal and a dual equilibrium problem for elements in $\DDD$. In this section, we identify assumptions under which these equilibrium problems admit a solution. Our main application for the existence results presented in this section will be the generating bifunction $\sfB$ introduced in \eqref{introeq:generatingB}. Indeed, at each timestep $n$, the solution $X_\tau^n$ to the variational movement scheme (VMS) is precisely the solution to the primal equilibrium problem from Section~\ref{sec:primal} below, for the choice of $\sfa=\sfB$ in \eqref{introeq:generatingB} defined in terms of $\sfb$ that is driving \eqref{introeq:EVI}. And the solution to the dual variational inequality \eqref{introeq:dual} is precisely the solution to the dual equilibrium problem from Section~\ref{sec:dual} below, for the same choice $\sfa=\sfB$ in \eqref{introeq:generatingB}. In this way, the primal and dual equilibrium problems (Problem~\ref{prob:equilibrium-1} and Problem~\ref{prob:equilibrium-2}) are the key tools for ensuring existence of solutions to the discrete $\bif$-EVI
\eqref{introeq:bEVI}. This approach is strongly inspired
by the exposition of \cite{Baiocchi-Capelo84}. 
We can also apply the results in this section to the bifunction $\sfa=\sfb$ directly, thus providing a solution to the static game defined by $\sfb$.

Showing 
existence of solutions to the primal formulation (Section~\ref{sec:primal}) relies on the notion of barycentric quasi-convexity and
on compactness. Existence of solutions to the dual formulation (Section~\ref{sec:dual}) on the other hand requires the bifunction to also satisfy an interaction dissipativity condition and stronger barycentric convexity, however no compactness assumption is needed to solve the dual formulation. In Section~\ref{sec:saddle}, we discuss the case of antisymmetric bifunctions $\sfa$; in this case the primal and dual formulations are equivalent. Finally, in Section~\ref{sec:equilibria-ex}, we illustrate settings in which the primal and dual equilibrium problems can be solved.

\renewcommand{\bif}{\mathsf a}

\medskip 
We will always assume that $\bif$ is nonnegative on the diagonal:
\begin{equation}
	\label{eq:diagonal}
	\bif(x,x)\ge0\quad\text{for every }x\in \DDD.
\end{equation}

\subsection{Primal Problem}\label{sec:primal}
Let us first study a one-sided equilibrium problem induced by a bifunction $\bif:\DDD\times \DDD \to\R$.
\begin{problem}[Primal Equilibrium]
	\label{prob:equilibrium-1}
	We say that $x^*\in\DDD$
	is a solution of the Primal Equilibrium Problem if 
	\begin{equation}
	\label{eq:equilibrium}
	\tag{EP}
		\bif(y,x^*)\ge 0
		\quad \text{for every $y\in \DDD$.}
\end{equation}
\end{problem}
\begin{theorem}[Existence of primal equilibria]\label{thm:primal_soln}
    Let us suppose $\bif:\DDD \times \DDD \to \R$ satisfies \eqref{eq:diagonal} and that 
	\begin{enumerate}
		\item
		the functions $y\mapsto \bif(y,x)$, $x\in \DDD$,
		are jointly barycentrically quasi-convexlike with respect to
		a system of
		barycentric maps $\barysystem{}$ in $\DDD;$
		\item
		for every $y\in \DDD$ the function $x\mapsto \bif(y,x)$ is
		upper hemicontinuous with respect to $\barysystem{}$
		(Definition~\ref{def:hemicontinuity});
		\item
		the superlevels $\sigma^+(y):=\{x\in \DDD:
		\bif(y,x)\ge 0\}$ are
		closed in $\DDD$ and at least one of them is compact.
	\end{enumerate}
	Then Problem \ref{prob:equilibrium-1}
	has a solution.
\end{theorem}
Notice that Condition 2 holds, in particular, when the maps of
$\barysystem{}$ are continuous and, for every $y\in \DDD$, the function
$x\mapsto \bif(y,x)$ is upper semicontinuous.
To show this theorem, we will make use of the Knaster-Kuratowski-Mazurkiewicz (KKM) Lemma.
\begin{lemma}[KKM Lemma \cite{knaster1929beweis}] 
\label{le:KKM}
    Let $\simplex J$ be the $J-1$-dimensional simplex with $J$ vertices and let $C_1,\dots,C_J \subset \simplex J$ be a closed 
    KKM covering of $\simplex J$, i.e.~satisfying 
    \begin{equation}
        \label{eq:KKM}
        \operatorname{co}(\sfe_J^{k}:k\in B)\subset 
        \bigcup_{k\in B} C_{k}\quad\text{for every subset } 
        B\subset \{1,\dots,J\}.
    \end{equation}
    Then $\cap_{j=1}^J C_j \neq \emptyset$.
\end{lemma}

\begin{proof}[Proof of Theorem~\ref{thm:primal_soln}]
	It is sufficient to check that the collection
	of 
	sets $\sigma^+(y)$, $y\in \DDD$,
	have the finite intersection property.
	Let $y_1,\cdots,y_J\in \DDD$,
	let $\baryname y \in \barysystem{J}$
	such that $y_k=\bary y{\sfe_J^k}$,
	and consider the 
	sets 
	$$\cS(k)\subset \simplex J,\quad 
	\cS(k):=
	\Big\{\bbeta\in \simplex J:
	\bif( y_k,\bary y\bbeta )\ge0\Big\}.$$
	Thanks to Condition 2, the compositions
	$\bbeta\mapsto \bif(y_k,\bary y\bbeta)$ are upper semicontinuous, so
	that every $\cS(k)$ is closed. \nc
	Let us check that $\{\cS(k)\}_{k=1}^J$ satisfies the assumption of the 
	KKM lemma (see Lemma \ref{le:KKM} above):
	if  $B$
    is a subset of 
	 $\{1,\cdots, J\}$
	we need to prove that 
	$\operatorname{co}(\sfe_J^{k}:k\in B
    )\subset 
	\cup_{k\in B} \cS(k)$.
	By contradiction, if 
	$\simplex J\ni \aalpha =
	\sum_{k=1}^{ J} \alpha_k \sfe_J^{k}
	\not\in \cup_{k\in B} \cS(k)$, 
	for some 
	coefficients 
	$\alpha_k$ with $\alpha_k=0$ if $k\not\in B$, 
	then 
	$\bif ( y_{k},\bary y\aalpha )<0$
	for every $k\in B$
	and therefore, since
	$y\mapsto \bif( y,\bary y\aalpha )$ is
	quasi-convexlike 
    with respect to
	$\barysystem J$, 
	$\bif(\bary y\aalpha,\bary y\aalpha)
	\le \max_{k\in B}
	\bif (y_{k},\bary y\aalpha)<0$,
	a contradiction with \eqref{eq:diagonal}. Hence, $\{\cS(k)\}_{k=1}^J$ is a KKM covering and there exists $\beta^*\in \cap_{k=1}^J \cS(k)$. Let $x^*=
    \bary y {\beta^*} $. 
    Then $
    \bif ({y_k}, {x^*})  \ge 0$ for all $k\in [1,\dots,J]$,
    i.e.~$x^*\in \cap_{k=1}^J \sigma^+(y_k).$
\end{proof}
\begin{remark}
	\label{rem:coercive-primal}
    If $\sfa$ is $\xi$ radially convex according to Definition~\ref{def:radial-convexity} and $\sfa(x^*,x^*)=0$, then
	the primal inequality \eqref{eq:equilibrium} improves to a coercive
	one: any solution $x^*$ of Problem~\ref{prob:equilibrium-1} satisfies
	\begin{equation}
		\label{eq:coercive-primal-remark}
		\bif(y,x^*)\ge \frac\xi2\,\dX^2(y,x^*)
		\qquad\text{for every }y\in\DDD.
	\end{equation}
	Indeed, taking \eqref{eq:radial-convex} with $x=x^*$, using
	$\bif(x^*,x^*)=0$ and $\bif(z_t,x^*)\ge0$ (since $z_t\in\DDD$ and $x^*$
	solves \eqref{eq:equilibrium}), we get
	$0\le t\,\bif(y,x^*)-\frac\xi2\,t(1-t)\,\dX^2(y,x^*)$ for $t\in(0,1]$;
	dividing by $t$ and letting $t\downarrow0$ gives
	\eqref{eq:coercive-primal-remark}.
\end{remark}

\subsection{Dual problem for $\eta$-interaction dissipative bifunctions}\label{sec:dual}
We want to study now the relation between 
Problem \ref{prob:equilibrium-1} 
and its dual formulation.
\begin{problem}[Dual Equilibrium]
	\label{prob:equilibrium-2}
	We say that $y_*\in \DDD$
	is a solution of the Dual Equilibrium Problem if 
	\begin{equation}
	\label{eq:equilibrium'}
	\tag{EP'}
		\bif(y_*,x)\le 0
		\quad \text{for every $x\in \DDD$}.
\end{equation}
\end{problem}
Before showing  existence of solutions to the Dual Equilibrium Problem~\eqref{eq:equilibrium'}, we investigate under which conditions the primal and the dual problem are equivalent. 
\begin{remark}
	\label{rem:saddle}
	$x_*=x^*$ is a joint solution
	of the two Equilibrium Problems
	\ref{prob:equilibrium-1} and
	\ref{prob:equilibrium-2} 
    if and only if
	$x_*$ is a diagonal saddle point of $\bif$, i.e.
	\begin{equation}
	\label{eq:saddle}
	\tag{Saddle}
		\bif(y,x_*)\ge 0=\bif(x_*,x_*)\ge \bif(x_*,x)
		\quad \text{for every $x,y\in \DDD$.}
\end{equation}
Let us observe that we are looking for saddle points of $\bif$ 
on the diagonal of $\DDD\times \DDD$.
Therefore, if we introduce the marginal functions
\begin{equation}
	\label{eq:marginals}
	f(x):=\inf_{y\in \DDD}\bif(y,x),\quad
	g(y):=\sup_{x\in \DDD}\bif(y,x)
\end{equation}
which satisfy
\begin{equation}
	\label{eq:minmax1}
	\max_{x\in \DDD}f(x)=\max_{x\in \DDD}
	\inf_{y\in\DDD}\bif(y,x),\quad
	\min_{y\in \DDD}g(y)=\min_{y\in \DDD}
	\sup_{x\in\DDD}\bif(y,x),
\end{equation}
it is immediate to observe that 
\eqref{eq:diagonal} yields
\begin{equation}
	\label{eq:lu-bounds}
	g(y)\ge 0\ge f(x)\quad\text{for every }x,y\in \DDD.
\end{equation}
\eqref{eq:saddle} is thus equivalent to
\begin{equation}
	\label{eq:saddle2}
	g(x_*)=f(x_*).
\end{equation}
\end{remark}
The link between the primal and the dual formulation rests on a
dissipativity property of $\bif$: 
\begin{definition}\label{def:int-dissipative}
Given $\eta\in\R$, we say that $\bif$
is \emph{$\eta$-interaction dissipative} if
\begin{equation}
	\label{eq:bif-monotone}
	\bif(x,y)+\bif(y,x)\le \eta \dX^2(y,x)\quad\text{for every }
	x,y\in \DDD.
\end{equation}  
\end{definition}
Together with the standing assumption \eqref{eq:diagonal}, $\eta$-interaction dissipativity implies $\sfa(x,x)=0$ for all $x\in\DDD$. Hence, a vanishing diagonal is the standing assumption in Section~\ref{sec:soln_to_discrete_EVI}, see \eqref{eq:b-regular}.
When $\eta=0$, the passage from the primal to the dual problem is
immediate: if $x^*$ solves \eqref{eq:equilibrium}, then
\begin{displaymath}
	\bif(x^*,x)\le -\bif(x,x^*)\le 0\quad\text{for every }x\in \DDD,
\end{displaymath}
so that $x^*$ also solves \eqref{eq:equilibrium'}; the case of a
general $\eta$ will be discussed in
Proposition~\ref{prop:primal-implies-dual} below.
The 
following metric version of Minty's trick
shows conditions under which (dual $\Rightarrow$ primal), i.e. under which solutions
to Problem \ref{prob:equilibrium-2} 
are also solution to 
Problem~\ref{prob:equilibrium-1}.
\begin{lemma}[A metric Minty's lemma]
	\label{le:metric-Minty}
	Let
    $y_*$ be a solution
    of the dual equilibrium problem 
    \eqref{eq:equilibrium'}
    and let us suppose that 
	\begin{enumerate}
		\item [(1)] 
		the functions $y\mapsto \bif(y,x)$, $x\in \DDD$, 
		are jointly convexlike with respect to
		a 
        system of 
		$2$-barycentric maps $\barysystem{2}$ in $\DDD.$
		\item [(2)] for every $y\in \DDD$ the function $x\mapsto \bif(y,x)$ is
        upper hemicontinuous with respect to $\barysystem{2}$.
	\end{enumerate}
	Then $y_*$ is 
	also a solution of 
	Problem \ref{prob:equilibrium-1}.
\end{lemma}
\begin{proof}
	Let $y_*$ be a solution to 
	Problem \ref{prob:equilibrium-2}, let $y\in \DDD$
	and let $\baryname x\in \barysystem 2$
	satisfying 
	$\baryc x1=y_*$, $\baryc x2=y$.
	We set $x_t:=\bary x{(1-t),t}$.
	Since $y_*$ is a solution to Problem \ref{prob:equilibrium-2}
	we have
	\begin{displaymath}
		\bif(y_*,x_t)\le 0.
	\end{displaymath}
	On the other hand \eqref{eq:diagonal}
	yields 
	\begin{displaymath}
		\bif(x_t,x_t)\ge 0.
	\end{displaymath}
	Since the functions $y\mapsto \bif(y,x)$ are jointly convexlike,
	applying \eqref{eq:B2-convex} to $\baryname x$ with $\aalpha=(1-t,t)$
	gives
	\begin{displaymath}
		\bif(x_t,x_t)\le (1-t)\,\bif(y_*,x_t)+t\,\bif(y,x_t).
	\end{displaymath}
	Since $\bif(x_t,x_t)\ge0$, $\bif(y_*,x_t)\le0$ and $1-t>0$, we obtain
	$0\le t\,\bif(y,x_t)$, hence $\bif(y,x_t)\ge0$.
	Passing to the limit as $t\down0$ and 
	using the upper-semicontinuity of 
	$t\mapsto \bif(y,x_t)$ at $t=0$ we 
	obtain $\bif(y,y_*)\ge0$.
\end{proof}

Notice that the proof just requires 
upper semicontinuity of $t\mapsto \bif(y,\baryname x(1-t,t))$ at $t=0$.

\par\medskip
The coercive form of Remark~\ref{rem:coercive-primal} yields the converse
implication (primal $\Rightarrow$ dual) for every
$\eta$-interaction dissipative bifunction. Indeed, we will see in the next proposition that when $\eta\le0$, it holds without any
convexity assumption; when $\eta>0$ it requires the radial convexity
condition \eqref{eq:radial-convex} of Definition~\ref{def:radial-convexity} with $\xi\ge 2\eta$.
\begin{proposition}[Primal equilibria are dual]
	\label{prop:primal-implies-dual}
	Suppose $\bif$ satisfies the $\eta$-interaction dissipativity
	\eqref{eq:bif-monotone}, and let $x^*\in\DDD$ be a solution of the
	primal problem \eqref{eq:equilibrium}.
	\begin{enumerate}
		\item[(i)] If $\eta\le0$, then $x^*$ is also a solution of the
		dual problem \eqref{eq:equilibrium'}.
		\item[(ii)] If $\eta>0$ and, in addition, $\bif$ satisfies the
		radial convexity \eqref{eq:radial-convex} with $\xi\ge 2\eta$,
        then $x^*$
		is also a solution of the dual
		problem \eqref{eq:equilibrium'}.
	\end{enumerate}
	In either case $x^*$ is a diagonal saddle point \eqref{eq:saddle}.
\end{proposition}
\begin{proof}
	For every $x\in\DDD$, \eqref{eq:bif-monotone} for the pair $(x^*,x)$ gives
	\begin{displaymath}
		\bif(x^*,x)\le \eta\,\dX^2(x^*,x)-\bif(x,x^*).
	\end{displaymath}
	If $\eta\le0$, then $\bif(x,x^*)\ge0$ because $x^*$ solves
	\eqref{eq:equilibrium}, hence $\bif(x^*,x)\le\eta\,\dX^2(x^*,x)\le0$.
	If $\eta>0$, recall that interaction dissipativity (no matter the sign of $\eta$) together with the standing assumption \eqref{eq:diagonal} yield $\sfa(x,x)=0$ for all $x\in\DDD$; then \eqref{eq:coercive-primal-remark} of
	Remark~\ref{rem:coercive-primal} with $y=x$ gives
	$\bif(x,x^*)\ge\frac\xi2\,\dX^2(x,x^*)$, so
	\begin{displaymath}
		\bif(x^*,x)\le\Big(\eta-\frac\xi2\Big)\dX^2(x^*,x)\le0,
	\end{displaymath}
	because $\xi\ge2\eta$. In both cases $x^*$ solves
	\eqref{eq:equilibrium'}, and being a solution of both problems it is a
	diagonal saddle point by Remark~\ref{rem:saddle}.
\end{proof}

If in addition to the conditions in Proposition~\ref{prop:primal-implies-dual} the mild regularity assumptions of Lemma~\ref{le:metric-Minty} hold, then the primal and dual problems are in fact equivalent.

We now state our second existence result,
when $\bif$ satisfies two conditions which are more restrictive compared to the primal existence result Theorem~\ref{thm:primal_soln}: dissipativity  and uniform barycentric convexity. 
In this case, we can remove the compactness assumption and relax the 
upper semicontinuity condition as well.
For every $x\in \DDD$ we set 
	\begin{equation}
		\label{eq:soluno}
		\sigma_-(x):=\Big\{y\in \DDD:
		\bif(y,x)\le 0\Big\}.
	\end{equation}
\begin{theorem}[Existence of dual equilibria]
	\label{thm:equilibrium2}
	Let $\xi> 2\eta_+$ 
    and  
	let $\bif:\DDD\times \DDD\to \R$ 
	be
	a function
	satisfying 
	the following properties:
	\begin{enumerate}
		\item[(i)] $\bif$ is $\eta$-interaction dissipative
		according to \eqref{eq:bif-monotone};
		\item[(ii)] The maps
		$y\mapsto \bif(y,x)$, $x\in \DDD$, 
		have complete sublevels $\sigma_-(x)$ and are jointly
		barycentrically {$\xi$}-convexlike with respect to a system of 
		 barycentric maps $\barysystem{}$. 
	\end{enumerate}
	 Then Problem \ref{prob:equilibrium-2} has a unique solution $y_*$. 
\end{theorem}
\begin{proof}
It is not restrictive to assume $\eta\ge0$: the $\eta$-interaction dissipativity \eqref{eq:bif-monotone} implies the $\eta_+$-interaction dissipativity, and the assumption $\xi>2\eta_+$ only involves $\eta_+$.
Notice that applying 
\cite[Lemma 2.4.8]{AGS08} we 
deduce that 
\begin{equation}
    \label{eq:lower-bound}
    \text{for every }x_0\in \DDD
		 \text{ we have $\inf_{y\in \DDD}\bif(y,x_0)>-\infty.$}
\end{equation}

\textbf{Step 1.} \emph{The sets $\sigma_-(x)$, $x\in \DDD$,
	have the finite intersection property.}
	
	Let $x_1,x_2,\cdots,x_J\in \DDD$
    and let $\baryname x\in \barysystem J$ such that $x_k=\bary x{\sfe_J^k}$, $k=1,\cdots,J$. 
	We consider 
    the function $S(\aalpha):= 
    \sum_{j,k}\alpha_j\alpha_k \dX^2(x_j,x_k)$.
    For $k=1,\cdots,J$ we define the sets
    \begin{equation}
        \label{eq:sets}
        C(k):=\Big\{\aalpha\in \simplex J:
        \sum_j\alpha_j\bif(x_j,x_k) \le \tfrac \xi 4\, S(\aalpha)\Big\}.
    \end{equation}
    and we want to prove that the collection $(C_k)_k$ is a closed KKM covering according to the definition in \eqref{eq:KKM}. 
    
    Clearly, each $C_k$ is closed.
    Whenever $\aalpha\in \mathrm{co}(\sfe_J^k:k\in B)$ for some subset $B\subset \{1,\cdots,J\}$ (so that $\alpha_k=0$ if $k\not\in B$) we get 
    \begin{align*}
        \sum_{k\in B} \alpha_k\Big[\sum_j \alpha_j\bif(x_j,x_k) \Big]
        &= \sum_{j,k} \alpha_j\alpha_k \bif(x_j,x_k)
        =
        \frac 12 \sum_{j,k} \alpha_j\alpha_k 
        \Big[\bif(x_j,x_k)+\bif(x_k,x_j)\Big]\\
        &\le  \frac{\eta}{2}  \sum_{j,k} \alpha_j\alpha_k 
        \dX^2(x_j,x_k) = \frac \eta 2 S(\aalpha)\,\le \frac \xi 4 S(\aalpha)\,.
    \end{align*} 
    It follows that there exists $k\in B$ such that
    \begin{displaymath}
        \sum_j \alpha_j\bif(x_j,x_k) \le \frac \xi 4 S(\aalpha),
    \end{displaymath}
    i.e.~$\aalpha\in C(k)$, so that $\mathrm{co}(\sfe_J^k:k\in B)\subset \cup_{k\in B} C(k)$ and the KKM property is satisfied. By Lemma \ref{le:KKM} we deduce that there exists $\aalpha^*\in \cap_{k=1}^J C(k)$, so that 
    \begin{equation}
        \label{eq:sets2}
        \sum_j \alpha_j^*\bif(x_j,x_k) \le \tfrac \xi 4 S(\aalpha^*)\quad\text{for every }k=1,\cdots,J.
    \end{equation}  
    Setting $x_*:=\bary x{\aalpha^*}$, since $y\mapsto \bif(y,x_k)$ are jointly $\xi$ barycentrically convexlike we have 
    \begin{align*}
        \bif(x_*,x_k) \le \sum_j \alpha_j^*\bif(x_j,x_k) - \tfrac \xi 4 S(\aalpha^*)
        \le 0\quad\text{for every }k=1,\cdots,J\,,
    \end{align*}
   and therefore $x_*\in \bigcap_{k=1}^J \sigma_-(x_k)$.

	\smallskip\noindent
	\textbf{Step 2.} \emph{$\bigcap_{x\in \DDD}\sigma_-(x)$ is nonempty and contains a unique element}.

    Fix $x_0\in \DDD$.
    For every $\Omega\subset \DDD$ with $\Omega\ni x_0$, where $\inf_{y} \bif( y,{x_0})>-\infty$ from
    \eqref{eq:lower-bound}, 
	 we
	set 
	\begin{equation}
	\label{eq:marginal}
	\bif_\Omega(y):=\sup_{x\in \Omega}\bif(y,x),\quad
	\sigma_-(\Omega):=
	\Big\{y\in \DDD:\bif_\Omega(y)\le 0\Big\}\,.
	\end{equation}
	The function $\bif_\Omega$ is $\xi$-convexlike and uniformly bounded from below since
	$\bif_\Omega\ge \bif(\cdot,x_0)\ge \inf_y \skd y{x_0}>-\infty$, and together with \eqref{eq:bif-monotone}, we conclude that $\bif_\Omega(\cdot)$ is proper.
	Using the assumption that $\bif(\cdot,x)$ has complete sublevel sets, $\bif_\Omega$  has a unique minimizer $y_\Omega$ by Proposition~\ref{prop:minima}.
	In the particular case when $\Omega=\DDD$ 
	we recover the marginal 
	function $g$ of \eqref{eq:marginal} with $g(y)=\bif_\DDD(y)=
	\sup_{x\in \DDD}\bif(y,x)
	\ge \bif(x,x)\ge0$. 
	Therefore $\sigma_-(\DDD)$ 
	is nonempty if $\min \bif_\DDD=0$.
	
		Let $\Lambda=\{\Omega\subset \DDD\, : \, x_0\in \Omega\}$, restricted such that $\Omega_1,\Omega_2\in \Lambda$ implies $\Omega_1 \subseteq \Omega_2$ or $\Omega_2 \subseteq \Omega_1$, be  the collection of 
	 subsets with a finite number of elements in $\DDD$, containing $x_0$, and ordered by inclusion.
	$\Lambda$ is a directed set. 
	Every set $\sigma_-(\Omega)$, $\Omega\in \Lambda$,
	 is nonempty by Step 1 because each $\Omega\in\Lambda$ is finite and it is 
	 included in $\sigma_-(x_0)$, which is a complete set.
	We have $m_\Omega:=\bif_\Omega(y_\Omega)\le 0$ since $\sigma_-(\Omega)$ is nonempty by Step 1;
	in particular $y_\Omega\in \sigma_-(\Omega)$
	for every $\Omega\in \Lambda$.

	The claim of Step 2 then follows if we show that $\Omega\mapsto y_\Omega$ is a Cauchy net.
	We know that $\Omega\mapsto m_\Omega$ is bounded from above and increasing: if $\Omega_1 \subseteq \Omega_2$, then
    \begin{align*}
        m_{\Omega_1} =\min_{y\in \DDD} \bif_{\Omega_{1}}(y) \le \min_{y \in \DDD}\bif_{\Omega_2}(y) = m_{\Omega_{2}}\,.
    \end{align*}
    From this, $\lim_{\Omega\in \Lambda}m_\Omega=\sup_{\Omega\in \Lambda}m_\Omega=:m\le 0$, where the limit is taken in the sense of nets. 
	Given $\eps>0$ we can find $\Omega^\eps\in \Lambda$ such that 
	$m_{\Omega^\eps}\ge m-\eps$
	and therefore
	for every $\Omega_1,\Omega_2$ containing $\Omega^\eps$ 
	with $y_k:=y_{\Omega_k}$ we have
	$$ m_{\Omega^\eps}\le \bif_{\Omega^\eps}(y_k)\le
	\bif_{\Omega_k}(y_k)=m_{\Omega_k}\le m
	\le m_{\Omega^\eps}+\eps\,.
	$$
    We now use that $y \mapsto \bif( y,x)$  is $\xi$-barycentrically convexlike, selecting $J=2$ and test points $y_1,y_2$:
    \begin{align*}
        \bif ({\bary y{0.5,0.5}}, x) \le \frac12 \bif( { {y_1}}, x) + \frac12 \bif({ {y_2}}, x) - \frac{\xi\nc}{8} \dX^2(y_1,y_2)\,.
    \end{align*}
    Using that for all $x\in \Omega^\eps$, $\bif ({y_1},x) \le m_{\Omega^\eps} + \eps$ and $\bif({y_2},x) \le m_{\Omega^\eps} + \eps$, along with $\bif ( {\bary y{0.5,0.5}} ,x) \ge \bif_{\Omega^\eps}(y_{\Omega^\eps})$, we have  
    \begin{align*}
        \dX^2(y_1,y_2) \le \frac{8\eps}\xi\,.
    \end{align*}
    The above estimate in combination with $y_\Omega\in\sigma_-(x_0)$, a complete set, allows us to conclude that $\Omega\mapsto y_\Omega$ is a Cauchy net converging to a limit $\bar y \in \sigma_-(x_0)$. Since $\Omega \cup \{x\} \mapsto y_{\Omega \cup \{x\}}$ is a subnet of $\Omega\mapsto y_\Omega$ for any $x\in \DDD$, if $\Omega \mapsto y_\Omega$ converges to $\bar y$, then $\Omega \cup \{x\} \mapsto y_{\Omega\cup \{x\}}$ also converges to $\bar y$. Lower semicontinuity of $\bif( \cdot, x)$ follows from the assumption that $\sigma_-(x)$ are complete. Therefore,
	for every $x\in \DDD$, 
	$\bif(\bar y,x)\le \liminf_{\Omega\in \Lambda}\bif(y_\Omega,x)
	\le \liminf_{\Omega\in \Lambda}
	\bif_\Omega(y_\Omega)=m\le 0$
	so that $\bar y\in \bigcap_{x\in \DDD}\sigma_-(x)$. Hence, $\bar y$ solves the Dual Equilibrium Problem \ref{prob:equilibrium-2}.
\end{proof}
We can now collect the consequences of the previous results for the primal
problem.
\begin{corollary}[Alternative existence result for the primal problem]
	\label{cor:primal-consequences}
	Under the assumptions of Theorem~\ref{thm:equilibrium2}, let $y_*$ be
	the unique solution of the dual problem \eqref{eq:equilibrium'}.
	\begin{enumerate}
		\item[(i)] The primal problem \eqref{eq:equilibrium} has at most
		one solution, which must coincide with $y_*$.
		\item[(ii)] If in addition the upper-hemicontinuity
		condition~(2) of Lemma~\ref{le:metric-Minty} holds, then $y_*$
		solves \eqref{eq:equilibrium} as well; it is then the unique
		solution of both problems and the unique diagonal saddle point
		\eqref{eq:saddle} of $\bif$.
	\end{enumerate}
\end{corollary}
\begin{proof}
	The barycentric $\xi$-convexity assumed in
	Theorem~\ref{thm:equilibrium2} implies the radial convexity
	\eqref{eq:radial-convex} (Remark~\ref{rem:coercive-primal}), with
	$\xi>2\eta_+$.

	\emph{(i)} If $x^*$ solves \eqref{eq:equilibrium}, then
	Proposition~\ref{prop:primal-implies-dual} shows that $x^*$ solves
	\eqref{eq:equilibrium'} (the case $\eta\le0$ needs no convexity, the
	case $\eta>0$ uses $\xi>2\eta$); by uniqueness of the dual solution
	$x^*=y_*$.

	\emph{(ii)} The convexity condition~(1) of
	Lemma~\ref{le:metric-Minty} follows from the $\xi$-convexity, so if
	condition~(2) holds as well, Lemma~\ref{le:metric-Minty} shows that
	$y_*$ solves \eqref{eq:equilibrium}. Together with \emph{(i)} and
	Remark~\ref{rem:saddle}, $y_*$ is the unique solution of both
	problems and the unique diagonal saddle point.
\end{proof}
\nc
\subsection{Saddle points of antisymmetric bifunctions}\label{sec:saddle}
The antisymmetric case when
	\begin{equation}
		\label{eq:antisymmetry}
		\bif(x,y)+\bif(y,x)=0
		\quad\text{for every }x,y\in \DDD,
	\end{equation}
    deserves a specific analysis, since we can considerably simplify the convexity assumptions on $\bif$. 
 Notice in particular that $\bif$ is $0$-interaction dissipative according to \eqref{eq:bif-monotone}.
It is clear that in this case
Problems
	\ref{prob:equilibrium-1} and 
	\ref{prob:equilibrium-2} are equivalent 
	and any (common) solution
	is a saddle point according to the Remark
	\ref{rem:saddle}.	
\begin{theorem}[Existence of primal and dual equilibria under antisymmetry]
	\label{thm:antisymmetric}
	Let $\bif:\DDD\times \DDD\to \R$ 
	be 
	a function
	satisfying 
	the antisymmetric identity
	\eqref{eq:antisymmetry}.
	If 
	\begin{quote}
		the maps
		$y\mapsto \bif(y,x)$, $x\in \DDD$, 
		have complete sublevels 
		and are jointly 
        $\xi$-convexlike 
         with respect to some family 
        of 
        $2$-barycentric maps $\cB_2$
        for some $\xi>0.$
        \end{quote}
	then Problems
	\ref{prob:equilibrium-1} and 
	\ref{prob:equilibrium-2} have a unique saddle solution $x^*=y_*$.
\end{theorem}
\begin{proof}
	We adapt the argument of \cite[Chapter 1, Theorem 1.1]{Brezis73}
	and \cite[Chap. 6, Theorem 2.1]{Ekeland-Temam74} due to Shiffman.
    
    Let us first observe that 
    for $x_0\in \DDD$ 
		 we have $
		 -I:=\inf_{y\in \DDD}\bif(y,x_0)>-\infty$
         thanks to Proposition \ref{prop:minima}. 

	Let us define $
	g(y)=\sup_{x\in \DDD}\bif(y,x)$
	as in \eqref{eq:marginals}.
	Notice that \eqref{eq:antisymmetry}
	yields $f(x)=-g(x)$.
	We observe that the function
	$y\mapsto g(y)$ is $\xi$-convexlike,
	it has complete sublevels, and it is 
	bounded from below since
	$g(y)\ge \bif(y,y)=0$.
	Moreover
	\begin{displaymath}
		\bif(x_0,x)=
		-\bif(x,x_0)\le I\quad\text{for every }x\in \DDD,
	\end{displaymath}
	so that 
	$g(x_0)\le I<+\infty$ and therefore
	$g$ is not identically $+\infty$
	it attains its minimum at some point $\bar y\in \DDD.$ We want to show that 
	$g(\bar y)=0$.
	Let us fix  
	$y\in \DDD$ and 
	let $\baryname y\in \barysystem 2$ 
	satisfy $\baryname y_1=\bar y$,
	$\baryname y_2=y$;
	for every $\vartheta \in (0,1)$ 
	we set $y_\vartheta =\bary{y}{(1-\vartheta),\vartheta}$.
	Since $g(y_\vartheta)=
	\sup_x \bif(y_\vartheta,x)$,
	and $x\mapsto -\bif(y,x)=\bif(x,y)$ 
	is $\xi$-convexlike with respect to $\barysystem 2$, 
	there exists a unique point $x_\vartheta$
	such that 
	$g(y_\vartheta)=\bif(y_\vartheta,x_\vartheta)$.
	We have
	\begin{align*}
		g(\bar y)&\le g(y_\vartheta)=
		\bif(y_\vartheta,x_\vartheta)
		\stackrel{(1)}\le 
		(1-\vartheta)
				\bif(\bar y,x_\vartheta)
				+\vartheta \bif(y,x_\vartheta)
				\stackrel{(2)}\le 
				(1-\vartheta)
			g(\bar y)
				+\vartheta \bif(y,x_\vartheta)
	\end{align*}
	where we used the convexity inequality 
	in (1)
	and the minimality property of $g$ in (2) ($g(\bar y)\ge 
	\skd{\bar y}{x}$ for every $x\in \DDD$).
	We deduce that 
	\begin{equation}
		\label{eq:1bis}
		g(\bar y)\le \bif(y,x_\vartheta)
		\quad\text{for every }\vartheta\in (0,1).
	\end{equation}
	\sloppy As $\vartheta\downarrow0$ 
	$y_\vartheta$ is a minimizing 
	sequence for $g$ since
	$g(y_\vartheta)\le 
	(1-\vartheta)g(\bar y)+\vartheta g(y)$ 
	so that 
	$\liminf_{\vartheta\down0}g(y_\vartheta)=g(\bar y)$.
	We deduce that $y_\vartheta\to \bar y$ as $\vartheta\down0$.
	
	We claim that $x_\vartheta$ 
	converges to 
	the unique maximizer $\bar x$ 
	of $x\mapsto \bif(\bar y,x)$
	as $\vartheta\down0$. In fact
	\begin{align*}
			(1-\vartheta)\bif (\bar y,x_\vartheta)+
		\vartheta g(y)&\ge 
		(1-\vartheta)\bif(\bar y,x_\vartheta)+
		\vartheta \bif(y,x_\vartheta)
		\ge \bif (y_\vartheta,x_\vartheta)
		=g(y_\vartheta)
		\\&\ge \bif(y_\vartheta,x)\quad\text{for every }x\in \DDD,
	\end{align*}
	so that 
	\begin{displaymath}
		\limsup_{\vartheta\downarrow0}
		\bif(\bar y,x_\vartheta)\ge 
\limsup_{\vartheta\downarrow0}
\bif(y_\vartheta,x)\ge \bif(\bar y,x)
		\quad\text{for every }x\in \DDD,
	\end{displaymath}
	and therefore
	\begin{displaymath}
		\limsup_{\vartheta\downarrow0}
		\bif(\bar y,x_\vartheta)\ge 
		\sup_{x\in \DDD} \bif(\bar y,x)=g(\bar y)
	\end{displaymath}
	showing that $x_\vartheta$ is a maximizing sequence
	for 
	the $\xi$-concavelike function 
	$x\mapsto \bif(\bar y,x)$. Proposition \ref{prop:minima}
	then shows that $x_\vartheta\to \bar x$.
	We can eventually pass to the limit in \eqref{eq:1bis}
	as $\vartheta\downarrow0$ using the lower semicontinuity of $\bif$ w.r.t.~$x$
	obtaining
	\begin{displaymath}
			g(\bar y)=\bif(\bar y,\bar x)\le 
	\bif (\bar x,y)
		\quad \text{for every }y\in \DDD,
	\end{displaymath}
	so that $g(\bar y)\le f(\bar x)\le 0$.
	
	To show uniqueness, if $(y',y')$ 
	is another saddle point for $\bif$,
	then 
    \begin{equation*}
	g(y')=\bif(y',y') =0 
    \end{equation*}
    yields
	$g(y')=g(\bar y)=0$ 
	showing that $\bar y=y'$ thanks to the 
	$\xi$-convexity of $g$.

    The fact that $\bar y$ is also
    a solution of the primal problem follows
    by Lemma 
    \ref{le:metric-Minty},
    since the map 
    $x\mapsto \bif(y,x)=
    -\bif(x,y)$ is upper semicontinuous.
\end{proof}

\subsection{Examples}\label{sec:equilibria-ex}
We illustrate settings in which the primal and dual equilibrium problems can be solved, 
using our
four running examples.

\subsubsection{Running Examples: Existence of equilibria}\label{sec:ex-existence}

In some settings, showing the primal problem properties is easier due to compactness. In other settings, showing compactness may be nontrival, and in those cases, with additional dissipativity information, we can check the requirements for existence of solutions to the dual problem instead.

\begin{itemize}
\item \textbf{Example (I): Monotone operator.} 
In the closed convex subset $\DDD$ of the Hilbert space $H$ 
consider, with $\xi\ge 0$,
$$\bif({y},{x}) := \tfrac \xi2\big(|y|^2-|x|^2\big)+\<\sfA(x),y-x>,\quad 
x,y\in \DDD.$$ 
The primal equilibrium is a solution $x^*$ of
the system of variational inequalities:
\begin{align*}
    \<\sfA(x^*),y-x^*> +\tfrac \xi2|y|^2\ge \tfrac \xi 2|x^*|^2 \quad \text{for every } y \in \DDD\,,
\end{align*}
whereas the dual problem is solved for $y_*\in \DDD$ satisfying
\begin{align*}
     \<\sfA(x),y_*-x> +\tfrac \xi2|y_*|^2\le \tfrac \xi2|x|^2 \quad \text{for every } x \in \DDD\,.
\end{align*}
For solving the primal problem, note that $y\mapsto \sfa(y,x)$ is jointly convex if and only if $\xi\ge 0$.
We are in the classic framework of 
\cite{Baiocchi-Capelo84}. 
If the sublevels
$\{x\in \DDD  : \tfrac \xi2|x|^2+ \langle \sfA(x),x-y\rangle\le c\}$ are compact for every $y\in \DDD,c\in \R$, then the primal problem has a solution
(cf. \cite[Theorem 10.3]{Baiocchi-Capelo84}). This is exactly the statement of our Theorem~\ref{thm:primal_soln}. For the sublevels to be bounded, we require $\psi(x):=\tfrac \xi2|x|^2+ \langle \sfA(x),x-y\rangle$ to be coercive; this follows if $\sfA$ is $\eta$-interaction dissipative with $\xi>2\eta$. For the sublevels to be closed, we require $\psi$ to be lower semicontinuous, which is a consequence of the monotonicity ($\eta\le 0$) and hemicontinuity of $\sfA$. In summary, to show existence of equilibria via compactness (Theorem~\ref{thm:primal_soln}), we require $\xi\ge 0$, $\sfA$ monotone ($\eta\le 0$) and hemicontinuous, and $\xi>2\eta$. We will see that the compactness requirement for the primal problem imposes similar conditions as the existence proof via the dual problem.

Turning to the dual problem,
Theorem \ref{thm:equilibrium2} and Corollary \ref{cor:primal-consequences} 
correspond to 
the classical existence theory for variational inequalities,
involving monotonicity ($\eta\le 0$) and hemicontinuity of $\sfA$, 
i.e.~
     for all $x_0,x_1\in\DDD$, $v\in H$
     the map $t \mapsto  
     \langle \sfA((1-t)x_0+t x_1)),v\rangle$ is 
     continuous in $[0,1]$. 
 Here, we need to assume $\xi>0$ in order to satisfy $\xi>2\eta_+$, but using Theorem~\ref{thm:equilibrium2}, every $\eta\in\R$ is admissible (no requirement that $\sfA$ is monotone).

\item \textbf{Example (II): Single species gradient flow.} When $\bif({y},{x}) := \varphi(y)-\varphi(x)$ with $x,y\in \DDD$, 
the primal (dual) equilibrium is a solution $x^*$ ($y_*$ resp.) of:
\begin{align*}
   \varphi(y)\ge \varphi(x^*) \quad \text{for every } y \in \DDD\,, \quad   \varphi(y_*)\le 
     \varphi(x)\quad \text{for every } x \in \DDD\,.
\end{align*}
If the sublevels of $\varphi$, 
$\{x\in \DDD  \ : \
\varphi(x)\le c\}$,
are compact for every $c\in \R\nc$,
then the primal problem has a solution \cite[Lemma 2.4.8]{AGS08}. This is exactly the statement of our Theorem~\ref{thm:primal_soln}.

For the dual problem, Theorem~\ref{thm:equilibrium2} requires
that $\DDD$ is closed, 
    $\varphi$ is lower semicontinuous and 
    $\xi$-uniformly convex for some $\xi>0$. We automatically have $\eta=0$ interaction dissipativity, because $\bif(y,x) + \bif(x,y)=0$.
    
Concerning the $\xi$-uniform convexity of $\varphi$, 
it is worth noticing that 
for the application to evolution problems
the function $\varphi$ is typically perturbed by 
the squared norm: for $\tau>0$ 
the variational movement scheme involves
$\varphi_\tau(x):=\frac 1{2\tau} \|x-\bar x\|^2+\varphi(x)$,
which is $1/\tau$-uniformly convex 
whenever $\varphi$ is convex.

\item \textbf{Example (III): Min-max gradient flows.} The choice of $\bif (y,x)=F(y^{(1)},x^{(2)})-F(x^{(1)},y^{(2)})$ is an example of an antisymmetric bifunction.  To apply Theorem~\ref{thm:antisymmetric}, it is sufficient for $F$ to be lower semicontinuous in the first argument, upper semicontinuous in the second argument, and $\xi>0$ convex-concave, that is, $x \mapsto F(x,y)$ is $\xi$-barycentrically convexlike and $y \mapsto -F(x,y)$ is $\xi$ barycentrically convexlike, both with respect to some family of barycentric maps $\cB_2$. Then both a unique primal and dual solution exists. The antisymmetric structure allows us to relax the convexity assumption in the more general Theorem~\ref{thm:equilibrium2} from barycentric maps interpolating among $J\ge 2$ points to just two points. 

\item \textbf{Example (IV): Multispecies gradient flows.} 
In the multispecies Wasserstein-2 gradient flow example, 
we consider the conditions under which the primal equilibrium problem has a solution $x^*$:
\begin{align*}
    \sum_{i=1}^N   F^{(i)}(y^{(i)},x^{*(-i)})-F^{(i)}(x^{*(i)},x^{*(-i)}) \ge 0 \quad \text{for every } y \in \DDD\,, 
\end{align*}
and conditions under which the dual problem has a solution $y_*$:
\begin{align*}
    \sum_{i=1}^N F^{(i)}(y^{(i)}_*,x^{(-i)})-F^{(i)}(x^{(i)},x^{(-i)}) \le 0\,, \quad \text{for every } x \in \DDD\,.
\end{align*}
The solution to the primal problem is immediately a Nash equilibrium for the game $(F^{(i)})$, whereas the solution for the dual problem $y_*$ is not necessarily a Nash equilibrium. For the primal problem, Theorem~\ref{thm:primal_soln} requires $y\mapsto \sum_{i=1}^N F^{(i)}(y^{(i)},x^{(-i)})-F^{(i)}(x^{(i)},x^{(-i)})$ to be jointly barycentrically quasi-convexlike with respect to a system of continuous barycentric maps $\cB$ in $\DDD$; this holds, in particular, whenever each $F^{(i)}(\,\cdot\,,x^{(-i)})$ is barycentrically $\cvx$-convexlike for some $\cvx\ge0$, since barycentric convexity---unlike quasi-convexity---is additive, so that the sum is barycentrically $\cvx$-convexlike and hence barycentrically quasi-convexlike. 
The second condition is harder to show; the superlevels
\begin{align*}
    \sigma^+(y) = \left\{x \in\DDD\ : \ \sum_{i=1}^N F^{(i)}(y^{(i)},x^{(-i)})-F^{(i)}(x^{(i)},x^{(-i)}) \ge 0 \right\}
\end{align*}
need to be closed in $\DDD$ and at least one compact.

If instead, we aim to show existence of a solution to the dual equilibrium problem, we check if the requirements of Theorem~\ref{thm:equilibrium2} are satisfied.
The $\eta$ interaction dissipativity and $\cvx$-barycentric convexity hold if $(F^{(i)})$ are $(\eta,\cvx)$-monotone according to Definition~\ref{def:monotonicity_zeroth_order}, with $\lambda=\eta-\cvx$.  We also require that
$y\mapsto \sum_{i=1}^N F^{(i)}(y^{(i)},x^{(-i)})-F^{(i)}(x^{(i)},x^{(-i)})$ have complete sublevel sets, which holds if $\XXX$ is complete, $\DDD$ is closed, and each $F^{(i)}(\cdot,x^{(-i)})$ is lower semicontinuous and bounded from below.
\end{itemize}

\section{Variational Movement Scheme (VMS)}\label{sec:soln_to_discrete_EVI}
\renewcommand{\bif}{\mathsf b}

In this section we introduce the \textit{variational movement scheme} (VMS) generated by a driving bifunction $\sfb$, a generalization of the minimizing movement scheme for gradient flows: it produces a sequence $(X_\tau^n)_{n\ge 0}$ which approximates the solution of the continuous-time evolution \eqref{introeq:EVI} for a given time step $\tau>0$. In Section~\ref{sec:driving-b} we introduce the generating bifunction $\sfB$ associated with $\sfb$ and we show how its structural properties and, above all, its convexity can be derived from corresponding \emph{joint} properties of $\sfb$ and of the squared distance along a common system of barycentric maps depending on the reference point. In Section~\ref{sec:resolvent} we apply the variational equilibrium results of Section~\ref{sec:equilibrium_problems} to $\sfB$ and solve a single step of the scheme---its \emph{resolvent}---using three different approaches (compactness, dissipativity, antisymmetry). In Section~\ref{sec:VMS} we iterate the resolvent to define the variational movement scheme and we derive the discrete Evolution Variational Inequality satisfied by its iterates. There we also introduce the local and global slopes of $\sfb$, we show that under suitable convexity assumptions the global slope can be recovered from the local one, and we prove that---without any dissipativity assumption---all the iterates of the scheme lie in the domain of the global slope.  For dissipative bifunctions which are quadratically bounded from below, one step of the scheme approximates every initial point at rate $\sqrt\tau$, so that the domain of the global slope is automatically dense in $\DDD$ (Corollary~\ref{cor:slope-density}). We conclude the section with examples illustrating the interplay between the convexity of $\sfb(\cdot,x)$ and $\dX^2(\cdot,x)$ with respect to various families of barycentric maps $\cB$ (Section~\ref{sec:ex-VMS}), including our four running examples.

\subsection{The generating bifunction $\sfB$}\label{sec:driving-b}
Let $(\XXX,\dX)$ be a metric space, let $\DDD\subset \XXX$, and let
$\sfb:\DDD\times \DDD\to \R$ be the driving bifunction of
\eqref{introeq:EVI}. We will always assume that $\sfb$ vanishes on the
diagonal:
\begin{equation}
	\label{eq:b-regular}
	\skd xx=0
	\qquad\text{for every }x\in \DDD.
\end{equation}
For instance, nonegativity on the diagonal~\eqref{eq:diagonal} together with $\eta$ interaction dissipativity \eqref{eq:bif-monotone} directly implies \eqref{eq:b-regular}.
For $\cvx\in\R$, we also set
\begin{equation}
	\label{eq:s-variant1}
	\lskd yx:=\skd yx-\frac{\cvx}{2}\dX^2(x,y),
	\qquad
	x,y\in \DDD.
\end{equation}
For every time step $\tau>0$ and every base point $\bar x\in \DDD$ we
introduce the generating function
$\sfB(\tau,\bar x;\cdot,\cdot):\DDD\times \DDD\to \R$, already announced
in \eqref{introeq:generatingB}, defined by
\begin{equation}
	\label{eq:generating}
	\sfB(\tau,\bar x;y,x):=
	\frac 1{2\tau}\dX^2(\bar x,y)
	-
	\frac 1{2\tau}\dX^2(\bar x,x)
	+
	\skd yx,\qquad y,x\in \DDD.
\end{equation}
A single step of the scheme, starting from the base point
$\bar x=X^{n-1}_\tau$, consists in solving the primal equilibrium
problem \eqref{eq:equilibrium} for the bifunction
$\sfa=\sfB(\tau,\bar x;\cdot,\cdot)$. The existence results of
Section~\ref{sec:equilibrium_problems} will be applied through the
following properties of $\sfB$, which are directly inherited from
$\sfb$.
\begin{lemma}[Structural properties of $\sfB$]
	\label{le:B-structure}
	Let $\sfb$ satisfy \eqref{eq:b-regular}, let $\tau>0$, $\bar x\in \DDD$,
	and let $\sfB$ be defined by \eqref{eq:generating}.
	\begin{enumerate}
		\item[\rm(1)] For every $x,y\in \DDD$
		\begin{equation}
			\label{eq:B-symmetrized}
			\sfB(\tau,\bar x;x,x)=0,\qquad
			\sfB(\tau,\bar x;y,x)+\sfB(\tau,\bar x;x,y)=\skd yx+\skd xy.
		\end{equation}
		In particular, for every $\eta\in\R$,
		$\sfB(\tau,\bar x;\cdot,\cdot)$ satisfies the $\eta$-interaction
		dissipativity \eqref{eq:bif-monotone} if and only if $\sfb$ does,
		and $\sfB(\tau,\bar x;\cdot,\cdot)$ is antisymmetric
		\eqref{eq:antisymmetry} if and only if $\sfb$ is.
		\item[\rm(2)] If the sublevels of $\sfb(\cdot,x)$ are complete, i.e.
		\begin{equation}
			\label{eq:b-sublevels}
			\text{for every }x\in \DDD\quad
			\text{the map }y\mapsto \skd yx\ \text{has complete sublevels,}
		\end{equation}
		then for every $x\in \DDD$ the map
		$y\mapsto \sfB(\tau,\bar x;y,x)$ has complete sublevels as well.
	\end{enumerate}
\end{lemma}
\begin{proof}
	Claim (1) follows by a direct computation, since the terms
	$\frac 1{2\tau}\dX^2(\bar x,\cdot)$ cancel in the symmetrized sum.

	(2) The map $y\mapsto \sfB(\tau,\bar x;y,x)$ is the sum of the lower
	semicontinuous map $y\mapsto \skd yx$, of the continuous map
	$y\mapsto \frac 1{2\tau}\dX^2(\bar x,y)$, and of a constant, so that
	its sublevels are closed in $\DDD$. Since $\dX^2(\bar x,\cdot)\ge 0$,
	for every $c'\in\R$
	\begin{displaymath}
		\Big\{y\in \DDD:\sfB(\tau,\bar x;y,x)\le c'\Big\}\subset
		\Big\{y\in \DDD:\skd yx\le c'+\tfrac 1{2\tau}\dX^2(\bar x,x)\Big\};
	\end{displaymath}
	the latter set is complete by \eqref{eq:b-sublevels} and a closed
	subset of a complete set is complete.
\end{proof}
The convexity of the maps $y\mapsto \sfB(\tau,\bar x;y,x)$ required by
the existence results of Section~\ref{sec:equilibrium_problems} is a
more subtle matter, since $\sfB$ couples $\sfb$ with the squared
distance from the base point $\bar x$. The natural assumption is the
\emph{joint} convexity of the two families along common interpolating
curves in $\DDD$, which are allowed to depend on $\bar x$.
The simplest formulation
involves pairs of points:
\begin{quote}
    Given $\cvx\in\R$, 
    for every $\bar x, x_0,x_1\in \DDD$ there exists a
    family $(z_t)_{t\in[0,1]}$ in $\DDD$, with $z_0=x_0$ and $z_1=x_1$,
such that for every $t\in[0,1]$
\begin{align}
	\label{eq:gg-dist-2}
	\tag{$\mathrm{C}^{\mathsf d}_{\bar x,2}$}
	&\dX^2(\bar x,z_t)\le (1-t)\,\dX^2(\bar x,x_0)+t\,\dX^2(\bar x,x_1)
	-t(1-t)\,\dX^2(x_0,x_1),
	\\
	\label{eq:gg-bif-2}
	\tag{$\mathrm{C}^{\sfb}_{\cvx,\bar x,2}$}
	&\skd{z_t}x\le (1-t)\,\skd{x_0}x+t\,\skd{x_1}x
	-\frac\cvx2\,t(1-t)\,\dX^2(x_0,x_1)
	\quad\text{for every }x\in \DDD.
\end{align}
\end{quote}
Setting
$\bary z{(1-t),t}:=z_t$ the two conditions say precisely that 
\begin{quote}
    for every base point
$\bar x\in \DDD$, the
    function $x\mapsto \tfrac12\dX^2(\bar x,x)$ and the maps
    $y\mapsto \skd yx$, $x\in \DDD$, are jointly $1$-convexlike and
    $\cvx$-convexlike with respect to a common system
    $\barysystem{\bar x,2}$ of $2$-barycentric maps in $\DDD$, possibly
    depending on $\bar x$
(Definitions~\ref{def:barycentric maps} and~\ref{def:convexlike}).
\end{quote}
The KKM-based results of Section~\ref{sec:equilibrium_problems} involve
arbitrary finite collections of points, and thus the stronger
barycentric versions of the previous conditions: 
\begin{quote}
for every base point
$\bar x\in \DDD$, there exists a system $\barysystem{\bar x}$ of barycentric maps
in $\DDD$ (Definition~\ref{def:barycentric maps}), possibly depending
on $\bar x$, such that 
\begin{align}
	\label{eq:gg-dist}
	\tag{$\mathrm{C}^{\mathsf d}_{\bar x}$}
	&\text{the map $x\mapsto \tfrac12\dX^2(\bar x,x)$ is barycentrically
	$1$-convexlike w.r.t.~$\barysystem{\bar x}$,}
	\\
	\label{eq:gg-bif}
	\tag{$\mathrm{C}^{\sfb}_{\cvx,\bar x}$}
	&\text{the maps $y\mapsto \skd yx$, $x\in \DDD$, are 
	barycentrically $\cvx$-convexlike w.r.t.~$\barysystem{\bar x}$.}
\end{align}
\end{quote}
Recalling \eqref{eq:convex}, condition \eqref{eq:gg-dist} explicitly
reads
\begin{equation}
	\label{eq:generalized_metric_convexity}
	\dX^2(\bar x,\bary x\aalpha)\le
	\sum_{j=1}^J \alpha_j\, \dX^2(\bar x,\baryc xj)
	-\frac1{2}\sum_{j,k=1}^J \alpha_j\alpha_k\, \dX^2(\baryc xj,\baryc xk)
\end{equation}
for every $\baryname x\in \barysystem{\bar x}$ and
$\aalpha\in \simplex J$. Restricting \eqref{eq:gg-dist} and
\eqref{eq:gg-bif} to the $2$-barycentric maps of $\barysystem{\bar x}$
interpolating between $x_0$ and $x_1$, and choosing
$z_t:=\bary z{(1-t),t}$, we clearly recover \eqref{eq:gg-dist-2} and
\eqref{eq:gg-bif-2}.
\begin{proposition}[Convexity of the generating bifunction]
	\label{prop:B-convexity}
	Let $\bar x\in \DDD$, $\cvx\in\R$, $\tau>0$, let $\sfB$ be defined by
	\eqref{eq:generating}, and let $\xi:=\dfrac 1\tau+\cvx$.
	\begin{enumerate}
		\item[\rm(1)] If the two-point conditions \eqref{eq:gg-dist-2}
		and \eqref{eq:gg-bif-2} hold, then the maps
		$y\mapsto \sfB(\tau,\bar x;y,x)$, $x\in \DDD$, are jointly
		$\xi$-convexlike with respect to the same system
		$\barysystem{\bar x,2}$:
        for every $x_0,x_1\in \DDD$ the family
		$(z_t)_{t\in[0,1]}$ of \eqref{eq:gg-dist-2}, \eqref{eq:gg-bif-2}
		satisfies, for every $t\in[0,1]$,
		\begin{equation}
			\label{eq:B-two-convex}
			\tag{$\mathrm{C}^{\sfB}_{\xi,2}$}
			\begin{aligned}
				\sfB(\tau,\bar x;z_t,x)\le{}&
				(1-t)\,\sfB(\tau,\bar x;x_0,x)+t\,\sfB(\tau,\bar x;x_1,x)
				\\&-\frac\xi2\,t(1-t)\,\dX^2(x_0,x_1)
				\qquad\text{for every }x\in \DDD.
			\end{aligned}
		\end{equation}
		In particular $\sfB(\tau,\bar x;\cdot,\cdot)$ satisfies the
		radial convexity \eqref{eq:radial-convex} with the same $\xi$.
		\item[\rm(2)] If a system $\barysystem{\bar x}$ of barycentric
		maps in $\DDD$ satisfies \eqref{eq:gg-dist} and
		\eqref{eq:gg-bif}, then the maps
		$y\mapsto \sfB(\tau,\bar x;y,x)$, $x\in \DDD$, are jointly
		barycentrically $\xi$-convexlike with respect to
		$\barysystem{\bar x}$: for every
		$\baryname z\in \barysystem{\bar x}$ and $\aalpha\in \simplex J$
		\begin{equation}
			\label{eq:B-bary-convex}
			\tag{$\mathrm{C}^{\sfB}_{\xi}$}
			\sfB(\tau,\bar x;\bary z\aalpha,x)\le
			\sum_{j=1}^J \alpha_j\,\sfB(\tau,\bar x;\baryc zj,x)
			-\frac{\xi}{4}\sum_{j,k=1}^J
			\alpha_j\alpha_k\,\dX^2(\baryc zj,\baryc zk)
			\quad\text{for every }x\in \DDD.
		\end{equation}
	\end{enumerate}
\end{proposition}
\begin{proof}
	(1) Given $x_0,x_1\in \DDD$, let $(z_t)_{t\in[0,1]}$ be a family
	satisfying \eqref{eq:gg-dist-2} and \eqref{eq:gg-bif-2}. Multiplying
	\eqref{eq:gg-dist-2} by $\frac 1{2\tau}$, adding \eqref{eq:gg-bif-2},
	and subtracting the constant $\frac 1{2\tau}\dX^2(\bar x,x)$ from
	both sides, we obtain \eqref{eq:B-two-convex}. The choice $x_0:=x$,
	$x_1:=y$, together with $\sfB(\tau,\bar x;x,x)=0$, yields the radial
	convexity \eqref{eq:radial-convex} of
	Remark~\ref{rem:coercive-primal}.

	(2) Let us fix $x\in \DDD$, $J\in \N_{\ge 2}$,
	$\baryname z\in \barysystem{\bar x}$, and $\aalpha\in \simplex J$.
	Multiplying \eqref{eq:generalized_metric_convexity}, written for
	$\baryname z$, by $\frac 1{2\tau}$, adding the inequality
	\eqref{eq:convex} for the maps $y\mapsto \skd yx$, and subtracting
	the constant $\frac 1{2\tau}\dX^2(\bar x,x)$ from both sides, we
	obtain \eqref{eq:B-bary-convex}.
\end{proof}
Clearly \eqref{eq:B-bary-convex} implies \eqref{eq:B-two-convex} (it is
sufficient to restrict to the $2$-barycentric maps of the system), and
\eqref{eq:B-two-convex} implies the radial convexity
\eqref{eq:radial-convex} of $\sfB(\tau,\bar x;\cdot,\cdot)$. In the
next section, \eqref{eq:B-bary-convex} and \eqref{eq:B-two-convex}
will also be used as assumptions on $\sfB$ for general values
$\xi\in\R$ and systems $\barysystem{}$, $\barysystem{2}$ possibly
depending on $\tau$ and $\bar x$, without any reference to
\eqref{eq:gg-dist} and \eqref{eq:gg-bif}.
\begin{remark}[Metric space curvature]\label{rm:metric_space_curvature}
	Condition \eqref{eq:gg-dist} constrains the geometry of
	$(\XXX,\dX)$ rather than the driving bifunction, and it captures the
	part of the convexity of $\sfB$ which is uniform with respect to the
	time step. Indeed, suppose that for a base point $\bar x\in \DDD$, a
	system $\barysystem{}$ independent of $\tau$, and a vanishing sequence
	$\tau_n\downarrow 0$, the maps $y\mapsto \sfB(\tau_n,\bar x;y,x)$,
	$x\in \DDD$, are jointly barycentrically
	$(\tau_n^{-1}+\cvx)$-convexlike with respect to $\barysystem{}$.
	Writing \eqref{eq:convex} for these maps along a fixed
	$\baryname z\in \barysystem{}$ and $\aalpha\in \simplex J$, multiplying
	by $2\tau_n$, and letting $n\to\infty$, the contribution of the
	real-valued $\sfb$ vanishes in the limit and we recover exactly
	\eqref{eq:generalized_metric_convexity}.
	For $J=2$ and geodesic interpolants,
	\eqref{eq:generalized_metric_convexity} is the NPC comparison
	inequality \eqref{eq:cn} of Example~\ref{ex:NPC}: it holds with
	equality along linear interpolants in Hilbert (and all flat) spaces,
	and it fails in positively curved spaces. This is the reason for
	allowing the system $\barysystem{\bar x}$ to depend on the base point
	and to collect non-geodesic interpolants: in the Wasserstein space
	$(\PP_2(H),\Wass_2)$, which is positively curved, the generalized
	geodesics based at $\bar x$ of Example~\ref{ex:generalized-geo}
	satisfy \eqref{eq:gg-dist}, whereas geodesics do not; NNCC spaces
	(Example~\ref{ex:NNCC}) exhibit the same mechanism.
\end{remark}

\subsection{The single-step resolvent}\label{sec:resolvent}
Throughout this section $\sfb:\DDD\times \DDD\to \R$ satisfies
\eqref{eq:b-regular} and $\sfB$ is the generating bifunction defined by
\eqref{eq:generating}. A single step of the variational movement scheme
with step size $\tau>0$, starting from the base point $\bar x\in \DDD$,
consists in finding a solution $x_\tau\in \DDD$ of the primal
equilibrium problem
\begin{equation}
	\label{eq:B-primal}
	\sfB(\tau,\bar x;y,x_\tau)\ge 0
	\qquad\text{for every }y\in \DDD;
\end{equation}
we will also consider the corresponding dual problem
\begin{equation}
	\label{eq:B-dual}
	\sfB(\tau,\bar x;x_\tau,x)\le 0
	\qquad\text{for every }x\in \DDD.
\end{equation}
Combining Lemma~\ref{le:B-structure} and
Proposition~\ref{prop:B-convexity} with the results of
Section~\ref{sec:equilibrium_problems}, we obtain the solvability of
\eqref{eq:B-primal} along three different routes: via compactness and
the primal problem (Theorem~\ref{thm:VMS-compact}), via interaction
dissipativity and the dual problem (Theorem~\ref{thm:discrete-SEVI}),
and via antisymmetry (Theorem~\ref{thm:VMS-antisym}).
\begin{theorem}[One step via compactness]
	\label{thm:VMS-compact}
	Let $\tau>0$ and $\bar x\in \DDD$, and assume that
	\begin{enumerate}
		\item[\rm(i)] the maps $y\mapsto \sfB(\tau,\bar x;y,x)$,
		$x\in \DDD$, are jointly barycentrically quasi-convexlike with
		respect to a {\em continuous} system $\barysystem{\bar x}$ 
        of barycentric maps in $\DDD$;
		\item[\rm(ii)] for every $y\in \DDD$ the map $x\mapsto \skd yx$ is
		upper semicontinuous; 
		\item[\rm(iii)] there exists $y_0\in \DDD$ such that the
		superlevel set
		$\sigma^+(y_0)=\big\{x\in \DDD:\sfB(\tau,\bar x;y_0,x)\ge 0\big\}$
		is compact.
	\end{enumerate}
	Then the primal problem \eqref{eq:B-primal} has at least one solution.
\end{theorem}
\begin{proof}
	We apply Theorem~\ref{thm:primal_soln} to
	$\sfa:=\sfB(\tau,\bar x;\cdot,\cdot)$, which vanishes on the diagonal
	by Lemma~\ref{le:B-structure}(1). Note that the first assumption of Theorem~\ref{thm:primal_soln} is precisely (i) in the statement above.
	For every $y\in \DDD$ the map $x\mapsto \sfB(\tau,\bar x;y,x)$ is the
	sum of the upper semicontinuous map $x\mapsto \skd yx$ by assumption (ii) above, of the
	continuous map $x\mapsto -\frac 1{2\tau}\dX^2(\bar x,x)$, and of a
	constant, hence it is upper semicontinuous; since the system
	$\barysystem{}$ is continuous, the second assumption of
	Theorem~\ref{thm:primal_soln} follows, and every superlevel
	$\sigma^+(y)$ is closed in $\DDD$. By (iii) one of them is compact,
	so Theorem~\ref{thm:primal_soln} provides a solution of
	\eqref{eq:B-primal}.
\end{proof}

    Theorem~\ref{thm:VMS-compact} requires no dissipativity of $\sfb$,
	and uniqueness may fail. 
	Condition (i) holds, in particular, when
	\eqref{eq:B-bary-convex} is satisfied with some $\xi\ge 0$,
	e.g.~under \eqref{eq:gg-dist} and \eqref{eq:gg-bif} with
	$\xi=\tau^{-1}+\cvx\ge 0$ thanks to Proposition~\ref{prop:B-convexity}. The following lemma provides a sufficient condition for Condition (iii) in Theorem~\ref{thm:VMS-compact} to hold uniformly for $\tau\in(0,\tau_o)$.

\begin{lemma}\label{lem:compact}
Assume
    \begin{equation}
        \label{eq:b-growth}
        \text{for some $y_0\in \DDD$ and $\tau_o>0$,
	the superlevel of $x\mapsto \skd{y_0}x-\tfrac1{2\tau_o}\dX^2(y_0,x)$
	is compact.}
    \end{equation}
 Then Condition (iii) in Theorem~\ref{thm:VMS-compact} holds for every $\bar x\in \DDD$ and every
	$\tau\in(0,\tau_o)$.
\end{lemma}
\begin{proof}
    By \eqref{eq:generating}, $\sigma^+(y_0)$ is a
	superlevel set of $x\mapsto \skd{y_0}x-\tfrac 1{2\tau}\dX^2(\bar x,x)$. From Young's inequality for $p,q\in\R$ and $\theta>0$,
    \begin{align*}
        (p+q)^2 \le (1+\theta) p^2 + (1+\theta^{-1})q^2\,,
    \end{align*}
	applied with $q=\dX(y_0,\bar x)$, $p=\dX(\bar x,x)$, and $\theta=\frac {\tau_o}\tau-1$ and dividing by $1/2\tau_o$, we have
 \begin{align*}
        \frac1{2\tau_o}\dX^2(y_0,x)\le \tfrac{1}{2\tau}\dX^2(\bar x,x)+C\,, \quad C=C(y_0,\bar x, \tau_o, \tau):=\frac{1+\theta^{-1}}{2\tau_o}\dX^2(y_0,\bar x)\,.
    \end{align*}
    Hence, 
	\begin{displaymath}
		\sigma^+(y_0)\subset
		\Big\{x\in \DDD:\skd{y_0}x-\tfrac1{2\tau_o}\dX^2(y_0,x)
		\ge c\Big\}
	\end{displaymath}
    for any small enough constant $c\in\R$. Then, using  \eqref{eq:b-growth}, it follows that Condition (iii) holds for every $\bar x\in \DDD$ and every
	$\tau\in(0,\tau_o)$.
\end{proof}

A notable class of examples for \eqref{eq:b-growth} is provided by
bifunctions admitting a \emph{Lyapunov decomposition}: 
\begin{quote}
There exist a
function $V:\DDD\to\R$ with compact sublevels (in particular $V$ is
lower semicontinuous and bounded from below),
constants $C_1,C_2\ge 0$, and a bifunction $\bif':\DDD\times \DDD\to\R$
upper semicontinuous with respect to its second variable,
 such
that for every $x,y\in \DDD$
\begin{equation}
	\label{eq:lyapunov}
	\tag{$\mathrm V$}
	\skd yx= V(y)-V(x)+\bif'(y,x)
	\quad\text{with}\quad
	\bif'(y,x)\le C_1\,\dX(x,y)+C_2\,\dX^2(x,y).
\end{equation}
\end{quote}
Notice that $\bif'(x,x)=0$ since \eqref{eq:b-regular} holds for $\sfb$, and that the Lyapunov part cancels in the
symmetrized sum, $\bif'(y,x)+\bif'(x,y)=\skd yx+\skd xy$, so that
interaction dissipativity and antisymmetry are entirely carried by
$\bif'$. In the model case of a gradient flow (Example II),
$\skd yx=\varphi(y)-\varphi(x)$, condition \eqref{eq:lyapunov} holds
with $V:=\varphi$ (no continuity of $\varphi$ is required) and
$\bif'\equiv 0$.

Under \eqref{eq:lyapunov} the topological assumptions of
Theorem~\ref{thm:VMS-compact} simplify. By \eqref{eq:generating}, for
every $\tau>0$, $\bar x\in \DDD$, and $y\in \DDD$ the superlevel set
$\sigma^+(y)$ is a sublevel set of the lower semicontinuous map
\begin{equation}
	\label{eq:V-lsc-map}
	x\mapsto V(x)+\frac1{2\tau}\dX^2(\bar x,x)-\bif'(y,x),
\end{equation}
and is therefore closed.

Since the sublevels $\big\{x\in \DDD: V(x)\le c\big\}$, $c\in\R$,
are compact, \eqref{eq:lyapunov} implies \eqref{eq:b-growth} for
every $y_0\in \DDD$ and every $\tau_o>0$ with $2\tau_o C_2<1$: setting
$\delta:=\frac 1{2\tau_o}-C_2>0$, Young's inequality
$C_1\,\dX(y_0,x)\le \delta\,\dX^2(y_0,x)+\frac{C_1^2}{4\delta}$ gives
\begin{displaymath}
	\skd{y_0}x-\frac1{2\tau_o}\dX^2(y_0,x)\le
	V(y_0)+\frac{C_1^2}{4\delta}-V(x),
\end{displaymath}
so that the closed superlevels in \eqref{eq:b-growth} are 
contained in compact sublevels of $V$.
The decomposition \eqref{eq:lyapunov} will play an important role also
in the asymptotic analysis of
Section~\ref{sec:metric_dissipative_evolutions}: the resulting
one-sided bound $\skd yx\le V(y)-V(x)+C_1\,\dX(x,y)+C_2\,\dX^2(x,y)$
provides the a priori estimates for the convergence of the scheme,
while the semicontinuity properties of $V$ and $\bif'$ allow to pass
to the limit as the step size vanishes.
\begin{theorem}[One step via dissipativity]
	\label{thm:discrete-SEVI}
	Let $\sfb$ satisfy \eqref{eq:b-sublevels} and the $\eta$-interaction
	dissipativity
	\begin{equation}
		\label{eq:b-dissipative}
		\skd yx+\skd xy\le \eta\,\dX^2(x,y)
		\qquad\text{for every }x,y\in \DDD,
	\end{equation}
	let $\tau>0$, $\bar x\in \DDD$, $\xi>2\eta_+$, and assume that
	\begin{enumerate}
		\item[\rm(i)] the convexity condition \eqref{eq:B-bary-convex}
		holds for a system $\barysystem{}$ of barycentric maps in $\DDD$,
		possibly depending on $\tau$ and $\bar x$;
		\item[\rm(ii)] for every $y\in \DDD$ the map
		$x\mapsto \sfB(\tau,\bar x;y,x)$ is upper hemicontinuous with
		respect to the system $\barysystem{2}$ of $2$-barycentric maps of
		$\barysystem{}$ (Definition~\ref{def:hemicontinuity}).
	\end{enumerate}
	Then the primal problem \eqref{eq:B-primal} and the dual problem
	\eqref{eq:B-dual} have a unique common solution $x_\tau\in \DDD$, and
	$(x_\tau,x_\tau)$ is the unique diagonal saddle point
	\eqref{eq:saddle} of $\sfB(\tau,\bar x;\cdot,\cdot)$.
\end{theorem}
\begin{proof}
	We set $\sfa:=\sfB(\tau,\bar x;\cdot,\cdot)$ and check the assumptions
	of Theorem~\ref{thm:equilibrium2}: $\sfa$ satisfies \eqref{eq:diagonal}
	and the $\eta$-interaction dissipativity \eqref{eq:bif-monotone} by
	Lemma~\ref{le:B-structure}(1) and \eqref{eq:b-dissipative}; the maps
	$y\mapsto \sfa(y,x)$, $x\in \DDD$, have complete sublevels by
	\eqref{eq:b-sublevels} and Lemma~\ref{le:B-structure}(2), and are
	jointly barycentrically $\xi$-convexlike with $\xi>2\eta_+$ by (i).
	Theorem~\ref{thm:equilibrium2} thus yields a unique solution $x_\tau$
	of the dual problem \eqref{eq:B-dual}. By (ii), $\sfa$ satisfies the
	upper hemicontinuity condition (2) of Lemma~\ref{le:metric-Minty} with
	respect to the system $\barysystem{2}$ of condition (1) of that Lemma
	(guaranteed by (i)); Corollary~\ref{cor:primal-consequences}(ii) then
	shows that $x_\tau$ is also the unique solution of the primal problem
	\eqref{eq:B-primal} and the unique diagonal saddle point of
	$\sfB(\tau,\bar x;\cdot,\cdot)$.
\end{proof}
\begin{remark}\label{rem:B-hemicontinuity}
	When the maps of $\barysystem{2}$ are continuous, condition
	{\rm(ii)} of Theorem~\ref{thm:discrete-SEVI} is equivalent to the
	$\barysystem{2}$-upper hemicontinuity of the maps
	$x\mapsto \skd yx$, $y\in \DDD$, since the squared-distance terms in
	\eqref{eq:generating} are then continuous along each map.
\end{remark}
\begin{theorem}[One step under antisymmetry]
	\label{thm:VMS-antisym}
	Let $\sfb$ satisfy \eqref{eq:b-sublevels} and the antisymmetry
	condition \eqref{eq:antisymmetry}, let $\tau>0$, $\bar x\in \DDD$,
	$\xi>0$, and assume that the convexity condition
	\eqref{eq:B-two-convex} holds for a system $\barysystem{2}$ of
	$2$-barycentric maps in $\DDD$, possibly depending on $\tau$ and
	$\bar x$. Then the problems \eqref{eq:B-primal} and \eqref{eq:B-dual}
	have a unique common solution $x_\tau$, and $(x_\tau,x_\tau)$ is the
	unique diagonal saddle point \eqref{eq:saddle} of
	$\sfB(\tau,\bar x;\cdot,\cdot)$.
\end{theorem}
\begin{proof}
	By Lemma~\ref{le:B-structure}, $\sfB(\tau,\bar x;\cdot,\cdot)$ is
	antisymmetric and the maps $y\mapsto \sfB(\tau,\bar x;y,x)$ have
	complete sublevels; by \eqref{eq:B-two-convex} they are jointly
	$\xi$-convexlike with respect to $\barysystem{2}$. Theorem~\ref{thm:antisymmetric}
	then applies.
\end{proof}
Independently of the route followed to solve \eqref{eq:B-primal}, a
two-point convexity property of $\sfB$ improves the primal inequality
to a coercive one, which is the seed of the discrete Evolution
Variational Inequality studied in
Section~\ref{sec:metric_dissipative_evolutions}.
\begin{proposition}[Discrete Evolution Variational Inequality]
	\label{prop:discrete-EVI}
	Let $\tau>0$, $\cvx\in\R$, $\bar x\in \DDD$, set
	$\xi:=\dfrac 1\tau+\cvx>0$, and suppose that the convexity condition
	\eqref{eq:B-two-convex} holds. Then every solution $x_\tau$ of the
	primal problem \eqref{eq:B-primal} satisfies
	\begin{equation}
		\label{eq:resolvent-evi}
		\frac1{2\tau}
		\Big(\dX^2(x_\tau,y)
		-\dX^2(\bar x,y)+
		\dX^2(\bar x,x_\tau)\Big)
		\le \blskd y{x_\tau}
		\qquad\text{for every }y\in \DDD.
	\end{equation}
\end{proposition}
\begin{proof}
	By Proposition~\ref{prop:B-convexity}, condition
	\eqref{eq:B-two-convex} implies the radial convexity
	\eqref{eq:radial-convex} of $\sfB(\tau,\bar x;\cdot,\cdot)$ with the
	same $\xi$; Remark~\ref{rem:coercive-primal} then improves the primal
	inequality \eqref{eq:B-primal} to the coercive form
	\eqref{eq:coercive-primal-remark}
	\begin{displaymath}
		\sfB(\tau,\bar x;y,x_\tau)\ge
		\frac{\xi}2\,\dX^2(y,x_\tau)
		\qquad\text{for every }y\in \DDD,
	\end{displaymath}
	which coincides with \eqref{eq:resolvent-evi} thanks to
	\eqref{eq:generating} and \eqref{eq:s-variant1}.
\end{proof}
\begin{remark}
	\label{rem:radial-enough}
	The proof of Proposition~\ref{prop:discrete-EVI} only uses the radial
	convexity \eqref{eq:radial-convex} of $\sfB(\tau,\bar x;\cdot,\cdot)$
	with $\xi=\frac 1\tau+\cvx$, which is weaker than \eqref{eq:B-two-convex}
	and implied by it (Proposition~\ref{prop:B-convexity}); we have stated
	the assumption as \eqref{eq:B-two-convex} only because radial convexity
	is less transparent. Relaxing to radial convexity is useful, for example, when $\cB_{\bar x}$ are variational $c$-segments in a nonnegative cross curvature metric setting (see Example~\ref{ex:NNCC}) where \eqref{eq:gg-dist-2} does not hold in general; it holds only for maps in $\cB_{\bar x}$ with starting or ending point $\bar x$. 
    In turn, \eqref{eq:B-two-convex} follows from the
	two-point conditions \eqref{eq:gg-dist-2} and \eqref{eq:gg-bif-2}
	thanks to Proposition~\ref{prop:B-convexity}(1). Inequality
	\eqref{eq:resolvent-evi} is precisely a single step of the discrete
	Evolution Variational Inequality \eqref{eq:D-SEVI2}.
\end{remark}\nc 
\begin{remark}
	\label{rem:explain}
	In the typical situation where \eqref{eq:gg-dist} and
	\eqref{eq:gg-bif} hold for every base point $\bar x\in \DDD$ with
	$\xi=\tfrac1\tau+\cvx$, Proposition~\ref{prop:B-convexity} provides
	both the barycentric convexity \eqref{eq:B-bary-convex} required by
	Theorem~\ref{thm:discrete-SEVI} and the two-point condition
	\eqref{eq:B-two-convex} of Proposition~\ref{prop:discrete-EVI}. When
	$\sfb$ is $\eta$-interaction dissipative, the requirement
	$\xi>2\eta_+$ of Theorem~\ref{thm:discrete-SEVI} then reduces to a
	constraint on the step size only: $\tau<\tau_o$, where
	$\tau_o^{-1}:=2\eta_+-\cvx$ if $\cvx<2\eta_+$ and $\tau_o:=+\infty$
	otherwise.
\end{remark}
\nc

\subsection{The variational movement scheme}\label{sec:VMS}
The single-step resolvent of Section~\ref{sec:resolvent} can now be
iterated: given a step size $\tau>0$ and an initial datum
$X^0_\tau\in \DDD$, at each step we take the previous iterate as base
point, $\bar x:=X^{n-1}_\tau$, and solve the corresponding primal
problem.
\begin{definition}[Variational Movement Scheme]\label{def:VMS}
	Let $\tau>0$ be a given step size and let $X^0_\tau\in \DDD$.
	A sequence $(X^n_\tau)_{n\in \N}$ in $\DDD$ is a solution
	of the recursive Variational Movement Scheme (VMS) if for every $n>0$,
	\begin{equation}
		\label{eq:explicit-VMS-primal}
		\sfB(\tau,X^{n-1}_\tau; y,X^{n}_\tau)\ge 0
		\quad\text{for every }y\in \DDD.
	\end{equation}
\end{definition}
The characterization \eqref{eq:resolvent-evi} in Proposition~\ref{prop:discrete-EVI} allows us to write a discrete Evolution Variational Inequality for the (VMS) iterates defined above. 
\begin{definition}[Discrete $\sfb^\cvx$-EVI]\label{def:discrete-bEVI}
	Let $\tau>0$ and $\cvx\in\R$. A sequence $(X^n_\tau)_{n\in \N}$ in
	$\DDD$ is a solution of the discrete $\sfb^\cvx$-Evolution Variational
	Inequality if for every $n>0$,
	\begin{equation}
		\label{eq:D-SEVI2}
		\frac1{2\tau}\Big(\dX^2(X^n_\tau,y)-\dX^2(X^{n-1}_\tau,y)\Big)
		+
		\frac1{2\tau}\dX^2(X^n_\tau,X^{n-1}_\tau)
		\le
		\lskd y{X^n_\tau}
		\qquad\text{for every }y\in \DDD.
	\end{equation}
\end{definition}
\begin{theorem}[Existence for the VMS]
	\label{thm:saddle-sEVI}
	Let $\sfb$ satisfy \eqref{eq:b-regular}, let $\tau>0$, and suppose
	that for every base point $\bar x\in \DDD$ the assumptions of one of
	Theorems~\ref{thm:VMS-compact}, \ref{thm:discrete-SEVI},
	or~\ref{thm:VMS-antisym} hold. 
    \begin{enumerate}
        \item[\rm(i)] 
    For every initial datum
	$X^0_\tau\in \DDD$ the VMS scheme \eqref{eq:explicit-VMS-primal}
	admits a solution $(X^n_\tau)_{n\in\N}$. 
    \item[\rm(ii)] If the assumptions of
	Theorem~\ref{thm:discrete-SEVI} or of Theorem~\ref{thm:VMS-antisym}
	hold at every step, the solution is unique and each iterate
	$X^n_\tau$ also solves the dual problem
    \begin{equation}
		\label{eq:explicit-VMS-dual}
		\sfB(\tau,X^{n-1}_\tau; X^{n}_\tau,x)\le 0
		\quad\text{for every }x\in \DDD.
	\end{equation}
     \item[\rm(iii)] 
	If
    for some $\cvx\in\R$ and every $\bar x\in \DDD$, the
	convexity condition \eqref{eq:B-two-convex} holds with
	$\xi=\tfrac1\tau+\cvx>0$ (see also 
	Remark~\ref{rem:radial-enough}), then $(X^n_\tau)$ defined via the VMS scheme \eqref{eq:explicit-VMS-primal} also
	solves the discrete $\sfb^\cvx$-EVI \eqref{eq:D-SEVI2}.
    \end{enumerate}
\end{theorem}
\begin{proof}
	At each step $n>0$ we apply one of Theorems~\ref{thm:VMS-compact},
	\ref{thm:discrete-SEVI}, or~\ref{thm:VMS-antisym} to the bifunction
	$\sfB(\tau,X^{n-1}_\tau;\cdot,\cdot)$, obtaining $X^n_\tau$ solving the
	primal problem \eqref{eq:explicit-VMS-primal}.
    
    In the cases of
	Theorems~\ref{thm:discrete-SEVI} and~\ref{thm:VMS-antisym} the
	solution is unique and also solves \eqref{eq:explicit-VMS-dual}. 
    
    Under the
	additional convexity assumption, Proposition~\ref{prop:discrete-EVI}
	applied with
	$\bar x:=X^{n-1}_\tau$ turns \eqref{eq:explicit-VMS-primal} into a
	single step of \eqref{eq:D-SEVI2}.
\end{proof}
\nc

\begin{remark}[Choice of Discrete Scheme for Multispecies Flows]\label{rmk:choice-scheme}
Recall the choice of bifunction for multispecies coupled gradient flows \eqref{intro:bifunction_multispecies}:
\begin{align*}
    \skd yx = \sum_{i=1}^N 
    F^{(i)}(y^{(i)},
    x^{(-i)} )-F^{(i)}(x^{(i)}, x^{(-i)} )\,, \quad \text{for all }x,y\in\DDD\,.
\end{align*}
 From a discrete-time gradient flow perspective, consider the following two natural options for how each species $(i)$ can update to move in the direction of steepest descent of $F^{(i)}$. One option is minimizing $F^{(i)}(\cdot,(X_\tau^{(-i),n-1}))$ with a distance penalty,
\begin{align*}
         X_\tau^{(i),n}= \argmin_{x^{(i)}\in\XXX^{(i)}}\Big\{  F^{(i)}(x^{(i)},X_\tau^{(-i),n-1} ) + \frac{1}{2\tau} \dXi^2(x^{(i)},X_\tau^{(i),n-1})\Big\} \,, \quad \text{for every }i\in[1,\dots,N]\,.
\end{align*}
This update is implicit in $x^{(i)}$ but explicit in $x^{(-i)}$ because of the dependence on $X_\tau^{(-i),n-1}$. The second option is to solve
\begin{align*}
        X_\tau^{n} \quad \text{is a Nash equilibrium of} \quad \Big\{x \mapsto F^{(i)}(x^{(i)},x^{(-i)}) + \frac{1}{2\tau} \dXi^2(x^{(i)},X_\tau^{(i),n-1}) \Big\}_i\,, 
\end{align*}
which is a fully implicit update. Proving the existence of minimizers for the partially-explicit update is easier than proving existence of a Nash equilibrium for a set of energies. On the other hand, we expect better stability properties from the fully implicit method. In \cite{di_francesco_measure_2013}, the implicit-explicit update scheme was used along with Lipschitz bounds on the energy functionals to prove existence of solutions for a two-species PDE system with nonlocal interaction potentials and linear coupling potentials. 
To illustrate the lack of stability of the first method, consider the min-max game where $\XXX^{(1)}=\XXX^{(2)}=\R$, and $F^{(1)}(x^{(1)},x^{(2)})=-F^{(2)}(x^{(2)},x^{(1)}) = \frac{b}{2} (x^{(1)})^2+ax^{(1)}x^{(2)}-\frac{b}{2}(x^{(2)})^2$ with $b>0$ and $a\in\R$. The corresponding continuous-time gradient flow solution converges exponentially to the Nash equilibrium $x=0$ with rate $b$.
The schemes via the implicit-explicit and the Nash updates are
\begin{align*}
   &\text{(implicit-explicit) }&\ X^{n+1}_\tau &= (1+\tau b)^{-1} \begin{bmatrix}
         1 & -\tau a \\ \tau a & 1 
    \end{bmatrix} X_\tau^n\,,\\
    &\text{(fully implicit) }&\ X^{n+1}_\tau &= ((1+\tau b)^2 + \tau^2 a^2) ^{-1} \begin{bmatrix}
         1+\tau b & -\tau a \\ \tau a & 1+\tau b
    \end{bmatrix} X^{n}_\tau\,. 
\end{align*}
The maximum eigenvalue magnitudes  for each system are $\sqrt{a^2 \tau^2 + 1}(b \tau + 1)^{-1}$ and $[\tau^2 (a^2 + b^2) + 2 b \tau + 1]^{-1/2}$, with exponential stability if the magnitudes are less than one. If the coupling term $a$ is larger than the convexity parameter $b$, the first system is exponentially stable only for a small enough $\tau$, dependent on $a$ and $b$, whereas the  second system is exponentially stable for all $\tau > 0$, $a \in \R$, and $b>0$.
This mirrors the need for an upper-Lipschitz bound on the gradient of the energy for explicit Euler for the stability of gradient descent, along with a small enough time step relative to the Lipschitz constant. This observation motivates our choice of saddle point scheme (VMS) for general bifunctions $\sfb$, which is precisely the fully implicit Nash equilibrium scheme above in the case of multispecies coupled gradient flows.
\end{remark}

We conclude the subsection with another characterization of the solutions of
the discrete EVI \eqref{eq:D-SEVI2}.
It is convenient to
normalize $\sfb$ by the distance, introducing the function
$\skuname:\DDD\times \DDD\to \R$ defined for every $x,y\in \DDD$ by
\begin{equation}
	\label{eq:link}
\begin{aligned}
	\sku yx:={}&\frac{\skd yx}{\dX(x,y)}\quad\text{if }x\neq y,&
	\sku xx:={}&
    0,
\end{aligned}
\end{equation}
together with its perturbations, as in \eqref{eq:s-variant1},
\begin{equation}
	\label{eq:s-variant2}
	\lskuname(y,x):=\sku yx-\frac{\cvx}{2}\dX(x,y),
	\quad
	x,y\in \DDD,
\end{equation}
so that $\lsku yx=\lskd yx/\dX(x,y)$ whenever $x\neq y$.
The discrete EVI \eqref{eq:D-SEVI2}, which controls the squared
distance of the iterates from an arbitrary point, has a first-order
counterpart involving the distance itself.
\begin{lemma}
	\label{le:simple-but-tricky}
	If a sequence $(X^n_\tau)_{n\in \N}$ is a solution
	of \eqref{eq:D-SEVI2} then it also solves
	\begin{equation}
		\label{eq:D-SEVI1}
		\frac1{\tau}\Big(\dX(X^n_\tau,Y)-\dX(X^{n-1}_\tau,Y)\Big)
	\le
		\lsku Y{X^n_\tau}
		\qquad\text{for every }Y\in \DDD.
	\end{equation}
\end{lemma}
\begin{proof}
	Let us first suppose $Y\neq X^n_\tau.$
	Setting $a:=\dX(X^n_\tau,Y)$,
	$b:=\dX(X^{n-1}_\tau,Y)$,
	$c=\lsku Y{X^n_\tau}$,
	since $|a-b|\le \dX(X^n_\tau,X^{n-1}_\tau)$,
	\eqref{eq:D-SEVI2} implies
	\begin{displaymath}
		\frac 1{2\tau} (a^2-b^2+(a-b)^2)\le c a,
	\end{displaymath}
	which, after developing the squares, yields
	\begin{displaymath}
		\frac 1{\tau} (a^2-ab)\le c a;
	\end{displaymath}
	dividing by $a>0$ we get \eqref{eq:D-SEVI1}.
	The case $Y=X^n_\tau$ is trivial, since the left-hand side of
	\eqref{eq:D-SEVI1} is nonpositive and
	$\lsku{X^n_\tau}{X^n_\tau}=0$ by \eqref{eq:link}.
\end{proof}
\subsection{Local and global slopes for a bifunction}\label{sec:slope} 
The right-hand side of \eqref{eq:D-SEVI1} naturally leads to a
quantitative control of $\sfb$ near the diagonal, expressed by a
suitable notion of slope. Recall the definition of $\sfb^\kappa$ from \eqref{eq:s-variant1}.
\begin{definition}[Local and global slopes of $\bif$]\label{def:bounded_Delta}
    For every $x\in \DDD$ and $\cvx\in\R$ we set
   \begin{gather}
	\label{eq:local-slope}
	\slope x:=\limsup_{y\to x}
	\frac{\big(\skd yx\big)_-}{\dX(x,y)},\qquad 
	\skslope{\cvx}{x}:=
    \sup_{y\in \DDD}\frac{{\big(\lskd yx\big)}_-}{\dX(x,y)},
\end{gather}
with values in $[0,+\infty]$.
 We set $\domainslope{\cvx}:=\Big\{x\in \DDD:
\skslope{\cvx}{x}<\infty\Big\}$ and $\domainslope{}:=\Big\{x\in \DDD:
\slope{x}<\infty\Big\}$.
\end{definition}
When $x$ is an isolated point of $\DDD$ we adopt the convention
$\slope x:=0$.
Replacing $\skd yx$ with $\lskd yx$ modifies the difference quotient in
\eqref{eq:local-slope} by the vanishing term $\frac{|\cvx|}2\dX(x,y)$,
so that the local slope admits the equivalent representations
\begin{equation}
	\label{eq:local-slope-kappa}
	\slope x=\limsup_{y\to x}\frac{\big(\lskd yx\big)_-}{\dX(x,y)}
	=\limsup_{y\to x}\big(\lsku yx\big)_-
	\qquad\text{for every }\cvx\in\R;
\end{equation}
in particular
\begin{equation}
	\label{eq:loc-le-glob}
	\slope x\le \skslope\cvx x\quad\text{for every }\cvx\in\R.
\end{equation}
\begin{remark}
    In order to better understand
    the role of $\slope{x}$
    it could be useful to observe that
    in the simple case
    $\skd yx=\varphi(y)-\varphi(x)$,
    $\slope x$ coincides with
    the metric slope of $\varphi$
    \cite{AGS08}
    \begin{displaymath}
        \slope x=
        |\partial\varphi|(x)=
        \limsup_{y\to x}
        \frac{(\varphi(y)-\varphi(x))_-}{\dX(y,x)}
    \end{displaymath}
    and $\skslope \kappa x=|\partial\varphi|(x)$  
    whenever $\varphi$ is $\cvx$-convex (this follows from Proposition~\ref{prop:local-global-slope} below). 
    In the operator-Hilbertian case
    $\skd yx=\langle A(x),y-x\rangle$
    where $A$ is globally defined,
    $\skslope \kappa x=\slope x=\|A(x)\|$
    for every $\kappa\le 0$.
    In both cases $\domainslope{}$
    is dense in $\DDD$; see also Corollary~\ref{cor:slope-density} for a more general statement.
\end{remark}
The identity $\skslope\kappa x=\slope x$ of the previous
remark is a particular case of a general localization principle,
which
holds under the same convexity condition \eqref{eq:B-two-convex} used
to derive the discrete EVI.
\begin{proposition}[Localization of the global slope]
	\label{prop:local-global-slope}
	Let $x\in \DDD$, $\cvx\in\R$, $\tau_o>0$, and suppose that there
	exists a system $\barysystem{2}$ of $2$-barycentric maps in $\DDD$,
	independent of $\tau$, such that the convexity condition
	\eqref{eq:B-two-convex} holds at the base point $\bar x:=x$ with
	$\xi:=\frac1\tau+\cvx$ for every $\tau\in(0,\tau_o)$. Then
	\begin{equation}
		\label{eq:local-global-slope}
		\skslope\cvx x=\slope x.
	\end{equation}
\end{proposition}
\begin{proof}
	By \eqref{eq:loc-le-glob} it is sufficient to prove that
	$\skslope\cvx x\le \slope x$, and we may assume that
	$S:=\slope x$ is finite.
	Let us fix $y\in \DDD$, $y\neq x$, and a $2$-barycentric map of
	$\barysystem{2}$ connecting $x_0:=x$ to $x_1:=y$; we set
	$z_t:=\bary z{(1-t),t}$, $t\in [0,1]$.
	Evaluating \eqref{eq:B-two-convex} at the point $x$, since
	$\sfB(\tau,x;x,x)=0$ and
	$\sfB(\tau,x;z,x)=\skd zx+\frac1{2\tau}\dX^2(x,z)$
	by \eqref{eq:generating}, we obtain, for every $\tau\in(0,\tau_o)$
	and $t\in[0,1]$,
	\begin{equation}
		\label{eq:radial-B-at-x}
		\skd{z_t}x+\frac1{2\tau}\dX^2(x,z_t)\le
		t\,\skd yx+\frac{t^2}{2\tau}\dX^2(x,y)
		-\frac\cvx2\,t(1-t)\,\dX^2(x,y).
	\end{equation}
    Multiplying \eqref{eq:radial-B-at-x} by $2\tau$ and letting
	$\tau\down0$, the contribution of the real-valued $\sfb$ vanishes
	and, as in Remark~\ref{rm:metric_space_curvature}, we recover the
	metric control
	\begin{displaymath}
		\dX(x,z_t)\le t\,\dX(x,y);
	\end{displaymath}
	in particular $z_t\to x$ as $t\down0$. On the other hand, keeping
	$\tau$ fixed in \eqref{eq:radial-B-at-x}, neglecting the nonnegative
	term $\frac1{2\tau}\dX^2(x,z_t)$, and recalling
	\eqref{eq:s-variant1}, we get
	\begin{displaymath}
		\skd{z_t}x\le
		t\,\lskd yx
		+\Big(\frac\cvx2+\frac1{2\tau}\Big)\,t^2\,\dX^2(x,y). 
	\end{displaymath}
	For every $\eps>0$ there exists $r>0$ such that
	$\big(\skd zx\big)_-\le (S+\eps)\,\dX(x,z)$ whenever
	$z\in \DDD$, $\dX(x,z)\le r$; hence, for $t$ small enough,
	\begin{displaymath}
		-t\,\lskd yx
		-\Big(\frac\cvx2+\frac1{2\tau}\Big)\,t^2\,\dX^2(x,y)\le
		-\skd{z_t}x\le
		\big(\skd{z_t}x\big)_-\le
		(S+\eps)\,\dX(x,z_t)\le
		(S+\eps)\,t\,\dX(x,y).
	\end{displaymath}
	Dividing by $t>0$ and passing to the limit as $t\down0$ we obtain
	$\big(\lskd yx\big)_-\le (S+\eps)\,\dX(x,y)$. Dividing by $\dX(x,y)$ and taking the supremum
	with respect to $y\in \DDD$ and letting $\eps\down0$ we conclude.
\end{proof}
\begin{remark}
	\label{rem:local-global-weaker}
	By Proposition~\ref{prop:B-convexity}(1), the assumption of
	Proposition~\ref{prop:local-global-slope} is satisfied when the
	two-point conditions \eqref{eq:gg-dist-2} and \eqref{eq:gg-bif-2}
	hold at the base point $\bar x:=x$, since they provide
	\eqref{eq:B-two-convex} for every $\tau>0$ along the same system.
\end{remark}
The first-order inequality \eqref{eq:D-SEVI1} immediately shows that,
whatever the initial datum $X^0_\tau\in \DDD$, all the iterates of the
scheme lie in the domain of the global slope.
\begin{proposition}[The iterates lie in the domain of the slope]
	\label{prop:iterates-slope}
	Let $\cvx\in\R$, $\tau>0$, and let $(X^n_\tau)_{n\in\N}$ be a
	solution of the discrete $\sfb^\cvx$-EVI \eqref{eq:D-SEVI2}. Then
	\begin{equation}
		\label{eq:iterates-slope}
		\tau\,\skslope{\cvx}{X^n_\tau}\le \dX(X^n_\tau,X^{n-1}_\tau)
		\qquad\text{for every }n\ge 1.
	\end{equation}
	In particular $X^n_\tau\in \domainslope\cvx$ for every $n\ge 1$.
\end{proposition}
\begin{proof}
	By \eqref{eq:D-SEVI1} and the triangle inequality, for every
	$Y\in \DDD$ we have
	\begin{displaymath}
		-\lsku Y{X^n_\tau}\le
		\frac1{\tau}\Big(\dX(X^{n-1}_\tau,Y)-\dX(X^{n}_\tau,Y)\Big)
		\le \frac{\dX(X^n_\tau,X^{n-1}_\tau)}\tau;
	\end{displaymath}
	taking the supremum with respect to
	$Y\in \DDD\setminus\{X^n_\tau\}$ and recalling
	\eqref{eq:local-slope} we obtain \eqref{eq:iterates-slope}.
\end{proof}
Let us stress that \eqref{eq:iterates-slope} is a consequence of the
discrete EVI alone. The converse estimate---a discrete Lipschitz bound
propagating the value of the slope along the iterates starting from
$X^0_\tau\in \domainslope\cvx$---requires the dissipativity of $\sfb$
and will be obtained in Proposition~\ref{le:apriori}, as a first step of the
convergence analysis in
Section~\ref{sec:metric_dissipative_evolutions}.

Under the dissipativity of $\sfb$, a single step of the scheme also
approximates its initial point with an explicit rate. Combined with
Proposition~\ref{prop:iterates-slope}, we will show that the density of
$\domainslope\cvx$ in $\DDD$ is a \emph{consequence} of the solvability
of the scheme (Corollary \eqref{cor:slope-density}) and need not be assumed a priori, in analogy with the
theories of gradient flows and of monotone operators. 

In order to highlight the essential condition for this property to hold, we say that
$y\mapsto \bif(y,\bar x)$, $\bar x\in \DDD$, is \emph{quadratically bounded from below}, if
there exists $\bar\tau=\bar\tau(\bar x)>0$  such that
\begin{equation}
	\label{eq:MY}
	\bif_{\bar \tau}(\bar x):=
	\inf_{y\in \DDD}\sfB(\bar \tau,\bar x;y,\bar x)=
	\inf_{y\in \DDD}\Big(\frac1{2\bar \tau}\dX^2(\bar x,y)+\skd y{\bar x}\Big)
	>-\infty.
\end{equation}
Since $\bif_{\bar\tau'}(\bar x)\ge \bif_{\bar\tau}(\bar x)$ whenever
$0<\bar\tau'\le \bar\tau$, condition \eqref{eq:MY} is stable under
reducing the step size $\bar\tau$.
The next lemma, a direct application of
Proposition~\ref{prop:minima}, shows that \eqref{eq:MY} is a
consequence of the uniform convexity of
$y\mapsto \sfB(\bar \tau,\bar x;y,\bar x)$ and of the completeness of its
sublevels; in particular, it imposes no additional restriction under
the convexity assumptions of our well-posedness theorems.
\begin{lemma}[Uniform convexity implies \eqref{eq:MY}]
	\label{le:MY-convex}
	Let $\sfb$ satisfy \eqref{eq:b-regular} and \eqref{eq:b-sublevels},
	let $\bar\tau>0$, $\cvx\in\R$, $\bar x\in \DDD$ with
	$\xi=\frac 1{\bar \tau}+\cvx>0$,
    and assume that the convexity condition
	\eqref{eq:B-two-convex} holds for a system $\barysystem{2}$ of
	$2$-barycentric maps in $\DDD$.
    Then the map
	$y\mapsto \sfB(\bar \tau,\bar x;y,\bar x)$ attains its minimum on
	$\DDD$; in particular \eqref{eq:MY} holds at $\bar x$ with step
	size $\bar \tau$.
\end{lemma}
\begin{proof}
	By \eqref{eq:B-two-convex} with $x:=\bar x$, the map
	$\psi(y):=\sfB(\bar \tau,\bar x;y,\bar x)$ is $\xi$-convexlike with
	$\xi=\bar \tau^{-1}+\cvx>0$ with respect to $\barysystem{2}$; it is
	proper, since $\psi(\bar x)=0$ by \eqref{eq:b-regular}, and it has
	complete sublevels by \eqref{eq:b-sublevels} and
	Lemma~\ref{le:B-structure}(2). Proposition~\ref{prop:minima} then
	shows that $\psi$ attains its minimum on $\DDD$, so that
	$\bif_{\bar \tau}(\bar x)>-\infty$.
\end{proof}
\begin{lemma}[One-step approximation of the initial point]
	\label{le:one-step}
	Let $\sfb$ satisfy \eqref{eq:b-regular} and the $\eta$-interaction
	dissipativity \eqref{eq:b-dissipative}, let $\bar\tau>0,\,\bar x\in \DDD$ satisfy
	\eqref{eq:MY} and 
    $\bar \tau(2\eta-\cvx)_+ \le 1$.  
    If
	$\tau\in(0,\bar\tau/2)$ and $x_\tau\in \DDD$ solves a single step of the
	discrete $\sfb^\cvx$-EVI \eqref{eq:resolvent-evi},
    then
	\begin{equation}
		\label{eq:one-step-rate}
		\dX^2(\bar x,x_\tau)\le
		2\,\big({-}\bif_{\bar \tau}(\bar x)\big)\,\tau.
	\end{equation}
\end{lemma}
\begin{proof}
	Choosing $y:=\bar x$ in \eqref{eq:resolvent-evi} and recalling
	\eqref{eq:s-variant1} and \eqref{eq:b-regular} we obtain
	\begin{displaymath}
		\frac1{\tau}\dX^2(x_\tau,\bar x)\le
		\lskd{\bar x}{x_\tau}=
		\skd{\bar x}{x_\tau}-\frac\cvx2\dX^2(x_\tau,\bar x).
	\end{displaymath}
	The dissipativity \eqref{eq:b-dissipative} gives
	$\skd{\bar x}{x_\tau}\le
	\eta\,\dX^2(\bar x,x_\tau)-\skd{x_\tau}{\bar x}$, while \eqref{eq:MY}
	yields
	$-\skd{x_\tau}{\bar x}\le
	\frac1{2\bar \tau}\dX^2(\bar x,x_\tau)-\bif_{\bar \tau}(\bar x)$; hence
	\begin{displaymath}
		\Big(\frac1\tau+\frac\cvx2-\eta-\frac1{2\bar \tau}\Big)
		\dX^2(x_\tau,\bar x)\le -\bif_{\bar \tau}(\bar x).
	\end{displaymath}
	Since $\tau<\bar\tau/2$ we have
	$\frac1\tau+\frac\cvx2-\eta-\frac1{2\bar \tau}\ge \frac 1{2\tau}+
    \frac1{2\bar \tau}
    +\frac\cvx2-\eta\ge \frac 1{2\tau}$ and we
	obtain \eqref{eq:one-step-rate}; notice that
	$\bif_{\bar \tau}(\bar x)\le \sfB(\bar \tau,\bar x;\bar x,\bar x)=0$.
\end{proof}
\begin{corollary}[Density of the domain of the slope]
	\label{cor:slope-density}
	Let $\sfb$ satisfy \eqref{eq:b-regular}, \eqref{eq:b-sublevels}, and \eqref{eq:b-dissipative},
    and assume that for every $\bar x\in \DDD$ there exist
	$\bar\tau>0$ with $\xi=\frac 1{\bar \tau}+\cvx>0$ and a system
	$\barysystem{2}$ of $2$-barycentric maps in $\DDD$ for which the
	convexity condition \eqref{eq:B-two-convex} holds.
	If for every $\bar x\in \DDD$ there exists a sequence
	$\tau_\ell\down0$ such that the discrete $\sfb^\cvx$-EVI
	\eqref{eq:D-SEVI2} admits a one-step solution with initial datum
	$\bar x$ and step size $\tau_\ell$, then $\domainslope\cvx$ is dense in
	$\DDD$.
\end{corollary}
\begin{proof}
	Let $\bar x\in \DDD$. Lemma~\ref{le:MY-convex} provides
	\eqref{eq:MY} at $\bar x$ and, since \eqref{eq:MY} is stable under
	reducing the step size, we can also assume
	$\bar\tau(2\eta-\cvx)_+\le1$. Since the first step of
	\eqref{eq:D-SEVI2} coincides with \eqref{eq:resolvent-evi}, the
	solutions $X^1_{\tau_\ell}$ belong to $\domainslope\cvx$ by
    Proposition~\ref{prop:iterates-slope}
    and converge to $\bar x$ by
	Lemma~\ref{le:one-step}.
\end{proof}
The one-step solvability required by Corollary~\ref{cor:slope-density}
is guaranteed by Theorem~\ref{thm:saddle-sEVI}; explicit
conditions on $\sfb$ will be provided by the well-posedness results of
Section~\ref{sec:metric_dissipative_evolutions} (Theorems
\ref{thm:main_existence} and \ref{thm:antisym_existence}), under whose
assumptions the density of $\domainslope\cvx$ thus holds automatically.
\nc

\subsection{Examples}\label{sec:ex-VMS}
We discuss examples illustrating the interplay between the convexity of $y \mapsto\sfb(y,x)$ and $\dX^2(\cdot,\bar x)$ with respect to various families of barycentric maps $\cB$. The joint convexity of $y \mapsto\sfb(y,x)$ and $\dX^2(\cdot,\bar x)$ with respect to the same $\cB$ plays a role in the well-posedness
of the variational movement scheme; see Sections~\ref{sec:resolvent} and \ref{sec:VMS}.

\subsubsection{Convexity of squared distance function}\label{sec:convexity-squared-d}
In the first example, we illustrate how the choice of curves along which the squared distance function is convex for the spherical Hellinger metric impacts the choice of energy functionals in practical applications. In the second example, convexity of the squared distance function in Wasserstein-2 space along generalized geodesics allows us to prove existence of solutions without strong ($\lambda <0$) dissipativity due to compensation from the $\frac{1}{2\tau} \dX^2(\cdot,\bar x)$ term.
\begin{example}[Convexity with Spherical Hellinger/Fisher-Rao Metric]\label{ex:convexity-spherical-Hellinger}
    Consider a system of coupled gradient flows with respect to the spherical Hellinger metric $\HH$, which gives rise to a space with nonnegative cross curvature (see Example~\ref{ex:NNCC} in Section~\ref{sec:convexity-ex}): for all $i\in [1,\dots,N]$,
    \begin{align*}
        \partial_t \rho^{(i)}(x) = -\nabla_{\HH,\rho^{(i)}}F^{(i)}[\rho^{(1)},\dots, \rho^{(N)} ] = -\Big(\delta_{\rho^{(i)}}F^{(i)}[\rho](x) -\int \delta_{\rho^{(i)}}F^{(i)}[\rho]\, \d\rho^{(i)}\Big) \rho^{(i)}(x)\,.
    \end{align*}
    We would like to prove well-posedness of the solutions to the above system of evolutions. While results in \cite{carrillo_fisher-rao_2024} show that a class of $f$-divergences are geodesically convex in the Hellinger metric, existence of solutions was left as an open problem because $\dX^2(\cdot,\rho)$ is not 1-convex along geodesics in the spherical Hellinger metric. In \cite{lascu_fisher-rao_2024}, the authors propose a min-max flow system for use in machine learning applications, studying its long time behavior.
    In recent work, \cite{Leger-Todeschi-Vialard25} showed the existence of (non-unique) curves along which the squared spherical Hellinger distance is 1-convex. These curves could serve as condidates for a suitable choice of barycentric system $\cB$. However, it is not known which nontrivial energy functionals are convex with respect to these curves. 
    This illustrates how the choice of $\cB$ can be challenging for spaces that lack non-positive curvature.
\end{example}
\begin{example}[Nonconvexity and Existence of Solutions in the Wasserstein-2 Space]
    For $\rho^{(1)},\rho^{(2)}\in\PP_2(\R)$, $a_1,a_2\in\R$, and $\cvx>0$, consider the energy functionals
    \begin{align*}
        F^{(1)}(\rho^{(1)},\rho^{(2)}) &= \frac{\cvx}{2}\int z_1^2\d \rho^{(1)}(z_1)+ 2a_1\iint z_1 z_2 \d \rho^{(1)}(z_1) \d \rho^{(2)}(z_2)\,, \\
        F^{(2)}(\rho^{(2)},\rho^{(1)}) &=\frac{\cvx}{2}\int z_2^2 \d \rho^{(2)}(z_2)+ 2a_2\iint z_1  z_2 \d \rho^{(1)}(z_1) \d \rho^{(2)}(z_2) \,.
    \end{align*}
    It was shown in \cite[Example 4.5]{conger_monotone_2025} that this system is not necessarily monotone, even though $F^{(1)}$ and $F^{(2)}$ are strongly $\kappa$-displacement convex in $\rho^{(1)}$ and $\rho^{(2)}$ respectively.  In fact, for the choice of $a_1=a_2$, the corresponding coupled Wasserstein-2 gradient flow can be written as a joint gradient flow of an energy functional that has no minimizer.  In other words, the static game has no solution (no Nash equilibrium). However, the existence theorems in Section~\ref{sec:resolvent}
     can still be applied to provide existence of solutions to the dynamic game (see Section~\ref{sec:metric_dissipative_evolutions}). Here, let us discuss existence of solutions to the VMS scheme \eqref{eq:explicit-VMS-primal} and the discrete EVI \eqref{eq:D-SEVI2} approximating the solutions to the dynamic game.
     Let $\skd \mu \rho =\sum_{i=1}^2 F^{(i)}(\mu^{(i)},\rho^{(-i)})-F^{(i)}(\rho^{(i)},\rho^{(-i)})$. Computing the interaction dissipativity according to Definition~\ref{def:monotonicity_zeroth_order}, we have, with $\gamma^{(1)}\in\Gamma_o(\rho^{(1)},\mu^{(1)})$ and $\gamma^{(2)}\in\Gamma_o(\rho^{(2)},\mu^{(2)})$ optimal transport plans,
     \begin{align*}
         \skd \rho \mu + \skd \mu \rho &=  2(a_1+a_2)\iiiint \left( z_1 z_2'+ z_1'z_2-z_1z_2- z_1'z_2' \right)\d \gamma^{(1)}(z_1,z_1') \d \gamma^{(2)} (z_2,z_2') \\
         &= 2(a_1+a_2) \iiiint (z_1-z_1')(z_2'-z_2)\d \gamma^{(1)}(z_1,z_1') \d \gamma^{(2)} (z_2,z_2')  \le |a_1+a_2| \dX^2(\rho,\mu)\,,
     \end{align*}
     where the inequality follows from Young's inequality. Therefore, $\sfb$ satisfies \eqref{eq:b-dissipative} with $\eta=|a_1+a_2|$. In particular, $\sfb$ is not antisymmetric, unless $a_1+a_2=0$ or one of the pairs $(\rho^{(i)}, \mu^{(i)})$ has the same means. Thus, we consider the existence of discrete iterates using Theorem~\ref{thm:saddle-sEVI} via the dissipativity conditions in Theorem~\ref{thm:discrete-SEVI}. For this example, we select $\cB$ as generalized geodesics (Example~\ref{ex:generalized-geo}), then $y \mapsto \sfb(y,x)$ is $\cvx$-barycentrically convexlike with respect to $\cB$ and \eqref{eq:B-bary-convex} holds. Therefore, $\lambda=|a_1+a_2|-\kappa$. Depending on the choice of $a_1$, $a_2$, and $\kappa$, it is possible that $\lambda>0$. 
     Selecting $\tau_o:=(2\eta_+)^{-1}$ results in $\xi=\tau^{-1}+\cvx > 2\eta_+$ for every $\tau<\tau_o$. For any $\rho_n \to \rho$ we have $\sfb(\rho_n,\mu) \to \sfb(\rho,\mu)$ because $\Wass_2$ convergence implies convergence of second moment terms  
     and the linear terms have Lipschitz integrands. Thus, the map $y \mapsto \sfb(y,x)$ is $\Wass_2$-continuous, and 
    it has complete sublevels due to the coercive quadratic terms and continuity of $y \mapsto \sfb(y,x)$. The map $x\mapsto \sfb(y,x)$ is continuous along all curves in $\cB$.
      The requirements for Theorem~\ref{thm:discrete-SEVI} are thus satisfied and we can apply Theorem~\ref{thm:saddle-sEVI}.
\end{example}

\subsubsection{Running Examples: Discrete $\sfb$-EVI}
\begin{itemize}
    \item \textbf{Example (I):}
    The main difference in applying Theorem~\ref{thm:equilibrium2} to $\sfa=\sfB$ rather than $\sfa=\sfb$ to achieve the result in Theorem~\ref{thm:discrete-SEVI} is the presence of $\frac{1}{2\tau}\dX^2(\cdot,\bar x)$ in the bifunction. In order for $\skd yx = \<\sfA(x),y-x>$ to satisfy the criteria for Theorem~\ref{thm:discrete-SEVI}, we can simply apply the same requirements as discussed in Section~\ref{sec:ex-existence}, because norms on Hilbert spaces satisfy the desired 1-convexity property and continuity properties. 

    \item \textbf{Example (II):} In the single species gradient flow setting, $\XXX=\XXX^{(1)}$ and energy $\varphi:\XXX \to \R$. Solving the primal equilibrium problem \eqref{eq:B-primal} for $\sfB(\tau,X^{n-1}_\tau;y,x)$ reduces to solving
\begin{align*}
     &\frac{1}{2\tau}\dX^2(X^{n-1}_\tau,y) - \frac{1}{2\tau} \dX^2(X^{n-1}_\tau,X^n_\tau)+\varphi(y) - \varphi(X^n_\tau) \ge 0\,, \quad \forall\, y\in \XXX\,,  \\
     &\quad \Rightarrow \quad \frac{1}{2\tau}\dX^2(X^{n-1}_\tau,y)+\varphi(y) \ge \frac{1}{2\tau} \dX^2(X^{n-1}_\tau,X^n_\tau) + \varphi(X^n_\tau) \,, \quad \forall\, y\in \XXX\,,
\end{align*}
which is the same minimizing movement scheme \eqref{eq:JKO} in \cite{AGS08}, known as the JKO scheme in Wasserstein-2 spaces \cite{JordanKinderlehrerOtto98}.

\item \textbf{Example (III):} 
The bifunction is given by $\skd yx = F(y^{(1)},x^{(2)})-F(x^{(1)},y^{(2)})$, and the generating bifunction is
\begin{align*}
    \begin{split}
        \sfB(\tau,X^{n-1}_\tau;y,x):=
	\frac 1{2\tau}\dX^2(y,X_\tau^{n-1})-
    \frac1{2\tau}\dX^2(x,X_\tau^{n-1}) + F(y^{(1)}, x^{(2)})-F(x^{(1)},y^{(2)})\,.
    \end{split}
\end{align*}
Let us assume sufficient regularity of $F$ so that $\sfb(\cdot,x)$ has complete sublevels, and $F$ is convex-concave such that the convexity condition \eqref{eq:B-two-convex} holds for some $\cB_{\bar x,2}$. 
Since $\sfb$ is antisymmetric, we can apply Theorem~\ref{thm:VMS-antisym} to obtain existence of $(X_\tau^n)$ solving \eqref{eq:explicit-VMS-primal} with 2-point barycentric convexity instead of the stronger barycentric convexity required for Theorem~\ref{thm:discrete-SEVI}. By Theorem~\ref{thm:saddle-sEVI}, $(X_\tau^n)$ also solves the discrete EVI \eqref{eq:D-SEVI2}.

\item \textbf{Example (IV):} We will check the convexity along generalized geodesics in the multispecies coupled gradient flow setting. 
    Let $(F^{(i)})$ be a coupled gradient flow with each species evolving in the Wasserstein-2 metric space. To show existence of a global Nash equilibrium, we can 
    check conditions on $(F^{(i)})$ so that the maps $y \mapsto \skd yx$ are $\cvx$-barycentrically convexlike with respect to $\Wass_2$  generalized geodesics (Example~\ref{ex:generalized-geo}). 
    This choice of curves applies to Examples (II)-(III) as well in the Wasserstein-2 metric setting. Dissipativity \eqref{eq:b-dissipative} in combination with upper-semicontinuity of each $F^{(i)}(y^{(i)},\cdot)-F^{(i)}(\cdot,\cdot)$ is sufficient to satisfy the conditions of Theorem~\ref{thm:discrete-SEVI}. Then by Theorem~\ref{thm:saddle-sEVI}, the VMS scheme admits a solution $(X_\tau^n)$, each iterate $X_\tau^n$ solves the dual equilibrium problem \eqref{eq:explicit-VMS-dual}, and since $\frac12 \Wass_2^2(\cdot,x)$ is $1$-convex with respect to generalized geodesics, $(X_\tau^n)$ solves the discrete $\sfb^\cvx$-EVI \eqref{eq:D-SEVI1}. 
   
\end{itemize}

\section{Metric evolutions }\label{sec:metric_dissipative_evolutions}
In this section we collect our main results concerning 
convergence of the VMS method and existence of continuous solutions $\sfx:[0,+\infty)\to \DDD$ of the
$\skdname$-Evolution Variational Inequality \eqref{introeq:EVI}
introduced in Section~\ref{sec:intro}:
\begin{equation}
\label{eq:SEVI2}
	\frac12\frac{\d} \dt \dX^2(\sfx(t),y)\le \blskd y{\sfx(t)}\quad\text{$t\in (0,+\infty)$,}\quad
	\text{for every }y\in \DDD.
	\tag{EVI$_{2,\cvx}$}
\end{equation}
\begin{definition}[Strong solutions]
	\label{def:solutions}
	A \emph{strong solution} to \eqref{eq:SEVI2} is a locally absolutely
	continuous curve $\sfx:[0,+\infty)\to \DDD$ satisfying
	\eqref{eq:SEVI2} pointwise a.e.~in $(0,+\infty)$.
\end{definition}
If $\sfx$ is locally absolutely continuous, then for every $y\in \DDD$
the map $t\mapsto \dX^2(\sfx(t),y)$ is locally absolutely continuous as
well, so that the left-hand side of \eqref{eq:SEVI2} is well defined
almost everywhere.
In the next subsections we will follow two different routes to study
\eqref{eq:SEVI2}: a compactness argument, which provides existence of
strong solutions without any dissipativity assumption (and without
uniqueness), and the dissipative theory, which yields uniqueness,
exponential contraction, and local Lipschitz regularity in time for
solutions starting from the domain of the slope. 

\subsection{Existence of Solutions to (EVI) via Compactness}\label{sec:existence-compactness}
Our first existence result for \eqref{eq:SEVI2} does not require any
dissipativity of $\sfb$ (and, correspondingly, does not provide
uniqueness). Instead, the Lyapunov decomposition
\eqref{eq:lyapunov}
of Section~\ref{sec:resolvent}, which there guaranteed the solvability
of a single step of the scheme via compactness, provides the a
priori estimates needed to pass to the limit as the step size
vanishes.

To describe the asymptotic behaviour of the scheme
we introduce the piecewise constant interpolants
	\begin{equation}\label{eq:interpolation}
    \begin{split}
		&\ol X_\tau(0):=X^0_\tau,\quad
		\ol X_\tau(t):=X^n_\tau\quad \ \ \ \text{if }t\in ((n-1)\tau,n\tau]\,, \\
        &\ul X_\tau(0):=X^0_\tau,\quad
        \ul X_\tau(t):=X^{n-1}_\tau \quad  \text{if }t\in[(n-1)\tau,n\tau)\,.
    \end{split}
	\end{equation}
\begin{lemma}[A priori estimates under the Lyapunov decomposition]
	\label{le:V-estimates}
	Let $\cvx\in\R$, let $\sfb$ satisfy \eqref{eq:lyapunov},
    and let
	$\tau_o>0$ satisfy $\tau_o\,(C_2-\cvx/2)_+\le 1/4$. Every solution
	$(X^n_\tau)_{n\in\N}$ of the discrete $\sfb^\cvx$-EVI
	\eqref{eq:D-SEVI2} with $\tau\in(0,\tau_o]$ satisfies
	\begin{equation}
		\label{eq:V-energy}
		V(X^n_\tau)+\frac1{2\tau}\,\dX^2(X^n_\tau,X^{n-1}_\tau)\le
		V(X^{n-1}_\tau)+\tau C_1^2
		\qquad\text{for every }n\ge 1.
	\end{equation}
	In particular, setting for a given time horizon $T>0$
	\begin{equation}
		\label{eq:V-constants}
		N_\tau:=\lceil T/\tau\rceil,\quad
		V_T:=V(X^0_\tau)+(T+\tau_o)C_1^2,\quad
		E_T:=2\Big(V(X^0_\tau)-\inf_\DDD V+(T+2\tau_o)C_1^2\Big),
	\end{equation}
	we have
	\begin{gather}
		\label{eq:V-bound}
		V(X^n_\tau)\le V_T,\qquad
		\dX^2(X^n_\tau,X^{n-1}_\tau)\le \tau\,E_T
		\qquad\text{for every }1\le n\le N_\tau,
		\\
		\label{eq:V-energy-sum}
		\sum_{n=1}^{N_\tau}
		\frac{\dX^2(X^n_\tau,X^{n-1}_\tau)}{\tau}\le E_T,
		\\
		\label{eq:V-hoelder}
		\dX\big(\ol X_\tau(t),\ol X_\tau(s)\big)\le
		\sqrt{(t-s+\tau)\,E_T}
		\qquad\text{for every }0\le s\le t\le T.
	\end{gather}
\end{lemma}
\begin{proof}
	Choosing $Y:=X^{n-1}_\tau$ in \eqref{eq:D-SEVI2} and recalling
	\eqref{eq:s-variant1} we obtain, with
	$d_n:=\dX(X^n_\tau,X^{n-1}_\tau)$,
	\begin{displaymath}
		\frac1\tau\,d_n^2\le
		\lskd{X^{n-1}_\tau}{X^n_\tau}
		=\skd{X^{n-1}_\tau}{X^n_\tau}-\frac\cvx2\,d_n^2.
	\end{displaymath}
	By \eqref{eq:lyapunov} and Young's inequality
	$C_1\,d_n\le \frac1{4\tau}d_n^2+\tau C_1^2$,
	\begin{displaymath}
		\Big(\frac3{4\tau}-C_2+\frac\cvx2\Big)d_n^2\le
		V(X^{n-1}_\tau)-V(X^n_\tau)+\tau C_1^2,
	\end{displaymath}
	and since $\tau\,(C_2-\cvx/2)_+\le 1/4$ we have
	$\frac3{4\tau}-C_2+\frac\cvx2\ge \frac1{2\tau}$, which yields
	\eqref{eq:V-energy}.
	Summing \eqref{eq:V-energy} over $l$ for $1\le l\le n\le N_\tau$ and using
	$n\tau\le T+\tau_o$ we get
	\begin{displaymath}
		V(X^n_\tau)+\frac12\sum_{l=1}^n \frac{d_l^2}\tau\le
		V(X^0_\tau)+n\tau C_1^2\le V_T,
	\end{displaymath}
	which gives the first bound in \eqref{eq:V-bound} and, since
	$V(X^n_\tau)\ge \inf_\DDD V$, also \eqref{eq:V-energy-sum}. The
	second bound in \eqref{eq:V-bound} follows from
	\eqref{eq:V-energy}, \eqref{eq:V-bound}, and
	$V(X^n_\tau)\ge\inf_\DDD V$.
	Finally, if $0\le s\le t\le T$ with $s\in ((m-1)\tau,m\tau]$ and
	$t\in ((n-1)\tau,n\tau]$, then $(n-m)\tau\le t-s+\tau$ and the
	Cauchy--Schwarz inequality gives
	\begin{displaymath}
		\dX^2\big(\ol X_\tau(t),\ol X_\tau(s)\big)\le
		\Big(\sum_{l=m+1}^n d_l\Big)^2\le
		(n-m)\sum_{l=m+1}^n d_l^2\le
		\frac{t-s+\tau}\tau\,\tau\,E_T=(t-s+\tau)\,E_T.
		\qedhere
	\end{displaymath}
\end{proof}
The next lemma isolates the mechanism by which limits of discrete
solutions produce strong solutions; it will be applied both here and,
in the dissipative setting. 
\begin{lemma}[Convergence to strong solutions]
	\label{le:discrete-to-continuous}
	Let $\cvx\in\R$, let $\tau_j\down0$, and for every $j\in\N$ let
	$(X^n_{\tau_j})_{n\in\N}$ be a solution of the discrete
	$\sfb^\cvx$-EVI \eqref{eq:D-SEVI2} with step size $\tau_j$.
	Suppose that there exists a curve $\sfx:[0,+\infty)\to \DDD$ such
	that for every $T>0$:
	\begin{enumerate}
		\item[\rm(i)] $\ol X_{\tau_j}(t)\to \sfx(t)$ as $j\to\infty$,
		for every $t\in [0,T]$;
		\item[\rm(ii)] There exists $\tilde E_T<\infty$ such that $\displaystyle
		\sup_{j\in\N}\
		\sum_{n=1}^{\lceil T/\tau_j\rceil}
		\frac{\dX^2(X^n_{\tau_j},X^{n-1}_{\tau_j})}{\tau_j}
		\le \tilde E_T$;
		\item[\rm(iii)] for every $y\in \DDD$,
		$\displaystyle\sup_{j\in\N}\ \sup_{t\in[0,T]}
		\skd y{\ol X_{\tau_j}(t)}<+\infty$;
		\item[\rm(iv)] for every $y\in \DDD$ and a.e.~$t\in (0,T)$,
		$\displaystyle\limsup_{j\to\infty}
		\skd y{\ol X_{\tau_j}(t)}\le \skd y{\sfx(t)}$.
	\end{enumerate}
	Then $\sfx$ is a strong solution to \eqref{eq:SEVI2}; moreover its
	metric derivative satisfies
	$\int_0^T|\dot\sfx|^2(t)\,\d t\le \tilde E_T$ for every $T>0$.
\end{lemma}
\begin{proof}
	Throughout the proof we fix $T>0$ and set
	$d_n^j:=\dX(X^n_{\tau_j},X^{n-1}_{\tau_j})$.

	\emph{Step 1: Regularity of $\sfx$.}
	From (ii), arguing as in the proof of \eqref{eq:V-hoelder}, we have for large enough $j\in\N$ that 
	\begin{equation}
		\label{eq:hoelder-j}
		\dX\big(\ol X_{\tau_j}(t),\ol X_{\tau_j}(s)\big)\le
		\sqrt{(t-s+\tau_j)\tilde E_T}\qquad\text{for }0\le s\le t\le T,
	\end{equation}
	so that, by (i), $\dX(\sfx(t),\sfx(s))\le \sqrt{(t-s)\tilde E_T}$: $\sfx$
	is continuous. Let
	$m_j:(0,T)\to[0,+\infty)$ be the piecewise constant function equal
	to $d^j_n/\tau_j$ on $((n-1)\tau_j,n\tau_j)$; by (ii)
	$\|m_j\|^2_{L^2(0,T)}\le \tilde E_T$, so a subsequence (not relabelled)
	converges weakly in $L^2(0,T)$ to some $m$ with
	$\|m\|^2_{L^2(0,T)}\le \liminf_j \|m_j\|^2_{L^2(0,T)}\le \tilde E_T$. 
	For $0\le s\le t\le T$, the triangle inequality gives
	$\dX(\ol X_{\tau_j}(t),\ol X_{\tau_j}(s))\le
	\int_{s'}^{t'} m_j(r)\,\d r$ for some $s',t'$ with $|s'-s|,|t'-t|\le \tau_j$, and
	passing to the limit
	\begin{displaymath}
		\dX(\sfx(t),\sfx(s))\le \int_s^t m(r)\,\d r,
	\end{displaymath}
    so that $\sfx$ is absolutely continuous in $[0,T]$ with metric
	derivative $|\dot\sfx|\le m$ a.e.

	\emph{Step 2: Integral form of the discrete EVI.}
	Fix $y\in \DDD$ and let $\zeta_j(\cdot,y):[0,T]\to\R$ be the
	piecewise affine interpolant of the values
	$\big(\dX^2(X^n_{\tau_j},y)\big)_n$ at the nodes $(n\tau_j)_n$.
	For $t\in ((n-1)\tau_j,n\tau_j)$, the inequality
	\eqref{eq:D-SEVI2} gives
	\begin{displaymath}
		\partial_t \zeta_j(t,y)=
		\frac{\dX^2(X^n_{\tau_j},y)-\dX^2(X^{n-1}_{\tau_j},y)}{\tau_j}
		\le 2\,\lskd y{X^n_{\tau_j}}-\frac{(d^j_n)^2}{\tau_j}
		\le 2\,\lskd y{\ol X_{\tau_j}(t)},
	\end{displaymath}
	so that, for every $0\le s\le t\le T$,
	\begin{equation}
		\label{eq:integrated-discrete}
		\zeta_j(t,y)-\zeta_j(s,y)\le
		2\int_s^t \lskd y{\ol X_{\tau_j}(r)}\,\d r.
	\end{equation}

	\emph{Step 3: Passage to the limit.}
    By the triangle inequality and \eqref{eq:hoelder-j} applied with $s=0$, we have
    	\begin{equation*}
        \dX\big(\ol X_{\tau_j}(t),y\big)
        \le 
		\dX\big(\ol X_{\tau_j}(t),\ol X_{\tau_j}(0)\big)
        +  \dX\big(\ol X_{\tau_j}(0),y\big)
        \le
		\sqrt{(T+\tau_j)\tilde E_T} +  \dX\big(\ol X_{\tau_j}(0),y\big)\,,
	\end{equation*}
    which is bounded above uniformly in $t\in[0,T]$ and $j\in\N$.
	Since $\max_n d^j_n\le \sqrt{\tau_j \tilde E_T}\to0$ by (ii) and
	$\sup_{j\in\N} \sup_{[0,T]}\dX(\ol X_{\tau_j}(\cdot),y)$ is bounded, we have
	$\sup_{[0,T]}\big|\zeta_j(\cdot,y)-
	\dX^2(\ol X_{\tau_j}(\cdot),y)\big|\to 0$, so that by (i) the
	left-hand side of \eqref{eq:integrated-discrete} converges to
	$\dX^2(\sfx(t),y)-\dX^2(\sfx(s),y)$.
	The integrands in the right-hand side are measurable (being
	piecewise constant) and uniformly bounded above by (iii) and the
	boundedness of $\dX(\ol X_{\tau_j}(\cdot),y)$; Fatou's Lemma
	(in limsup form) then yields
	\begin{displaymath}
		\dX^2(\sfx(t),y)-\dX^2(\sfx(s),y)\le
		2\int_s^t g_y(r)\,\d r,\qquad
		g_y(r):=\limsup_{j\to\infty}\lskd y{\ol X_{\tau_j}(r)},
	\end{displaymath}
	where $g_y$ is measurable, locally bounded above, and satisfies
	$g_y(r)\le \lskd y{\sfx(r)}$ for a.e.~$r\in (0,T)$ by (iv) and (i).
	Since $t\mapsto \dX^2(\sfx(t),y)$ is absolutely continuous,
	differentiating at every $t$ which is both a differentiability
	point and a Lebesgue point of $g_y$ we obtain
	\begin{displaymath}
		\frac12\frac\d\dt \dX^2(\sfx(t),y)\le g_y(t)\le
		\lskd y{\sfx(t)}
		\qquad\text{for a.e.~}t\in (0,T).
	\end{displaymath}
	As $y\in \DDD$ and $T>0$ are arbitrary, $\sfx$ is a strong solution
	to \eqref{eq:SEVI2}.
    Finally, since $|\dot\sfx|(t)\le m(t)$ for a.e. $t\in[0,T]$ by Step 1, and $\|m\|^2_{L^2(0,T)}\le \tilde E_T$, we conclude $\int_0^T|\dot\sfx|^2(t)\,\d t\le \int_0^T|m(t)|^2\,\d t\le\tilde E_T$.
\end{proof}
The structure of \eqref{eq:lyapunov} also provides the upper
semicontinuity of $\sfb$ with respect to its second variable required
by condition (iv) of Lemma~\ref{le:discrete-to-continuous}: if
$x_n\to x$ in $\DDD$ then, by the lower semicontinuity of $V$ and the
upper semicontinuity of $\bif'(y,\cdot)$,
\begin{equation}
	\label{eq:V-usc}
	\limsup_{n\to\infty}\skd y{x_n}\le
	V(y)-\liminf_{n\to\infty}V(x_n)+
	\limsup_{n\to\infty}\bif'(y,x_n)\le
	\skd yx
	\qquad\text{for every }y\in \DDD.
\end{equation}
Now, we are ready to show existence of strong solutions to \eqref{eq:SEVI2} via compactness by applying Theorem~\ref{thm:VMS-compact}. Note that here, compact sublevels of $V$ play a crucial role, whereas the a priori estimates in Lemma~\ref{le:V-estimates} only used the upper bound on $\sfb$ in \eqref{eq:lyapunov}, not the compactness.
\begin{theorem}[Existence via compactness]
	\label{thm:existence-compact}
	Let $\cvx\in\R$ and let $\sfb$ satisfy \eqref{eq:b-regular} and the
	Lyapunov decomposition \eqref{eq:lyapunov} with compact sublevels
	$\big\{x\in \DDD: V(x)\le c\big\}$, $c\in\R$.
	Let $\tau_o>0$ satisfy $\tau_o\,(C_2-\cvx/2)_+\le 1/4$,
	$2\tau_o C_2<1$, and
	$\tau_o^{-1}+\cvx>0$, and suppose that for every $\tau\in(0,\tau_o]$
	and every base point $\bar x\in \DDD$
	\begin{enumerate}
		\item[\rm(1)] 
        the maps $y\mapsto \sfB(\tau,\bar x;y,x)$,
		$x\in \DDD$, are jointly barycentrically quasi-convexlike with
		respect to a {\em continuous} system $\barysystem{}$ of barycentric maps in $\DDD$;
		\item[\rm(2)] the convexity condition \eqref{eq:B-two-convex}
		holds with $\xi:=\dfrac1\tau+\cvx$.
	\end{enumerate}
	Then for every $x_0\in \DDD$ there exist a vanishing sequence
	$\tau_j\down0$ and a strong solution $\sfx$ to \eqref{eq:SEVI2}
	such that $\sfx(0)=x_0$, $|\dot\sfx|\in L^2_{\rm loc}([0,+\infty))$,
	and
	$\ol X_{\tau_j}(t)\to\sfx(t)$ locally uniformly in $[0,+\infty)$,
	where $(X^n_{\tau_j})$ solve \eqref{eq:D-SEVI2} with
	$X^0_{\tau_j}=x_0$.
\end{theorem}
\begin{proof}
	Since $2\tau_oC_2<1$, we can choose $\tau_o'>\tau_o$ with
	$2\tau_o'C_2<1$; as explained in the paragraph after \eqref{eq:V-lsc-map}, the Lyapunov decomposition \eqref{eq:lyapunov} with
	compact sublevels implies \eqref{eq:b-growth} 
    with the threshold
	$\tau_o'$, so that, by Lemma~\ref{lem:compact}, condition (iii) of
	Theorem~\ref{thm:VMS-compact} holds for every $\tau\in(0,\tau_o]$
	and every base point $\bar x\in \DDD$. Moreover, under
	\eqref{eq:lyapunov} the map $x\mapsto \skd yx=V(y)-V(x)+\bif'(y,x)$
	is upper semicontinuous for every $y\in \DDD$, since $V$ is lower
	semicontinuous and $\bif'(y,\cdot)$ is upper semicontinuous; hence
	condition (ii) of Theorem~\ref{thm:VMS-compact} holds as well.
	Together with (1), all the assumptions of
	Theorem~\ref{thm:VMS-compact} are thus satisfied at every base point;
	by (2) and Theorem~\ref{thm:saddle-sEVI}, for every
	$\tau\in(0,\tau_o]$ there exists a solution $(X^n_\tau)$ of the VMS
	with $X^0_\tau=x_0$, which solves the discrete $\sfb^\cvx$-EVI
	\eqref{eq:D-SEVI2}.
    
    Let $V_T,E_T$ be
	defined as in \eqref{eq:V-constants} (they do not depend on $\tau$).
	By Lemma~\ref{le:V-estimates}, for every $T>0$ the iterates
	$X^n_\tau$, $0\le n\le N_\tau$, belong to the compact sublevel
	$K_T:=\{V\le V_T\}$, and the interpolants satisfy the H\"older
	bound \eqref{eq:V-hoelder}.

    By a refined version of Ascoli-Arzel\`a Theorem (see, e.g., \cite[Thm.~3.3.1]{AGS08}),
    we can find a sequence $\tau_j\down0$ and a continuous curve $\sfx:[0,+\infty)\to \DDD$ such that
    (i) of Lemma~\ref{le:discrete-to-continuous} holds, and the convergence is locally uniform in $[0,+\infty)$.

	It remains to check conditions (ii)--(iv) of
	Lemma~\ref{le:discrete-to-continuous}: (ii) follows from \eqref{eq:V-energy-sum} with $\tilde E_T=E_T$. For (iii), the
	Lyapunov decomposition \eqref{eq:lyapunov} gives, for every $y\in \DDD$
	and $t\in [0,T]$,
	\begin{displaymath}
		\skd y{\ol X_{\tau_j}(t)}\le
		V(y)-\inf_\DDD V+C_1\,R_T+C_2 \,R_T^2,\qquad
		R_T:=\sup_{j\in\N}\sup_{t\in[0,T]}\dX\big(\ol X_{\tau_j}(t),y\big)<+\infty.
	\end{displaymath}
    Here, the finiteness of $R_T$ follows since $       \dX\big(\ol X_{\tau_j}(t),y\big)
        \le 
		\dX\big(\ol X_{\tau_j}(t),x_0\big)
        +  \dX\big(x_0,y\big)
        \le
		\sqrt{(T+\tau_o)\tilde E_T} +  \dX\big(x_0,y\big)$ by \eqref{eq:hoelder-j}.
	Finally, (iv) follows from \eqref{eq:V-usc}, since
	$\ol X_{\tau_j}(t)\to\sfx(t)$ for every $t$.
	Lemma~\ref{le:discrete-to-continuous} then shows that $\sfx$ is a
	strong solution.
\end{proof}

\subsection{Properties of Dissipative EVI Solutions}\label{sec:admissible-b}
In this subsection we study the \emph{structural} properties of the EVI
solutions of \eqref{eq:SEVI2} in the $\eta$-interaction dissipative case, when
\begin{subequations}
    \label{subeq:structure-dissipativity}
\begin{equation}
	\label{eq:S-lambda}
	\skd xx=0,\qquad
	\skd yx+\skd xy\le \eta\,\dX^2(x,y)
	\quad\text{for every }x,y\in \DDD;
\end{equation}
in particular, we will discuss the equivalent
first-order and dual formulations, the contraction between solutions,
uniqueness, and the resulting EVI flow. Here the generating bifunction
$\sfB$ and the variational movement scheme play no role: everything is
expressed directly in terms of $\bif$, of its dissipativity, and of its
slope. Existence of solutions is a separate matter, obtained either by
compactness (Section~\ref{sec:existence-compactness}) or by the explicit error estimates for the Variational Movement Scheme (Section~\ref{sec:soln_to_discrete_EVI} and Theorems~\ref{thm:tau_to_zero}, \ref{thm:main_existence}, and \ref{thm:antisym_existence}).

Throughout this subsection we fix $\cvx,\eta\in\R$ and 
in addition to 
\eqref{eq:S-lambda} we will assume that
$\bif:\DDD\times \DDD\to\R$ satisfies the complete
sublevel property (cf.~\eqref{eq:b-sublevels})
		\begin{equation}
			\label{eq:b-sublevels-bis}
			\text{for every }x\in \DDD\quad
			\text{the map }y\mapsto \skd yx\ \text{has complete sublevels,}
		\end{equation}
and that the following conditional upper semicontinuity holds:
\begin{equation}
\label{eq:conditional-usc}
	y,x,x_n\in \DDD,\ x_n\to x,\ \sup_n\skslope\cvx{x_n}<+\infty\quad
	\Rightarrow
	\quad
	\limsup_{n\to\infty}\skd y{x_n}\le \skd yx.
\end{equation}
Finally, we introduce
the contraction rate $\lambda$ 
of the dissipative evolution 
and the \emph{$\eta$-conjugate bifunction}
$\cbskdname$ 
obtained from $\bif$ by exchanging its arguments and reversing the
sign, up to the dissipativity shift $\eta\,\dX^2$.
\begin{equation}
	\label{eq:parameters}
	\lambda:=\eta-\cvx,\qquad
    \cbskdname(y,x):=\eta\,\dX^2(x,y)-\skd xy,\qquad x,y\in\DDD.
\end{equation}
\end{subequations}
$\cbskdname$ vanishes on the
diagonal and the $\eta$-conjugation is involutive (conjugating $\cbskdname$
returns $\bif$); its $\cvx$-perturbation satisfies
\begin{equation}
	\label{eq:conjugate-identity}
	\cblskd yx=\lambda \dX^2(x,y)-
    \blskd{x}{y}=-\bif^{\cvx_{*}}(x,y),\qquad
	\cvx_{*}:=2\eta-\cvx=\cvx+2\lambda,
\end{equation}
where the conjugate parameter $\cvx_{*}$ is the reflection of $\cvx$
about $\eta$.\nc
\begin{remark}\label{rem:skew-symmetric}
The simplest situation is when $\bif$ is \emph{skew-symmetric}, i.e.
\begin{equation}
	\label{eq:skew-symmetric}
	\skd yx+\skd xy=0,\quad\text{i.e.}\quad
	\skd yx=-\skd xy,\quad\text{for every }
	x,y\in \DDD.
\end{equation}
In this case $\eta=0$ and $\cbskdname(y,x)=-\skd xy=\skd yx$, so
that $\bif$ is \emph{self-conjugate}; moreover $x\mapsto \skd yx$ is upper semicontinuous
and therefore the conditional upper semicontinuity
\eqref{eq:conditional-usc}
is a consequence of \eqref{eq:b-sublevels-bis}.
\end{remark}
More generally, $\cblskdname$ satisfies the dissipativity property
\begin{equation}
    \label{eq:general-skew-symmetry}
    \blskd yx+\cblskd xy= \lambda\dX^2(x,y)
    \quad\text{for every $x,y\in \DDD$.}
\end{equation}
\begin{remark}\label{rem:slope-lsc}
It is worth noticing that \eqref{eq:conditional-usc}
implies that
$\skslope\cvx\cdot$ is lower semicontinuous.
In fact, if $(x_n)_{n\in \N}$ is a sequence in $\DDD$ converging to $x\in \DDD$
such that
$\skslope\cvx{x_n}\le L$, then
we have
\begin{displaymath}
    \frac{\kappa}2\dX^2(y,x_n)\le
    L\dX(y,x_n)+\skd y{x_n}
\end{displaymath}
and we can pass to the limit in the inequality
as $n\to\infty$ thanks to \eqref{eq:conditional-usc},
obtaining $\skslope\cvx{x}\le L$ as well.
\end{remark}
We now consider a dual and relaxed formulation of
\eqref{eq:SEVI2}, which involves the dissipativity of $\sfb$.
First of all we observe that, 
thanks to the $\eta$-interaction dissipativity 
\eqref{eq:S-lambda}
and \eqref{eq:conjugate-identity}, we have
\begin{equation}
    \label{eq:obvious}
    \blskd xy +\blskd yx\le \lambda\dX^2(x,y),
    \qquad
    \blskd yx\le \lambda \dX^2(y,x)
    -\blskd xy=
    \cblskd yx,
\end{equation}
so that 
any solution to \eqref{eq:SEVI2}
also solves 
the dual EVI formulation
\begin{equation}
    \label{eq:cSEVI2}
	\frac12\frac{\d} \dt \dX^2(\sfx(t),y)\le \cblskd y{\sfx(t)}\quad\text{$t\in (0,+\infty)$,}\quad
	\text{for every }y\in
     \domainslope{\cvx}.
	\tag{d-EVI$_{2,\cvx}$}
\end{equation}
Note that the primal and dual formulations are indeed equivalent in the skew-symmetric case (Remark~\ref{rem:skew-symmetric}).
The advantage of using the conjugate bifunction $\cblskdname$ is twofold: first, thanks to 
\begin{align}
    \cblskd yx=
    -\blskd xy+\lambda \dX^2(x,y)
    \label{eq:above}
    \le \skslope{\cvx}{y}\dX(x,y)+
    \lambda \dX^2(x,y),
\end{align}
the right-hand side of \eqref{eq:cSEVI2} is bounded from above by a quadratic function of $\dX(x,y)$, which makes it easier to handle.
Second, the upper semicontinuity of $\cbskdname(y,\cdot)$  provides a better stability of the solutions. 

Since we want to allow the solutions to also take values in
$\overline\DDD$, we extend $\cbskdname(y,\cdot)$, $y\in \DDD$, to
$\overline\DDD$ by upper semicontinuity, i.e.~setting
$\cbskdname(y,x):=-\infty$ for every $x\in \overline\DDD\setminus\DDD$:
indeed, the completeness of the sublevels \eqref{eq:b-sublevels-bis}
forces $\skd {x_n}y\to+\infty$, i.e.~$\cbskdname(y,x_n)\to-\infty$,
whenever $x_n\in \DDD$ converge to a point $x\in
\overline\DDD\setminus\DDD$, since any subsequence with
$\sup_k\skd{x_{n_k}}y<\infty$ would lie in a complete sublevel of
$\skd\cdot y$, forcing $x\in \DDD$.

 \begin{definition}[Weak dual EVI solutions and EVI flows]
	\label{def:weak-solutions}
    Let us assume that $\bif$ satisfies 
    {\em (\ref{subeq:structure-dissipativity}a-d)}.
    A \emph{weak dual} solution is a continuous curve
	$\sfx:[0,+\infty)\to \overline\DDD$ that satisfies
	\eqref{eq:cSEVI2} in the distributional sense.

	A \sEVI\ flow in $\DDD\subset \XXX$ is a family of continuous maps
	$\sfS_t:
    \overline{\DDD}\to \overline{\DDD}$ such that for every $x_0\in \overline{\DDD}$ 
	\begin{equation}
		\sfS_{t+h}(x_0)=\sfS_h(\sfS_t(x_0)),\quad 
		\lim_{h\down0}\sfS_h(x_0)=x_0,\quad
		t\mapsto \sfS_t(x_0)\text{ is a weak solution to\ } \eqref{eq:cSEVI2}
	\end{equation}
    and
    \begin{equation}
        \label{eq:strong-solution}
       t\mapsto  \sfS_t(x_0)\text{ is a strong solution to \eqref{eq:SEVI2} whenever $x_0\in \domainslope{\cvx}$.}
    \end{equation}
\end{definition} 
We already observed that, using \eqref{eq:obvious}, 
every strong solution is also a weak dual solution. 
According to Lemma \ref{le:ulDini},
weak solutions can be characterized by
the integral inequalities
\begin{equation}
    \label{eq:integral}
    \frac12\dX^2(y,\sfx(t))-
    \frac12\dX^2(y,\sfx(s))\le
    \int_s^t
    \cblskd y{\sfx(r)}
    \,\d r
    \quad
    \text{for every }y\in \domainslope{\cvx},\quad
    0<s<t.
\end{equation}
Note that for every $y\in \domainslope{\cvx}$ the map
$x\mapsto \cblskd yx$ is upper semicontinuous in $\overline\DDD$ and,
along locally uniformly convergent curves, locally uniformly bounded
from above thanks to \eqref{eq:above}. 
Passing to the limit in the integral formulation \eqref{eq:integral}
it is clear that
the class of weak solutions is stable under locally uniform
convergence (in particular, locally uniform limits of strong solutions
are weak solutions).
\begin{remark}\label{rem:Lip-slope}
    If $\sfx$ is a strong solution
    which is locally Lipschitz continuous
     in $(0,+\infty)$ 
    and $L_{a,b}$ is its Lipschitz constant
    in $[a,b]\subset (0,+\infty)$
    we easily deduce
    \begin{displaymath}
        -\blskd y{\sfx(t)}\le L_{a,b}\dX(y,\sfx(t))
        \quad\text{for every }y\in \DDD,\ t\in [a,b],
    \end{displaymath}
    so that 
    $\sup_{t\in [a,b]}\skslope{\cvx}{\sfx(t)}\le L_{a,b}$.
    It follows that the map 
    $t\mapsto \blskd y{\sfx(t)}$
    is upper semicontinuous in $(0,+\infty)$ and in particular
    locally bounded from above for every $y\in \DDD.$ Thus, it satisfies the assumptions of Lemma~\ref{le:ulDini}.

    If $\sfx$ is a weak solution,
    the map $t\mapsto \cblskd y{\sfx(t)}$ is upper semicontinuous,
    so that Lemma~\ref{le:ulDini}
    can be applied as well.
\end{remark}
\begin{remark}[Weak solutions take values in $\DDD$ for a.e.~time]
    \label{rem:weak-in-D}
    If $\sfx$ is a weak solution, then $\sfx(t)\in \DDD$ for
    a.e.~$t>0$. Indeed, the left-hand side of \eqref{eq:integral} is
    finite and, by \eqref{eq:above}, the integrand is locally bounded
    from above, so that $t\mapsto \cblskd y{\sfx(t)}$ is locally
    integrable in $(0,+\infty)$; in particular, 
     it is finite for a.e.~$t>0$, i.e.\ $\sfx(t)\in\DDD$ for a.e.~$t>0$.
     Moreover, $\sfx(\bar t)\notin\DDD$, i.e.\ $\cblskd y{\sfx(\bar t)}=-\infty$, may only occur at times $\bar t$ where the upper
    left Dini derivative of $t\mapsto \frac12\dX^2(y,\sfx(t))$ is
    $-\infty$  (by the upper semicontinuity of $t\mapsto\cblskd y{\sfx(t)}$ and Lemma~\ref{le:ulDini}), i.e.~where $\sfx$ has infinite left metric speed; it is
    thus excluded for locally Lipschitz weak solutions. 
\end{remark}
\begin{remark}
    \label{rem:exponential-trick}
    Notice that we can obtain a more symmetric
    formulations of \eqref{eq:SEVI2}
    and \eqref{eq:cSEVI2}
    by multiplying both inequalities by
    $\rme^{-\lambda t}$ obtaining
    \begin{equation}
        \label{eq:more-symm}
        \begin{aligned}
            \frac 12\frac \d{\d t}
            \Big(\rme^{-\lambda t}
            \dX^2(\sfx(t),y)\Big)
            &\le 
            \rme^{-\lambda t}
            \bif^{\eta}(y,\sfx(t)),
            \\
            \frac 12\frac\d{\d t}
            \Big(\rme^{-\lambda t}
            \dX^2(\sfx(t),y)\Big)
            &\le 
            \rme^{-\lambda t}
            \hatsk^{\eta}(y,\sfx(t)),
        \end{aligned}
        \qquad
        \bif^{\eta}(y,x)+\hatsk^{\eta}(x,y)= 0.
    \end{equation}
\end{remark}

We conclude with the equivalent first-order formulations.
Recall the normalized bifunction $\skuname$ and its perturbation
$\lskuname$ introduced in \eqref{eq:link} and \eqref{eq:s-variant2}.
The dissipativity condition \eqref{eq:S-lambda} yields
\begin{equation}
	\label{eq:S-eta1}
	\sku yx+	\bsku xy\le \eta \dX(x,y),
\end{equation}
and therefore, with $\lambda:=\eta-\cvx$,
\begin{align}\label{eq:S-lambda1}
    \lsku yx + \blsku xy \le \lambda \dX(x,y).
\end{align}
\begin{lemma}
	\label{le:equivalence12}
	A locally absolutely continuous
    curve $\sfx:[0,+\infty)\to
    \DDD$ is a strong solution to
	\eqref{eq:SEVI2}
	if and only if it is a solution to
	\begin{equation}
\label{eq:SEVI1}
	\frac\d \dt \dX(\sfx(t),y)\le \blsku y{\sfx(t)}\quad\text{ a.e.~in $(0,\infty)$, 
    for every }y\in \DDD.
	\tag{EVI$_{1,\cvx}$}
\end{equation}
(or, equivalently, according to one of the forms of Lemma \ref{le:ulDini}).

A continuous curve $\sfx:[0,+\infty)\to \overline\DDD$ 
is a weak solution to 
	\eqref{eq:cSEVI2}  
	if and only if it is a solution to 
	\begin{equation}
\label{eq:cSEVI1}
	\frac\d \dt \dX(\sfx(t),y)\le \cblsku y{\sfx(t)}
    \quad\text{for every }y\in
     \domainslope{\cvx}.
	\tag{d-EVI$_{1,\cvx}$}
\end{equation}
according to one of the equivalent forms of Lemma \ref{le:ulDini}.
\end{lemma}
\begin{proof}
 For strong solutions, the equivalence
is easy since
the local absolute continuity implies
\begin{displaymath}
    \frac 12\frac\d{\d t}\dX^2(\sfx(t),y)
    =
    \dX(\sfx(t),y)
    \cdot
    \frac\d{\d t}\dX(\sfx(t),y),\quad
    \frac\d{\d t}\dX(\sfx(t),y)=0
    \text{ a.e.\ where $\sfx(t)=y$}.
\end{displaymath}
In the case of weak solutions, 	we apply 
	Lemma \ref{le:ulDini} and we prove the equivalence 
	of \eqref{eq:cSEVI1} and \eqref{eq:cSEVI2}
	using the lower left Dini derivative.

	The implication \eqref{eq:cSEVI1}$\Rightarrow$\eqref{eq:cSEVI2}
	is an immediate consequence of the formula
	\begin{equation}
	\label{eq:liminf}
		\frac 12 \Lld\zeta^2(t)=
		\liminf_{h\down0}\frac{\zeta^2(t)-\zeta^2(t-h)}{2h}=
		\zeta(t)		\liminf_{h\down0}\frac{\zeta(t)-\zeta(t-h)}{h}=
		\zeta(t)\,\Lld\zeta(t)
	\end{equation}
	which holds for every nonnegative continuous function $\zeta$.
	
	Let us now prove the converse implication:
	we suppose that 
     \eqref{eq:cSEVI2} holds, 
    we fix $y\in \domainslope{\cvx}$, and we want to prove that
	\begin{equation}
		\label{eq:toprove}
		\Lld \dX(\sfx(t),y)\le
        \cblsku y{\sfx(t)}  \quad\text{for every }t\in (0,\infty)\text{ such that }\sfx(t)\neq y,
	\end{equation} 
    and that 
    the right-hand side of 
    \eqref{eq:cSEVI1} is locally bounded from above.
	\eqref{eq:toprove} can be obtained by applying
	\eqref{eq:liminf}.
    We also have
    \begin{align*}
        \cblsku y{\sfx(t)}
        =
        -\blsku {\sfx(t)}y
        +\lambda \dX(\sfx(t),y)
        \le \skslope{\cvx}{y}
        +\lambda \dX(\sfx(t),y),
    \end{align*}
    which clearly shows the 
    required local boundedness.
\end{proof}
The dissipativity of $\bif$ yields the contraction between EVI
solutions, their uniqueness, and the existence of their associated EVI flows. We split the
argument in three statements: strong solutions are contractive among
themselves and enjoy a regularizing effect, quantified by the
exponential control \eqref{eq:mder-decay} of the upper right metric
derivative; the contraction estimate then extends to the comparison
between a strong and a weak solution; the existence of strong
solutions starting from every point of $\domainslope\cvx$ 
(dense in $\overline\DDD$)
generates the \sEVI\ flow in $\overline\DDD$.
\begin{theorem}[Contraction and regularization of strong solutions]
	\label{thm:strong-solutions}
	Let $\sfb$ satisfy the dissipativity \eqref{eq:S-lambda}  
    and let 
	$\lambda=\eta-\cvx$. The following properties hold.
	\begin{enumerate}
		\item\label{item:strong-contraction}
		If $\sfx,\sfy$ are strong solutions to \eqref{eq:SEVI2}, then
		\begin{equation}
			\label{eq:estimate-strong}
			\dX(\sfy(t),\sfx(t))\le
			\rme^{\lambda(t-s)}\dX(\sfy(s),\sfx(s))
			\quad\text{whenever}\quad 0\le s\le t;
		\end{equation}
		in particular, for every initial datum there is at most one
		strong solution to \eqref{eq:SEVI2}.
		\item\label{item:mder-decay}
		Every strong solution $\sfx$ to \eqref{eq:SEVI2} satisfies
		\begin{equation}
			\label{eq:mder-decay}
			\urmder \sfx t\le \rme^{\lambda(t-s)}
			\urmder \sfx s
			\quad \text{whenever}\quad
			0\le s\le t;
		\end{equation}
        in particular $\sfx$ is locally Lipschitz in $(0,+\infty)$.
        \item\label{item:Lip-at-zero}
		If $\sfx(0)\in \domainslope{\cvx}$, then any strong solution $\sfx$ is
		locally Lipschitz in $[0,+\infty)$ and
		\begin{equation}
			\label{eq:Lip-at-zero}
			\urmder\sfx t\le \rme^{\lambda t}\,\skslope\cvx{\sfx(0)}
			\quad\text{for every }t\ge0.
		\end{equation}
		\item 
        \label{item:slope} Finally, if in addition the conditional upper semicontinuity \eqref{eq:conditional-usc} holds, then
		\begin{equation}
			\label{eq:slope-mder}
			\skslope\cvx{\sfx(t)}= \urmder\sfx t
			\quad\text{for every }t\ge0,
		\end{equation}
		so that $\sfx$ takes values in $\domainslope\cvx$ for $t>0$ whenever $\sfx(0)\in\domainslope\cvx$.
	\end{enumerate}
\end{theorem}\nc
\begin{proof}
	{\em Claim \ref{item:strong-contraction}.}
	Since $\sfx$ and $\sfy$ are locally absolutely continuous,
	using \cite[Lemma 4.3.4]{AGS08}, setting $g(t,s)=\dX(\sfy (t), \sfx(s))$ and using the equivalence of right- and left-hand derivatives,
    \begin{align*}
        \frac{\d}{\d t} g(t,t) &\le \limsup_{h \downarrow 0} \frac{\dX(\sfx(t),\sfy(t))-\dX(\sfx(t-h),\sfy(t))}{h} + \limsup_{h \downarrow 0} \frac{\dX(\sfx(t),\sfy(t+h))-\dX(\sfx(t),\sfy(t))}{h} \\
        & = {\dds} \big[\dX(\sfx(s),\sfy(t)) \big]\big|_{s=t} + {\dds} \big[\dX(\sfx(t),\sfy(s)) \big]\big|_{s=t}\,,
    \end{align*}
    we get by Lemma \ref{le:equivalence12} and \eqref{eq:S-lambda1}
	\begin{equation}
		\label{eq:d-inequality}
		\frac\d{\dt}\dX(\sfy(t),\sfx(t))\le \lsku{\sfy(t)}{\sfx(t)}+
		\lsku{\sfx(t)}{\sfy(t)}
		\le \lambda \dX(\sfy(t),\sfx(t))
		\quad\text{a.e.~in }(0,+\infty),
	\end{equation}
	which immediately yields \eqref{eq:estimate-strong}.

	{\em Claim \ref{item:mder-decay}.}
	Applying \eqref{eq:estimate-strong} to the strong solutions $\sfx(t)$ and $\sfx_h(t):=\sfx(t+h)$
	we get
	\begin{equation*}
		\dX(\sfx(t+h),\sfx(t))\le \rme^{\lambda(t-s)}
		\dX(\sfx(s+h),\sfx(s))
		\quad\text{if }0\le s\le t;
	\end{equation*}
	dividing by $h$ and passing to the limit as $h\down0$ we get
	\eqref{eq:mder-decay}.
    The fact that $\sfx$ is locally
	Lipschitz in $(0,+\infty)$ follows: 
    \eqref{eq:mder-decay} shows that
	$\urmder\sfx t
    <\infty$ for every $t>0$ and that the restriction of
	$\sfx$ to every interval $[t,T]$, $0<t\le T$, is Lipschitz with
	constant
	$L_{t,T}\le 
    \inf_{0<s\le t}\rme^{\lambda_+(T-s)}\,\urmder\sfx s$.
	
	{\em Claim \ref{item:Lip-at-zero}.}
	Let us set $S:=\skslope\cvx{\sfx(0)}$ and
	$\delta(t):=\dX(\sfx(t),\sfx(0))$.
	Choosing $y:=\sfx(0)$ in \eqref{eq:SEVI1} and using
	\eqref{eq:S-lambda1} and the definition \eqref{eq:local-slope} of
	the global slope, we get
	\begin{displaymath}
		\frac\d{\dt}\delta(t)\le \lsku {\sfx(0)}{\sfx(t)}\le
		\lambda\,\delta(t)+S
		\quad\text{a.e.~in }(0,+\infty).
	\end{displaymath}
	since $\delta(0)=0$, a
	standard comparison argument yields
	$\delta(h)\le S\int_0^h\rme^{\lambda (h-r)}\,\d r$.
    Dividing by $h>0$ and taking the superior limit as $h\down0$ we obtain
    \begin{displaymath}
        \urmder\sfx 0=\limsup_{h\down0}\frac{\delta(h)}{h}\le S.
    \end{displaymath}
	Applying \eqref{eq:mder-decay}  
    we obtain \eqref{eq:Lip-at-zero} 
    and the fact that $\sfx$ is Lipschitz in every
	bounded interval $[0,T]$.

    {\em Claim \ref{item:slope}.}   
     Assuming $\urmder\sfx t<\infty$, we can apply Remark~\ref{rem:Lip-slope} together with Remark~\ref{rem:slope-lsc}
     (which needs \eqref{eq:conditional-usc})
     we obtain $\skslope\cvx{\sfx(t)}\le 
     \rme^{\lambda_+(T-t)}\urmder\sfx t$,
	and 
    $\skslope\cvx{\sfx(t)}\le 
    \urmder\sfx t$ follows 
    letting $T\down t$. 
    \eqref{eq:slope-mder} then follows from \eqref{eq:mder-decay} and \eqref{eq:Lip-at-zero} 
    (for the curve $s\mapsto \sfx(t+s)$).
    \nc
\end{proof}
We can now compare a strong with a weak solution. Thanks to the
regularizing effect of Theorem~\ref{thm:strong-solutions}, no
condition on the initial datum $\sfx(0)$ of the strong solution is
required.
\begin{theorem}[Comparison between strong and weak solutions]
	\label{thm:cEVI_solutions}
    Let $\sfb$ satisfy \eqref{eq:S-lambda},
	\eqref{eq:b-sublevels-bis} and \eqref{eq:conditional-usc} with
	$\lambda:=\eta-\cvx$.
	Let $\sfx$ be a strong solution to \eqref{eq:SEVI2} and let $\sfy$
	be a weak solution to \eqref{eq:cSEVI2}. Then
	\begin{equation}
		\label{eq:estimate}
		\dX(\sfy(t),\sfx(t))\le \rme^{\lambda(t-s)}\dX(\sfy(s),\sfx(s))\quad\text{whenever}\quad
		0\le s\le t.
	\end{equation}
	In particular, a strong and a weak solution starting from the same
	initial datum coincide.
\end{theorem}
\begin{proof}
	By Theorem~\ref{thm:strong-solutions}, $\sfx$ is locally Lipschitz
	in $(0,+\infty)$ and $\sfx(t)\in \domainslope\cvx$ for every $t>0$.
	Since $\sfx$ and $\sfy$ are continuous, it is sufficient to prove
	\eqref{eq:estimate} when $0<s\le t$; up to replacing the two
	solutions with $\sfx(s+\cdot)$ and $\sfy(s+\cdot)$, we can assume
	that $s=0$, that $\sfx(0)\in \domainslope{\cvx}$, and that $\sfx$
	is Lipschitz in $[0,T]$ for every $T>0$.
	We argue by a standard doubling variables
	technique\nc\ (see e.g.~\cite{Nochetto-Savare06}
    and Appendix~\ref{app:cauchy_estimate_proof}).
    We extend $\sfx$ in the interval $(-\infty,0)$
    by setting $\sfx(t):=\sfx(0)$ for $t<0.$
    Let us set
    $\zeta(s,t):=
    \rme^{-\frac 12\lambda(s+t)}\dX(\sfx(s),
    \sfy(t))$;  we easily see from \eqref{eq:SEVI1}, \eqref{eq:cSEVI1} and \eqref{eq:above} that
    \begin{displaymath}
        \frac{\partial}{\partial s}
        \zeta(s,t)+
        \frac{\partial}{\partial t}
        \zeta(s,t)\le f(s,t)
        \quad\text{in }\mathscr D'(\R\times (0,+\infty))
    \end{displaymath}
    where
    \begin{displaymath}
        f(s,t):=
        \begin{cases}
            0&\text{if }s>0,\\
            -
            \rme^{-\frac\lambda 2(s+t)}
            \lsku{\sfy(t)}{\sfx(0)}
            &\text{if }s\le 0.
        \end{cases}
    \end{displaymath}
    Indeed, since $\sfx$ is Lipschitz, for a.e.~$t$ (where
    $\sfy(t)\in\DDD$, Remark~\ref{rem:weak-in-D}) the map
    $s\mapsto\zeta(s,t)$ satisfies the corresponding a.e.~inequality by
    \eqref{eq:SEVI1}, while for every $s$ the map $t\mapsto\zeta(s,t)$
    satisfies the distributional inequality given by \eqref{eq:cSEVI1}
    and Lemma~\ref{le:ulDini}, whose right-hand side is locally
    integrable in $t$ and, by \eqref{eq:above} and
    \eqref{eq:Lip-at-zero}, locally uniformly bounded from above;
    testing with a nonnegative $\phi\in C^\infty_\rmc(\R\times(0,+\infty))$
    and combining the two inequalities via Fubini's theorem, the
    first-order form of the identity \eqref{eq:general-skew-symmetry}
    yields $f$ as above. Notice that for a.e.~$t>0$ we have
    $-\lsku{\sfy(t)}{\sfx(0)}\le \skslope\cvx{\sfx(0)}$ by
    Definition~\ref{def:bounded_Delta}.
    For $0<\sigma<T$ we integrate the inequality over the strip
    $\{(s,t):\sigma<t<T,\ t-\eps<s<t\}$; the flux of the field
    $(\zeta,\zeta)$ through the two oblique edges vanishes, and, letting
    $\sigma\down0$, by the continuity of $\zeta$ up to $t=0$ we obtain,
    for every $0<\eps\le T$,
    \begin{displaymath}
        \int_{T-\eps}^{T}\zeta(s,T)\,\d s\le
        \int_{-\eps}^{0}\zeta(s,0)\,\d s
        +\iint_{Q(\eps)}f(s,t)\,\d s\,\d t,
    \end{displaymath}
    where $Q(\eps):=\{(s,t):-\eps<s<0,\ 0<t<s+\eps\}$ is the portion of
    the strip where $s\le 0$. Using \eqref{eq:above} and $\sfx(0)\in\domainslope\cvx$, $f$ is bounded from above by $e^{\lambda \eps/2}\skslope\cvx{\sfx(0)}$ on
    $Q(\eps)$. Since $\sfx(0)\in \domainslope{\cvx}$, and
    $|Q(\eps)|=\eps^2/2$, dividing by $\eps>0$ and passing to the limit
    as $\eps\down0$ we obtain
    $\rme^{-\lambda T}\dX(\sfy(T),\sfx(T))\le \dX(\sfy(0),\sfx(0))$,
    i.e.~\eqref{eq:estimate}.\nc
\end{proof}
Finally, we combine the previous results with the stability
of weak solutions under locally uniform convergence, assuming that a
strong solution exists for every initial datum in $\domainslope\cvx$
and that the latter is dense in $\DDD$.
\begin{theorem}[Generation of the \sEVI\ flow]
	\label{thm:EVI-flow}
    Let $\sfb$ satisfy \eqref{eq:S-lambda},
	\eqref{eq:b-sublevels-bis} and \eqref{eq:conditional-usc} with
	$\lambda:=\eta-\cvx$.
	Let us suppose that the domain of the global slope is dense
	in $\DDD$,
	\begin{equation}
		\label{eq:domain-slope}
		\text{the set }\domainslope{\cvx}\text{ (Definition~\ref{def:bounded_Delta}) is dense in }\DDD,
	\end{equation}
	and that\nc\ for every $x_0\in \domainslope{\cvx}$ there
	exists a strong solution to \eqref{eq:SEVI2} starting from $x_0$.
	Then there exists a unique \sEVI\ flow
	$\sfS_t:\overline\DDD\to\overline\DDD$ according to
	Definition~\ref{def:weak-solutions}; it satisfies the contraction
	property
	\begin{equation}
		\label{eq:flow-contraction}
		\dX(\sfS_tx_0,\sfS_ty_0)\le \rme^{\lambda t}\,\dX(x_0,y_0)
		\quad\text{for every }x_0,y_0\in \overline\DDD,\ t\ge0,
	\end{equation}
	and for every $x_0\in \overline\DDD$ the curve
	$t\mapsto \sfS_tx_0$ is the unique weak solution to
	\eqref{eq:cSEVI2} starting from $x_0$.
\end{theorem}
\begin{proof}
	If $x_0\in \domainslope\cvx$, the strong solution starting from
	$x_0$ is unique by
	Theorem~\ref{thm:strong-solutions}(\ref{item:strong-contraction})
	and we denote it by $t\mapsto \sfS_tx_0$;
	\eqref{eq:estimate-strong} yields \eqref{eq:flow-contraction} for
	initial data in $\domainslope\cvx$. Since
	$\sfS_hx_0\in \domainslope\cvx$ for every $h>0$ by
	Theorem~\ref{thm:strong-solutions}(\ref{item:slope}) and
	$t\mapsto \sfS_{t+h}x_0$ is a strong solution starting from
	$\sfS_hx_0$, uniqueness yields the semigroup property
	$\sfS_{t+h}x_0=\sfS_t(\sfS_hx_0)$.

	If $x_0\in \overline\DDD$, the density \eqref{eq:domain-slope} of
	$\domainslope\cvx$ in $\DDD$ (and thus in $\overline\DDD$) provides
	a sequence $x_{0,n}\in \domainslope\cvx$ converging to $x_0$; by
	\eqref{eq:flow-contraction} the curves $t\mapsto \sfS_tx_{0,n}$
	converge locally uniformly to a continuous curve
	$t\mapsto \sfS_tx_0$ starting from $x_0$, which does not depend on
	the approximating sequence and still satisfies
	\eqref{eq:flow-contraction} and the semigroup property. Being a
	locally uniform limit of strong solutions, it is a weak solution
	to \eqref{eq:cSEVI2} by the stability observed after
	\eqref{eq:integral}.

	Finally, if $\sfy$ is a weak solution starting from
	$x_0\in \overline\DDD$, then Theorem~\ref{thm:cEVI_solutions}
	yields 
	\begin{displaymath}
		\dX(\sfy(t),\sfS_tx_0)=
		\lim_{n\to\infty}\dX(\sfy(t),\sfS_tx_{0,n})\le
		\limsup_{n\to\infty}\rme^{\lambda t}\,\dX(x_0,x_{0,n})=0,
	\end{displaymath}
	so that $\sfy(t)$ coincides with $\sfS_t x_0$ for all $t\ge 0$: for every initial
	datum in $\overline\DDD$ there is exactly one weak solution. The
	uniqueness of the \sEVI\ flow follows as well: any two \sEVI\
	flows consist of weak solutions, so that they coincide on $\DDD$,
	and thus on $\overline\DDD$ by the continuity of the maps
	$\sfS_t$.
\end{proof}
\begin{remark}[Use of the assumptions]
	\label{rem:which-assumptions}
	Claims (\ref{item:strong-contraction})--(\ref{item:Lip-at-zero})
	of Theorem~\ref{thm:strong-solutions} 
    rely on the dissipativity
	\eqref{eq:S-lambda} only; the conditional upper semicontinuity
	\eqref{eq:conditional-usc} enters just in the proof of
	\eqref{eq:slope-mder}, through Remark~\ref{rem:Lip-slope}. The
	density condition \eqref{eq:domain-slope} is only needed to
	generate the flow on the whole of $\overline\DDD$ in
	Theorem~\ref{thm:EVI-flow}. Let us
	recall that, by Corollary~\ref{cor:slope-density}, it is
	automatically satisfied whenever every $\sfb(\cdot,\bar x)$,
	$\bar x\in \DDD$, is quadratically bounded from below \eqref{eq:MY}
	and the scheme admits one-step solutions from $\bar x$ for
	arbitrarily small step sizes.\nc
\end{remark}
    
\subsection{Convergence of the discrete $\sk$-EVI}
Recall from Definition~\ref{def:discrete-bEVI} the discrete
$\sfb$-Evolution Variational Inequality \eqref{eq:D-SEVI2} generated by
$\sfb$ with step size $\tau$, obtained in Section~\ref{sec:VMS} as the
recursive form of the single-step resolvent inequality
\eqref{eq:resolvent-evi}:
\begin{equation*}
		\frac1{2\tau}\Big(\dX^2(X^n_\tau,y)-\dX^2(X^{n-1}_\tau,y)\Big)
		+
		\frac1{2\tau}\dX^2(X^n_\tau,X^{n-1}_\tau)
		\le
		\lskd y{X^n_\tau}
		\qquad\text{for every }y\in \DDD\,.
	\end{equation*}
In this subsection 
we assume a priori the solvability of the scheme (studied in
Section~\ref{sec:soln_to_discrete_EVI} by the equilibrium theory of
Section~\ref{sec:equilibrium_problems}) 
and we prove a
\emph{structural} result: whenever the discrete
$\sfb$-EVI admits solutions for every sufficiently small step size and
$\sfb$ is dissipative, the piecewise constant interpolants converge,
with an explicit rate of order $\sqrt\tau$, to the unique strong
solution of \eqref{eq:SEVI2}. It is worth noticing that 
all the information is coded in \eqref{eq:D-SEVI2} and 
no
compactness and  no convexity are needed
here. The completeness of the sublevels \eqref{eq:b-sublevels-bis} is
sufficient, since the dissipativity confines the discrete solutions in
a complete sublevel of $\sfb(\cdot,x_0)$.

The starting point is a discrete counterpart of the regularization
estimates of Theorem~\ref{thm:strong-solutions}. Recall the
first-order form \eqref{eq:D-SEVI1} of the scheme
(Lemma~\ref{le:simple-but-tricky}) and the slope bound
\eqref{eq:iterates-slope} of Proposition~\ref{prop:iterates-slope},
which hold for every solution of \eqref{eq:D-SEVI2} without any
dissipativity assumption on $\sfb$. Combining \eqref{eq:D-SEVI1} with
the dissipativity condition \eqref{eq:S-lambda1}, we obtain a discrete
Lipschitz estimate, propagating the initial value of the slope along
the iterates.
\begin{proposition}[Stability of the slope along discrete iterates]
	\label{le:apriori}
	Let $\sfb$ satisfy the dissipativity condition \eqref{eq:S-lambda},
	let $\lambda:=\eta-\cvx$, and let $\tau>0$ with $\lambda\tau<1$,
	$\tilde\tau:=\big(\frac1\tau-\lambda\big)^{-1}$. Every solution
	$(X^n_\tau)_{n\in\N}$ of \eqref{eq:D-SEVI2} with
	$X^0_\tau\in \domainslope\cvx$ takes values in $\domainslope\cvx$
	and satisfies, for every $n\ge1$,
	\begin{gather}
		\label{eq:discrete-Lipschitz}
		\tau\,\skslope{\cvx}{X_\tau^n}\le
		\dX(X_\tau^n,X_\tau^{n-1})\le
		\tilde\tau\,\skslope{\cvx}{X_\tau^{n-1}},
		\\
		\label{eq:slope-stability}
		\skslope{\cvx}{X^n_\tau}\le
		(1-\lambda\tau)^{-n}\,\skslope{\cvx}{X^0_\tau}.
	\end{gather}
	In particular, for every $T>0$ and $\tau_o>0$ with
	$\lambda\tau_o<1$ we have
	\begin{equation}
		\label{eq:monotonicity}
		\skslope{\cvx}{X^n_\tau}\le c_0\,\skslope{\cvx}{X^0_\tau}
		\qquad\text{for every }\tau\in(0,\tau_o],\
		0\le n\le\Big\lceil\frac T\tau\Big\rceil,
	\end{equation}
	where
	\begin{equation*}
		c_0=c_0(\tau_o,T,\lambda):=
		\begin{cases}
			1&\text{ if }\lambda\le0,\\
			(1-\lambda\tau_o)^{-(T/\tau_o+1)}&\text{ if }\lambda>0.
		\end{cases}
	\end{equation*}
\end{proposition}
\begin{proof}
	Choosing $Y:=X^{n-1}_\tau$ in \eqref{eq:D-SEVI1} and using
	\eqref{eq:S-lambda1} we get
	\begin{equation}
		\label{eq:discrete-derivative}
		\frac{\dX(X^n_\tau,X^{n-1}_\tau)}\tau
		\le \lsku{X^{n-1}_\tau}{X^n_\tau}
		\le -\lsku{X^{n}_\tau}{X^{n-1}_\tau}
		+\lambda\,\dX(X_\tau^n,X_\tau^{n-1})
		\le \skslope{\cvx}{X^{n-1}_\tau}
		+\lambda\,\dX(X_\tau^n,X_\tau^{n-1}),
	\end{equation}
	where the last inequality follows from the definition
	\eqref{eq:local-slope} of the global slope. Since $\lambda\tau<1$,
	\eqref{eq:discrete-derivative} yields the second inequality of
	\eqref{eq:discrete-Lipschitz}; the first one is
	\eqref{eq:iterates-slope}. Combining the two inequalities we obtain
	\begin{equation*}
		\skslope{\cvx}{X^n_\tau}\le
		\frac{\tilde\tau}\tau\,\skslope{\cvx}{X^{n-1}_\tau}
		=(1-\lambda\tau)^{-1}\,\skslope{\cvx}{X^{n-1}_\tau},
	\end{equation*}
	and an induction gives \eqref{eq:slope-stability}. If $\lambda\le0$
	then $(1-\lambda\tau)^{-n}\le1$; if $\lambda>0$, for every
	$n\le\lceil T/\tau\rceil$ an elementary computation shows that
	$(1-\lambda\tau)^{-n}\le(1-\lambda\tau)^{-(T/\tau+1)}\le
	(1-\lambda\tau_o)^{-(T/\tau_o+1)}$, whence \eqref{eq:monotonicity}.
\end{proof}
To approximate a continuous-time solution, we use the piecewise
constant interpolants $\ol X_\tau,\ul X_\tau$ introduced in
\eqref{eq:interpolation}. The key ingredient is the following Cauchy
estimate, comparing two discrete solutions with different step sizes;
its proof, which relies on the dissipativity of $\sfb$ only (through
\eqref{eq:S-lambda1} and Proposition~\ref{le:apriori}), is postponed
to Appendix~\ref{app:cauchy_estimate_proof} and is an adaptation of
the approach in \cite{Nochetto-Savare06}.\nc
\begin{proposition}[Cauchy estimate]\label{prop:cauchy_condition}
Let $\sfb$ satisfy the dissipativity condition
\eqref{eq:S-lambda} and let $\lambda:=\eta-\kappa$.\ Fix $T>0$ and $\tau_o$ small enough such that $\lambda<1/\tau_o$. For any $\tau_1, \tau_2\in (0,\tau_o]$,
	if $X^0_{\tau_1},Y^0_{\tau_2}$ belong to $\domainslope{\cvx}$, then the curves
	$\overline X_{\tau_1}(t), \overline Y_{\tau_2}(t)$  constructed from  the discrete solution to \eqref{eq:D-SEVI2}
	satisfy the Cauchy condition
    \begin{align*}
    \dX(\ol X_{\tau_1}(T),\ol Y_{\tau_2}(T))
    \le  e^{\lambda T}\dX(X^0_{\tau_1},Y^0_{\tau_2}) + C_\tau(T)\,.
\end{align*}
where for $\tilde\tau:=\max\{\tilde\tau_1, \tilde\tau_2\}$ and $\tilde\tau_i := \left(\frac{1}{\tau_i}- \lambda \right)^{-1}$ we define
\begin{gather*}
    C_\tau(T)= a_\tau(T) +\lambda \int_0^T a_\tau(t) e^{\lambda(T-t)}\d t\\
    a_\tau(T)= 2c_T \sqrt{\tilde\tau} \left[\sqrt{\tilde\tau} +
  \sqrt{\left(1+\frac{\tilde \tau}{\tau_1 }\right)(h_{\tau_1}(T)+h_{\tau_2}(T))}
  \right]
\left(\Delta^\kappa (X_{\tau_1}^0) +\Delta^\kappa (Y_{\tau_2}^0)\right)\,.
\end{gather*}
with $h_\tau(T)=\frac{T}{2}+\frac \tau2\ell_\tau(T)(1-\ell_\tau(T))$ for $\ell_\tau(T):=T/\tau-n$ if $n\tau< T\le (n+1)\tau$, and $c_T>0$ a constant that only depends on $\tau_o, T, \lambda$ and increases in $T$. In particular, 
\begin{equation}\label{eq:Cauchy-est}
    \dX(\ol X_{\tau_1}(T),\ol Y_{\tau_2}(T))-e^{\lambda T}\dX(X^0_{\tau_1},Y^0_{\tau_2})=\mathcal{O}\left(\max\{\tau_1,\tau_2\}/\sqrt{\tau_1}\right)\to 0
\end{equation}
for any $\tau_1,\tau_2 \to 0$ satisfying  $\tau_2^2/\tau_1\to 0$.
\end{proposition}
We can now state the main result of this subsection.
\begin{theorem}[Convergence to the strong solution]\label{thm:tau_to_zero}
	Let $\sfb$ satisfy \eqref{eq:S-lambda},
	\eqref{eq:b-sublevels-bis} and \eqref{eq:conditional-usc}, and
	assume that $\sfb(\cdot,\bar x)$ is quadratically bounded from below
	\eqref{eq:MY} for every $\bar x\in\DDD$. Let
	$\lambda:=\eta-\cvx$, $x_0\in \domainslope\cvx$, and let $\tau_o>0$
	with $\lambda\tau_o<1$. If for every $\tau\in(0,\tau_o]$ there
	exists a solution $(X^n_\tau)_{n\in\N}$ of \eqref{eq:D-SEVI2} with
	$X^0_\tau=x_0$, then \eqref{eq:SEVI2} admits a unique strong
	solution $\sfx$ with $\sfx(0)=x_0$ and for every $T>0$ we have
	\begin{equation}
		\label{eq:tau-rate}
		\dX(\ol X_\tau(t),\sfx(t))\le C\sqrt\tau\,\skslope\cvx{x_0}
		\quad\text{for every }t\in[0,T],\ \tau\in(0,\tau_o],
	\end{equation}
	with a constant $C=C(\tau_o,T,\lambda)$; in particular
	$\ol X_\tau,\ul X_\tau\to\sfx$ uniformly in $[0,T]$ as
	$\tau\down0$.
\end{theorem}
\begin{proof}
	Throughout the proof we fix $T>0$, we set
	$\Delta_0:=\skslope\cvx{x_0}$, and we denote by $C$ a generic
	positive constant depending only on $\tau_o,T,\lambda$.

	{\em Step 1: Convergence of the interpolants.}
	Let $0<\tau_2\le\tau_1\le\tau_o$. The two discrete solutions with time steps $\tau_1$ and $\tau_2$ share
	the initial datum $x_0$; applying
	Proposition~\ref{prop:cauchy_condition} in the interval $[0,t]$,
	$t\le T$, and using the monotonicity of the constants with respect
	to the time horizon, the explicit form of the error term
	$C_\tau(t)$ (where now $\tilde\tau=\tilde\tau_1\le
	C\tau_1$ and $h_{\tau_i}(t)\le C$) yields
	\begin{equation}
		\label{eq:Cauchy-rate}
		\dX(\ol X_{\tau_1}(t),\ol X_{\tau_2}(t))\le
		C\sqrt{\tau_1}\,\Delta_0
		\quad\text{for every }t\in[0,T],
	\end{equation}
	so that $(\ol X_\tau(t))_{\tau\in(0,\tau_o]}$ satisfies the Cauchy
	condition as $\tau\down0$, uniformly with respect to $t\in[0,T]$.
	Moreover, \eqref{eq:discrete-Lipschitz} and \eqref{eq:monotonicity}
	give, for $t\in((n-1)\tau,n\tau]\cap[0,T]$,
	\begin{equation}
		\label{eq:interp-bounds}
		\skslope\cvx{\ol X_\tau(t)}\le L:=c_0\Delta_0,\qquad
		\dX(\ol X_\tau(t),x_0)\le
		\sum_{i=1}^{n}\dX(X^i_\tau,X^{i-1}_\tau)\le
		\frac{L\,(t+\tau)}{1-\lambda\tau}\le R,
	\end{equation}
	with $R:=\max\{1,(1-\lambda\tau_o)^{-1}\}\,L\,(T+\tau_o)$. We claim
	that the values of all the interpolants lie in a complete sublevel
	of $\sfb(\cdot,x_0)$: indeed, by \eqref{eq:S-lambda} and the
	definition \eqref{eq:local-slope} of the global slope, every
	$x\in\DDD$ with $\skslope\cvx x\le L$ and $\dX(x,x_0)\le R$
	satisfies
	\begin{equation*}
		\sfb(x,x_0)\le \eta\,\dX^2(x,x_0)-\sfb(x_0,x)
		\le |\eta|\,R^2+LR+\frac{|\cvx|}2R^2=:b^*,
	\end{equation*}
	since $-\sfb(x_0,x)\le \big(\lskd{x_0}x\big)_-
	+\frac{|\cvx|}2\dX^2(x_0,x)\le
	\skslope\cvx x\,\dX(x_0,x)+\frac{|\cvx|}2\dX^2(x_0,x)$. The
	sublevel $\{y\in\DDD:\sfb(y,x_0)\le b^*\}$ is complete by
	\eqref{eq:b-sublevels-bis}; therefore for every $t\in[0,T]$ the
	limit $\sfx(t):=\lim_{\tau\down0}\ol X_\tau(t)$ exists in $\DDD$,
	the convergence is uniform in $[0,T]$, and \eqref{eq:tau-rate}
	follows from \eqref{eq:Cauchy-rate} letting $\tau_2\down0$ with
	$\tau_1=\tau$. Since
	$\dX(\ol X_\tau(t),\ul X_\tau(t))\le\tilde\tau L\to0$ by
	\eqref{eq:discrete-Lipschitz}, the interpolants $\ul X_\tau$
	converge to the same limit; clearly $\sfx(0)=x_0$.

	{\em Step 2: Regularity of the limit.} If $0\le s<t\le T$ with
	$s\in((m-1)\tau,m\tau]$ and $t\in((n-1)\tau,n\tau]$, then
	$(n-m)\tau < t-s+\tau$ and \eqref{eq:discrete-Lipschitz},
	\eqref{eq:monotonicity} give
	\begin{equation*}
		\dX(\ol X_\tau(t),\ol X_\tau(s))\le
		\sum_{i=m+1}^n\dX(X^i_\tau,X^{i-1}_\tau)\le
		(n-m)\,\tilde\tau\,L<\frac{L\,(t-s+\tau)}{1-\lambda\tau};
	\end{equation*}
	passing to the limit as $\tau\down0$ we obtain
	$\dX(\sfx(t),\sfx(s))\le L\,(t-s)$, so that $\sfx$ is Lipschitz in
	$[0,T]$. Since $\skslope\cvx\cdot$ is lower semicontinuous
	(Remark~\ref{rem:slope-lsc}), \eqref{eq:interp-bounds} also yields
	$\skslope\cvx{\sfx(t)}\le L$ for every $t\in[0,T]$.

	{\em Step 3: The limit is a strong solution.} 
    We apply
	Lemma~\ref{le:discrete-to-continuous} along any sequence
	$\tau_j\down0$. Condition~(i) is the uniform convergence of Step~1.
	For~(ii), by \eqref{eq:discrete-Lipschitz} and
	\eqref{eq:monotonicity} we have
	$\dX(X^n_\tau,X^{n-1}_\tau)\le\tilde\tau\,\skslope\cvx{X^{n-1}_\tau}
	\le\tilde\tau L$, whence
	\begin{equation*}
		\sum_{n=1}^{\lceil T/\tau\rceil}
		\frac{\dX^2(X^n_\tau,X^{n-1}_\tau)}{\tau}
		\le \Big\lceil\tfrac T\tau\Big\rceil\,\frac{\tilde\tau^2}{\tau}\,L^2
		\le \frac{(T+\tau_o)\,L^2}{(1-\lambda_+\tau_o)^2}=:E_T'.
	\end{equation*}
	 For~(iii), writing $R_Y:=R+\dX(x_0,Y)$ and choosing
	$\bar\tau=\bar\tau(Y)$ as in \eqref{eq:MY}, the dissipativity
	\eqref{eq:S-lambda}, the lower bound \eqref{eq:MY} at $\bar x:=Y$,
	and the uniform bounds \eqref{eq:interp-bounds} give,
	for every $Y\in\DDD$ and $t\in[0,T]$,
	\begin{equation*}
		\skd Y{\ol X_\tau(t)}\le
		\eta\,\dX^2(\ol X_\tau(t),Y)-\skd{\ol X_\tau(t)}Y
		\le \Big(|\eta|+\tfrac1{2\bar\tau}\Big)R_Y^2-\bif_{\bar\tau}(Y),
	\end{equation*}
	which is finite and independent of $\tau,t$. Finally, condition~(iv)
	is precisely the conditional upper semicontinuity
	\eqref{eq:conditional-usc}, valid at every $t$ because
	$\ol X_{\tau_j}(t)\to\sfx(t)$ and
	$\sup_j\skslope\cvx{\ol X_{\tau_j}(t)}\le L<\infty$ by
	\eqref{eq:interp-bounds}. Lemma~\ref{le:discrete-to-continuous} then
	shows that $\sfx$ is a strong solution to \eqref{eq:SEVI2} (with
	$|\dot\sfx|\in L^2(0,T)$); uniqueness follows from
	Theorem~\ref{thm:strong-solutions}(\ref{item:strong-contraction}).
\end{proof}

\subsection{Well-posedness of the continuous $\sk$-EVI}
 We finally combine the three main ingredients developed so far:
(i) the solvability of the single-step problem
(Section~\ref{sec:resolvent}), iterated by the Variational Movement
Scheme (Theorem~\ref{thm:saddle-sEVI}), (ii) the structural convergence of
the discrete solutions (Theorem~\ref{thm:tau_to_zero}), (iii) and the
dissipative theory of Section~\ref{sec:admissible-b}. We state two
well-posedness results, corresponding to the dissipative route of
Theorem~\ref{thm:discrete-SEVI} and to the antisymmetric route of
Theorem~\ref{thm:VMS-antisym}; the compact route of
Theorem~\ref{thm:existence-compact} is discussed in
Remark~\ref{rem:compact-dissipative}, and the generation of the
\sEVI\ flow is stated in Corollary~\ref{cor:well-posed-flow}.

Let us first observe that, in the presence of the two-point
convexity conditions underlying all these results, the global slope
$\skslope\cvx\cdot$ of Definition~\ref{def:bounded_Delta} coincides
with the local slope $\slope\cdot$. Indeed, if the two-point conditions
\eqref{eq:gg-dist-2} and \eqref{eq:gg-bif-2} hold at the base point
$\bar x:=x$ with some parameter $\kappa_0$, they also hold for every
$\cvx\le \kappa_0$, since \eqref{eq:gg-bif-2} weakens as $\cvx$
decreases; Proposition~\ref{prop:local-global-slope} and
Remark~\ref{rem:local-global-weaker} then yield
$\skslope\cvx x=\slope x$ for every $\cvx\le \kappa_0$, so that the
choice of the parameter $\cvx$ in the initial condition
$x_0\in \domainslope\cvx$ is immaterial. In particular, under the
assumptions of Theorem~\ref{thm:main_existence} or of
Theorem~\ref{thm:antisym_existence} below, the domain
$\domainslope\cvx$ coincides with the domain $\domainslope{}$ of the
local slope.
\begin{theorem}[Well-posedness: the dissipative case]
	\label{thm:main_existence}
	Let $\sfb:\DDD\times \DDD\to\R$ satisfy \eqref{eq:S-lambda},
	\eqref{eq:b-sublevels-bis}, and \eqref{eq:conditional-usc}, let
	$\lambda:=\eta-\cvx$, and let 
    $\tau_o^{-1}:=2\eta_+-\cvx$ if
	$\cvx<2\eta_+$ and $\tau_o:=+\infty$ otherwise. Suppose that for
	every $\bar x\in \DDD$ there exists a system $\barysystem{\bar x}$
	of \emph{continuous} barycentric maps in $\DDD$ such that
	\begin{enumerate}
		\item[(1)] the convexity conditions \eqref{eq:gg-dist} and
		\eqref{eq:gg-bif} hold for $\barysystem{\bar x}$;
		\item[(2)] for every $y\in \DDD$ the map
		$x\mapsto \skd yx$ is upper hemicontinuous with
		respect to the system of $2$-barycentric maps of
		$\barysystem{\bar x}$ (Definition~\ref{def:hemicontinuity}). 
	\end{enumerate}
	Then the following hold.
	\begin{enumerate}
		\item[\rm(i)] For every $\tau\in(0,\tau_o)$ and every
		$X^0_\tau\in \DDD$ 
        the scheme (VMS) admits a unique solution
		$(X^n_\tau)_{n\in\N}$, which solves the discrete EVI
		\eqref{eq:D-SEVI2}.
		\item[\rm(ii)] For every $x_0\in \domainslope{}$ there
		exists a unique strong solution $\sfx$ to \eqref{eq:SEVI2} with
		$\sfx(0)=x_0$; the interpolants $\ol X_\tau$ of the discrete
		solutions starting from $x_0$ converge to $\sfx$ uniformly on
		compact intervals as $\tau\down0$, with the rate 
        	\begin{equation*}
		\dX(\ol X_\tau(t),\sfx(t))\le C\sqrt\tau\,\skslope\cvx{x_0}
		\quad\text{for every }t\in[0,T],\ \tau\in(0,\tau_o],
	\end{equation*} 
        and the contraction estimate \eqref{eq:estimate-strong} holds
		between any two strong solutions.
	\end{enumerate}
\end{theorem}
\begin{proof}
	By Proposition~\ref{prop:B-convexity}, Condition (1) implies that for
	every $\tau>0$ and $\bar x\in \DDD$ the maps
	$y\mapsto \sfB(\tau,\bar x;y,x)$, $x\in \DDD$, satisfy the
	barycentric convexity \eqref{eq:B-bary-convex} and the two-point
	condition \eqref{eq:B-two-convex} with $\xi=\frac1\tau+\cvx$; if
	$\tau<\tau_o$ then $\xi>2\eta_+$ by Remark~\ref{rem:explain}.
	Since the maps of $\barysystem{\bar x}$ are continuous, Condition (2)
	and Remark~\ref{rem:B-hemicontinuity} yield the upper
	hemicontinuity of $x\mapsto \sfB(\tau,\bar x;y,x)$ with respect to
	the $2$-barycentric maps of $\barysystem{\bar x}$. Together with
	\eqref{eq:S-lambda} and \eqref{eq:b-sublevels-bis}, the assumptions
	of Theorem~\ref{thm:discrete-SEVI} thus hold for every base point,
	so that Theorem~\ref{thm:saddle-sEVI} provides a unique solution of
	(VMS), which solves \eqref{eq:D-SEVI2}: this proves {\rm(i)}.

	Since $\tau_o^{-1}\ge 2\eta_+-\cvx\ge \eta-\cvx=\lambda$, every
	$\tau<\tau_o$ satisfies $\lambda\tau<1$: the existence and
	convergence in claim {\rm(ii)} then follow from
	Theorem~\ref{thm:tau_to_zero},  whose lower bound \eqref{eq:MY}
	at every $\bar x\in\DDD$ is provided by Lemma~\ref{le:MY-convex}
	thanks to Condition (1) and \eqref{eq:b-sublevels-bis}, while its uniqueness and the
	contraction estimate follow from
	Theorem~\ref{thm:strong-solutions}(\ref{item:strong-contraction}).
\end{proof}

When $\sfb$ is antisymmetric the assumptions simplify considerably in requiring only two points rather than $J\ge 2$ points for convexity and completeness of the sublevels.
\begin{theorem}[Well-posedness: the antisymmetric case]
	\label{thm:antisym_existence}
	Let $\sfb:\DDD\times \DDD\to\R$ be antisymmetric
	\eqref{eq:antisymmetry} and satisfy \eqref{eq:b-sublevels-bis}, let
	$\cvx\in\R$, $\lambda:=-\cvx$, and suppose that for every
	$\bar x\in \DDD$ and every $x_0,x_1\in \DDD$ there exists a family
	$(z_t)_{t\in[0,1]}$ satisfying the two-point convexity conditions
	\eqref{eq:gg-dist-2} and \eqref{eq:gg-bif-2}. Then the conclusions
	{\rm(i)--(ii)} of Theorem~\ref{thm:main_existence} hold with
	$\eta=0$ and $\tau_o$ characterized by
	$\tau_o^{-1}=(-\cvx)_+=\lambda_+$ if $\cvx<0$, and $\tau_o=+\infty$ otherwise.
\end{theorem}
\begin{proof}
	The antisymmetry implies the $0$-interaction dissipativity and the
	diagonal condition, i.e.~\eqref{eq:S-lambda} with $\eta=0$;
	moreover, by Remark~\ref{rem:skew-symmetric}, the conditional upper
	semicontinuity \eqref{eq:conditional-usc} is a consequence of
	\eqref{eq:b-sublevels-bis}. By
	Proposition~\ref{prop:B-convexity}(1), the two-point conditions
	imply \eqref{eq:B-two-convex} with $\xi=\frac1\tau+\cvx$, which is
	positive if and only if $\lambda\tau<1$, i.e.~$\tau<\tau_o$.
	Theorems~\ref{thm:VMS-antisym} and~\ref{thm:saddle-sEVI} then
	provide, for every $X^0_\tau\in \DDD$, a unique solution of (VMS)
	solving \eqref{eq:D-SEVI2}, and we conclude as in the proof of
	Theorem~\ref{thm:main_existence} by applying
	Theorems~\ref{thm:tau_to_zero}
	and~\ref{thm:strong-solutions}(\ref{item:strong-contraction}).
\end{proof}
\begin{remark}[The compact case]
	\label{rem:compact-dissipative}
	Theorem~\ref{thm:existence-compact} provides strong solutions to
	\eqref{eq:SEVI2} starting from \emph{every} $x_0\in \DDD$ under
	compactness and convexity assumptions, with no dissipativity of
	$\sfb$, but without uniqueness. If $\sfb$ is in addition
	$\eta$-interaction dissipative, i.e.~\eqref{eq:S-lambda} holds,
	then Theorem~\ref{thm:strong-solutions} immediately upgrades that
	existence result: strong solutions are unique, satisfy the
	contraction estimate \eqref{eq:estimate-strong}, and enjoy the
	regularizing estimates \eqref{eq:mder-decay} and
	\eqref{eq:Lip-at-zero} (recall
	Remark~\ref{rem:which-assumptions}). If moreover
	\eqref{eq:b-sublevels-bis} and \eqref{eq:conditional-usc} hold,
	then for every $x_0\in \domainslope{}$ the \emph{whole} family of
	interpolants $\ol X_\tau$ converges to the solution with the rate
	\eqref{eq:tau-rate}, by Theorem~\ref{thm:tau_to_zero}: notice that
	the possible non-uniqueness of the discrete iterates plays no role
	there. 
\end{remark}
\begin{corollary}[Generation of the \sEVI\ flow]
	\label{cor:well-posed-flow}
	Under the assumptions of Theorem~\ref{thm:main_existence} or of
	Theorem~\ref{thm:antisym_existence}, there exists a unique
	\sEVI\ flow $\sfS_t:\overline\DDD\to\overline\DDD$ in the sense of
	Definition~\ref{def:weak-solutions}, satisfying the contraction
	property \eqref{eq:flow-contraction}. In particular, for every
	$x_0\in \overline\DDD$ there exists a unique weak solution to
	\eqref{eq:cSEVI2} starting from $x_0$.
\end{corollary}
\begin{proof}
	The convexity and completeness assumptions of  Theorems~\ref{thm:main_existence} and \ref{thm:antisym_existence}
	guarantee, via Proposition~\ref{prop:B-convexity}(1), the two-point
	convexity \eqref{eq:B-two-convex} at every base point
	$\bar x\in \DDD$; since both theorems also provide, for every
	$\bar x\in \DDD$ and every $\tau\in(0,\tau_o)$, a one-step solution
	of the discrete $\sfb^\cvx$-EVI \eqref{eq:D-SEVI2}, the density
	condition \eqref{eq:domain-slope} follows from
	Corollary~\ref{cor:slope-density}.
	Theorems~\ref{thm:main_existence} and~\ref{thm:antisym_existence}
	provide a strong solution starting from every initial datum in
	$\domainslope\cvx$, so that all the assumptions of
	Theorem~\ref{thm:EVI-flow} are satisfied.
\end{proof}

\subsection{Running Examples: Strong and Weak EVI Solutions}
\begin{itemize}
    \item \textbf{Example (I): Monotone operator.}
	In the monotone operator setting,
	\begin{equation*}
		\skd yx:=\langle \sfA (x),y-x\rangle,
        \quad \cvx=0\,.
	\end{equation*}
	 Assume that $\<\sfA(x)-\sfA(y),x-y> \ge -\eta \norm{x-y}^2$ for some $\eta\in\R$ (with $\eta\le 0$ if $\sfA$ is monotone), so that $\sfb$ is $\eta$-interaction dissipative and $\lambda=\eta$; notice also that $\skslope0x=\norm{\sfA(x)}$, so that the dual form \eqref{eq:cSEVI2} is tested on the domain of $\sfA$. The strong form \eqref{eq:SEVI2} writes
     \begin{align*}
        \frac12 \ddt \norm{\sfx(t)-y}^2 \le \<\sfA(\sfx(t)),y-\sfx(t)> \quad \text{for all }y\in \DDD\,,
     \end{align*}
     while the weak dual form \eqref{eq:cSEVI2} means $\sfx(t)$ satisfies in the sense of distributions
     \begin{align*}
       \frac12  \ddt \norm{\sfx(t)-y}^2 \le \<\sfA(y),y-\sfx(t)> +\eta \norm{\sfx(t)-y}^2\quad \text{for all }y\in  \domainslope0\nc\,.
     \end{align*}
     Notice that if the domain of $\sfA$ requires some regularity, the weak dual form allows for less regularity in $\sfx$ because the operator $\sfA$ is applied only to $y$, whereas in the strong form it is applied to $\sfx$.
    \item \textbf{Example (II): Single species gradient flow.} Choosing $\skd yx:=\varphi(y)-\varphi(x)$
	\eqref{eq:SEVI2} corresponds to the 
	{\rm EVI$_{\cvx}$} formulation of metric gradient flows
	\cite{AGS08,Muratori-SavareI}. In this setting, $\eta=0$, and the strong \eqref{eq:SEVI2} and weak dual \eqref{eq:cSEVI2} evolutions are the same,
    \begin{align*}
       \frac{1}{2} \ddt \dX^2(\sfx(t),y) + \frac{\kappa}{2}\dX^2(\sfx(t),y) \le \varphi(y)-\varphi(\sfx(t)) \quad \text{for all }y\in\DDD\,,
    \end{align*}
    because $\lskd yx = \varphi(y)-\varphi(x) - \frac\cvx2 \dX^2(y,x) = -\sfb^{-\cvx}(x,y)= \cblskd yx$.
    \item \textbf{Example (III): Min-max gradient flows.} With  $\skd yx := F(y^{(1)},x^{(2)})-F(x^{(1)},y^{(2)})$, again we have $\eta=0$. The strong \eqref{eq:SEVI2} and weak dual \eqref{eq:cSEVI2} evolutions are again the same, because $\sfb^\cvx(y,\sfx)=-\sfb^{\cvx}(x,y)=\cblskd y\sfx$:
    \begin{align*}
        \frac12 \ddt \dX^2(\sfx(t),y) + \frac{\cvx}2\dX^2(\sfx(t),y) \le F(y^{(1)},\sfx^{(2)}(t))-F(\sfx^{(1)}(t),y^{(2)}) \quad \text{for all } y\in \domainslope\cvx\,.
    \end{align*}
    \item \textbf{Example (IV): Multispecies gradient flows.} Here, $\skd yx  = \sum_{i=1}^N F^{(i)}(y^{(i)},x^{(-i)})-F^{(i)}(x^{(i)},x^{(-i)})$ and the conjugate bifunction $\cblskdname$ is given by
    \begin{align*}
        \cblskd yx = \sum_{i=1}^N F^{(i)}(y^{(i)},y^{(-i)}) -F^{(i)}(x^{(i)},y^{(-i)}) + \frac{\cvx_{*}}{2} \dX^2(x,y)\,.
    \end{align*}
    We observe that \eqref{eq:cSEVI2} is given by  
    \begin{align*}
        \frac 12 \ddt \dX^2(\sfx(t),y) + \frac{\cvx-2\eta}{2} \dX^2(\sfx(t),y) \le \sum_{i=1}^N F^{(i)}(y^{(i)},y^{(-i)}) -F^{(i)}(x^{(i)},y^{(-i)}) \quad  \text{for all } y\in \domainslope\cvx\nc\,,
    \end{align*}
    which parallels the structure of Example (II), with a different convexity parameter and with $F^{(i)}$ dependent on the test point $y^{(-i)}$. The function
    $x \mapsto  \sum_{i=1}^N F^{(i)}(y^{(i)},y^{(-i)}) -F^{(i)}(x^{(i)},y^{(-i)})$ is $\cvx$-barycentrically concavelike, as is $x\mapsto \varphi(y)-\varphi(x)$.
\end{itemize}

\section{Multispecies Systems}\label{sec:multispecies}
In this section, we provide additional results and examples for multispecies coupled gradient flows, including a proof that $\eta\ge 0$ (Section~\ref{sec:role-eta-multispecies}), and that the steady state of the system is the unique Nash equilibrium of the corresponding energy functionals (Section~\ref{sec:multispecies-Nash}). Then we discuss multispecies gradient flows in Euclidean space (Section~\ref{sec:Euclidean_flows}) and the Wasserstein-2 metric (Section~\ref{sec:W2-flow}). In both settings, we show that the zeroth-order notion of monotonicity that we propose in Definition~\ref{def:monotonicity_zeroth_order} implies the respective first order conditions. In the Euclidean setting, we show how coupling can worsen overall dissipativity, and in the Wasserstein-2 setting we provide an example of energy functionals with a commonly-analyzed nonlocal interaction term.

\subsection{Properties of Multispecies Systems in General Metric Spaces}\label{sec:coupled_GF}
In the setting where $\sk$ specifies a coupled $N$-species gradient flow system, where each species aims to minimize $F^{(i)}:\DDD \to \R$, with $\DDD:=\prod_{i=1}^n \DDD^{(i)}$ and $\sk:\DDD\times\DDD\to \R$ taking the form
\begin{align}\label{eq:sk_coupled_gradient_flow}
    \skd yx = \sum_{i=1}^N F^{(i)}(y^{(i)},x^{(-i)})-F^{(i)}(x^{(i)},x^{(-i)})\,.
\end{align}
In the game theory literature, strong monotonicity of a game is a sufficient condition for existence of a Nash equilibrium. In Euclidean space, a game $(F^{(i)})$ is $-\lambda$-monotone if 
\begin{align}\label{eq:Euclidean_monotonicity}
    \sum_{i=1}^N \<\nabla_{x^{(i)}}F^{(i)}(x)-\nabla_{x^{(i)}}F^{(i)}(y), x^{(i)}-y^{(i)}> \ge -\lambda \norm{x-y}^2 \quad \text{for all }x,y\in\R^d\,.
\end{align}
When $\lambda<0$, it can be shown that coupled gradient descent converges to the unique Nash equilibrium of the game $(F^{(i)})$; however, this definition relies on a notion of gradient of $F^{(i)}$. The $\eta$ interaction dissipativity  and barycentrically $\cvx$-convexlike conditions for the dual equilibrium problem point toward a zeroth-order condition for monotonicity. Recall from Section~\ref{sec:contributions} our proposed  zeroth order definition of monotonicity for coupled games in  metric spaces. 
\begin{definition}[Definition~\ref{def:monotonicity_zeroth_order}: Game-Theoretic Monotonicity]
    A set of energy functionals $(F^{(i)})$ are $(\eta,\cvx)$-monotone, with rate $\lambda = \eta-\cvx$,  if
    \begin{itemize}
        \item $\cvx$ convexity: $y\mapsto \bif(y,x)$ is barycentrically $\cvx$-convexlike with respect to a system $\cB$.
        \item $\eta$ dissipativity: $\skd yx + \skd xy \le \eta \dX^2(x,y)$.
    \end{itemize}
    Equivalently, each 
    $F^{(i)}(\cdot,x^{(-i)})$
    is barycentrically $\cvx$-convexlike and for every $x,y \in \DDD$,
        \begin{align*}
    \sum_{i=1}^N F^{(i)}(y^{(i)},x^{(-i)}) - F^{(i)}(x^{(i)},x^{(-i)})+
   F^{(i)}(x^{(i)},y^{(-i)})-  F^{(i)}(y^{(i)},y^{(-i)}) \le \eta \dX^2(x,y)\,.
    \end{align*}
\end{definition}
In Lemma~\ref{lem:Euclidean_monotonicity}, we show that Definition~\ref{def:monotonicity_zeroth_order} implies the standard Euclidean game theory monotonicity inequality \eqref{eq:Euclidean_monotonicity} when $\cB_2$ includes linear interpolations, and in Lemma~\ref{lem:W2_monotonicity} we show Definition~\ref{def:monotonicity_zeroth_order} implies the $-\lambda$-monotonicity as defined for coupled Wasserstein-2 gradient flows in \cite{conger_monotone_2025} when $\cB_2$ includes geodesics for $J=2$. In the Wasserstein case, the set of curves needed for first-order monotonicity are the geodesics, $\cG_2$, which are a  2-point, rather than $J$-point, interpolation. For many common functionals such as potential, self-interacting, and entropy functionals, geodesic convexity implies convexity along generalized geodesics \cite[Section 4.2.2]{AGS08}, which can then be used to show barycentric convexity.

The reverse direction in both Euclidean and Wasserstein-2 spaces does not necessarily hold. This is because the zeroth order condition requires an upper-bound on $F^{(i)}$, whereas the first order condition requires a lower bound on $\nabla F^{(i)}$, and so the value of $\lambda$ obtained from the first order condition can be tighter. See Example~\ref{ex:zeroth_vs_first_order_monotonicity} in Section~\ref{sec:Euclidean_flows} for an example of this gap. 

\subsubsection{Interaction Dissipativity for Coupled Gradient Flow System}\label{sec:role-eta-multispecies}
The  parameter $\kappa$ specifies the convexity of each $F^{(i)}(\cdot, x^{(-i)})$, 
  and $\eta$ is the dissipativity lost (or gained, if $\eta<0$) from the coupling structure.  While our results are formulated for $\eta\in\R$ for the case of general $\sk$, we can show that in the coupled gradient flow setting, $\eta \ge 0$, that is, the coupling structure cannot improve system dissipativity in comparison with each species evolving separately. 
\begin{proposition}\label{prop:eta_ge0}
    Let $\sk$ take the coupled gradient flow form \eqref{eq:sk_coupled_gradient_flow}. Then $\eta\ge 0$.
\end{proposition}
\begin{proof}
    Suppose $\sfb$ is $\eta$-interaction dissipative according to \eqref{eq:bif-monotone},
    \begin{align*}
        \skd yx + \skd xy \le \eta \dX^2(x,y)\,, \qquad \forall\, x,y\in \DDD \,,
    \end{align*}
    and fix $x,y\in \DDD$ with $\dX(x,y)>0$. Let the swap operator $S^{(i)}:\DDD \times \DDD \to \DDD \times \DDD$, which switches the $i$-th entry in the two input elements, be defined as
    \begin{align*}
        S^{(i)}(x,y) = (\tilde x, \tilde y) \,, \qquad \tilde x^{(j)}= x^{(j)}, \ \tilde y^{(j)}=y^{(j)} \quad \forall j\ne i\,, \quad \text{and }\tilde x^{(i)}=y^{(i)}, \ \tilde y^{(i)} = x^{(i)}\,.
    \end{align*}
    Since $\DDD=\prod_i\DDD^{(i)}$ is a product, $S^{(i)}(x,y)\in\DDD\times\DDD$, and swapping coordinates leaves $\dX(x,y)$ unchanged. We rewrite
    \begin{align*}
       \skd yx+\skd xy&= \sum_{i=1}^N G_i(x,y)\,,\\
       G_i(x,y)&:=   F^{(i)}(x^{(i)},y^{(-i)})-F^{(i)}(x^{(i)},x^{(-i)}) +F^{(i)}(y^{(i)},x^{(-i)})- F^{(i)}(y^{(i)},y^{(-i)})\,,
    \end{align*}
    and note that each $G_i$ changes sign when its own coordinate is swapped: $G_i(S^{(i)}(x,y))=-G_i(x,y)$.

    We consider the case $N=2$ for simplicity. Then the single swap $S^{(1)}$ flips \emph{both} terms: $G_1$ by the identity above, and $G_2$ because species $1$ is the opponent of species $2$, so $S^{(1)}$ swaps precisely the two values $x^{(1)},y^{(1)}$ appearing in $G_2$. Hence $\skd yx+\skd xy$ is antisymmetric under $S^{(1)}$,
    \begin{align*}
        \big(\skd yx+\skd xy\big)+\big(\skd{\tilde y}{\tilde x}+\skd{\tilde x}{\tilde y}\big)=0\,,\qquad (\tilde x,\tilde y):=S^{(1)}(x,y)\,.
    \end{align*}
    Applying \eqref{eq:bif-monotone} to the pairs $(x,y)$ and $(\tilde x,\tilde y)$ and adding, the left-hand side vanishes while, since $\dX(\tilde x,\tilde y)=\dX(x,y)$, the right-hand side equals $2\eta\,\dX^2(x,y)$; therefore $0\le 2\eta\,\dX^2(x,y)$ and $\eta\ge 0$.

     For general $N$ one sums \eqref{eq:bif-monotone} over the $2^N$ pairs $S^A(x,y)$, $A\subset\{1,\dots,N\}$, where $S^A$ composes the swaps $S^{(i)}$, $i\in A$ (all pairs lie in $\DDD\times\DDD$ at the same distance $\dX(x,y)$, since $\dX^2=\sum_i\sfd^2_{\XXX^{(i)}}$): as $G_i\circ S^{A\triangle\{i\}}=-G_i\circ S^A$, pairing $A$ with $A\triangle\{i\}$ cancels the contributions of $G_i$ for every $i$, and $0\le 2^N\eta\,\dX^2(x,y)$ follows.
\end{proof}

\subsubsection{Steady State is Nash Equilibrium for Coupled Multispecies Systems}\label{sec:multispecies-Nash}

To show the existence of a Nash equilibrium, we have two options: proceed via the Dual Equilibrium Problem using Theorem~\ref{thm:equilibrium2}, and then appeal to Minty's Lemma to obtain a solution to the Primal Equilibrium Problem, or directly show existence for the Primal Equilibrium Problem using Theorem~\ref{thm:primal_soln}. Note that the second approach avoids convexity assumptions at the cost of requiring some compactness, and it only yields existence, not uniqueness. Here, we present both approaches in Corollary~\ref{cor:ss_is_Nash} and Corollary~\ref{cor:ss_is_Nash-2} below.

\begin{remark}[Steady states and Nash equilibria]\label{rem:steady-nash}
	Let $x_*\in\DDD$ be a Nash equilibrium of $(F^{(i)})$, i.e.\
	$\skd y{x_*}\ge0$ for every $y\in\DDD$. Then the constant curve
	$\sfx\equiv x_*$ is a steady state of \eqref{eq:SEVI2} in either of
	the following situations.
	\begin{enumerate}
		\item[\rm(i)] If $y\mapsto\skd y{x_*}$ is $\cvx$-barycentrically
		convexlike, then applying the two-point convexity inequality to
		the interpolation between $x_*$ and $y$, using $\skd{x_*}{x_*}=0$
		and $\skd\cdot{x_*}\ge0$, and letting the interpolation parameter
		tend to $0$, gives $\skd y{x_*}\ge\frac\cvx2\dX^2(y,x_*)$ for
		every $y$, which is the steady-state form of \eqref{eq:SEVI2}.
		\item[\rm(ii)] If $\sfb$ is $\eta$-interaction dissipative and the
		flow of \eqref{eq:SEVI2} is well posed,  and $x_*\in\domainslope\cvx$
		(automatic if $\cvx\le0$, or under the convexity of (i)),
		so that the strong solution issued from $x_*$ exists, then the solution $\sfx$
		issued from $x_*$ is stationary: testing \eqref{eq:SEVI2} at
		$y=x_*$ and using $\skd{\sfx(t)}{x_*}\ge0$ and the dissipativity,
		\begin{displaymath}
			\tfrac12\ddt\dX^2(\sfx(t),x_*)\le
			\skd{x_*}{\sfx(t)}-\tfrac\cvx2\dX^2(\sfx(t),x_*)
			\le\Big(\eta-\tfrac\cvx2\Big)\dX^2(\sfx(t),x_*),
		\end{displaymath}
		so that $\dX^2(\sfx(t),x_*)\le\rme^{(2\eta-\cvx)t}\dX^2(x_*,x_*)=0$.
		More generally, this holds whenever
		$\skd y{x_*}+\beta\,\dX^2(y,x_*)\ge0$ for every $y$ and some $\beta\in\R$.
	\end{enumerate}
	Conversely, if $\cvx\ge0$, every steady state $\bar x$ satisfies
	$\skd y{\bar x}\ge\frac\cvx2\dX^2(y,\bar x)\ge0$ and is therefore a
	Nash equilibrium; steady states and Nash equilibria then coincide.
\end{remark}

   \begin{corollary}[Steady State is Nash Equilibrium I]\label{cor:ss_is_Nash}
  Let $(F^{(i)})$ be $(\eta,\cvx)$-monotone according to Definition~\ref{def:monotonicity_zeroth_order} with respect to some $\cB$, with rate $\lambda=\eta-\kappa$ and $\kappa>2\eta_+$. Assume that the maps $y\mapsto \skd yx$, $x\in\DDD$, have complete sublevels and that, for every $y\in\DDD$, the map $x\mapsto\skd yx$ is upper hemicontinuous with respect to the $2$-barycentric maps $\cB_2$ of $\cB$.\nc 
   Then there exists a unique steady state $x_\infty\in \DDD$, which is also the unique Nash equilibrium of $(F^{(i)})$.
\end{corollary}
\begin{proof}     From Theorem~\ref{thm:equilibrium2}, there exists a unique $y_*\in\DDD$ such that 
    \begin{align*}
        \sum_{i=1}^N F^{(i)}(y^{(i)}_*,x^{(-i)})- F^{(i)}(x^{(i)},x^{(-i)})\le 0 \quad \text{for all }x\in\DDD\,.
    \end{align*}
    In order for $y_*$ to be a solution to the primal problem, we need that for $\bary x t\in\cB_2$ such that $\bary x {(1,0)} =y_*$ and $\bary x {(0,1)}=y$, we have $t \mapsto \sum_{i=1}^N F^{(i)}(y^{(i)},\bary x {(1-t,t)}^{(-i)})- F^{(i)}(\bary x {(1-t,t)}^{(i)},\bary x {(1-t,t)}^{(-i)})$ is upper-hemicontinuous at $t=0$ for all $y\in\DDD$, which we obtain from the upper-hemicontinuity assumption.  By Lemma~\ref{le:metric-Minty}, $y_*$ solves the primal problem. Conversely, every solution $x^*$ of the primal problem is a solution of the dual problem by Proposition~\ref{prop:primal-implies-dual}, since the $\cvx$-barycentric convexity of $y\mapsto\skd yx$ yields the radial convexity \eqref{eq:radial-convex} with $\xi=\cvx\ge 2\eta_+$; the uniqueness part of Theorem~\ref{thm:equilibrium2} then gives $x^*=y_*$. Thus $x^*:=y_*$ is the unique solution of the primal problem and satisfies
    \begin{align*}
    \sum_{i=1}^N F^{(i)}(y^{(i)},x^{*(-i)}) - F^{(i)}(x^{*(i)},x^{*(-i)})\ge 0 \quad \text{for all } y\in \DDD\,,
    \end{align*}
    showing that $x^*\in\DDD$ is the unique Nash equilibrium for $(F^{(i)})$. 

    Finally, since $y\mapsto\skd y{x^*}$ is $\cvx$-barycentrically convexlike, Remark~\ref{rem:steady-nash}{\rm(i)} shows that $x^*$ is a steady state; and, since $\cvx>2\eta_+\ge0$, the converse in Remark~\ref{rem:steady-nash} shows that every steady state is a Nash equilibrium, hence coincides with $x^*$. Therefore $x^*$ is the unique steady state.
\end{proof}

   \begin{corollary}[Steady State is Nash Equilibrium II]\label{cor:ss_is_Nash-2}
   Assume $(F^{(i)})$ satisfies for every $J\ge 2$, $x\in\DDD$, $\baryname y\in\cB_J$ and $\aalpha\in \simplex {J}$,
     \begin{align*} 
        \sum_{i=1}^N F^{(i)}(\bary {y^{(i)}} \aalpha,x^{(-i)})\le
        \max\left\{\sum_{i=1}^N F^{(i)}(\baryname {y^{(i)}}_j ,x^{(-i)}):
	j\in \{1,\cdots,J\},
	\alpha_j>0\right\}
    \end{align*}
    Further, assume that the superlevel sets
    \begin{align*}
        \sigma^+(y):=\left\{x\in \DDD:
		 \sum_{i=1}^N F^{(i)}(y^{(i)},x^{(-i)})\ge \sum_{i=1}^N F^{(i)}(x^{(i)},x^{(-i)})\right\}
    \end{align*} are
		closed in $\DDD$ and at least one of them is compact. 
        Assume moreover that, for every $y\in\DDD$, the map $x\mapsto\sum_{i=1}^N F^{(i)}(y^{(i)},x^{(-i)})-F^{(i)}(x^{(i)},x^{(-i)})$ is upper hemicontinuous 
        and $\cB$ is continuous.
   Then there exists a Nash equilibrium $x^*\in\DDD$ of $(F^{(i)})$; if in addition $\sfb$ is $\eta$-interaction dissipative with $\cvx\ge0$ and the flow of \eqref{eq:SEVI2} is well posed, then $x^*$ is also a steady state, and conversely every steady state is a Nash equilibrium.
\end{corollary}
\begin{proof}
    The existence of a Nash equilibrium $x^*$ follows directly from Theorem~\ref{thm:primal_soln}. The steady-state property, and its converse, then follow from Remark~\ref{rem:steady-nash}.
\end{proof}

\subsection{Euclidean Multispecies Gradient Flows}\label{sec:Euclidean_flows}
First, we show that monotonicity according to Definition~\ref{def:monotonicity_zeroth_order} (zeroth order) implies the classical game-theoretic notion of monotonicity for games in Euclidean spaces (first order). The converse does not hold in general. In Example~\ref{ex:zeroth_vs_first_order_monotonicity}, we illustrate a case where the first-order condition does not imply the zeroth order condition. Finally, in Example~\ref{ex:eta-effect}, we demonstrate that coupling between species \emph{in Euclidean space} causes contractivity to slow down.

\begin{lemma}[Monotone Functions in $\R^d$]\label{lem:Euclidean_monotonicity}
    Let $\XXX^{(i)}=\R^{d_i}$, $\XXX=\R^d$ with $d=\sum_i d_i$, and $\dXi$ be the Euclidean norm. If $(F^{(i)})$ are $(\eta,\kappa)$-monotone with respect to linear interpolations according to Definition~\ref{def:monotonicity_zeroth_order}, and $F^{(i)}(\cdot,x^{(-i)}) \in C^1(\R^{d_i};\R)$ for all $i$, then the following classical definition of monotonicity holds:
    \begin{align*}
       \sum_{i=1}^N \<\nabla_{x^{(i)}} F^{(i)}(x^{(i)},x^{(-i)})-\nabla_{y^{(i)}} F^{(i)}(y^{(i)},y^{(-i)}),  x^{(i)}-y^{(i)}> \ge -\lambda \norm{x-y}^2\,.
   \end{align*}
\end{lemma} 
\begin{proof}  
   To show this, we first use that $y \mapsto \skd yx$ is $\kappa$-barycentrically convex, selecting the family linear interpolations for $\cB_2$. The first-order $\kappa$ convexity inequality, comparing $(y,x)$ and $(x,x)$, gives
   \begin{align*}
        \sum_{i=1}^N F^{(i)}(y^{(i)},x^{(-i)}) - F^{(i)}(x^{(i)},x^{(-i)}) \ge \sum_{i=1}^N \bigg(F^{(i)}(x^{(i)},x^{(-i)}) - F^{(i)}(x^{(i)},x^{(-i)})
        \\
       + \nabla_{y^{(i)}} (F^{(i)}(y^{(i)},x^{(-i)}) - F^{(i)}(x^{(i)},x^{(-i)}))\big|_{y=x} (y^{(i)}-x^{(i)}) \bigg) + \frac{\cvx}{2} \norm{x-y}^2\,.
   \end{align*}
   which simplifies to
    \begin{align*}
        \sum_{i=1}^N F^{(i)}(y^{(i)},x^{(-i)}) - F^{(i)}(x^{(i)},x^{(-i)}) \ge \sum_{i=1}^N \nabla_{x^{(i)}} F^{(i)}(x^{(i)},x^{(-i)}) \cdot  (y^{(i)}-x^{(i)})  + \frac{\cvx}{2} \norm{x-y}^2\,.
   \end{align*}
   Repeating the calculation with test points $(x,y)$ and $(y,y)$ and summing, we have
   \begin{align*}
       &\sum_{i=1}^N F^{(i)}(y^{(i)},x^{(-i)}) - F^{(i)}(x^{(i)},x^{(-i)}) + F^{(i)}(x^{(i)},y^{(-i)}) - F^{(i)}(y^{(i)},y^{(-i)}) \\ 
       & \quad \ge 
        \sum_{i=1}^N \<\nabla_{x^{(i)}} F^{(i)}(x^{(i)},x^{(-i)})-\nabla_{y^{(i)}} F^{(i)}(y^{(i)},y^{(-i)}),  y^{(i)}-x^{(i)}> + \cvx \norm{x-y}^2\,.
   \end{align*}
   Now we apply the $\eta$ interaction dissipativity bound to the left-hand side and rearrange to obtain
   \begin{align*}
       \sum_{i=1}^N \<\nabla_{x^{(i)}} F^{(i)}(x^{(i)},x^{(-i)})-\nabla_{y^{(i)}} F^{(i)}(y^{(i)},y^{(-i)}),  x^{(i)}-y^{(i)}> \ge (\cvx-\eta) \norm{x-y}^2\,,
   \end{align*}
   which is exactly the first-order game theoretic monotonicity condition with $\lambda=\eta-\kappa$.

\end{proof}

\begin{example}[Zeroth- and first-order monotonicity discrepancy]\label{ex:zeroth_vs_first_order_monotonicity}
    This example illustrates the gap between first order and zeroth order monotonicity. Consider the two-player game in $X^{(1)}=\R^{d_1}$, $X^{(2)}=\R^{d_2}$ with energies
    \begin{align*}
    F^{(1)}(z) &= \frac{1}{2} z^{(1)^\top} Q_{1} z^{(1)} + z^{(1)^\top} P z^{(2)}\,, \qquad 
    F^{(2)}(z) = \frac{1}{2} z^{(2)^\top} Q_{2} z^{(2)} 
    \end{align*}
    The convexity parameter is $\kappa=\lambda_{\underline Q}$  (for symmetric $Q_1$, $Q_2$), where  $\lambda_{\underline Q}$ is the smallest eigenvalue among the eigenvalues of $Q_1$ and $Q_2$.
    The first-order monotonicity condition depends on the smallest eigenvalue of $M$, denoted $\lambda_{\underline M}$, where
    \begin{align*}
    \sum_{i=1}^2 \langle \nabla_{z^{(i)}} F^{(i)}(z) - \nabla_{z^{(i)}} F^{(i)}(y), z^{(i)}-y^{(i)} \rangle &= (z-y)^\top M (z-y)\ge \lambda_{\underline M}\norm{z-y}^2\,, \quad M = \begin{bmatrix}
        Q_1 & \frac{1}{2} P \\ \frac{1}{2} P^\top & Q_2 
    \end{bmatrix}\,.
\end{align*}
The interaction dissipativity condition is
\begin{align*}
&\sfb(y,z)+\sfb(z,y)\\
    & \ =\sum_{i=1}^2  F^{(i)}(y^{(i)},z^{(-i)})-  F^{(i)}(z^{(i)},z^{(-i)})+
  F^{(i)}(z^{(i)},y^{(-i)}) -F^{(i)}(y^{(i)},y^{(-i)})= -(z^{(1)}-y^{(1)})^\top P (z^{(2)}-y^{(2)}) \\
    &\ \le \sigma_{\max}(P)\,\norm{z^{(1)}-y^{(1)}}\,\norm{z^{(2)}-y^{(2)}} \le \tfrac12\sigma_{\max}(P) \norm{z-y}^2\,,
\end{align*}
where $\sigma_{\max}(P)=\norm{P}_{\mathrm{op}}$ is the largest singular value (equivalently, the operator norm) of $P$ and the bound is sharp, so that $\eta = \tfrac12\sigma_{\max}(P)$. According to the first-order monotonicity definition \eqref{eq:Euclidean_monotonicity}, $\lambda=-\lambda_{\ul M}$, where $\lambda_{\underline M}$ is the smallest eigenvalue of $M$, whereas according to the zeroth order definition (Definition~\ref{def:monotonicity_zeroth_order}), $\lambda=\eta-\kappa=\tfrac12\sigma_{\max}(P)-\lambda_{\underline Q}$. Recall Weyl's inequality: for symmetric matrices $A,B$ one has $\lambda_{\min}(A+B)\ge\lambda_{\min}(A)-\norm{B}_{\mathrm{op}}$. Since $M$ decomposes as $\operatorname{diag}(Q_1,Q_2)$ plus the off-diagonal coupling block $E$ (with $\norm{E}_{\mathrm{op}}=\tfrac12\sigma_{\max}(P)$), it gives $\lambda_{\underline M}\ge \lambda_{\underline Q}-\tfrac12\sigma_{\max}(P)$, so the first-order condition is always at least as tight, and strictly tighter in general.
\end{example}

The following example illustrates the role of $\eta$ interaction dissipativity; here it quantifies how much dissipativity is lost due to the coupling between the two players.
\begin{example}[Non-symmetric two-player game]\label{ex:eta-effect}
   This example shows the setting when $\eta>0$, that is, the coupling term causes contractivity to slow down compared to the convexity of each term independently. Let $\XXX^{(1)}=\XXX^{(2)}=\R$, and $a, b\in\R$, and consider the functionals
    \begin{align*}
        F^{(1)}(x^{(1)},x^{(2)}) = \frac{a}{2}x^{{(1)}^2} + x^{(1)} x^{(2)}\,, \qquad F^{(2)}(x^{(2)},x^{(1)}) = \frac{a}{2}(x^{(2)}-b)^2\,.
    \end{align*}
    The second species is independent of the first species, but the first species evolves based on the second. The first order notion of monotonicity, according to \eqref{eq:Euclidean_monotonicity}, is
    \begin{align*}
        &\< \begin{bmatrix}
            \nabla_{x^{(1)}}F^{(1)}(x^{(1)},x^{(2)}) - \nabla_{y^{(1)}} F^{(1)}(y^{(1)},y^{(2)}) \\
            \nabla_{x^{(2)}}F^{(2)}(x^{(2)},x^{(1)}) - \nabla_{y^{(2)}} F^{(2)}(y^{(2)},y^{(1)})
        \end{bmatrix}, \begin{bmatrix}
            x^{(1)}  - y^{(1)} \\ x^{(2)} - y^{(2)}
        \end{bmatrix} >\\
        &= 
        a (x^{(1)}-y^{(1)})^2 + (x^{(2)}-y^{(2)})(x^{(1)}-y^{(1)}) + a(x^{(2)}-y^{(2)})^2  \\
        &= \left(a-\frac{1}{2}\right)\norm{x-y}^2 + \frac{1}{2}(x^{(2)}-y^{(2)}+x^{(1)}-y^{(1)})^2 
         \ge \left(a-\frac{1}{2}\right)\norm{x-y}^2\,,
    \end{align*}
    and therefore $\lambda=-(a-1/2)$. Moreover, $y\mapsto b(y,x)= \sum_{i=1}^N F^{(i)}(y^{(i)},x^{(-i)})-F^{(i)}(x^{(i)},x^{(-i)})$ is $a$-barycentrically convex with respect to the linear barycentric system in $\R^2$, and checking monotonicity according to the zeroth-order definition, we have
    \begin{align*}
        &\sum_{i=1,2}  F^{(i)}(x^{(i)},y^{(-i)})-F^{(i)}(y^{(i)},y^{(-i)}) +  F^{(i)}(y^{(i)},x^{(-i)})-F^{(i)}(x^{(i)},x^{(-i)})\\
        & = x^{(1)} y^{(2)}-y^{(1)}y^{(2)} +y^{(1)} x^{(2)} -  x^{(1)} x^{(2)}
        = \frac{1}{2} \norm{x-y}^2 -\frac{1}{2}(x^{(1)}-y^{(1)}+x^{(2)}-y^{(2)})^2 \le \frac{1}{2} \norm{x-y}^2\,,
    \end{align*}
    which satisfies the $\eta$-dissipativity condition in Definition~\ref{def:monotonicity_zeroth_order} with $\eta=1/2$. Note that $\kappa=a$ must be large enough for $\lambda<0$ to obtain contraction; this reflects that the strength of the convexity must be large enough to overcome the  effects of the coupling. Additionally, any value of $\lambda\ge -a+1/2$ maintains the first-order monotonicity inequality, in the same way that a $\kappa$-convex function is also $\kappa-\varepsilon$ convex, for any $\varepsilon \ge 0$.
\end{example}

\subsection{Multispecies Coupled Wasserstein-2 Gradient Flows}\label{sec:W2-flow}
In this section, we discuss two different notions of monotonicity for evolving systems of measures, illustrate how the zeroth-order definition of monotonicity (Definition~\ref{def:monotonicity_zeroth_order}) implies the notion of monotonicity proposed in \cite{conger_monotone_2025} for coupled Wasserstein-2 gradient flows (Lemma~\ref{lem:W2_monotonicity}), formally show that the Wasserstein-2 gradient flow PDEs satisfy the EVI (Lemma~\ref{lem:W2GF_satisfies_EVI}), and illustrate these results with the example of a system of coupled nonlocal interaction equations (Example~\ref{ex:W2-interactions}).

\subsubsection{Wasserstein-2 Displacement Monotonicity}\label{sec:W2-monotonicity}

In this subsection, we focus on the displacement monotone setting, with $\cB_2=\mathcal{G}_2$ geodesic curves in Wasserstein-2 space.

\begin{lemma}[Wasserstein-2 Displacement Monotonicity]\label{lem:W2_monotonicity}
    Let $\XXX^{(i)}=\PP_2(\R^{d_i})$, $\dXi=\Wass_2$, and let each $F^{(i)}$ be smooth enough so that $\nabla_{x^{(i)}} \delta_{\rho^{(i)}} F^{(i)}[\rho^{(i)},\rho^{(-i)}](x^{(i)})$ is weakly well-posed. If $(F^{(i)})$ are $(\eta,\cvx)$-monotone with respect to geodesics $\cG_2$ according to Definition~\ref{def:monotonicity_zeroth_order}, then the monotonicity condition in \cite{conger_monotone_2025} holds,
    \begin{align*}
        \sum_{i=1}^N
        \int\< \nabla_{z_i} \delta_{\rho^{(i)}} F^{(i)}[\rho^{(i)},\rho^{(-i)}] (z_i)-\nabla_{v_i} \delta_{\mu^{(i)}} F^{(i)}[\mu^{(i)},\mu^{(-i)}] (v_i), z_i-v_i> \d \gamma^{(i)}(z_i,v_i)\\ \ge  -\lambda \sum_{i=1}^N \Wass_2^2(\rho^{(i)},\mu^{(i)})\,,
    \end{align*}
    where $\gamma^{(i)}\in\Gamma_o(\rho^{(i)},\mu^{(i)})$ are Wasserstein-2 optimal plans with marginals $\rho^{(i)}$ and $\mu^{(i)}$.
\end{lemma} 
\begin{proof}
    Using the $\kappa$-convexity of $y \mapsto \skd yx$ along geodesics $\cG_2$, using $(\mu,\rho)$ and $(\rho,\rho)$ as comparison points in the convexity inequality,
    we have for $(\rho^{(i)}_s)_{s\in[0,1]}$ a geodesic between $\rho^{(i)}$ and $\mu^{(i)}$,
    \begin{align*}
        &\sum_{i=1}^N F^{(i)}(\mu^{(i)},\rho^{(-i)}) -F^{(i)}(\rho^{(i)},\rho^{(-i)}) 
        \ge \frac{d}{ds}\left.\sum_{i=1}^N F^{(i)}\left(\rho^{(i)}_s,\rho^{(-i)}\right) \right|_{s=0} 
        + \frac{\kappa}{2} \Wass_2^2(\rho,\mu) 
    \end{align*}
    which can be rewritten, with enough smoothness in $(F^{(i)})$, using the optimal maps $\gamma^{(i)}\in\Gamma_o(\rho^{(i)},\mu^{(i)})$, as
    \begin{align}\label{eq:W2_convexity_inequality}
        &\sum_{i=1}^N F^{(i)}(\mu^{(i)},\rho^{(-i)}) -F^{(i)}(\rho^{(i)},\rho^{(-i)}) \notag
        \\
        &\quad \ge
          \sum_{i=1}^N \bigg(  \int\< \nabla_{z_i} \delta_{\rho^{(i)}}F^{(i)}[\rho^{(i)},\rho^{(-i)}](z_i), v_i-z_i> \d \gamma^{(i)}(z_i,v_i)  +\frac{\kappa}{2} \Wass_2^2(\rho^{(i)},\mu^{(i)})\bigg)
    \end{align}
    Now using the convexity inequality again with test points $(\rho,\mu)$ and $(\mu,\mu)$ and summing,
        \begin{align*}
         \sum_{i=1}^N F^{(i)}(\mu^{(i)},\rho^{(-i)}) -F^{(i)}(\rho^{(i)},\rho^{(-i)}) + F^{(i)}(\rho^{(i)},\mu^{(-i)}) -F^{(i)}(\mu^{(i)},\mu^{(-i)})
         \ge \sum_{i=1}^N \bigg(  \kappa \Wass_2^2(\rho^{(i)},\mu^{(i)})
        \\-  \int\< \nabla_{z_i} \delta_{\rho^{(i)}} F^{(i)}[\rho^{(i)},\rho^{(-i)}] (z_i)-\nabla_{v_i} \delta_{\mu^{(i)}} F^{(i)}[\mu^{(i)},\mu^{(-i)}] (v_i), z_i-v_i> \d \gamma^{(i)}(z_i,v_i) \bigg) \,.
    \end{align*}
   \sloppy From the $\eta$-dissipativity of $\skd yx$, we upper-bound the left-hand side of the inequality by $\eta \sum_{i=1}^N \Wass_2^2(\rho^{(i)},\mu^{(i)})$ to obtain
    \begin{align*}
      & ( \eta-\kappa) \sum_{i=1}^N \Wass_2^2(\rho^{(i)},\mu^{(i)})
        \\ &\quad \ge \sum_{i=1}^N \bigg(
        -  \int\< \nabla_{z_i} \delta_{\rho^{(i)}} F^{(i)}[\rho^{(i)},\rho^{(-i)}] (z_i)-\nabla_{v_i} \delta_{\mu^{(i)}} F^{(i)}[\mu^{(i)},\mu^{(-i)}] (v_i), z_i-v_i> \d \gamma^{(i)}(z_i,v_i) \bigg) \,,
    \end{align*}
    and inverting the inequality results in exactly $-\lambda=\kappa-\eta$ monotonicity of $(F^{(i)})$ from \cite{conger_monotone_2025},
    \begin{align*}
        \sum_{i=1}^N
        \int\< \nabla_{z_i} \delta_{\rho^{(i)}} F^{(i)}[\rho^{(i)},\rho^{(-i)}] (z_i)-\nabla_{v_i} \delta_{\mu^{(i)}} F^{(i)}[\mu^{(i)},\mu^{(-i)}] (v_i), z_i-v_i> \d \gamma^{(i)}(z_i,v_i)\\ \ge  -\lambda \sum_{i=1}^N \Wass_2^2(\rho^{(i)},\mu^{(i)})\,.
    \end{align*}
\end{proof}
In the next result, we show formally that multispecies gradient flow dynamics of the form $\ddt \rho^{(i)}(t)=\div{\rho^{(i)}\nabla_{x^{(i)}}\delta_{\rho^{(i)}}F^{(i)}[\rho]}$ for all $i\in[1,\dots,N]$ satisfy 
\eqref{eq:SEVI2}.
\begin{lemma}[Multispecies Wasserstein-2 Gradient Flow Satisfies EVI (Formal)]\label{lem:W2GF_satisfies_EVI}
    Consider dynamics given by
    \begin{align}\label{eq:W2_dynamics}
        \ddt \rho^{(i)}=\div{\rho^{(i)}\nabla_{x^{(i)}}\delta_{\rho^{(i)}}F^{(i)}(\rho)} \quad \text{for all } i\in[1,\dots,N]
    \end{align}
    for functional $F^{(i)}:\DDD\to\R$.
    Let $(F^{(i)})$ be $(\eta,\cvx)$-monotone according to Definition~\ref{def:monotonicity_zeroth_order} with $\cB_2=\mathcal{G}_2$ given by Wasserstein-2 geodesics. Any (weak, in the sense of distributions) solution to \eqref{eq:W2_dynamics} satisfies the EVI, for all $\mu\in\DDD$ and all $t \ge 0$,
    \begin{align*}
       \frac12 \ddt \sum_{i=1}^N\Wass_2^2(\rho^{(i)}(t),\mu^{(i)}) + \frac \cvx 2  \sum_{i=1}^N\Wass_2^2(\rho^{(i)}(t),\mu^{(i)}) \le \sum_{i=1}^N F^{(i)}(\mu^{(i)},\rho^{(-i)}(t)) - F^{(i)}(\rho^{(i)}(t),\rho^{(-i)}(t)) \,.
    \end{align*}
\end{lemma}
\begin{proof}
    Since $\mu\mapsto \skd \mu {\rho(t)}$ is $\kappa$-convex with respect to  geodesics we can use again \eqref{eq:W2_convexity_inequality}
    \begin{align*}
        &\frac \cvx 2 \sum_{i=1}^N \Wass_2^2(\rho^{(i)}(t),\mu^{(i)}) + \sum_{i=1}^N\int \<\nabla_{z_i} \delta_{\rho^{(i)}}F^{(i)}[\rho^{(i)}(t),\rho^{(-i)}(t)](z_i),v_i-z_i> \d \gamma^{(i)}(t,z_i,v_i)\\
        &\le \sum_{i=1}^N F^{(i)}(\mu^{(i)},\rho^{(-i)}(t)) - F^{(i)}(\rho^{(i)}(t),\rho^{(-i)}(t)) \quad \text{for all } \mu \in\DDD\,,
    \end{align*}
    where $\gamma^{(i)}(t) \in \Gamma_o(\rho^{(i)}(t),\mu^{(i)})$. Identifying 
    \begin{align*}
        \frac{1}{2} \ddt \Wass_2^2(\rho^{(i)}(t),\mu^{(i)}) = \int \<\nabla_{z_i} \delta_{\rho^{(i)}}F^{(i)}[\rho^{(i)}(t),\rho^{(-i)}(t)](z_i),v_i-z_i> \d \gamma^{(i)}(t,z_i,v_i)\,,
    \end{align*}
    using \cite[Lemma 8.4.7]{AGS08} allows us to conclude the EVI.
\end{proof}

The following example illustrates how Definition~\ref{def:monotonicity_zeroth_order} can be applied to a specific choice of coupling between two species in the Wasserstein-2 space.
\begin{example}[Attractive-Repulsive Kernels]\label{ex:W2-interactions} 
Let $\hat F^{(1)}:\mathcal{P}_2(\mathbb{R}^d)\to \mathbb{R}$ and $\hat F^{(2)}:\mathcal{P}_2(\mathbb{R}^d) \to \mathbb{R}$ be $\kappa_1>0$ and $\kappa_2>0$ displacement convex. Consider the setting with a nonlocal interaction kernel $W\in C^2(\mathbb{R}^d,\mathbb{R})$ that is a symmetric function $W(x)=W(-x)$, and energy functionals
\begin{align*}
    F^{(1)}(\rho^{(1)},\rho^{(2)}) = \hat F^{(1)}(\rho^{(1)}) + \iint W(x_1-x_2) \, \mathrm{d} \rho^{(1)}(x_1) \, \mathrm{d}\rho^{(2)}(x_2)\,, \qquad F^{(2)}(\rho) = \hat F^{(2)}(\rho^{(2)})\,.
\end{align*}
First we check the $\eta$ interaction dissipativity inequality from Definition~\ref{def:monotonicity_zeroth_order}. Only the cross-interaction term coupling the two species appears:
\begin{align*}
&\sfb(\rho, \mu)+\sfb(\mu,\rho)\\
    &=-\iint [ W(y_1 - y_2)-W(x_1-y_2) +  W(x_1-x_2)-W(y_1 - x_2)]\, \mathrm{d} \rho^{(1)}(x_1)\, \mathrm{d}\mu^{(1)}(y_1)\, \mathrm{d}\rho^{(2)}(x_2)\, \mathrm{d}\mu^{(2)}(y_2)  \\
    &=\iint  \int_0^1 \int_0^1
    \langle x_1 - y_1, \nabla^2 W(I(s,\tau,x,y)) \cdot (x_2-y_2) \rangle \, \mathrm{d} s \, \mathrm{d} \tau \,  
    \mathrm{d} \gamma^{(1)}(x_1,y_1) \, \mathrm{d} \gamma^{(2)}(x_2,y_2)
\end{align*}
where $I(s,\tau,x,y):=(1-s)x_1+sy_1-(1-\tau)x_2 - \tau y_2$ and $\gamma^{(1)}\in \Gamma_o(\rho^{(1)},\mu^{(1)})$, $\gamma^{(2)}\in \Gamma_o(\rho^{(2)},\mu^{(2)})$. The $\eta$ interaction dissipativity inequality holds if for a.e. $(x_1,y_1)\in\supp\gamma^{(1)}$, $(x_2,y_2)\in\supp\gamma^{(2)}$,
\begin{align}\label{eq:Hess-W-bound}
    \int_0^1 \int_0^1 \langle x_1 - y_1, \nabla^2 W(I(s,\tau,x,y)) \cdot (x_2-y_2) \rangle \, \mathrm{d} s \, \mathrm{d} \tau \le \eta \|x-y\|^2\,.
\end{align}
For example, consider the case where $W=W^{(m)}$ is an attractive-repulsive kernel, with $l_a>l_r>0$, $C_r>C_a>0$: 
\begin{align*}
    W^{(m)}(x) = C_r \exp\bigg(-\frac{\|x\|^2}{l_r}\bigg) - C_a \exp\bigg(-\frac{\|x\|^2}{l_a}\bigg)\,.
\end{align*}
This kernel has been studied in a number of applications; see \cite{balague_confinement_2012,craig_convergence_2016} and references therein. The Hessian of $W^{(m)}$ is uniformly bounded in $x$ because
\begin{align*}
    \nabla^2{W^{(m)}}(x) &= 2I_{d}\left( - \frac{C_r}{l_r} \exp(-\|x\|^2/l_r) + \frac{C_a}{l_a} \exp(-\|x\|^2/l_a )\right) 
    \\&\quad 
    + 4x x^\top\left(\frac{C_r}{l_r^2} \exp(-\|x\|^2/l_r) -  \frac{C_a}{l_a^2} \exp(-\|x\|^2/l_a )\right)\,.
\end{align*}
Then, since
\begin{align*}
    \langle x_1-y_1, \nabla^2 W(z)\,(x_2-y_2)\rangle
    \le \|\nabla^2 W\|_{L^\infty}\,|x_1-y_1|\,|x_2-y_2|
    \le \tfrac12\|\nabla^2 W\|_{L^\infty}\,\|x-y\|^2
\end{align*}
for every $z$ (recall $\|x-y\|^2=|x_1-y_1|^2+|x_2-y_2|^2$), the bound \eqref{eq:Hess-W-bound} holds with $\eta:=\tfrac12\|\nabla^2 W\|_{L^\infty}$; integrating this pointwise inequality against $\gamma^{(1)}\otimes\gamma^{(2)}$ and using the optimality of the plans reproduces the $\eta$-interaction dissipativity. For $W=W^{(m)}$, bounding the two terms of its Hessian separately---the isotropic one by $4\max\{C_r/l_r,C_a/l_a\}=4 C_r/l_r$ (opposite signs and $\rme^{-\|x\|^2/l}\le1$), the $4xx^\top$ one via $\|x\|^2\rme^{-\|x\|^2/l}\le l\,\rme^{-1}$ (as $\sup_{t\ge0}t\rme^{-t}=\rme^{-1}$)---gives
\begin{align*}
    \|\nabla^2 W^{(m)}\|_{L^\infty}\le \Big(4+\tfrac{4}{\rme}\Big)\tfrac{C_r}{l_r}
    \qquad \Rightarrow \qquad \eta \le 2\Big(1+\tfrac{1}{\rme}\Big)\tfrac{C_r}{l_r}\,.
\end{align*}
The coupling also affects the convexity parameter: freezing $\rho^{(2)}$, the map $\rho^{(1)}\mapsto F^{(1)}(\rho^{(1)},\rho^{(2)})$ carries the potential energy of $W\ast\rho^{(2)}$, whose displacement-convexity modulus is bounded below by $-c_W$, with $c_W:=\sup_z\big(\lambda_{\min}\nabla^2 W(z)\big)_-\le \|\nabla^2 W\|_{L^\infty}$. Hence the joint convexity parameter is $\kappa=\min\{\kappa_1-c_W,\kappa_2\}$ and
\begin{align*}
    \lambda = \eta-\kappa = \tfrac12\|\nabla^2 W\|_{L^\infty}-\min\{\kappa_1-c_W,\kappa_2\}\,,
\end{align*}
which is negative---so that the dynamics contracts---precisely when $\min\{\kappa_1-c_W,\kappa_2\}>\tfrac12\|\nabla^2 W\|_{L^\infty}$, i.e.\ when the convexity of $(\hat F^{(i)})$ dominates the Hessian of the interaction kernel. Here the only interaction between species is generated by the convolution potential $W$, and the $\eta$-interaction dissipativity is thus a Hessian control of $W$. The convexity above is stated along geodesics, which suffices for the monotonicity and the contraction estimate; well-posedness of the associated scheme instead requires convexity along generalized geodesics (Example~\ref{ex:generalized-geo}), since the term $\tfrac{1}{2\tau}\Wass_2^2(\cdot,\bar\rho)$ enters the generating bifunction. As the $\hat F^{(i)}$ and $W\ast\rho^{(-i)}$ are internal and potential energies, their geodesic convexity---and hence the parameter $\kappa$ above---transfers to generalized geodesics \cite[Section 9.3]{AGS08}.

Alternatively, one can also proceed via compactness, if $(\hat F^{(i)})$ and $W$ satisfy suitable assumptions. Let us assume for instance that $\hat F^{(i)}$ are proper, lower semicontinuous with respect to $\Wass_2$, and have compact sublevels, and $W$ satisfies $L_W:=\|\nabla W\|_{L^\infty}<\infty$ (these assumptions are chosen for convenience; weaker assumptions are possible). Then $V(\rho):=\hat F^{(1)}(\rho^{(1)})+\hat F^{(2)}(\rho^{(2)})$ has compact sublevels on $\mathcal{P}_2(\mathbb{R}^d)\times \mathcal{P}_2(\mathbb{R}^d)$, and we can write $\sfb(\mu,\rho)=V(\mu)-V(\rho)+\sfb'(\mu,\rho)$ with
\begin{align*}
    \sfb'(\mu,\rho)&:=\iint W(y_1-x_2)\d \mu^{(1)}(y_1)\d \rho^{(2)}(x_2) - \iint W(x_1-x_2)\d \rho^{(1)}(x_1)\d \rho^{(2)}(x_2) \\
    &=\int \Psi_2(x_1) \d (\mu^{(1)}-\rho^{(1)})(x_1)\,,\qquad
    \Psi_2(x_1):=(W\ast \rho^{(2)})(x_1)\\
    &\le L_W \Wass_1(\mu^{(1)}, \rho^{(1)})\le  L_W \Wass_2(\mu^{(1)}, \rho^{(1)})\,.
\end{align*}
Hence, the Lyapunov decomposition \eqref{eq:lyapunov} from Section~\ref{sec:resolvent} holds with $C_1=L_W$ and $C_2=0$. 
Upper semicontinuity of $\sfb'(\mu,\rho)$ in $\rho$ is immediate thanks to $W$ being Lipschitz continuous and bounded, so $\sfb'$ is even continuous under $\Wass_2$ convergence. 
In order to apply Theorem~\ref{thm:existence-compact}, additional assumptions on $(\hat F^{(i)})$ and $W$ are needed to guarantee Conditions (1) and (2). Let $\hat F^{(i)}(\rho^{(i)})=\int U_i(x_i)\d\rho^{(i)}(x_i)$ for $U_i\in C^2(\R^d, \R)$ with $\nabla^2 U_i \ge \cvx_i I$, and $\cvx:=\min\{\cvx_1-c_W,\cvx_2\}\ge 0$. Then 
Conditions (1) and (2) in Theorem~\ref{thm:existence-compact} hold along generalized Wasserstein-2 geodesics with base point $\bar\mu$ (Example~\ref{ex:generalized-geo}) for any $\tau_o>0$. We conclude that for any $x_0\in\DDD$ there exists a strong solution to \eqref{eq:SEVI2}.
\end{example}

\subsubsection{$L^2$ versus $\Wass_2$ Geometries}\label{sec:L2vsW2}

In this subsection, we apply our framework to provide well-posedness and existence results for the evolution considered in the recent work \cite{wang_local_2026} (Corollary~\ref{cor:wellposednessL2} and Corollary~\ref{cor:existenceL2}), by carefully relating properties in the $L^2$ and $\Wass_2$ geometries, filling a gap in the literature. In doing so, we illustrate the role of a suitable choice of curves for existence of solutions to the EVI in the $\Wass_2$ geometry when considering evolutions driven by general velocity fields: for convexity of $\Wass_2(\cdot,\bar\rho)$, we need to use generalized Wasserstein-2 geodesics with basepoint $\bar\rho$, whereas for convexity of the corresponding bifunction, a different basepoint is required. This can be addressed by switching the bifunction $\sfb$ to a choice derived from $L^2$ dynamics instead.

\begin{example}[Linear versus Displacement Monotonicity]\label{ex:L2_W2_monotonicity}
In this example, we study a set of $N$ measures evolving according to a set of velocity fields, which do not necessarily have a gradient flow structure. The flow considered in \cite{wang_local_2026} uses the $\Wass_2$ metric but with a notion of monotonicity corresponding to $L^2$ dynamics. To analyze this flow, we first consider dynamics in $\Wass_2$ and $L^2$ separately for measures in $\PP_2$.
Let $\XXX^{(i)}=\PP_2(\R^{d_i})$, $\DDD\subset \Pi_{i=1}^N \XXX^{(i)}$, and $f^{(i)}:\DDD \to C^1(\R^{d_i};\R)$ any potential (not necessarily coming from the first variation of an energy functional).
The $\Wass_2$ flow dynamics are given by
\begin{align}\label{eq:team_MFG}
    \dot \rho^{(i)} = \nabla \cdot (\rho^{(i)} \nabla f^{(i)}[\rho]) 
    \quad \text{for all }i\in[1,\dots,N]\,.
\end{align}
The $L^2$ flow dynamics are given by
\begin{align}\label{eq:L2_flow}
    \dot \rho^{(i)} = -f^{(i)}[\rho]
    \,, \quad \int f^{(i)}[\rho](v_i)\d v_i = 0 \quad \text{for all }\rho \in \DDD\,, \quad \text{for all }i\in[1,\dots,N]\,.
\end{align}
The work \cite{wang_local_2026} proposes a definition of \textit{linear monotonicity}, given by
\begin{align}\label{eq:linear_monotonicity}
   \sum_{i=1}^N \int \left(f^{(i)}[\rho](v_i)-f^{(i)}[\mu](v_i) \right) \d (\rho^{(i)}-\mu^{(i)})(v_i) \ge 0 \quad \text{for all }\rho, \mu\in\DDD\,.
\end{align}
This notion of monotonicity is closely tied to the $L^2$ metric, as we will see when considering the form of the EVI.
In \cite{conger_monotone_2025} a \textit{displacement monotonicity} definition is given by
\begin{align}\label{eq:displacement_monotonicity}
   \sum_{i=1}^N \int \<\nabla f^{(i)}[\rho](v_i)-\nabla f^{(i)}[\mu](z_i),v_i-z_i > \d \gamma^{(i)}(v_i,z_i) \ge 0\,,
\end{align}
where $\gamma^{(i)}\in\Gamma_o(\rho^{(i)},\mu^{(i)})$.
The following $L^2$-bifunction recovers the $L^2$-flow \eqref{eq:L2_flow} with metric $d_\XXX=L^2$:
\begin{align*}
    \sfb_{L^2}( \mu, \rho) = -\sum_{i=1}^N \int f^{(i)}[\rho](v_i) \,\d (\rho^{(i)}-\mu^{(i)})(v_i)\,.
\end{align*}
Similarly, the following $\Wass_2$-bifunction recovers the $\Wass_2$-flow \eqref{eq:team_MFG} with metric $d_\XXX=\Wass_2$:
\begin{align*}
    \sfb_{\Wass_2}( \mu, \rho) = -\sum_{i=1}^N \int \<\nabla f^{(i)}[\rho](v_i),v_i-z_i>\d\gamma^{(i)}(v_i,z_i) \,,
\end{align*}
where $\gamma^{(i)}\in\Gamma_o(\rho^{(i)},\mu^{(i)})$. In both cases, this choice of bifunction $\sfb$ recovers interaction dissipativity with $\eta=0$ since linear monotonicity \eqref{eq:linear_monotonicity} in $L^2$ (respectively displacement monotonicity \eqref{eq:displacement_monotonicity} in $\Wass_2$) implies
$$\skd \rho \mu + \skd \mu \rho \le 0\,.$$

Consider now the question of well-posedness for \eqref{eq:team_MFG}-\eqref{eq:L2_flow}. The bifunction 
 $\sfb_{L^2}$ is affine in $\mu$, hence $0$-barycentrically convexlike along linear interpolations $\mu_t=(1-t)\mu_0+t\mu_1$; along $\Wass_2$ generalized geodesics $\mu_t=((1-t)T_0+tT_1)_\#\nu$, however, $\int f\,\d\mu_t=\int f\circ((1-t)T_0+tT_1)\,\d\nu$, so that convexity of $\mu\mapsto\sfb_{L^2}(\mu,\rho)$ along them amounts to convexity of the potentials $f^{(i)}[\rho]$. Let $\cB_{\bar \rho}$ and $\cB_\rho$ be two families of generalized geodesics (Example~\ref{ex:generalized-geo} and Section~\ref{sec:ex-VMS}) with base points $\bar \rho$ and $ \rho$ respectively. The bifunction $\mu \mapsto \sfb_{\Wass_2}(\mu, \rho)$ is 0-barycentrically convexlike with respect to generalized geodesics $\cB_\rho$ with basepoint $\rho$, but not necessarily  with respect to $\cB_{\bar \rho}$. This poses an issue because $\frac{1}{2}\Wass_2^2(\cdot,\bar \rho)$ is $1$-barycentrically convexlike with respect to  $\cB_{\bar \rho}$, but not with respect to $\cB_{\rho}$ for any other basepoint $\rho$. Then $\mu \mapsto \sfB(\tau,\bar \rho;\mu,\rho)=\sfb_{\Wass_2}(\mu,\rho)-\frac{1}{2\tau}\Wass_2^2(\rho,\bar \rho)+\frac{1}{2\tau}\Wass_2^2(\mu,\bar \rho)$ is \textit{not} $1/\tau$ barycentrically convexlike with respect to either $\cB_{\bar \rho}$ or $\cB_\rho$. We will work around this issue later by using  $\sfb_{L^2}$ rather than  $\sfb_{\Wass_2}$, whilst still fixing the choice of metrix as $\Wass_2$. 
Checking the form of the EVIs, we compute the time derivative of the distance function in $\Wass_2$ and $L^2$ along solutions to \eqref{eq:team_MFG} and \eqref{eq:L2_flow} respectively,
\begin{align}
     \frac 12 \ddt \Wass_2^2(\rho(t),\mu) &= -\sum_{i=1}^N\int \<\nabla f^{(i)}[\rho](v_i),v_i-z^{(i)}> \d \gamma^{(i)}(v_i,z^{(i)}) \,, \label{eq:W2_EVI} \\
  \frac 12 \ddt \sum_{i=1}^N\int\norm{\rho^{(i)}(t,v_i)-\mu^{(i)}(v_i)}^2 \d v_i &= -\sum_{i=1}^N\int f^{(i)}[\rho](v_i)\, \d (\rho^{(i)}-\mu^{(i)})(v_i) \,. \label{eq:L2_EVI}
\end{align}
This shows that solutions to \eqref{eq:SEVI2} with $\sfb_{\Wass_2}$ solve \eqref{eq:team_MFG} and solutions to \eqref{eq:SEVI2} with $\sfb_{L^2}$ solve \eqref{eq:L2_flow}. However, in \cite{wang_local_2026}, the linear monotonicity condition \eqref{eq:linear_monotonicity}, which comes from the $L^2$ geometry, is used with $\Wass_2$ dynamics \eqref{eq:team_MFG}, with well-posedness and existence of equilibrium assumed. In particular, \cite{wang_local_2026} considers the EVI
\begin{align}\label{eq:mixed_metric}
   \frac 12 \ddt \Wass_2^2(\rho(t),\mu) \le \sfb_{L^2}(\mu,\rho(t))  \,.
\end{align}
To relate these dynamics to the $\Wass_2$-flow \eqref{eq:W2_EVI}, we derive a bound for the corresponding bifunctions. This will allow us to show well-posedness and existence of an equilibrium for the mixed dynamics \eqref{eq:mixed_metric}, a question that was left open in \cite{wang_local_2026}.
\begin{lemma}\label{lem:bifunction_ineq}
    Let $\nabla^2 f^{(i)}[\rho](z_i) \succeq 0$ for all $\rho\in\DDD$, $z_i\in\R^{d_i}$, and $i\in[1,\dots,N]$. Then $ \sfb_{\Wass_2}(\mu,\rho)\le \sfb_{L^2}(\mu,\rho)$ for all $\rho,\mu\in\DDD$.
\end{lemma}
\begin{proof}
First, we use that
\begin{align*}
    \int f^{(i)}[\rho](v_i) \d (\rho^{(i)}-\mu^{(i)}) = \int (f^{(i)}[\rho](v_i)-f^{(i)}[\rho](z_i)) \d \gamma^{(i)}(v_i,z_i)
\end{align*}
for $\gamma^{(i)}\in\Gamma_o(\rho^{(i)},\mu^{(i)})$.
Next we use the Taylor expansion 
\begin{align*}
    f^{(i)}[\rho](z_i) &= f^{(i)}[\rho](v_i) + \nabla f^{(i)}[\rho](v_i) \cdot (z_i-v_i) \\
    &+ (z_i-v_i) \cdot \int_0^1 \nabla^2 f^{(i)}[\rho]((1-s)v_i+sz_i)\cdot (z_i-v_i) \d s \\
    &\ge f^{(i)}[\rho](v_i) + \nabla f^{(i)}[\rho](v_i) \cdot (z_i-v_i) 
\end{align*}
using that $\nabla^2 f^{(i)}[\rho](z_i) \succeq 0$.    
\end{proof}
\begin{remark}
    The previous lemma illustrates the relationship between  Wasserstein-2 convexity and linear convexity. Suppose a sufficiently smooth  functional $F:\PP_2 \to \R$ satisfies displacement convexity:
\begin{align*}
    \text{Hess}_{\Wass_2} F[\nu](\Phi,\Phi) &= \int \Phi^\top(z) \nabla^2 \delta_{\nu}F[\nu](z)\cdot \Phi(z) \,\d \nu(z)
    \\
    &\quad
    +\iint \Phi^\top(z) \nabla^2_{z,\hat z } \delta^2_{\nu} F[\nu](z,\hat z) \cdot \Phi(\hat z) \, \d \nu(z)\d \nu(\hat z) \ge 0\,,  \quad \text{for all }\Phi\in L^2(\nu)\,.
\end{align*}
\sloppy 
Consider the $\Wass_2$-gradient flow of $F$, that is, \eqref{eq:team_MFG} with $N=1$ and $f^{(1)}=\delta_\rho F[\rho]$.
The linear monotonicity condition \eqref{eq:linear_monotonicity}  only constrains the second term, $\iint \Phi^\top(z) \nabla^2_{z,\hat z } \delta^2_{\nu} F[\nu](z,\hat z) \cdot \Phi(\hat z) \, \d \nu(z)\d \nu(\hat z)\ge 0$ (indeed, this can be derived by choosing $\rho=\nu_t:=\nu+t\sigma$ and $\mu=\nu$ in \eqref{eq:linear_monotonicity}, where $\sigma:=-\nabla \cdot(\nu\Phi)$ for $\Phi\in L^2(\nu)$; dividing by $t^2$, letting $t\to0$, and integrating by parts, one obtains the second term above), whereas the additional requirement in Lemma~\ref{lem:bifunction_ineq} provides the other term, $\int \Phi^\top(z) \nabla^2 \delta_{\nu}F[\nu](z)\cdot \Phi(z) \,\d \nu(z) \ge 0$.
\end{remark}
By Lemma~\ref{lem:bifunction_ineq}, $ \sfb_{\Wass_2}(\mu,\rho(t))\le \sfb_{L^2}(\mu,\rho(t))$, and so any solution to \eqref{eq:W2_EVI} solves \eqref{eq:mixed_metric}, and therefore solutions to \eqref{eq:team_MFG} also solve \eqref{eq:mixed_metric}. Conversely, if we can show that \eqref{eq:mixed_metric} has a unique strong solution, then we 
have also shown uniqueness of strong solutions for \eqref{eq:team_MFG}. Here, we will show existence of a unique strong solution to \eqref{eq:mixed_metric} directly via the dual formulation.
    \begin{corollary}[Well-Posedness]\label{cor:wellposednessL2}
        Let linear monotonicity, given by \eqref{eq:linear_monotonicity}, hold for all $\mu,\rho\in\DDD$, and $\nabla^2 f^{(i)}[\rho](z_i) \succeq 0$ for all $\rho\in\DDD$, $z_i\in\R^{d_i}$, and $i\in[1,\dots,N]$. Under \eqref{eq:b-sublevels-bis}-\eqref{eq:conditional-usc} and upper-hemicontinuity of $\rho\mapsto \sfb_{L^2}(\mu,\rho)$ for all $\mu \in \DDD$ with respect to 2-barycentric maps in $\cB_{2,\bar\mu}$, 
        there exists a unique strong solution $\rho(t)$ to \eqref{eq:mixed_metric} starting from $$\rho(0)\in\DDD(\Delta):=\{\rho\in\DDD\,:\,\sum_{i=1}^N\int_{\R^{d_i}}|\nabla f^{(i)}[\rho](v_i)|^2\,\d\rho^{(i)}(v_i) <\infty \}\,.$$
    \end{corollary}
    \begin{proof}
        Selecting $\cB_{\bar \mu}$ to be Wasserstein-2 generalized geodesics, we have that $\mu \mapsto \sfb_{L^2} (\mu, \rho)$ is $0$-convex with respect to $\cB_{\bar \mu}$ , since $\nabla^2 f^{(i)}[\rho]\succeq0$, and $\Wass_2^2(\cdot,\bar \mu)$ is 1-convex with respect to $\cB_{\bar \mu}$. The linear monotonicity condition \eqref{eq:linear_monotonicity} gives that $\sfb_{L^2}$ is $0$-interaction dissipative. Applying Theorem~\ref{thm:main_existence}, there exists a strong solution to \eqref{eq:mixed_metric}.
    \end{proof}

    \begin{corollary}[Existence of Equilibrium]\label{cor:existenceL2}
        Let $\rho \mapsto \sfb_{L^2}( \mu, \rho)$ be upper hemicontinuous with respect to a barycentric system 
        $\cB$  of linear interpolations $\mu_t=(1-t)\mu_0+t\mu_1$ (so that $\DDD$ is assumed to be convex under mixtures) and superlevels $\sigma^+(\mu)=\{\rho \in \DDD \ : \ \sfb_{L^2}(\mu,\rho) \ge 0\}$ be closed in $\DDD$ with a least one compact.
        Then the bifunction $\sfb_{L^2}$ has a $\rho^*$ such that $\sfb(\mu,\rho^*) \ge 0$ for all $\mu\in\DDD$. 
    \end{corollary}
    \begin{proof}
        In order to apply Theorem~\ref{thm:primal_soln}, we need that $\mu \mapsto \sfb_{L^2}(\mu,\rho)$ is jointly barycentrically quasi-convexlike with respect to 
        $\cB$.  Since $\mu \mapsto \sfb_{L^2}(\mu,\rho)$ is affine along linear interpolations, it is $0$-barycentrically convexlike with respect to $\cB$, hence also quasi-convexlike, and we apply Theorem~\ref{thm:primal_soln}.
    \end{proof}
    We remark that in order to use the linear monotonicity condition \eqref{eq:linear_monotonicity} to show existence of a \textit{unique} equilibrium point via the dual equilibrium problem, we would need $\mu \mapsto \sfb_{L^2}(\mu,\rho)$ to be $\kappa>0$ barycentrically convexlike with respect to some family $\cB$ in order to apply Theorem~\ref{thm:equilibrium2};  however, this fails along linear interpolations, since $\mu\mapsto\sfb_{L^2}(\mu,\rho)$ is affine, and along $\Wass_2$ generalized geodesics it requires the potentials $f^{(i)}[\rho]$ to be uniformly $\kappa$-convex.
\end{example}

\appendix

\section{Proof of Cauchy Estimate}\label{app:cauchy_estimate_proof}
In this section we present the proof of the Cauchy estimate,  Proposition~\ref{prop:cauchy_condition}; it relies on the discrete slope estimates of Proposition~\ref{le:apriori}.

 Recall that $\overline X_\tau(t)$ denotes the piecewise constant function
taking the value $X^n_\tau$ in $((n-1)\tau,n\tau]$, that
$\underline X_\tau(t)$ takes the value $X^{n-1}_\tau$ in $[(n-1)\tau,n\tau)$, and that $\overline X_\tau(0)=\underline X_\tau(0)=X^0_\tau$.
\begin{proof}[Proof of Proposition~\ref{prop:cauchy_condition}]
We now introduce the functions
\begin{equation*}
	\ell(t):=t-n\quad \text{if }n< t\le n+1,\quad 
	\ell_\tau(t):=\ell(t/\tau), \quad \ell_\tau(0):=0,
\end{equation*}
and for a discrete solution $X^n_{\tau_1}$ to \eqref{eq:D-SEVI2},
\begin{equation*}
	\zeta_{\tau_1}(s,Y):=(1-\ell_{\tau_1}(s))\dX(\underline X_{\tau_1}(s),Y)+
	\ell_{\tau_1}(s)\dX(\overline X_{\tau_1}(s),Y)
\end{equation*}
Similarly, for another discrete solution $Y^n_{\tau_2}$ to \eqref{eq:D-SEVI2}, we set
\begin{equation*}
	\omega_{\tau_2}(X,t):=(1-\ell_{\tau_2}(t))\dX(\ul Y_{\tau_2}(t),X)+
	\ell_{\tau_2}(t)\dX(\ol Y_{\tau_2}(t),X).
\end{equation*}
The functions $\zeta_{\tau_1}(s,Y)$ and $\omega_{\tau_2}(X,t)$ are continuous and piecewise linear; hence, they are locally Lipschitz. It immediately follows from  Lemma~\ref{le:simple-but-tricky} that
\begin{equation*}
	\partial_s \zeta_{\tau_1}(s,Y)\le \lsku Y{\overline X_{\tau_1}(s)}\,,\quad
	\partial_t \omega_{\tau_2}(X,t)\le \lsku X{\overline Y_{\tau_2}(t)}\,, \quad \text{for a.e. } t,s>0\,.
\end{equation*}
Choosing $Y:=\overline Y_{\tau_2}(t)$ in the first inequality and
$X:=\overline X_{\tau_1}(s)$ in the second one, using \eqref{eq:S-lambda1},
and extending $\overline X_{\tau_1}(s)\equiv X^0_{\tau_1}$
for $s<0$, we get, for a.e. $s, t$,
\begin{equation}\label{eq:discretization_error_estimate}
	\partial_s \zeta_{\tau_1}(s,\overline Y_{\tau_2}(t))+
	\partial_t \omega_{\tau_2}(\overline X_{\tau_1}(s),t)\le f(s,t)\,,
\end{equation}
where
\begin{equation}
    f(s,t) = \begin{cases}
        \lsku {X^0_{\tau_1}}{\overline Y_{\tau_2}(t)}
	& \quad\text{if }s<0,\\
	 \lambda \dX(\overline X_{\tau_1}(s),\overline Y_{\tau_2}(t)) & \quad \text{ otherwise}
    \end{cases}\,, \quad t\ge 0 \,.
\end{equation}
We will integrate over the strip $Q_{0,T}^\eps$ as in \cite{Nochetto-Savare06}:
\begin{align*}
    Q_{0,T}^\eps := \big\{ (s,t)\in\R^2 \ : \ 0\le t \le T, \quad t-\eps \le s \le t \big\}\,.
\end{align*}
The integral of \eqref{eq:discretization_error_estimate} over $Q_{0,T}^\eps$,
\begin{align}\label{eq:discretization_err_integral}
   \iint_{Q_{0,T}^\eps}  \text{div}_{s,t} \left( \begin{bmatrix}
       \zeta_{\tau_1}(s,\overline Y_{\tau_2}(t)) \\
	\omega_{\tau_2}(\overline X_{\tau_1}(s),t)
   \end{bmatrix} \right) \, \d s \d t \le \iint_{Q_{0,T}^\eps} f(s,t)\, \d s \d t:= a_0\,,
\end{align}
will require five estimates; one from the right hand side, and four over the boundary $\partial Q_{0,T}^\eps$ from the divergence theorem. In Figure~\ref{fig:strip}, the edges of $Q_{0,T}^\eps$ are labeled with the corresponding boundary integrals.  Applying the divergence theorem to the lower bound, let 
\begin{align*}
       \iint_{Q_{0,T}^\eps}  \text{div}_{s,t} \left( \begin{bmatrix}
       \zeta_{\tau_1}(s,\overline Y_{\tau_2}(t)) \\
	\omega_{\tau_2}(\overline X_{\tau_1}(s),t)
   \end{bmatrix} \right) \, \d s \d t  =\sum_{i=1}^4 a_i\,, \quad 
   a_1:=\int_{T-\eps}^T \omega_{\tau_2}(\overline X_{\tau_1}(s),T)\,\d s\,, \\ 
   a_2:= -\int_{-\eps}^0 \omega_{\tau_2}(\overline X_{\tau_1}(s),0)\,\d s \,,  \quad
   a_3:=
   \int_0^T\Big(\zeta_{\tau_1}(t,\overline Y_{\tau_2}(t))-\omega_{\tau_2}(\overline X_{\tau_1}(t),t)\Big)\,\d t\,,\\
   a_4 := 
   \int_0^T\Big(\omega_{\tau_2}(\overline X_{\tau_1}(t-\eps),t)-\zeta_{\tau_1}(t-\eps,\overline Y_{\tau_2}(t))
	\Big)\,\d t\,.
\end{align*} 
Therefore \eqref{eq:discretization_err_integral} can be written as
\begin{align}\label{eq:discretization_err_integral_coeffs}
    \sum_{i=1}^4 a_i \le a_0\,.
\end{align}

\begin{figure}
    \centering

\begin{tikzpicture}[>=latex, scale=1.1]

  \def\T{5}
  \def\eps{1}

  \coordinate (A) at (-\eps, 0);
  \coordinate (B) at (0, 0);
  \coordinate (C) at (\T, \T);
  \coordinate (D) at (\T-\eps, \T);

  \fill[gray!15] (A) -- (B) -- (C) -- (D) -- cycle;

  \draw[thick] (A) -- (B);
  \draw[thick] (B) -- (C);
  \draw[thick] (D) -- (C);
  \draw[thick] (A) -- (D);

  \draw[dashed] (0, \T) -- (\T, \T);   
  \draw[dashed] (\T, 0) -- (\T, \T);   

  \draw[->] (-\eps-0.4, 0) -- (\T+0.7, 0) node[below] {$s$};
  \draw[->] (0, -0.3)      -- (0, \T+0.7)  node[left]  {$t$};

  \node[left]  at (0, \T) {$t = T$};
  \node[below] at (\T, 0) {$s = T$};

  \draw[<->] (-\eps, -0.4) -- (0, -0.4)
    node[midway, below] {$\varepsilon$};

  \draw[<->] (\T-\eps, \T+0.4) -- (\T, \T+0.4)
    node[midway, above] {$\varepsilon$};


  \node[above] at ($0.5*(D)+0.5*(C)+(0,-0.1)$) {$a_1$};

  \node[below] at ($0.5*(A)+0.5*(B)+(0,0)$) {$a_2$};

  \node[right] at ($(B)!0.22!(C)+(0.15,-0.15)$) {$a_3$};

  \node[left]  at ($(A)!0.22!(D)+(-0.15, 0.15)$) {$a_4$};


  \node[left] at ($(A)!0.62!(D)+(-0.1, 0.2)$) {$s = t - \varepsilon$};

  \node[right] at ($(B)!0.38!(C)+(0.55,0.25)$) {$t = s$};

  \node at ($(A)!0.5!(C)+(0.0, 0)$) {$Q^{\varepsilon}_{0,T}$};

\end{tikzpicture}
    \caption{The strip $Q_{0,T}^\eps$.}
    \label{fig:strip}
\end{figure}
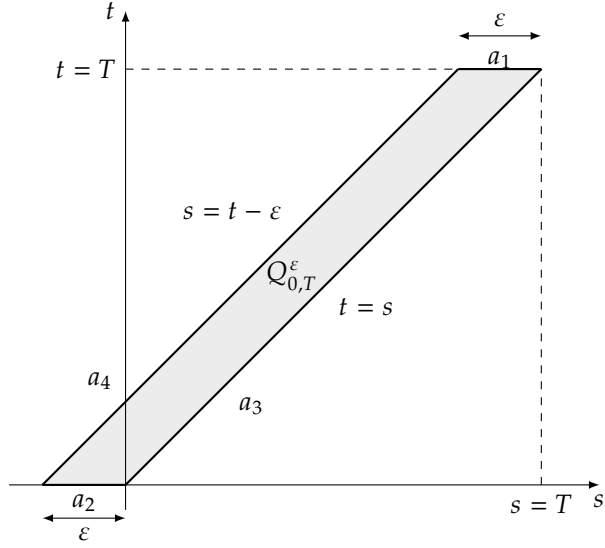
We compute estimates for $a_1,\dots, a_4$ from below, and for $a_0$ from above, to obtain the Cauchy estimate. We begin with $a_0$.
By Definition~\ref{def:bounded_Delta} with $x\in \domainslope{\cvx}$, we have 
    \begin{align}\label{eq:slope_estimate}
        \skslope{\cvx}x := \sup_{y'\in \DDD} \frac{{\lskd{y'}{x}}_-}{\dX(x,y')} \ge  -\lsku yx  \quad \text{for all } y\in\DDD\,.
    \end{align}
Combining this estimate with \eqref{eq:S-lambda1},
we construct the following bound on $\lsku {X^0_{\tau_1}}{\overline Y_{\tau_2}(t)}$:
\begin{align*}
    \lsku {X^0_{\tau_1}}{\overline Y_{\tau_2}(t)} \le \Delta^\kappa(X_{\tau_1}^0)  + \lambda \dX(X_{\tau_1}^0,\ol Y_{\tau_2}(t))\,.
\end{align*}
Applying this to the definition of $a_0$, and evaluating separately the part of $Q_{0,T}^\eps$ where $s<0$, we have
\begin{align*}
    a_0 &\le \int_{0}^\eps \int_{t-\eps}^{0} \skslope{\cvx}{X_{\tau_1}^0} \d s \d t + \lambda \iint_{Q_{0,T}^\eps} \dX(\overline X_{\tau_1}(s),\overline Y_{\tau_2}(t)) \d s \d t \\
    &=\frac{\eps^2}2 \skslope{\cvx}{X_{\tau_1}^0} + \lambda \iint_{Q_{0,T}^\eps} \dX(\overline X_{\tau_1}(s),\overline Y_{\tau_2}(t)) \d s \d t \,.
\end{align*}
The integral over the distance term has the upper bound
\begin{align*}
    \iint_{Q_{0,T}^\eps} \dX(\ol X_{\tau_1}(s),& \ol Y_{\tau_2}(t) ) \d s \d t 
    \le  \iint_{Q_{0,T}^\eps} \dX(\ol X_{\tau_1}(s),\ol X_{\tau_1}(t))  \d s \d t +\eps \int_0^T\dX(\ol X_{\tau_1}(t),\ol Y_{\tau_2}(t) ) \d t
\end{align*}
Let $n, m\in\N$ such that $t\in ((n-1)\tau_1,n\tau_1]$ and $s\in ((m-1)\tau_1,m\tau_1]$. Since each $(s,t)\in Q_{0,T}^\eps$ satisfies $t-\eps\le s\le t$, we conclude that $0\le n-m \le \eps/\tau_1 + 1$. Using the estimates in Proposition~\ref{le:apriori} and recalling $\tilde\tau:=\big(\frac1\tau-\lambda\big)^{-1}$, we have
\begin{align*}
    \dX(\ol X_{\tau_1}(s),\ol X_{\tau_1}(t))
    &\le  \dX(\ol X_{\tau_1}(s),X_{\tau_1}^m)+ \dX(X_{\tau_1}^m,X_{\tau_1}^n)+ \dX(X_{\tau_1}^n,\ol X_{\tau_1}(t)) =  \dX(X_{\tau_1}^m,X_{\tau_1}^n)\\
    &\le \sum_{i=m+1}^{n}  \dX(X_{\tau_1}^{i-1}, X_{\tau_1}^{i})
    \le \tilde \tau_1 \sum_{i=m+1}^{n}  \Delta^\cvx(X_{\tau_1}^{i-1})
    \le \tilde \tau_1 (n-m)  c_0 \Delta^\cvx(X_{\tau_1}^0)\\
    &\le \tilde \tau_1 (\eps/\tau_1 + 1)  c_0 \Delta^\cvx(X_{\tau_1}^0)\,.
\end{align*}
Combining with the previous estimates, and using that the volume of $Q_{0,T}^\eps$ is $\eps T$, we conclude that
\begin{align}\label{eq:f_upper_bd_Cauchy}
    a_0 
    &\le\frac{\eps^2}2 \skslope{\cvx}{X_{\tau_1}^0} + \lambda \tilde \tau_1(\eps/\tau_1+1)  \eps T  c_0 \Delta^\cvx(X_{\tau_1}^0)
    + \eps \lambda  \int_0^T\dX(\ol X_{\tau_1}(t),\ol Y_{\tau_2}(t) ) \d t \,.
\end{align}
Next, to bound $a_1$ from below, we use the triangle inequality twice to obtain
\begin{align*}
    \dX(\ol X_{\tau_1}(T),\ol Y_{\tau_2}(T)) &\le \dX(\ol X_{\tau_1}(s),\ol Y_{\tau_2}(T)) + \dX(\ol X_{\tau_1}(T),\ol X_{\tau_1}(s))
    \\ & \le\ell_{\tau_2}(T)\dX(\ol X_{\tau_1}(s),\ol Y_{\tau_2}(T))+ (1-\ell_{\tau_2}(T)) \left( \dX(\ol X_{\tau_1}(s), \ul Y_{\tau_2}(T))\right.\\
    &\left.\quad + \dX(\ul Y_{\tau_2}(T),\ol Y_{\tau_2}(T)) \right) + \dX(\ol X_{\tau_1}(T),\ol X_{\tau_1}(s)) \\
    & \le \omega_{\tau_2}(\ol X_{\tau_1}(s),T)+ \dX(\ol X_{\tau_1}(T),\ol X_{\tau_1}(s)) + \dX(\ol Y_{\tau_2}(T),\ul Y_{\tau_2}(T)) \,.
\end{align*}
Hence,
\begin{align*}
    a_1=\int_{T-\eps}^T \omega_{\tau_2}(\ol X_{\tau_1}(s),T) \d s
    \ge \eps \dX(\ol X_{\tau_1}(T),\ol Y_{\tau_2}(T)) - \int_{T-\eps}^T\dX(\ol X_{\tau_1}(T),\ol X_{\tau_1}(s))\d s - \eps\dX(\ol Y_{\tau_2}(T),\ul Y_{\tau_2}(T))
\end{align*}
To lower bound the two terms on the right hand side, we note that, again by Proposition~\ref{le:apriori},
\begin{align*}
    \int_{T-\eps}^T \dX(\ol X_{\tau_1}(T),\ol X_{\tau_1}(s)) \d s  \le \eps \tilde \tau_1 \left(\frac{\eps}{\tau_1}+1\right) c_0 \Delta^\kappa (X_{\tau_1}^0)\,,
\end{align*}
and for $n:= \lceil T/\tau_2\rceil$, we have
\begin{align*}
    \dX(\ol Y_{\tau_2}(T),\ul Y_{\tau_2}(T)) \le \dX(Y_{\tau_2}^{n},Y_{\tau_2}^{n-1}\nc)\le \tilde \tau_2 c_0 \Delta^\kappa (Y_{\tau_2}^0\nc)\,.
\end{align*}
Thus, we conclude
\begin{align*}
   a_1\ge  \eps\left[\dX(\ol X_{\tau_1}(T),\ol Y_{\tau_2}(T)) - c_0(\tilde \tau_1\left(\frac{\eps}{\tau_1}+1\right)+\tilde \tau_2) \left(\Delta^\kappa (X_{\tau_1}^0) +\Delta^\kappa (Y_{\tau_2}^0)\right) \right] \,.
\end{align*}
For $a_2$, we get the explicit expression
\begin{equation}
	a_2= -\int_{-\eps}^0 \omega_{\tau_2}(\overline X_{\tau_1}(s),0)\,\d s = -\int_{-\eps}^0 \dX(\overline Y_{\tau_2}(0),\overline X_{\tau_1}(0)\nc ) \d s =
	-\eps \dX(X^0_{\tau_1},Y^0_{\tau_2})\,.
\end{equation}
Next we will bound $a_3$ and $a_4$. 
We set
\begin{equation}
	\delta_{X,\tau_1}(s):=\sup_{Y\in\DDD} \Big|\zeta_{\tau_1}(s,Y)-\dX(\overline X_{\tau_1}(s),Y)\Big|,\quad
	\delta_{Y,{\tau_2}}(t):=\sup_{X\in\DDD} \Big|\omega_{\tau_2}(X,t)-\dX(\overline Y_{\tau_2}(t),X)\Big|
\end{equation}
observing that 
\begin{align*}
    &\ \delta_{Y,\tau_2}(t) = \sup_X \Big|(1-\ell_{\tau_2}(t))\dX(\ul Y_{\tau_2}(t),X)+
	\ell_{\tau_2}(t)\dX(\ol Y_{\tau_2}(t),X)-\dX(\overline Y_{\tau_2}(t),X)\Big| \\
    &= \sup_X \Big|(1-\ell_{\tau_2}(t))\big[\dX(\ul Y_{\tau_2}(t),X)-\dX(\overline Y_{\tau_2}(t),X)\big]\Big|\\
    &=(1-\ell_{\tau_2}(t)) \sup_X \Big|\dX(\ul Y_{\tau_2}(t),X)-\dX(\overline Y_{\tau_2}(t),X)\Big| \le (1-\ell_{\tau_2}(t)) \dX(\ul Y_{\tau_2}(t),\overline Y_{\tau_2}(t))\,,
\end{align*}
and similarly, $\delta_{X,\tau_1}(s)\le (1-\ell_{\tau_1}(s)) \dX(\ul X_{\tau_1}(s),\overline X_{\tau_1}(s))$. Using again Proposition~\ref{le:apriori}, we have
\begin{equation*}
\begin{split}
    	\delta_{X,\tau_1}(s)\le(1-\ell_{\tau_1}(s)) \tilde \tau_1 c_0 \Delta^{\cvx} (X^0_{\tau_1}),\\
	\delta_{Y,\tau_2}(t)\le (1-\ell_{\tau_2}(t)) \tilde \tau_2 c_0\Delta^{\cvx} (Y^0_{\tau_2})\,.
\end{split}
\end{equation*}
Define $h_\tau(t)=\frac t2+\frac \tau2\ell_\tau(t)(1-\ell_\tau(t))$. Then $h_\tau'(t)=1-\ell_\tau(t)$ for a.e. $t$, and $h_\tau(0)=0$. Hence $\int_0^T(1-\ell_\tau(t))=h_\tau(T)$, and so
\begin{align*}
	\rmI_{X,{\tau_1}} :={}&\int_0^T \delta_{X,\tau_1}(s)\,\d s \le \tilde \tau_1 c_0 h_{\tau_1}(T) \Delta^{\cvx} (X^0_{\tau_1}),\\
	\rmI_{Y,\tau_2}:=&
	\int_0^T \delta_{Y,{\tau_2}}(s)\,\d s\le 
	\tilde \tau_2 c_0 h_{\tau_2}(T)\Delta^{\cvx} (Y^0_{\tau_2}) \,.
\end{align*}
Therefore
\begin{align*}
	a_3 &=
    -\int_0^T\Big(\omega_{\tau_2}(\overline X_{\tau_1}(t),t)-
	\dX(\overline X_{\tau_1}(t),\overline Y_{\tau_2}(t))+
		\dX(\overline X_{\tau_1}(t),\overline Y_{\tau_2}(t))-
	\zeta_{\tau_1}(t,\overline Y_{\tau_2}(t))\Big)\,\d t	
	\\&\ge 
	-
    \int_0^T\Big(\delta_{X,{\tau_1}}(t)+
	\delta_{Y,{\tau_2}}(t)\Big)\,\d t=
    -
     \left(\rmI_{X,{\tau_1}}+\rmI_{Y,{\tau_2}} \right)\,.
\end{align*}
A similar calculation for $a_4$ gives
\begin{align*}
	a_4 &=
	-
    \int_0^T\Big(\zeta_{\tau_1}(t-\eps,\overline Y_{\tau_2}(t))-
	\dX(\overline X_{\tau_1}(t-\eps),\overline Y_{\tau_2}(t))+
		\dX(\overline X_{\tau_1}(t-\eps),\overline Y_{\tau_2}(t))
        -
	\omega_{\tau_2}(\overline X_{\tau_1}(t-\eps),t)\Big)\,\d t	
	\\&\ge -
	\int_0^T\Big(\delta_{X,{\tau_1}}(t-\eps)+
	\delta_{Y,{\tau_2}}(t)\Big)\,\d t=
    -
    \left(\rmI_{X,\tau_1}+\rmI_{Y,{\tau_2}} \right)\,.
\end{align*}
Combining the estimates for the five terms results in
\begin{align*}
&\eps\left[\dX(\ol X_{\tau_1}(T),\ol Y_{\tau_2}(T)) - c_0\left(\tilde \tau_1 \left(\frac{\eps}{\tau_1}+1\right)+\tilde \tau_2\right) \left(\Delta^\kappa (X_{\tau_1}^0) +\Delta^\kappa (Y_{\tau_2}^0)\right) \right] 
-\eps \dX(X^0_{\tau_1},Y^0_{\tau_2})
-
2 
\left(\rmI_{X,\tau_1}+\rmI_{Y,{\tau_2}} \right)\\
&\le
    \frac{\eps^2}2 \skslope{\cvx}{X_{\tau_1}^0} + \lambda \tilde \tau_1(\eps/\tau_1+1)  \eps T  c_0 \Delta^\cvx(X_{\tau_1}^0)
    + \eps \lambda  \int_0^T\dX(\ol X_{\tau_1}(t),\ol Y_{\tau_2}(t) ) \d t \,.
\end{align*}
Rearranging, dividing by $\eps$, and defining $\tilde\tau:=\max\{\tilde\tau_1, \tilde\tau_2\}$,  gives
\begin{align*}
\dX(\ol X_{\tau_1}(T),\ol Y_{\tau_2}(T)) \le &c_0\left(\tilde \tau_1\left(\frac{\eps}{\tau_1}+1\right)+\tilde \tau_2\right) \left(\Delta^\kappa (X_{\tau_1}^0) +\Delta^\kappa (Y_{\tau_2}^0)\right) 
+ \dX(X^0_{\tau_1},Y^0_{\tau_2})
+\frac{ 2}{\eps}\left(\rmI_{X,\tau_1}+\rmI_{Y,{\tau_2}} \right)\\
&+
    \frac{\eps}2 \skslope{\cvx}{X_{\tau_1}^0} + \lambda \tilde \tau_1(\eps/\tau_1+1)   T  c_0 \Delta^\cvx(X_{\tau_1}^0)
    +  \lambda  \int_0^T\dX(\ol X_{\tau_1}(t),\ol Y_{\tau_2}(t) ) \d t\\
 \le &c_0\left(\tilde \tau_1\left(\frac{\eps}{\tau_1}+1\right)+\tilde \tau_2\right) \left(\Delta^\kappa (X_{\tau_1}^0) +\Delta^\kappa (Y_{\tau_2}^0)\right) 
+ \dX(X^0_{\tau_1},Y^0_{\tau_2})\\
&+\frac{2}{\eps}(\tilde \tau_1 +\tilde \tau_2)c_0 (h_{\tau_1}(T)+h_{\tau_2}(T))( \Delta^{\cvx} (X^0_{\tau_1} +\Delta^{\cvx} (Y^0_{\tau_2}))\\
&+
    \frac{\eps}2 \skslope{\cvx}{X_{\tau_1}^0} + \lambda \tilde \tau_1(\eps/\tau_1+1)  T  c_0 \Delta^\cvx(X_{\tau_1}^0)
    +  \lambda  \int_0^T\dX(\ol X_{\tau_1}(t),\ol Y_{\tau_2}(t) ) \d t\\
\le
& c_T \tilde \tau \left(2+\frac{(h_{\tau_1}(T)+h_{\tau_2}(T))}{\eps} +\frac{\eps}{\tilde \tau} + \frac{\eps}{\tau_1}\right)
\left(\Delta^\kappa (X_{\tau_1}^0) +\Delta^\kappa (Y_{\tau_2}^0)\right)
\\&+\dX(X^0_{\tau_1},Y^0_{\tau_2})
+  \lambda  \int_0^T\dX(\ol X_{\tau_1}(t),\ol Y_{\tau_2}(t) ) \d t
\end{align*}
for some constant $c_T$ depending on $\tau_o, T, \lambda$ only. 
Optimizing w.r.t.~$\eps$, we choose $$\eps=(h_{\tau_1}(T)+h_{\tau_2}(T))^{1/2}\left(\frac{\tilde \tau}{1+\tilde \tau/\tau_1}\right)^{1/2}$$ to obtain
\begin{align*}
\dX(\ol X_{\tau_1}(T),\ol Y_{\tau_2}(T))
\le\dX(X^0_{\tau_1},Y^0_{\tau_2})+
 a_\tau(T)+  \lambda  \int_0^T\dX(\ol X_{\tau_1}(t),\ol Y_{\tau_2}(t) ) \d t\,,
\end{align*}
where
\begin{align*}
  a_\tau(T):= 2c_T \sqrt{\tilde\tau} \left[\sqrt{\tilde\tau} +
  \sqrt{\frac{\tau_1 + \tilde \tau}{\tau_1 }(h_{\tau_1}(T)+h_{\tau_2}(T))}
  \right]
\left(\Delta^\kappa (X_{\tau_1}^0) +\Delta^\kappa (Y_{\tau_2}^0)\right)\,.
\end{align*} 
Writing $g(t):=\dX(\ol X_{\tau_1}(t),\ol Y_{\tau_2}(t))$, the above estimate can be written as 
\begin{align*}
g(T)
\le g(0)+ a_\tau(T)+  \lambda  \int_0^T g(t) \d t\,.
\end{align*}
Applying Gr\"onwall's inequality, we conclude
\begin{align*}
    g(T)\le  e^{\lambda T} g(0) + C_\tau(T)\,,\qquad
    C_\tau(T):= a_\tau(T) +\lambda \int_0^T a_\tau(t) e^{\lambda(T-t)}\d t\,.
\end{align*}
Finally, we investigate the behavior of the upper bound as $\tau_1, \tau_2\to 0$. Notice that $\tilde\tau_i=\mathcal O(\tau_i)$ as $\tau_i\to0$. Since $h_{\tau_1}(T)+h_{\tau_2}(T)\to T$ as $\tau_1, \tau_2\to 0$, we have
\begin{align*}
    \sqrt{\tilde\tau}\left[\sqrt{\tilde\tau} +
  \sqrt{\frac{\tau_1 + \tilde \tau}{\tau_1 }(h_{\tau_1}(t)+h_{\tau_2}(t))}  
  \right] 
  &\sim 
  \sqrt{T \max\{\tau_1,\tau_2\}\left(1+\frac{\max\{\tau_1,\tau_2\}}{\tau_1}\right)} \\
 &\sim
  \begin{cases}
   \sqrt{2T\tau_1}   &\text{ if } \tau_1\ge\tau_2\,,\\
   \sqrt{2T\tau_2^2/\tau_1}   &\text{ if } \tau_1<\tau_2\,.
  \end{cases}
\end{align*}
Hence, as long as $\tau_2^2/\tau_1\to 0$, we have $a_\tau(t)=\mathcal O\left(\max\{\tau_1,\tau_2\}/\sqrt{\tau_1}\right)\to 0$ as $\tau_1, \tau_2\to 0$ uniformly for $t\in [0,T]$. We conclude that $C_\tau(T)\to0$ at the same rate. And so, 
$$\dX(\ol X_{\tau_1}(T),\ol Y_{\tau_2}(T))- e^{\lambda T}\dX(X^0_{\tau_1},Y^0_{\tau_2})=\mathcal O\left(\max\{\tau_1,\tau_2\}/\sqrt{\tau_1}\right)\to 0$$ 
goes to zero for $\tau_1,\tau_2 \to 0$ with $\tau_2^2/\tau_1\to 0$.
\end{proof}

\section*{Acknowledgments}
This collaboration was initiated at the conference "Aggregation-Diffusion Equations \& Collective Behavior:, Analysis, Numerics and Applications" (April 2024) at CIRM in Marseille, France, with continued discussion at the MFO workshop "Flows on Measure Spaces and Applications in Machine Learning" (March 2026) in Oberwolfach, Germany, with L.C.'s attendance supported by the US Junior Oberwolfach Fellows program, NSF award DMS-2230648. The authors are grateful for discussions with Guillaume Wang regarding Example~\ref{ex:L2_W2_monotonicity}.
L.C. is supported by NSF MSPRF 2602677, and was partially supported by a 2026 Carver Mead New Adventures Fund, by a US Air Force NDSEG Fellowship, and a PIMCO Fellowship. 
F.H. is supported by start-up funds at the California Institute of Technology and by NSF CAREER Award 2340762. 
 G.S.\ acknowledges support from IMATI (CNR), Pavia.
G.S.\ has been partially supported by funding from the European
Research Council (ERC) under the European Union's Horizon Europe
research and innovation programme (grant agreement No.\ 101200514,
project acronym OPTiMiSE).
Views and opinions expressed are however
those of the author(s) only and do not necessarily reflect those of the European Union or the European Research Council Executive Agency. Neither the European Union nor the granting authority can be held
responsible for them.

\section*{{Declaration on the use of AI tools}}
{This project was started in 2024 and its main results were obtained in 2025. During the final revision of this work, the authors used Claude (Anthropic) in a limited, auxiliary capacity: to suggest additional relevant citations which were then read and manually added to the introduction, to improve the language and readability of the authors' own text, to help word a few brief routine passages from material supplied by the authors, and to check arguments for consistency and flag possible gaps, which were then examined, corrected and verified by the authors.
All mathematical ideas, results, and proofs are the authors' own; the authors take full responsibility for the content of the manuscript.
}

\bibliographystyle{siam}
\bibliography{bibliografia2015}

\end{document}